\documentclass[11pt, reqno]{amsart}
\usepackage[a4paper, total={6in, 8in}]{geometry}
\usepackage{xcolor}
\usepackage{amsmath, amsfonts, amssymb, amsbsy, bigstrut, graphicx, enumerate,  upref, longtable, comment, csquotes, booktabs, array, verbatim}
\usepackage{bbm}
\usepackage{mathtools}
\usepackage{tabularx}
\usepackage{tikz}
\usepackage{float}
\usepackage{subcaption}
\usepackage{pstricks}
\usepackage{lscape}
\usepackage[breaklinks]{hyperref}

\usepackage[
  backend=biber,
  style=authoryear,
  giveninits=true,
  uniquename=init,
  natbib
]{biblatex}
\allowdisplaybreaks
\usepackage[toc,page]{appendix}
\usepackage{siunitx}
\newcommand{\abs}[1]{\lvert #1\rvert}

\newcommand{\bone}{\mathbbm{1}}
\newcommand{\cR}{\mathcal R}

\newtheorem{thm}{Theorem}[section]
\newtheorem{lmm}[thm]{Lemma}
\newtheorem{cor}[thm]{Corollary}
\newtheorem{prop}[thm]{Proposition}

\theoremstyle{definition}
\newtheorem{remark}[thm]{Remark}

\newtheorem{assump}[thm]{Assumption}

\newcommand{\cov}{\mathrm{Cov}}

\newcommand{\ee}{\mathbb{E}}

\newcommand{\pp}{\mathbb{P}}

\newcommand{\rr}{\mathbb{R}}

\newcommand{\var}{\mathrm{Var}}

\numberwithin{equation}{section}

\begin{document}

\title{
Kernel estimation of Chatterjee's dependence coefficient}
\author{Mona Azadkia}
\address{Department of Statistics, London School of Economics and Political Science}
\email{m.azadkia@lse.ac.uk}

\author{Holger Dette}
\address{Department of Mathematics, Ruhr-University Bochum}
\email{holger.dette@ruhr-uni-bochum.de}
\begin{abstract}

A dependence coefficient suggested  by  \citet{Chatterjee21}   
has attracted considerable attention in recent years. It takes values between $0$ and $1$, vanishes if and
only if the corresponding random variables $X$ and $Y$ are independent,
and equals $1$ if and only if $Y$ is almost surely a measurable function
of $X$.
The coefficient coincides, for distributions with continuous marginals, with a dependence  coefficient introduced in \citet{dette2013copula}. Several estimators of this measure have been proposed, including the
kernel-type estimator of \citet{dette2013copula}, Chatterjee's rank
correlation, and its nearest-neighbour modification by
\citet{lin2021boosting}. In this paper, we investigate the asymptotic properties and local power of the kernel-type estimator.

We first establish a non-degenerate central limit theorem under
independence: after suitable centring, the estimator is asymptotically
normal at the rate $\sqrt{n/h_1}$, where $h_1\to0$ is a bandwidth.
We then consider Gaussian-copula alternatives with correlation
parameter $\rho_n\to0$. By deriving  a second-order expansion of the
mean, establishing variance stability and proving a central limit
theorem under local alternatives, we obtain an explicit limiting
power function and identify $(h_1/n)^{1/4}$ as the detection  boundary.
Under suitable bandwidth conditions, 
the detection boundary can attain rates arbitrarily close to $n^{-3/8}$. The exponent $3/8$ exceeds
the exponent $1/4$ for the test based on Chatterjee's original
rank correlation and the exponent approaching $5/16$ for the
modified rank correlation of \citet{lin2021boosting} within
the range of their central limit theorem.

The proofs combine refined results for double-indexed permutation
statistics with Gaussian rank-score expansions and Hermite-chaos
projections. A simulation study illustrates the normal approximation
and the resulting power properties.
\end{abstract}

\maketitle

\textit{Keywords and Phrases: Dependence measure; Independence testing; Copulas;  Hermite chaos.}

\section{Introduction}
\label{sec1}

A fundamental problem in statistics is the quantification of the
dependence between two real-valued random variables $X$ and $Y$ based on a sample
of independent and identically distributed observations
$(X_1,Y_1),\ldots,(X_n,Y_n)$.
Classical dependence measures such as Pearson's and Spearman's correlation,
Kendall's $\tau$ and Gini's $\gamma$ take values in $[0,1]$ or $[-1,1]$ and
vanish under independence; moreover, they are highly sensitive to linear and
monotone relationships. However, it is well known that these measures may fail
to detect non-monotone forms of dependence, even in situations where
$Y=f(X)$ for some deterministic function $f$. In addition, a vanishing
classical dependence measure typically only implies independence under
restrictive structural assumptions, such as joint normality.

Motivated by these limitations, numerous alternative dependence measures
have been proposed in the literature, mainly for testing the hypothesis of
independence (see \citet{blum1961,rosenblatt1975,CSORGO1985} for some early
references and \citet{szekely2007,gretton2008,bergsma2014,zhang2019} for
more recent approaches relying on distance- and kernel-based methods).
As pointed out by \citet{Chatterjee21}, the ability to detect
dependence does not necessarily provide an easily interpretable
measure of its strength.

In recent years, considerable attention has been devoted to
dependence measures $\mu=\mu(X,Y)$ taking values in the interval  $[0,1]$,
vanishing if and only if $X$ and $Y$ are independent, and attaining
the value $1$ if and only if $Y$ is almost surely a measurable
function of $X$, provided that $Y$ is non-degenerate.
An early work in this direction is the paper by
\citet{dette2013copula}, who defined, for distributions with
continuous marginals, the association measure
\begin{align}
    \label{hd1}
    r=r(X,Y)
    =6\int_{[0,1]^2}(\partial_1 C(u,v))^2\,du\,dv-2,
\end{align}
which \citet{shidrtonhan2021a} called the
\textit{Dette--Siburg--Stoimenov correlation coefficient}.
Here $C$ denotes the copula of $(X,Y)$ and $\partial_1 C$ its partial
derivative with respect to the first coordinate.
In the same paper, the authors proposed a kernel-type estimator of this
measure and investigated its asymptotic behaviour at the $\sqrt{n}$ rate.
A different estimator was introduced by \citet{Chatterjee21}, who considered
the dependence coefficient
\begin{align}
    \label{hd2}
    \xi=\xi(X,Y)
    =
    \frac{
    \int \var
    \bigl(\ee[\bone(Y\geq y)\mid X]\bigr)\,dP_Y(y)
    }{
    \int \var(\bone(Y\geq y))\,dP_Y(y)
    }.
\end{align}
For continuous distributions, the two measures \eqref{hd1} and \eqref{hd2}
coincide. Chatterjee's estimator, commonly referred to as
\textit{Chatterjee's rank correlation}, is particularly attractive because
of its simple rank-based construction and computational efficiency.
It satisfies a central limit theorem at the $\sqrt{n}$ rate
(see \citealp{Chatterjee21}; \citealp{deb2020measuring};
\citealp{lin2025limittheoremschatterjeesrank};
\citealp{kroll2025asymptoticnormalitychatterjeesrank}).
However, its performance in testing independence against local alternatives
is limited. In particular, the detection boundary for Gaussian correlation
alternatives is of order $n^{-1/4}$.

To improve the power of tests based on estimates of  Chatterjee's dependence coefficient,
\citet{lin2021boosting} proposed a modification incorporating an increasing
number $M$ of right nearest neighbours.
Their approach substantially improves the detection rates against Gaussian
correlation alternatives. In particular, within the range
$M=o(n^{1/4})$ for which they establish a central limit theorem under
independence, the polynomial detection exponent can be made arbitrarily
close to $5/16$. By choosing $M$ substantially larger and using a
simulation-based test, their results also yield detection rates arbitrarily
close to the parametric rate $n^{-1/2}$.
These developments demonstrate that the choice of estimator can have a
substantial impact on the power of independence tests, even when the
underlying population dependence measure is the same.

In the present paper, we reconsider the kernel-type estimator of
\citet{dette2013copula} of Chatterjee's dependence coefficient and investigate its asymptotic behaviour under
independence and local alternatives.
Our first contribution is a non-degenerate central limit theorem under
independence. As observed by \citet{shidrtonhan2021a}, the asymptotic
variance of the original $\sqrt{n}$ limit vanishes in this case.
We show that, after appropriate centring, the kernel estimator is
asymptotically normal with the faster rate $\sqrt{n/h_1}$, where
$h_1\to0$ denotes a bandwidth.
The proof exploits a representation of the statistic in terms of a uniform
random permutation and requires a refined analysis of double-indexed
permutation statistics. An additional difficulty is caused by boundary
effects of the kernel estimator, which contribute at the leading order
and therefore must be retained in the centring term.

Our second contribution is a detailed analysis of the power of the resulting
independence test under Gaussian-copula alternatives with correlation
parameter $\rho_n\to0$.
We derive a second-order expansion of the mean of the kernel statistic,
establish the stability of its variance on the relevant local scale,
and prove a central limit theorem under these alternatives.
These results identify
\[
    \rho_n\asymp\left(\frac{h_1}{n}\right)^{1/4}
\]
as the Gaussian-copula detection boundary and yield an explicit expression
for the limiting power function.
Moreover, we prove consistency for correlations above this rate
within a suitable range of shrinking alternatives. The analysis under local alternatives requires substantially different
arguments from those used under independence, as the uniform-permutation
structure available under the null is lost, the leading mean shift appears
only at second order, and the variance has to be controlled uniformly on
the relevant local scale. We address these difficulties using Gaussian
rank-score expansions and Hermite-chaos projections.

We use these findings for a theoretical comparison with the tests
based on the estimators of \citet{Chatterjee21} and
\citet{lin2021boosting}.
Since $h_1\to0$, the detection boundary $(h_1/n)^{1/4}$ is smaller
than the $n^{-1/4}$ boundary of the test based on Chatterjee's
original rank correlation.
Moreover, under the bandwidth conditions imposed in this paper,
the polynomial detection exponent can be made arbitrarily close
to $3/8$. This goes beyond the exponent approaching $5/16$
obtained for the modified rank correlation of
\citet{lin2021boosting} within the range of their central limit
theorem. For substantially larger numbers of nearest neighbours,
their simulation-based test can achieve near-parametric detection
rates, but this regime is not covered by their central limit theorem.

\section{A kernel-type estimator and its  asymptotic null distribution}
\label{sec2}

Let $(X_1,Y_1),\ldots,(X_n,Y_n)$ be independent
and identically distributed random vectors with joint
distribution function $F$, continuous marginal distribution
functions $F_1$ and $F_2$, and copula $C$.
Let $\hat F_{n1}$ and $\hat F_{n2}$ denote the empirical distribution functions of the marginal samples $X_1,\dots,X_n$ and $Y_1,\dots,Y_n$, respectively.
\citet{dette2013copula}
proposed to estimate the partial derivative 
$\tau(u,v) ={\partial}_1
C(u,v) = \frac{\partial}{\partial u} C(u,v)$ of the copula  by the statistic 
\begin{align} \label{d2aa}
\hat\tau_n(u,v)
&= \frac{1}{nh_1}\sum_{i=1}^n
K\Bigl(\frac{u-\hat F_{n1}(X_i)}{h_1}\Bigr)
\bar K\Bigl(\frac{v-\hat F_{n2}(Y_i)}{h_2}\Bigr),
\end{align}
where $K$ is a kernel,
\[
\bar K(x) = \int_{-\infty}^{x} K(t)\,dt
\]
is the corresponding integrated kernel, and $h_1$ and $h_2$ are bandwidths. Throughout, we impose the following conditions on the kernel and the bandwidths.

\begin{assump}\label{asmp:kernel}
    The kernel $K$ is a symmetric, bounded, non-negative function supported on the interval $[-1,1]$ such that $\int K(t)\,dt = 1$. 
\end{assump}

\begin{assump}\label{asmp:band}
    The bandwidths $h_1$ and $h_2$ satisfy $h_1\rightarrow 0$ and $h_2\rightarrow 0$, and $nh_1^2\to\infty$ as $n\rightarrow\infty$. 
\end{assump}

It can be shown that (under suitable regularity conditions) $\hat\tau_n$ is a consistent estimator of $\partial_1 C$  and therefore \citet{dette2013copula} defined the estimator of the dependence measure \eqref{hd1} by 
\begin{align*}
\hat r_n = 6\hat\tau_n^2-2,
\end{align*}
where  
\begin{align}
     \label{d2} \hat\tau_n^2
= \int_0^1\int_0^1 \hat\tau_n^2(u,v)\,du\,dv
\end{align}
is the estimator of the integral 
$\tau^2 = \int_0^1\int_0^1 (\partial_1 C(u,v))^2\,du\,dv$. Note that under independence, that is $C(u,v) =uv$, we have $\tau^2= 1/3$, and therefore the coefficient $r$ in \eqref{hd1} vanishes in this case. From a computational perspective, the kernel estimator \eqref{d2} can be evaluated exactly in \(O(n\log n)\) operations for many commonly used piecewise-polynomial kernels, including the Epanechnikov and triangular kernels. Thus its computational complexity is of the same order as that of Chatterjee’s rank correlation; details are given in Section~\ref{app:epanechnikov-computation}.

The asymptotic properties of this estimator were investigated
by \citet{dette2013copula}, who developed a central limit theory
at the $\sqrt{n}$ rate. As pointed out by
\citet{shidrtonhan2021a}, the reported asymptotic variance
contains a sign error, and the corrected variance vanishes under
independence. We provide the completely corrected variance
formula in Appendix~\ref{app:corrected-variance}.
% of the online supplement. 
Under independence, a non-degenerate limit requires a finer
normalisation. Moreover, boundary effects introduce a
non-negligible bias, so that the estimator must be centred appropriately. Our first main result establishes the resulting
central limit theorem.

\begin{thm}\label{thm:taunCLT}
    Suppose that Assumptions~\ref{asmp:kernel} and~\ref{asmp:band} hold, and that $X$ and $Y$ have continuous marginal distributions and are independent. Then   
    \begin{align*}
        \sqrt{\frac{n}{h_1}}(\hat{r}_n-b_n)
        \stackrel{\mathcal{D}}{\rightarrow}\mathcal{N}(0,144\sigma_0^2),
    \end{align*}
    where 
    \begin{align}
    \nonumber 
        b_n
        &:= \frac{6}{n^2h_1^2}\sum_{i\neq j}\int_0^1
        K\Bigl(\frac{u-i/n}{h_1}\Bigr)
        K\Bigl(\frac{u-j/n}{h_1}\Bigr)\,du\\
        &\quad\times \frac{1}{n(n-1)}\sum_{i\neq j}\int_0^1
        \bar K\Bigl(\frac{v-i/n}{h_2}\Bigr)
        \bar K\Bigl(\frac{v-j/n}{h_2}\Bigr)\,dv
                \label{eq:bn}
                \\
        &\quad+\frac{6}{n^2h_1^2}
        \Biggl(\sum_{i=1}^n\int_0^1
        K^2\Bigl(\frac{u-i/n}{h_1}\Bigr)\,du\Biggr)
        \Biggl(\frac{1}{n}\sum_{i=1}^n\int_0^1
        \bar K^2\Bigl(\frac{v-i/n}{h_2}\Bigr)\,dv\Biggr)-2, \nonumber 
    \end{align}
   and
    \begin{align}\label{eq:sigma0}
        \sigma_0^2 := \frac{2}{45}\int_0^2\Bigl(1-\int_{-t}^1 K(v)\bar K(v+t)\,dv\Bigr)^2\,dt.
    \end{align}
\end{thm}

\begin{remark}[Approximation of the term $b_n$]
The centring  term $b_n$ appearing in Theorem~\ref{thm:taunCLT} is exactly the expectation under the null hypothesis of independence, that is 
\[
    b_n=\ee_0[\hat r_n].
\]
It can be approximated by a simpler expression, although this
approximation need not be sufficiently accurate on the scale of the
central limit theorem. More specifically, we have for the second factor of the first term in \eqref{eq:bn}

\[
\frac{1}{n(n-1)}
\sum_{k\ne\ell}
\int_0^1
\bar K\left(\frac{v-k/n}{h_2}\right)
\bar K\left(\frac{v-\ell/n}{h_2}\right)\,dv
=
\frac{n-2}{3n}+O(h_2),
\]
and the contribution  of the second term (corresponding to  the diagonal)  is of order 
\(O((nh_1)^{-1})\). It follows that
\[
        b_n-\tilde{b}_n
        =
        O(h_2)
        +
        O\left(\frac{1}{nh_1}\right), 
\]
where 
\[
\tilde{b}_n
:=
\frac{2(n-2)}{n^3h_1^2}
\sum_{i=1}^n\sum_{j\ne i}
\int_0^1
K\left(\frac{u-i/n}{h_1}\right)
K\left(\frac{u-j/n}{h_1}\right)\,du
-2.
\]
Thus, whenever \(h_2\to0\) and \(nh_1\to\infty\),
\(\tilde{b}_n - b_n=o(1)\). This, however, does not in general imply that \(\tilde{b}_n\) can replace \(b_n\) in Theorem~\ref{thm:taunCLT}. A sufficient condition for such a replacement is 
\[
\sqrt{\frac{n}{h_1}}
\Big \{h_2+\frac{1}{nh_1}\Big \}
= o(1) .
\]
This condition is not satisfied for all bandwidth choices
covered by the theorem. We therefore retain the exact
centring $b_n$. Finally, we note that the integral
\[
\frac{1}{h_1}
\int_0^1
K\left(\frac{u-i/n}{h_1}\right)
K\left(\frac{u-j/n}{h_1}\right)\,du
\]
should not in general be replaced by the corresponding full-line convolution $(K*K)((i-j)/(nh_1))$. Such a replacement would ignore the truncation at the
boundaries \(0\) and \(1\), and the resulting boundary corrections are relevant for the precise centring under the null hypothesis of independence.
\end{remark}

A test based on the statistic  $\hat{r}_n$ rejects the null hypothesis of independence if
\begin{align}
\sqrt{\frac{n}{h_1}}(\hat{r}_n - b_n) > 12 \sigma_0 u_{1-\alpha},
 \label{test}   
\end{align}
where $u_{1-\alpha}$ denotes the $(1-\alpha)$-quantile of the standard normal distribution. By Theorem~\ref{thm:taunCLT} this test has asymptotic level $\alpha$. Moreover, this result suggests that the test \eqref{test}
 may be sensitive to departures from
independence on a smaller scale than tests based on Chatterjee's rank
correlation. Indeed, under the null hypothesis of independence,
$\hat{r}_n$ fluctuates around its centring $b_n$ on the scale $ \left({h_1}/{n}\right)^{1/2}, $
which is smaller than the usual $n^{-1/2}$ rate since $h_1\to0$.
Thus, if under a sequence of local alternatives the mean of
$\hat{r}_n$ is shifted from its null value by an amount of this
order, one expects non-trivial power.

To establish such a local power result,  both the mean and the stochastic fluctuations of
$\hat{r}_n$ have to be controlled under the local alternatives. In the following section we  perform such an analysis  for Gaussian-copula
alternatives with correlation parameter $\rho_n\to0$. There we show
that the mean shift is of order $\rho_n^2$, while the variance remains
of order $h_1/n$. Consequently, $\rho_n\asymp\left({h_1}/{n}\right)^{1/4}$ is the natural signal-to-noise scale for the kernel statistic in this model. 

\section{Local power analysis}
\label{sec3}
Under fixed alternatives of dependence, the regularity conditions ensuring consistency of $\hat r_n$, together with the fact that $b_n\to0$, imply that the test~\eqref{test} is consistent, since $r>0$ whenever $X$ and $Y$ are dependent.
Theorem~\ref{thm:taunCLT} suggests that the smaller fluctuations under independence may allow the test based on
$\hat{r}_n$ to  detect departures from independence on a finer 
scale than the test based on Chatterjee's rank correlation, and in this section we investigate this question for a sequence of bivariate distributions with   Gaussian copulas whose correlation parameter $\rho_n$ converges to zero. More specifically, for each \(n\), let
\((X_i,Y_i)_{i=1}^n\) be an i.i.d.\ sample from a distribution
with continuous marginal distribution functions \(F_{1,n}\) and
\(F_{2,n}\) and Gaussian copula \(C_{\rho_n}\), where  \(\rho_n \in (-1, 1)\). Let \(\Phi\) denote
the standard normal distribution function and define
\[
        Z_i
        :=
        \Phi^{-1}\{F_{1,n}(X_i)\},
        \qquad
        W_i
        :=
        \Phi^{-1}\{F_{2,n}(Y_i)\}.
\]
Then \(\{(Z_i,W_i)\}_{i=1}^n\) are i.i.d.\ standard bivariate normal
vectors with correlation \(\rho_n \in (-1, 1)\). Moreover, continuity of the
marginal distributions implies that the ranks  of
\((X_1,\ldots,X_n)\) and \((Z_1,\ldots,Z_n)\) coincide almost surely,
and the same is true for the ranks of 
\((Y_1,\ldots,Y_n)\) and \((W_1,\ldots,W_n)\).
Consequently, almost surely $ \hat{r}_n(X,Y) = \hat{r}_n(Z,W)$. Thus, all results presented  below apply to arbitrary bivariate distributions with continuous marginals and  Gaussian copula $C_{\rho_n}$.
Throughout this section, we denote by 
$\ee_{\rho_n}$ the expectation with respect to the  Gaussian copula $C_{\rho_n}$ for the rank-based quantities under consideration, where the choice $\ee_0$
corresponds to the independent case discussed in Section \ref{sec2}. Our first result provides a second-order expansion of $\ee_{\rho_n}[\hat\tau_n^2]$  as $\rho_n \rightarrow 0 $. 

\begin{thm}\label{thm:mean-tau-square}
Suppose that Assumptions~\ref{asmp:kernel} and~\ref{asmp:band} hold. Then, under any bivariate distribution with continuous marginals and
Gaussian copula \(C_{\rho_n}\) satisfying
\begin{align}
    \label{condrho}
        \rho_n
        =
        O\Bigl\{\Bigl(\frac{h_1}{n}\Bigr)^{1/4}\Bigr\},
\end{align}
the estimator \(\hat\tau_n^2\) in
\eqref{d2} satisfies
\[
\ee_{\rho_n}[\hat\tau_n^2]
=
\ee_0[\hat\tau_n^2]
+
\frac{\rho_n^2}{2\pi\sqrt3}
+
      o\Bigl\{\Bigl(\frac{h_1}{n}\Bigr)^{1/2}\Bigr\}.
\]
\end{thm}

\begin{thm}\label{thm:variance-stability2} 
Let the assumptions of
Theorem~\ref{thm:mean-tau-square} be satisfied.
Assume additionally that $K$ is of bounded variation
and that $nh_1^{2+\eta}\to\infty$ for some fixed
$\eta\in(0,1)$. Then
\begin{equation}\label{eq:variance-stability-strong} \var_{\rho_n}(\hat\tau_n^2) = \var_0(\hat\tau_n^2) + o\left(\frac{h_1}{n}\right). 
\end{equation}
\end{thm}

\begin{thm}
\label{thm:Gaussian-local-CLT}
Let the assumptions of
Theorem~\ref{thm:mean-tau-square} be satisfied.
Assume additionally that $K$ is of bounded variation
and that $nh_1^{2+\eta}\to\infty$ for some fixed
$\eta\in(0,1)$. Then
\begin{equation*}
        \sqrt{\frac{n}{h_1}}
        \left(
        \hat\tau_n^2-\ee_{\rho_n}[\hat\tau_n^2]
        \right)
        \stackrel{\mathcal D}{\rightarrow}
         \mathcal{N}(0,4\sigma_0^2).
\end{equation*}
\end{thm}
An immediate consequence of
Theorems~\ref{thm:mean-tau-square}--\ref{thm:Gaussian-local-CLT} is an explicit characterization
of the local power of the independence test \eqref{test}
under Gaussian-copula alternatives. In particular,
we identify the  detection boundary and derive the
limiting power function.
 More specifically, if 
\begin{align*}
        |\rho_n| \big / 
        \Bigl(\frac{h_1}{n}\Bigr)^{1/4} 
        \longrightarrow c \ge  0,
\end{align*}
then by Theorem~\ref{thm:mean-tau-square} we obtain   
\[
        \ee_{\rho_n}[\hat r_n]-b_n
        =   \ee_{\rho_n}[\hat r_n]-\ee_0 [\hat r_n] =   \frac{\sqrt3}{\pi}\rho_n^2
        +  o\Bigl\{\Bigl(\frac{h_1}{n}\Bigr)^{1/2}\Bigr\},
\]
which gives
\begin{equation*}
\frac{1}{12\sigma_0}
\sqrt{\frac{n}{h_1}}
\left\{
\ee_{\rho_n}[\hat r_n]-b_n
\right\}
\longrightarrow
\frac{c^2}{4\pi\sqrt3\,\sigma_0}.
\end{equation*}
Consequently we obtain by   Theorem~\ref{thm:Gaussian-local-CLT} and Slutsky's theorem the following result.

\begin{cor}
\label{cor:loc} 
Let the assumptions of Theorem \ref{thm:Gaussian-local-CLT} be satisfied and 
suppose that
\[
    \frac{|\rho_n|}{(h_1/n)^{1/4}}
    \longrightarrow c\in[0,\infty).
\]
Then, under any bivariate distribution with continuous marginals and
Gaussian copula \(C_{\rho_n}\) satisfying  \eqref{condrho},
\begin{equation*}
\frac{1}{12\sigma_0}
\sqrt{\frac{n}{h_1}}
(\hat r_n-b_n)
\stackrel{\mathcal D}{\longrightarrow}
\mathcal N\left(
        \frac{c^2}{4\pi\sqrt3\,\sigma_0},
        1
\right),
\end{equation*}
and the limiting power of the test \eqref{test} is given by 
\begin{equation*}
\pp\Bigl(
   \sqrt{\frac{n}{h_1}}(\hat{r}_n - b_n) > 12 \sigma_0 u_{1-\alpha}    
        \Bigr)  \longrightarrow
  1-\Phi\Bigl(
        u_{1-\alpha}
        -
        \frac{c^2}{4\pi\sqrt3\,\sigma_0}
        \Bigr).
\end{equation*}
\end{cor}
Corollary \ref{cor:loc} shows that, if \(c=0\), the rejection probability of the test \eqref{test} converges to the nominal
level \(\alpha\). If instead  $c > 0$, the limiting power is strictly larger than \(\alpha\). Thus $(h_1/n)^{1/4}$ is the  detection boundary (with respect to Gaussian alternatives with copula $C_{\rho_n} $)  of the test 
\eqref{test} based on the statistic \(\hat r_n\).

This gives a sharper comparison with Chatterjee's rank correlation.
The exact detection boundary of the corresponding independence test
based on Chatterjee's estimator is \(n^{-1/4}\). Since \(h_1\to0\), $(h_1/n)^{1/4} = o(n^{-1/4})$. Therefore, there exist Gaussian alternatives with correlations of order
\((h_1/n)^{1/4}\) for which the test \eqref{test} based on the statistic \(\hat r_n\) has
non-trivial asymptotic power, whereas this rate lies strictly below
the detection boundary of the test based on Chatterjee's rank
correlation. In this sense, the test based on
\(\hat r_n\) is asymptotically more sensitive to local  Gaussian departures
from independence. 

Finally, the preceding local power calculation describes alternatives on the detection 
boundary. It is also natural to ask whether the test becomes
consistent once the correlation is larger than this rate. The following
proposition gives a positive answer on a super-local range. 
\begin{prop}
\label{prop:Gaussian-superlocal-power}
Suppose that Assumptions~\ref{asmp:kernel} and~\ref{asmp:band} hold. Additionally assume 
that $K$ is of bounded variation, and that
\[
    nh_1^{2+\eta}\longrightarrow\infty
\]
for some fixed $\eta\in(0,1)$.
Suppose that the observations have continuous marginal
distributions and Gaussian copula $C_{\rho_n}$, where
$\rho_n\in(-1,1)$ satisfies
\begin{equation}
\label{eq:superlocal-rho-range}
    \left(\frac{n}{h_1}\right)^{1/4}|\rho_n|
    \longrightarrow\infty,
    \qquad
    |\rho_n|=o(h_1^{1/2}).
\end{equation}
Then
\begin{equation}
\label{eq:superlocal-mean}
        \ee_{\rho_n}[\hat r_n]-b_n
        =
        \frac{\sqrt3}{\pi}\rho_n^2\{1+o(1)\}.
\end{equation}
Moreover, the independence test in \eqref{test} is consistent:
\begin{equation}
\label{eq:superlocal-power-one}
\pp_{\rho_n}
\left\{
        \sqrt{\frac{n}{h_1}}
        (\hat r_n-b_n)
        >
        12\sigma_0u_{1-\alpha}
\right\}
\longrightarrow1.
\end{equation}
\end{prop}
Proposition~\ref{prop:Gaussian-superlocal-power} strengthens the
preceding comparison above the local boundary. More precisely, note that 
\[
        \Bigl(\frac{h_1}{n}\Bigr)^{1/4}
        =
        o(h_1^{1/2}),
        \qquad
        \Bigl(\frac{h_1}{n}\Bigr)^{1/4}
        =
        o(n^{-1/4}),
\]
where the first relation follows from \(nh_1\to\infty\) and the
second from \(h_1\to0\). Hence, for every sequence \(\rho_n\) such
that
\[
        \Bigl(\frac{h_1}{n}\Bigr)^{1/4}
        \ll
        |\rho_n|
        \ll
        \min\{h_1^{1/2},n^{-1/4}\},
\]
Proposition~\ref{prop:Gaussian-superlocal-power} shows that the power
of the test based on \(\hat r_n\) converges to one, while the
corresponding correlations remain strictly below the \(n^{-1/4}\)
detection boundary of the test based on Chatterjee's rank
correlation. Thus the improvement is not restricted to alternatives
on the local boundary: there is a nonempty range of Gaussian
alternatives strictly below Chatterjee's detection boundary on which the test based on \(\hat r_n\) is already consistent.

\begin{remark}
{\rm 
We conclude this section with a comparison of the test \eqref{test} with a more powerful version of a test based on an estimator of Chatterjee's rank correlation, which has  recently been  investigated in \citet{lin2021boosting}.  These authors proposed a  statistic which incorporates \(M\) right nearest neighbours and proved a  central limit theorem under the
null  hypothesis  of independence  if $M \to \infty$ as $n\to \infty$ and  \(M=o(n^{1/4})\) \citep[see Theorem~3 in][]{lin2021boosting}.  For suitable
choices of \(M\), the corresponding test 
yields substantially improved detection rates against Gaussian copula alternatives  with
correlation parameter $\rho_n\to 0$. Within the  CLT regime, the best rate is $n^{-\gamma}$, where $\gamma \in (0, 5/16)$ can  be arbitrarily close to  \(5/16\).

The situation for the kernel statistic considered here is different.
Under the bandwidth condition $
    nh_1^{2+\eta}\longrightarrow\infty$ with
    $\eta\in(0, 1),$ 
the  detection boundary  (with respect to Gaussian copulas) is $(h_1/n)^{1/4}.$ 
Therefore, taking $ h_1=n^{-1/(2+\eta)}L_n$, where $L_n\rightarrow\infty$ sufficiently slowly, gives
\[
    \Big (\frac{h_1}{n}\Big )^{1/4}
    =
    n^{-(3+\eta)/\{4(2+\eta)\}}L_n^{1/4}.
\]
Consequently, for arbitrarily small fixed \(\eta>0\), the  exponent can be made arbitrarily close to \(3/8\). This means that 
for this choice of bandwidth
the detection rate of the test \eqref{test} is $n^{-\delta}$, where $\delta$ can be arbitrarily close to $3/8$. This goes strictly beyond the
\(5/16\) exponent available for the test of \cite{lin2021boosting} within the range of a central limit theorem.

We finally mention that \citet{lin2021boosting}
proved that  a permutation test based on their estimator yields detection rates arbitrarily close to the parametric rate \(n^{-1/2}\), if the number of nearest neighbours \(M\) is sufficiently close to \(n\). An interesting question for future research is whether a similar result can be shown for a permutation test based on the statistic $\hat r_n$.}
\end{remark}

\section{Outline of the proofs}
\label{sec:proof-outline}

In this section, we briefly describe the main ideas behind the proofs.
The detailed arguments require kernel-weight estimates, permutation
moment bounds, and, for the Gaussian local alternatives, likelihood,
rank-score, and Hermite expansions. The proofs under independence are
given in Section~\ref{sec4}, while the arguments for the Gaussian local
alternatives are developed in
Section~\ref{sec:proofs-gaussian-local}.

\subsection{The null central limit theorem}

Under independence, after ordering the observations according to
their \(X\)-ranks, the relative \(Y\)-rank vector is a uniform random
permutation. This converts \(\hat\tau_n^2\) into a permutation
statistic that can be decomposed into linear, quadratic, and diagonal
rank components. The rank representation and the deterministic
kernel-weight estimates needed for this decomposition are developed
in Subsection~\ref{sec:common-rank-bounds}, while the canonical
centering and permutation moment bounds used to control the
remainder terms are collected in
Subsection~\ref{subsec:null-canonical-permutation-bounds}.

The three components have different asymptotic orders. After centring, the quadratic and diagonal components are
negligible on the $\sqrt{n/h_1}$ scale, whereas the linear
permutation component carries the leading fluctuation.
 Its asymptotic variance is determined
by the boundary behaviour of the kernel weights, which contributes at
the leading order. A combinatorial central limit theorem then yields
the Gaussian limit for the linear component. Finally, Slutsky's
theorem gives the null central limit theorem for
\(\hat\tau_n^2\), and therefore for
\(\hat r_n=6\hat\tau_n^2-2\). The complete argument is given in
Subsection~\ref{sec:proof-null-clt}.

\subsection{Gaussian local alternatives}

While the statistic is distribution-free under the null, the analysis
of local alternatives relies essentially on the Gaussian copula
structure. In particular, the Gaussian model makes it possible to
expand the likelihood of the relative rank permutation around
independence and to express its derivatives in terms of Gaussian
rank scores. The Hermite and likelihood tools used throughout this
analysis are collected in
Subsection~\ref{subsec:gaussian-preliminaries}.

The analysis of shrinking correlations first determines how the mean
moves away from its null value. A Gaussian likelihood expansion shows
that for local alternatives satisfying
$|\rho_n|=O((h_1/n)^{1/4})$
the first-order contribution is negligible, whereas the
second-order contribution is non-zero. This yields
\[
        \ee_{\rho_n}[\hat\tau_n^2]
        -
        \ee_0[\hat\tau_n^2]
        =
        \frac{\rho_n^2}{2\pi\sqrt3}
        +
        o\left\{
        \Bigl(\frac{h_1}{n}\Bigr)^{1/2}
        \right\}.
\]
 The detailed proof of this expansion is given in
Subsection~\ref{subsec:mean-expansion}.

The second step is to control the variance as a function of the
correlation parameter. The main technical work shows that the
perturbation of the variance is negligible on the relevant scale:
\[
        \var_{\rho_n}(\hat\tau_n^2)
        =
        \var_0(\hat\tau_n^2)
        +
        o\left(\frac{h_1}{n}\right).
\]
This requires canonical centring of the rank statistic, contraction
bounds for the likelihood scores, and a fourth-order analysis of the
variance. These arguments are developed in
Subsection~\ref{subsec:variance-stability}.

Together, the mean and variance expansions identify
\[
        \Bigl(\frac{h_1}{n}\Bigr)^{1/4}
\]
as the rate at which the deterministic mean shift and the stochastic
fluctuations are comparable. The extension to correlations above this rate, which yields consistency of the test on the
super-local range considered in
Proposition~\ref{prop:Gaussian-superlocal-power}, is proved in
Subsection~\ref{subsec:superlocal-power}.

\subsection{The central limit theorem under local alternatives}

Under local alternatives, the rank decomposition used under the null
remains valid, but the relative rank permutation is no longer uniform,
since its distribution depends on \(\rho_n\). Consequently, the
combinatorial central limit theorem used under independence cannot be
applied directly. We first show that the canonical quadratic and
diagonal components remain negligible, so that the centred statistic
is asymptotically determined by the linear component.

To establish a central limit theorem for the linear component, we
condition on the ordered \(X\)-sample. Conditional on these variables,
the corresponding \(Y\)'s are independent Gaussian random variables.
Their rank scores are replaced by asymptotically equivalent
scores obtained from the conditional expectations of the
ranks, given the $X$-sample and the corresponding
$Y$-observation. This produces a surrogate linear statistic
with conditionally independent summands.
We then show
that its conditional variance converges to the same limit as under
independence and apply a conditional triangular-array central limit
theorem. Finally, the approximation error between the surrogate and
the original linear component is shown to be negligible. This yields
the local central limit theorem for \(\hat\tau_n^2\). The complete
reduction and conditional CLT argument are given in
Subsection~\ref{subsec:gaussian-local-clt}.

\section{Simulation}

In~\citet{shidrtonhan2021a}, it was shown theoretically that $\hat{\xi}_n$, introduced in~\citet{Chatterjee21}, can exhibit low power for testing independence. Subsequently, \citet{auddy2024exact} established that the exact detection boundary of a test based on $\hat{\xi}_n$ is $n^{-1/4}$, which is a suboptimal threshold for detecting dependence. On the other hand, empirical studies in~\citet{shidrtonhan2021a} suggested that a test based on the Dette--Siburg--Stoimenov coefficient $\hat{r}_n$ exhibits non-trivial power against certain alternatives.

Our results identify $(h_1/n)^{1/4}$ as the candidate Gaussian signal-to-noise rate of $\hat r_n$, which is smaller than that of $\hat \xi_n$. To illustrate this result, we consider a setting in which $X$ and $Y$ are standard normal random variables with correlation coefficient $\rho_n$, where $\rho_n \to 0$ as $n \to \infty$. 

For $\hat \xi_n$, the test is conducted using its asymptotic null distribution, while for $\hat r_n$ we rely on the asymptotic distribution established in Theorem~\ref{thm:taunCLT}. The significance level is set to $\alpha = 0.05$. The resulting empirical rejection probabilities are reported in Table~\ref{tb:asymp}. Throughout the simulations, we employ Epanechnikov and triangular kernels with bandwidths $h_1 = n^{-0.3}$ and $h_2 = n^{-0.8}$.

\begin{table}[h]
\centering
\begin{tabular}{
c
S[table-format=1.3, table-column-width=0.85cm]
S[table-format=1.3, table-column-width=0.85cm]
S[table-format=1.3, table-column-width=0.85cm]
S[table-format=1.3, table-column-width=0.85cm]
S[table-format=1.3, table-column-width=0.85cm]
S[table-format=1.3, table-column-width=0.85cm]
S[table-format=1.3, table-column-width=0.85cm]
S[table-format=1.3, table-column-width=0.85cm]
S[table-format=1.3, table-column-width=0.85cm]
}
\toprule
& \multicolumn{3}{c}{$\rho_n = n^{-1/4}$}
& \multicolumn{3}{c}{$\rho_n = (n/h_1)^{-1/4}$}
& \multicolumn{3}{c}{$\rho_n = 0$} \\
\cmidrule(lr){2-4} \cmidrule(lr){5-7} \cmidrule(lr){8-10}
$n$
& {$\hat r_n$[E]} & {$\hat r_n$[T]} & {$\hat \xi_n$}
& {$\hat r_n$[E]} & {$\hat r_n$[T]} & {$\hat \xi_n$}
& {$\hat r_n$[E]} & {$\hat r_n$[T]} & {$\hat \xi_n$} \\
\midrule
100   & 0.150 & 0.198 & 0.182 & 0.068 & 0.088 & 0.062 & 0.024 & 0.028 & 0.024 \\
500   & 0.310 & 0.358 & 0.228 & 0.108 & 0.124 & 0.108 & 0.038 & 0.041 & 0.054 \\
1000  & 0.380 & 0.430 & 0.190 & 0.103 & 0.118 & 0.084 & 0.035 & 0.037 & 0.032 \\
5000  & 0.577 & 0.633 & 0.210 & 0.129 & 0.142 & 0.074 & 0.043 & 0.048 & 0.036 \\
10000 & 0.694 & 0.749 & 0.214 & 0.144 & 0.152 & 0.072 & 0.044 & 0.042 & 0.044 \\
\bottomrule
\end{tabular}
\caption{\it Empirical rejection rates of the independence test \eqref{test} based on  $\hat r_n$ (using the Epanechnikov [E] and the triangular [T] kernel) and the test based on Chatterjee's rank correlation $\hat \xi_n$. The results are based on $2000$ iterations using the quantiles from the respective asymptotic distributions.}
\label{tb:asymp}
\end{table}
Table~\ref{tb:asymp} clearly indicates that when $\rho_n = n^{-1/4}$, the test \eqref{test} based on $\hat{r}_n$ exhibits substantially higher power than the test based on $\hat{\xi}_n$, with the power of the test based on $\hat{r}_n$ approaching 1 as the sample size $n$ increases. Moreover, when $\rho_n = (n/h_1)^{-1/4}$, the test based on $\hat{\xi}_n$ displays only trivial power, whereas the test based on $\hat{r}_n$ continues to demonstrate comparatively higher power.

To further illustrate the detection rate of the test \eqref{test} predicted by the local asymptotic theory, we parameterise the correlation under the Gaussian alternative as $\rho_n=c n^{-\beta}$, where $c>0$ is fixed (we let $c = 1.5$ in our simulations). If $h_1=n^{-a}$, the rate
\[
    \rho_n\asymp\left(\frac{h_1}{n}\right)^{1/4}
\]
corresponds to the exponent
\[
    \beta_*=\frac{1+a}{4}.
\]
Thus, alternatives with $\beta < \beta_*$ are stronger than the  detection boundary, whereas alternatives with $\beta > \beta_*$ are weaker. At $\beta=\beta_*$, the limiting power is non-trivial. The bandwidth condition $nh_1^{2+\eta}\to\infty$ permits $a$ to be chosen arbitrarily close to $1/2$, by taking $\eta>0$ sufficiently small. Consequently, $\beta_*=(1+a)/4$ can be made arbitrarily close to $3/8$. 

To compare the two procedures under both finite-sample and asymptotic calibration, we consider four versions of the test. For the kernel statistic, we retain the bandwidths $h_1=n^{-0.495}$ and $h_2=n^{-0.8}$, for which the Gaussian detection exponent is
\[
    \beta_*=\frac{1+0.495}{4}=0.37375,
\]
close to $3/8=0.375$. The simulated rejection probabilities of the test \eqref{test} are displayed in Table \ref{tab:kernel-comparison}. Panels~A and~B use the same statistic $\hat r_n$ and differ only in their calibration: Panel~A uses the simulated permutation null distribution, whereas Panel~B uses the asymptotic null distribution from Theorem~\ref{thm:taunCLT}.

For comparison with \citet{lin2021boosting}, we display in Table \ref{tab:LinHan-comparison} the rejection probabilities of their test. Panel~C uses their simulation-based test with $M=\lfloor n^{1/2}\rfloor$. This choice is relevant because their sufficient Gaussian detection boundary has exponent $3/8$ when $M=n^{1/2}$, and is therefore directly comparable with the exponent $0.37375$ of the kernel statistic. Panel~D instead restricts the Lin--Han statistic to the regime in which their central limit theorem applies. Specifically, we take $M=\lfloor n^{0.24}\rfloor$, so that $M=o(n^{1/4})$. For $M=n^\gamma$ with $\gamma=0.24$, the corresponding detection exponent is
\[
    \beta_*^{\mathrm{LH}} =\frac{1}{4}+\frac{\gamma}{4} =0.31,
\]
while $5/16=0.3125$ is the limiting exponent obtained as $\gamma\uparrow1/4$. We therefore include values of $\beta$ around both $0.31$ and $3/8$ in Tables~\ref{tab:kernel-comparison} and \ref{tab:LinHan-comparison}.
\begin{table}[t]
\begin{tabular}{ccccccccc}
  \toprule
$\beta\backslash n$ & 500 & 1000 & 2000 & 5000 & 10000 & 50000 & 100000 & 500000 \\ 
  \midrule
  \midrule
\multicolumn{9}{l}{\textit{Panel A: Permutation test based on $\hat r_n$}}\\
0.29 & 0.826 & 0.888 & 0.964 & 0.982 & 0.996 & 1.000 & 1.000 & 1.000 \\ 
  0.31 ($\beta_*^{\mathrm{LH}}$) & 0.670 & 0.750 & 0.794 & 0.866 & 0.942 & 0.988 & 0.996 & 1.000 \\ 
  0.3125 (5/16) & 0.658 & 0.742 & 0.792 & 0.866 & 0.906 & 0.988 & 0.994 & 1.000 \\ 
  0.33 & 0.482 & 0.628 & 0.652 & 0.642 & 0.726 & 0.868 & 0.904 & 0.956 \\ 
  0.35 & 0.366 & 0.446 & 0.420 & 0.444 & 0.490 & 0.578 & 0.594 & 0.664 \\ 
  0.37 & 0.298 & 0.340 & 0.300 & 0.302 & 0.334 & 0.394 & 0.320 & 0.372 \\ 
  0.37375 ($\beta_*$) & 0.280 & 0.328 & 0.282 & 0.324 & 0.278 & 0.338 & 0.300 & 0.340 \\ 
  0.375 (3/8) & 0.288 & 0.306 & 0.290 & 0.252 & 0.302 & 0.320 & 0.302 & 0.314 \\ 
  0.39 & 0.234 & 0.256 & 0.240 & 0.228 & 0.214 & 0.218 & 0.226 & 0.176 \\ 
  0.41 & 0.150 & 0.158 & 0.166 & 0.144 & 0.146 & 0.148 & 0.118 & 0.124 \\ 
    \midrule
  \multicolumn{9}{l}{\textit{Panel B: CLT test based on $\hat r_n$}}\\
0.29 & 0.876 & 0.948 & 0.966 & 0.996 & 1.000 & 1.000 & 1.000 & 1.000 \\ 
  0.31 ($\beta_*^{\mathrm{LH}}$) & 0.778 & 0.840 & 0.890 & 0.938 & 0.986 & 0.994 & 0.996 & 1.000 \\ 
  0.3125 (5/16) & 0.758 & 0.816 & 0.864 & 0.914 & 0.950 & 0.996 & 0.998 & 1.000 \\ 
  0.33 & 0.600 & 0.676 & 0.744 & 0.750 & 0.828 & 0.928 & 0.954 & 0.972 \\ 
  0.35 & 0.478 & 0.534 & 0.578 & 0.590 & 0.640 & 0.694 & 0.762 & 0.760 \\ 
  0.37 & 0.416 & 0.382 & 0.352 & 0.450 & 0.420 & 0.438 & 0.466 & 0.460 \\ 
  0.37375 ($\beta_*$) & 0.314 & 0.438 & 0.342 & 0.392 & 0.404 & 0.422 & 0.406 & 0.396 \\ 
  0.375 (3/8) & 0.360 & 0.354 & 0.384 & 0.414 & 0.428 & 0.414 & 0.404 & 0.408 \\ 
  0.39 & 0.346 & 0.314 & 0.330 & 0.304 & 0.264 & 0.298 & 0.246 & 0.250 \\ 
  0.41 & 0.218 & 0.254 & 0.238 & 0.224 & 0.228 & 0.186 & 0.184 & 0.198 \\ 
   \bottomrule
\end{tabular}
\centering
\caption{\it Empirical rejection probabilities of the independence test based on $\hat r_n$ under Gaussian-copula alternatives $\rho_n=1.5n^{-\beta}$. All rejection probabilities are estimated from $500$ Monte Carlo replications for each combination of $n$ and $\beta$, and the nominal level is $0.05$. Panel~A reports the permutation-calibrated test based on $\hat r_n$, with $h_1=n^{-0.495}$ and $h_2=n^{-0.8}$; its null distribution is approximated using $1999$ random permutations. Panel~B uses the same statistic and bandwidths but calibrates the test using the asymptotic null distribution in Theorem~\ref{thm:taunCLT} (this is the test \eqref{test}).} 
\label{tab:kernel-comparison}
\end{table}

\begin{table}[t]
\begin{tabular}{ccccccccc}
  \toprule
$\beta\backslash n$ & 500 & 1000 & 2000 & 5000 & 10000 & 50000 & 100000 & 500000 \\ 
  \midrule
  \midrule
\multicolumn{9}{l}{\textit{Panel C: Permutation test based on Lin--Han, $M=\lfloor n^{1/2}\rfloor$}}\\
0.29 & 0.988 & 0.998 & 0.998 & 1.000 & 1.000 & 1.000 & 1.000 & 1.000 \\ 
  0.31 ($\beta_*^{\mathrm{LH}}$)& 0.944 & 0.962 & 0.982 & 0.994 & 1.000 & 1.000 & 1.000 & 1.000 \\ 
  0.3125 (5/16) & 0.952 & 0.946 & 0.984 & 0.994 & 0.998 & 1.000 & 1.000 & 1.000 \\ 
  0.33 & 0.868 & 0.918 & 0.930 & 0.956 & 0.984 & 0.992 & 1.000 & 1.000 \\ 
  0.35 & 0.768 & 0.766 & 0.800 & 0.852 & 0.862 & 0.920 & 0.930 & 0.950 \\ 
  0.37 & 0.634 & 0.600 & 0.696 & 0.662 & 0.692 & 0.674 & 0.666 & 0.614 \\ 
  0.37375 ($\beta_*$) & 0.626 & 0.560 & 0.668 & 0.614 & 0.674 & 0.614 & 0.614 & 0.558 \\ 
  0.375 (3/8) & 0.612 & 0.608 & 0.652 & 0.612 & 0.624 & 0.606 & 0.626 & 0.560 \\ 
  0.39 & 0.484 & 0.518 & 0.552 & 0.490 & 0.516 & 0.412 & 0.436 & 0.368 \\ 
  0.41 & 0.414 & 0.388 & 0.360 & 0.364 & 0.324 & 0.264 & 0.276 & 0.180 \\ 
    \midrule
\multicolumn{9}{l}{\textit{Panel D: CLT test based on Lin--Han $M=\lfloor n^{0.24}\rfloor$}}\\
0.29 & 0.782 & 0.832 & 0.848 & 0.834 & 0.882 & 0.904 & 0.892 & 0.940 \\ 
  0.310 ($\beta_*^{\mathrm{LH}}$) & 0.624 & 0.658 & 0.646 & 0.642 & 0.630 & 0.604 & 0.600 & 0.602 \\ 
  0.3125 (5/16) & 0.658 & 0.626 & 0.624 & 0.592 & 0.594 & 0.612 & 0.572 & 0.542 \\ 
  0.33 & 0.518 & 0.562 & 0.470 & 0.454 & 0.422 & 0.384 & 0.324 & 0.312 \\ 
  0.35 & 0.366 & 0.396 & 0.306 & 0.298 & 0.256 & 0.228 & 0.200 & 0.168 \\ 
  0.37 & 0.288 & 0.226 & 0.236 & 0.182 & 0.196 & 0.130 & 0.116 & 0.104 \\ 
  0.37375 ($\beta_*$) & 0.296 & 0.256 & 0.244 & 0.170 & 0.132 & 0.154 & 0.102 & 0.084 \\ 
  0.375 (3/8) & 0.280 & 0.260 & 0.240 & 0.162 & 0.194 & 0.136 & 0.128 & 0.114 \\ 
  0.39 & 0.198 & 0.224 & 0.158 & 0.140 & 0.120 & 0.082 & 0.082 & 0.078 \\ 
  0.41 & 0.172 & 0.172 & 0.134 & 0.126 & 0.084 & 0.086 & 0.080 & 0.068 \\ 
   \bottomrule
\end{tabular}
\centering
\caption{\it Empirical rejection probabilities under Gaussian-copula alternatives $\rho_n=1.5n^{-\beta}$. All rejection probabilities are estimated from $500$ Monte Carlo replications for each combination of $n$ and $\beta$, and the nominal level is $0.05$. Panel~C reports the simulation-based test of \citet{lin2021boosting} with $M=\lfloor n^{1/2}\rfloor$,  using $1999$ random  permutations. Panel~D reports the CLT-based test of \citet{lin2021boosting} with $M=\lfloor n^{0.24}\rfloor$.} 
\label{tab:LinHan-comparison}
\end{table}

Several features of Tables~\ref{tab:kernel-comparison} and ~\ref{tab:LinHan-comparison} are consistent with the theoretical detection rates. First, the two versions of the  test based on the kernel estimator display the transition predicted around $\beta_*=0.37375$. For values well below this threshold, such as $\beta=0.29$, $0.31$, and $0.33$, the rejection probabilities increase rapidly towards one  with increasing sample size. Near $\beta_*$ and $3/8$, the power remains non-trivial over the range of sample sizes considered, whereas for $\beta>\beta_*$ it decreases as the alternatives become weaker. Comparing Panels~A and~B also shows that the asymptotic calibration produces noticeably larger rejection probabilities than the permutation calibration for the considered sample sizes, indicating that finite-sample calibration effects remain non-negligible for the bandwidth sequence $h_1=n^{-0.495}$.

Panel~C shows that the simulation-based Lin--Han test with $M=\lfloor n^{1/2}\rfloor$ has higher finite-sample power than the permutation-calibrated kernel test over most of the alternatives considered, including those close to the rate $n^{-3/8}$. This comparison should, however, be interpreted as a finite-sample comparison under the particular tuning choices $h_1=n^{-0.495}$ and $M=\lfloor n^{1/2}\rfloor$, rather than as an asymptotic ordering of the two procedures.

The behaviour in Panel~D is markedly different. When $M=\lfloor n^{0.24}\rfloor$, so that the Lin--Han central limit theorem applies, the relevant detection boundary is only $n^{-\beta_*^{\mathrm{LH}}}$ with $\beta_*^{\mathrm{LH}}=0.31$. Accordingly, power increases with $n$ for the stronger alternative $\beta=0.29$, remains non-trivial around $\beta=0.31$, and decreases towards the nominal level for alternatives with $\beta$ appreciably larger than $0.31$. In particular, at values of $\beta$ near $3/8$, the rejection probabilities in Panel~D decrease substantially with $n$. This illustrates the distinction between the CLT regime of \citet{lin2021boosting}, whose detection exponent can approach at most $5/16$, and their large-$M$ regime, in which faster detection rates such as $3/8$ are obtained but their null CLT is no longer available.

\section{Computational complexity for piecewise-polynomial kernels}
\label{app:epanechnikov-computation}

In this section we show that the kernel estimator $\hat r_n$ can be evaluated in $O(n\log n)$ operations for piecewise-polynomial kernels whose degree and number of polynomial pieces are fixed independently of $n$. For simplicity, we give the argument for the Epanechnikov kernel
\[
K(t)=\frac{3}{4}(1-t^2)\bone\{|t|\leq1\},
\qquad
\bar K(t)=
\begin{cases}
0, & t\leq-1,\\[1mm]
\displaystyle \frac{1}{2}+\frac{3}{4}t-\frac{1}{4}t^3, & -1<t<1,\\[2mm]
1, & t\geq 1;
\end{cases}
\]
 the extension to other such kernels follows in the same way.
 
After ordering the observations according to their $X$-ranks, let $\pi_i$ denote the corresponding $Y$-rank. Observing the definitions \eqref{d2} and \eqref{d2aa} we obtain 
\begin{align}
\label{comp2}
\hat\tau_n^2
=
\frac{1}{n^2h_1}
\sum_{i=1}^n\sum_{j=1}^n
a_{ij}B_{\pi_i\pi_j},
\end{align}
where $a_{i,j}$ and $B_{\pi_i,\pi_j}$ are defined  by
\begin{align*}
a_{ij} = \frac{1}{h_1}\int_0^1
K\left( \frac{u-i/n}{h_1}\right)
K\left( \frac{u-j/n}{h_1}\right)\,du, \quad
B_{ij}
=
\int_0^1
\bar K\left( \frac{v-i/n}{h_2}\right)
\bar K\left( \frac{v-j/n}{h_2}\right)\,dv.
\end{align*}
The argument has two steps. First, we show that for fixed $i$ the summand $a_{ij}B_{\pi_i\pi_j}$ has only a fixed number of polynomial forms as $(j,\pi_j)$ varies. Second, we show that the sums over the corresponding regions can all be evaluated jointly in $O(n\log n)$ operations.

For indices away from the boundary, let
\[
d=\frac{|i-j|}{nh_1}.
\]
Direct integration gives
\[
a_{ij}
=
\left\{
\frac{3}{5}
-\frac{3}{4}d^2
+\frac{3}{8}d^3
-\frac{3}{160}d^5
\right\}
\bone\{d\leq2\}.
\]
Thus, for fixed $i$, the formula for $a_{ij}$ can change only when $j$ crosses the points
\[
i-2nh_1,\qquad i,\qquad i+2nh_1.
\]
In particular, $a_{ij}$ is piecewise polynomial in $j/n$, with a number of pieces that does not depend on $n$.

Similarly, away from the boundary,
\[
B_{k\ell}
=
1-\frac{\max(k,\ell)}{n}
+
h_2
\left\{
\frac{d^7}{2240}
-\frac{3d^5}{160}
+\frac{d^4}{16}
-\frac{3d^2}{10}
+\frac{d}{2}
-\frac{9}{35}
\right\}
\bone\{d\leq2\},
\]
where
\[
d=\frac{|k-\ell|}{nh_2}.
\]
For fixed $k$, the term $\max(k,\ell)$ changes form only at $\ell=k$, while the indicator changes form at $\ell=k\pm2nh_2$. Hence the formula for $B_{k\ell}$ can change only at
\[
k-2nh_2,\qquad k,\qquad k+2nh_2.
\]
Thus $B_{k\ell}$ is also piecewise polynomial in $\ell/n$, with a fixed number of pieces. In \eqref{comp2}, where $k=\pi_i$ and $\ell=\pi_j$, these breakpoints become
\[
\pi_i-2nh_2,\qquad \pi_i,\qquad \pi_i+2nh_2.
\]

The truncation of the kernel windows at the boundaries of $[0,1]$ does not change this conclusion. To see this for $a_{ij}$, make the change of variables $t=(u-i/n)/h_1$. The integration interval then has endpoints
\[
\max\left\{
-\frac{i}{nh_1},
-1,
-1-\frac{i-j}{nh_1}
\right\}
\quad\text{and}\quad
\min\left\{
\frac{n-i}{nh_1},
1,
1-\frac{i-j}{nh_1}
\right\}.
\]
On any region on which the quantities attaining the maximum and minimum are fixed, these endpoints are affine functions of the ranks. Since the product of the two Epanechnikov kernels is polynomial on their common support, integration over such an interval gives a polynomial in the ranks. There are only finitely many possible choices for the maximum and minimum, so the number of resulting polynomial pieces is bounded independently of $n$. The same reasoning applies to $B_{ij}$, since $\bar K$ has only finitely many polynomial pieces. Thus the boundary changes the locations and formulas of some pieces, but not their number up to a constant independent of $n$.

It follows that, for each fixed $i$, the summation over $j$ in \eqref{comp2} can be divided into a fixed number of regions of the form
\[
L\leq j\leq U,
\qquad
S\leq\pi_j\leq T,
\]
such that on each region
\[
a_{ij}B_{\pi_i\pi_j}
=
\sum_{p=0}^5\sum_{q=0}^7
c_{pq}
\left(\frac{j}{n}\right)^p
\left(\frac{\pi_j}{n}\right)^q,
\]
where the coefficients $c_{pq}$ depend on $i$ and on the region. The  number of these regions is bounded by a constant that does not depend on $n$: the locations of the breakpoints vary with $i$, $h_1$, $h_2$, and $n$, but the number of breakpoints does not.

It is therefore enough to evaluate sums of the form
\begin{align} \label{comp1}
\sum_{\substack{L\leq j\leq U\\ S\leq\pi_j\leq T}}
\left(\frac{j}{n}\right)^p
\left(\frac{\pi_j}{n}\right)^q,
\qquad
0\leq p\leq5,\quad 0\leq q\leq7.
\end{align}
By inclusion--exclusion, such a sum is equal to
\begin{align*}
\sum_{\substack{j\leq U\\ \pi_j\leq T}}
\left(\frac{j}{n}\right)^p
\left(\frac{\pi_j}{n}\right)^q
-
\sum_{\substack{j\leq L-1\\ \pi_j\leq T}}
\left(\frac{j}{n}\right)^p
\left(\frac{\pi_j}{n}\right)^q
-
\sum_{\substack{j\leq U\\ \pi_j\leq S-1}}
\left(\frac{j}{n}\right)^p
\left(\frac{\pi_j}{n}\right)^q
+
\sum_{\substack{j\leq L-1\\ \pi_j\leq S-1}}
\left(\frac{j}{n}\right)^p
\left(\frac{\pi_j}{n}\right)^q.
\end{align*}
Hence the calculation of each sum in \eqref{comp1} requires the computation of four cumulative sums of the form
\[
\sum_{\substack{j\leq s\\ \pi_j\leq t}}
\left(\frac{j}{n}\right)^p
\left(\frac{\pi_j}{n}\right)^q.
\]

It is important that we do not compute these cumulative sums for all $n^2$ possible pairs $(s,t)$. For each $i$, the piecewise-polynomial decomposition produces only a fixed number of sums of the form \eqref{comp1}, and each sum  requires only four such pairs. Consequently, over all $i=1,\ldots,n$, only $O(n)$ pairs $(s,t)$ need to be queried.

The fact that only cumulative sums of this form are required is precisely what allows us to use a Fenwick tree. A Fenwick tree is a data structure designed to maintain cumulative sums under point updates, with both an update and the evaluation of a cumulative sum requiring only $O(\log n)$ operations \citep{Fenwick1994}. We now describe how all the $O(n)$ cumulative sums required above can be obtained using a single Fenwick tree.

 Collect all the required pairs $(s,t)$ in advance and sort them according to their first coordinate $s$. We then  process the observations in increasing order of the first rank, $j=1,\ldots,n$.
At each stage we maintain a Fenwick tree \citep{Fenwick1994} whose positions are indexed by the second rank $1,\ldots,n$. Thus the two rank coordinates play different roles: the first rank $j$ determines when an observation is inserted into the data structure, whereas the second rank $\pi_j$ determines where it is inserted.

When observation $j$ is processed, we insert at position $\pi_j$ the vector
\[
\left\{
\left(\frac{j}{n}\right)^p
\left(\frac{\pi_j}{n}\right)^q:
0\leq p\leq5,\;0\leq q\leq7
\right\}.
\]
A Fenwick tree stores sums over a collection of overlapping dyadic blocks of its index set. Hence inserting a value at a single position requires updating only $O(\log n)$ stored blocks. Each tree entry contains a vector of $48$ values, one for each of the $6\times8$ pairs $(p,q)$; since this dimension is fixed, it contributes only a constant factor.

After observations $1,\ldots,s$ have been inserted, the tree contains contributions only from indices satisfying $j\leq s$. At that moment, querying the tree up to position $t$ adds together the $O(\log n)$ stored blocks whose disjoint union is $\{1,\ldots,t\}$. The result is therefore
\[
\left\{
\sum_{\substack{j\leq s\\ \pi_j\leq t}}
\left(\frac{j}{n}\right)^p
\left(\frac{\pi_j}{n}\right)^q:
0\leq p\leq5,\;0\leq q\leq7
\right\}.
\]
Thus the condition $j\leq s$ is enforced simply by the order in which the observations are inserted, while the Fenwick-tree query enforces the condition $\pi_j\leq t$. There is no need to search simultaneously over both coordinates.

The computational cost can now be counted explicitly. There is exactly one insertion for each observation, hence $n$ insertions in total. There are also only $O(n)$ required cumulative queries, because each $i$ gives a fixed number of rectangular regions and each region gives four queries. Both an insertion and a query cost $O(\log n)$ in a Fenwick-tree \citep{Fenwick1994}. Therefore,
\[
\underbrace{n\,O(\log n)}_{\text{insertions}}
+
\underbrace{O(n)\,O(\log n)}_{\text{queries}}
=
O(n\log n),
\]
and it follows that the double sum in \eqref{comp2} can be evaluated in $O(n\log n)$ operations.

Computing the two marginal rank vectors also requires $O(n\log n)$ operations. Hence $\hat\tau_n^2$, and therefore the statistic
\[
\hat r_n=6\hat\tau_n^2-2,
\]
can be evaluated exactly in $O(n\log n)$ operations for the Epanechnikov kernel.

The same argument applies to any piecewise-polynomial kernel whose degree and number of pieces are fixed independently of $n$. In that case, the corresponding weights $a_{ij}$ and $B_{ij}$ again have a fixed piecewise-polynomial structure, so that the double sum reduces to a fixed number of polynomial range sums for each $i$ and can be evaluated in $O(n\log n)$ operations. This includes, for example, the triangular kernel. For a general kernel, however, the weights need not admit such a finite piecewise-polynomial representation, so the reduction used above is no longer available. Other kernels may still admit efficient algorithms based on additional structure or numerical approximation, but the $O(n\log n)$ complexity does not follow from the present argument.
\section*{Acknowledgements} 
The work of Holger Dette has been partially supported
by the Deutsche Forschungsgemeinschaft (DFG):
TRR 391 {\it Spatio-temporal Statistics for the Transition of Energy and Transport} (520388526);
Research Unit 5381 \textit{Mathematical Statistics in the Information Age} (460867398).

The authors acknowledge the use of ChatGPT-5.6 Sol (OpenAI) for assistance with proof organization, exposition, and language editing. All the content has been independently verified by the authors, who take full responsibility for the paper.

\printbibliography

\appendix
\newrefsection            
\markboth{\MakeUppercase{\shortauthors}}{\shorttitle}
\section{Proofs of the results in Section \ref{sec2}} \label{sec4}
\setcounter{equation}{0}
In this section we establish the asymptotic normality of the kernel estimator $\hat r_n = 6 \hat \tau_n^2 - 2 $ of Chatterjee's dependence measure. The explicit proof can be found in Section \ref{sec:proof-null-clt}. It requires several auxiliary results, which are presented first in Section \ref{sec:common-rank-bounds} and \ref{subsec:null-canonical-permutation-bounds}.
\subsection{Rank representation and deterministic kernel-weight bounds}
\label{sec:common-rank-bounds}

Let $\{(X_i,Y_i)\}_{i = 1}^n$ be an i.i.d. sample, and $R_i^x$, $R_i^y$
denote the ranks of $X_i$ and $Y_i$. Recall
\begin{align*}
\hat\tau_n(u,v)
=
\frac{1}{nh_1}
\sum_{i=1}^n
K\left( \frac{u-R_i^x/n}{h_1}\right)
\bar K\left( \frac{v-R_i^y/n}{h_2}\right).
\end{align*}
Its squared $L^2$-norm admits the representation
\begin{align}\label{eq:tau-square-rank-representation}
\hat{\tau}_n^2
:=
\int_0^1\int_0^1 \hat\tau_n^2(u,v)\,du\,dv
=
\frac{1}{n^2h_1}
\sum_{i=1}^n\sum_{j=1}^n
a_{R_i^x R_j^x}
B_{R_i^y R_j^y},
\end{align}
where
\begin{align}\label{eq:aij}
a_{ij} = \frac{1}{h_1}\int_0^1
K\left( \frac{u-i/n}{h_1}\right)
K\left( \frac{u-j/n}{h_1}\right)\,du,
\end{align}
and
\begin{align}\label{eq:Bij}
B_{ij}
=
\int_0^1
\bar K\left( \frac{v-i/n}{h_2}\right)
\bar K\left( \frac{v-j/n}{h_2}\right)\,dv.
\end{align}
Before we start, we collect several properties of the weights $a_{ij}$ and $B_{ij}$, which will be used frequently in the proofs.

\begin{lmm}[Deterministic kernel-weight bounds]
\label{lem:deterministic-kernel-weight-bounds}
Assume that \(K\) is bounded, symmetric, non-negative, supported on
\([-1,1]\), and normalized so that
\[
        \int_{\rr}K(u)\,du=1.
\]
Suppose that \(h_1\to0\) and \(nh_1\to\infty\), and let \(a_{ij}\) and
\(B_{k\ell}\) be defined as in \eqref{eq:aij} and \eqref{eq:Bij}.
Set
\[
        L_i:=\sum_{j\ne i}a_{ij},
        \qquad
        \bar L:=\frac{1}n\sum_{i=1}^nL_i,
        \qquad
        \bar b_k:=\frac{1}{n-1}\sum_{\ell\ne k}B_{k\ell},
        \qquad
        \bar b:=\frac{1}{n}\sum_{k=1}^n\bar b_k.
\]
Then the matrices \(a=(a_{ij})\) and \(B=(B_{k\ell})\) are symmetric
and, uniformly in their indices,
\begin{equation}\label{eq:kernel-weight-pointwise}
        0\le a_{ij}
        \le
        C\bone\{|i-j|\le2nh_1+2\},
        \qquad
        0\le B_{k\ell}\le1.
\end{equation}
Consequently,
\begin{align}
\max_{1\le i\le n}\sum_{j=1}^n|a_{ij}|
    &=O(nh_1),
&
\sum_{i,j=1}^n|a_{ij}|
    &=O(n^2h_1),
\label{eq:kernel-weight-row-sums}\\
\|a\|_{\mathrm F}^2
    =\sum_{i,j=1}^na_{ij}^2
    &=O(n^2h_1),
&
\|a\|_{\mathrm{op}}
    &=O(nh_1),
\label{eq:kernel-weight-a-norms}\\
\|B\|_{\mathrm F}^2
    =\sum_{k,\ell=1}^nB_{k\ell}^2
    &\le n^2,
&
\|B\|_{\mathrm{op}}
    &\le n.
\label{eq:kernel-weight-B-norms}
\end{align}
Let
\[
        m_n:=\lceil2nh_1\rceil,
        \qquad
        I_n:=\{m_n+1,\ldots,n-m_n\}.
\]
For all sufficiently large \(n\), there exists a deterministic number
\(\ell_n\) such that
\begin{equation}\label{eq:L-interior-approximation}
        L_i=\ell_n
        \quad\text{for every }i\in I_n,
        \qquad
        \ell_n=nh_1+o(nh_1).
\end{equation}
In particular,
\begin{equation}\label{eq:L-basic-bounds}
        \max_{1\le i\le n}L_i=O(nh_1),
        \qquad
        \bar L=nh_1+o(nh_1),
\end{equation}
and
\begin{equation}\label{eq:L-centered-bounds}
        \sum_{i=1}^n(L_i-\bar L)=0,
        \qquad
        \max_{1\le i\le n}|L_i-\bar L|=O(nh_1),
        \qquad
        \sum_{i=1}^n(L_i-\bar L)^2=O(n^3h_1^3).
\end{equation}
Finally,
\begin{equation}\label{eq:B-row-centered-bounds}
        \sum_{k=1}^n(\bar b_k-\bar b)=0,
        \qquad
        \max_{1\le k\le n}|\bar b_k-\bar b|\le1,
        \qquad
        \sum_{k=1}^n(\bar b_k-\bar b)^2\le n.
\end{equation}
\end{lmm}

\begin{proof} 
The change of variables \(t=(u-i/n)/h_1\) gives
\begin{equation}\label{eq:aij-convolution-form}
        a_{ij}
        =
        \int_{-i/(nh_1)}^{(n-i)/(nh_1)}
        K(t)
        K\left(t+\frac{i-j}{nh_1}\right)\,dt.
\end{equation}
Since \(K\) is bounded and supported on \([-1,1]\), the right-hand side
is uniformly bounded. Moreover, it can be non-zero only if there exists
\(t\) such that
\[
        |t|\le1,
        \qquad
        \left|t+\frac{i-j}{nh_1}\right|\le1,
\]
which implies \(|i-j|\le2nh_1\). This proves the first bound in
\eqref{eq:kernel-weight-pointwise}. Since \(K\ge0\) and
\(\int_{\mathbb R}K=1\), we have \(0\le\bar K\le1\), and hence the
definition of \(B_{k\ell}\) gives
\[
        0\le B_{k\ell}\le1.
\]

For each fixed \(i\), only \(O(nh_1)\) entries \(a_{ij}\) can be
non-zero. Together with their uniform boundedness, this proves
\eqref{eq:kernel-weight-row-sums} and the Frobenius-norm bound in
\eqref{eq:kernel-weight-a-norms}. Since \(a\) is symmetric,
\[
        \|a\|_{\mathrm{op}}
        \le
        \max_i\sum_{j=1}^n|a_{ij}|
        =
        O(nh_1).
\]
The bounds in \eqref{eq:kernel-weight-B-norms} follow similarly from
\(0\le B_{k\ell}\le1\) and symmetry of \(B\).

To study the row sums, define
\[
        G(s):=\int_{\mathbb R}K(t)K(t+s)\,dt.
\]
Because \(K\in L^2(\mathbb R)\), translation is continuous in \(L^2\),
and therefore
\[
        |G(s)-G(r)|
        \le
        \|K\|_2
        \|K(\,\cdot+s)-K(\,\cdot+r)\|_2
        \rightarrow0
        \qquad\text{as }s\to r.
\]
Thus \(G\) is continuous, and it is supported on \([-2,2]\).
By Fubini's theorem,
\[
        \int_{\mathbb R}G(s)\,ds
        =
        \left(\int_{\mathbb R}K(t)\,dt\right)^2
        =
        1.
\]
Since \(nh_1\to\infty\), the Riemann-sum approximation for the
continuous, compactly supported function \(G\) gives
\begin{equation}\label{eq:G-Riemann-sum}
        \frac{1}{nh_1}
        \sum_{r\in\mathbb Z}
        G\left(\frac{r}{nh_1}\right)
        \rightarrow1.
\end{equation}

If \(i\in I_n\), then the integration interval in
\eqref{eq:aij-convolution-form} contains \([-1,1]\), and every integer
\(r\) for which \(G(r/(nh_1))\ne0\) can be written as \(r=i-j\) for
some \(j\in\{1,\ldots,n\}\). Hence
\[
        L_i
        =
        \sum_{j\ne i}
        G\left(\frac{i-j}{nh_1}\right)
        =
        \sum_{r\in\mathbb Z\setminus\{0\}}
        G\left(\frac{r}{nh_1}\right)
        =:\ell_n.
\]
By \eqref{eq:G-Riemann-sum},
\[
        \frac{\ell_n}{nh_1}
        =
        \frac{1}{nh_1}
        \sum_{r\in\mathbb Z}
        G\left(\frac{r}{nh_1}\right)
        -
        \frac{G(0)}{nh_1}
        \rightarrow1,
\]
which proves \eqref{eq:L-interior-approximation}.

The complement of \(I_n\) has cardinality \(O(nh_1)\). Since
\(\max_iL_i=O(nh_1)\) and \(\ell_n=O(nh_1)\),
\[
        |\bar L-\ell_n|
        \le
        \frac{1}n\sum_{i\notin I_n}|L_i-\ell_n|
        =
        O(nh_1^2)
        =
        o(nh_1).
\]
This proves \eqref{eq:L-basic-bounds}. Also,
\(\sum_i(L_i-\bar L)=0\) by definition, and
\[
        \max_i|L_i-\bar L|
        \le
        \max_iL_i+|\bar L|
        =
        O(nh_1).
\]
Since \(\bar L\) minimizes \(c\mapsto\sum_i(L_i-c)^2\),
\begin{align*}
        \sum_{i=1}^n(L_i-\bar L)^2
        \le
        \sum_{i=1}^n(L_i-\ell_n)^2 
        =
        \sum_{i\notin I_n}(L_i-\ell_n)^2 
        =
        O(nh_1)O(n^2h_1^2)
        =
        O(n^3h_1^3),
\end{align*}
which proves \eqref{eq:L-centered-bounds}.

Finally, since \(0\le\bar b_k\le1\) and \(0\le\bar b\le1\), we have 
\[
        |\bar b_k-\bar b|\le1,
        \qquad
        \sum_{k=1}^n(\bar b_k-\bar b)=0.
\]
Consequently, we have  
\[
        \sum_{k=1}^n(\bar b_k-\bar b)^2\le n,
\]
which proves \eqref{eq:B-row-centered-bounds}.
\end{proof}
\subsubsection{Permutation representation under independence}
We first derive a decomposition of $\hat\tau_n^2$, and investigate both terms separately. For this purpose, we  define two random  permutations $\phi$ and $\lambda$ on $\{1, \ldots, n\}$ such that $\phi(i) := n\hat{F}_{n1}(X_i)$ and $\lambda(i) := n\hat{F}_{n2}(Y_i)$. If $X$ and $Y$ are independent, $\phi$ and $\lambda$ are independent as well and consequently $\pi := \lambda\phi^{-1}$ is a uniformly distributed random permutation on $\{1, \ldots, n\}$. For notational convenience, write $\pi_i = \pi(i)$. By \eqref{eq:tau-square-rank-representation}, after ordering the
observations according to their \(X\)-ranks,
\[
        \hat\tau_n^2
        =
        \frac{1}{n^2h_1}
        \sum_{i,j=1}^n
        a_{ij}B_{\pi_i\pi_j}.
\]
By dividing the indices into $i\neq j$ and $i = j$ we have
\begin{align*}
    \hat{\tau}^2_n &= S_n + b_{n0},
\end{align*}
where 
\begin{align*}
    S_n := \frac{1}{n^2h_1}\sum_{i = 1}^n\sum_{j \neq i}^n a_{ij}B_{\pi_i\pi_j}, \qquad b_{n0}:=\frac{1}{n^2h_1}\sum_{i = 1}^n a_{ii}B_{\pi_i\pi_i}.
\end{align*}
Note that 
\begin{align*}
    \ee[b_{n0}] = \frac{1}{n^2h_1^2}\Bigl(\sum_{i = 1}^n \int_{0}^1 K^2\Bigl(\frac{u - i/n}{h_1}\Bigr)\,du\Bigr)\Bigl(\frac{1}{n}\sum_{i = 1}^n \int_{0}^1\bar{K}^2\Bigl(\frac{v - i/n}{h_2}\Bigr)\,dv\Bigr),
\end{align*}
and
\begin{align*}
    \var(b_{n0}) = O \Bigl(\frac{1}{n^3h_1^2}\Bigr).
\end{align*}
Therefore it is enough to study the asymptotic behaviour of $S_n$. Defining 
\begin{align*}
    b_{ij} := B_{ij} - 
    \bar{b}_i - \bar{b}_j + \bar{b}
\end{align*}
with $\bar{b}_i$ and $\bar{b}$ defined in Lemma~\ref{lem:deterministic-kernel-weight-bounds} we obtain the representation
\begin{align*} 
    S_n = D_n + 2T_n - C_n,
\end{align*}
where 
\begin{align}\label{lochd4}
        D_n
        :=
        \frac{1}{n^2h_1}
        \sum_{i\ne j}
        a_{ij}b_{\pi_i\pi_j},
        \qquad
        T_n
        :=
        \frac{1}{n^2h_1}
        \sum_{i=1}^n
        L_i\bar b_{\pi_i},
        \qquad
        C_n
        :=
        \frac{\bar{b}}{n^2h_1}
        \sum_{i=1}^nL_i.
\end{align}
Since \(\ee[\bar b_{\pi_i}] = \bar{b}\), we have \(\ee[T_n] = C_n\).

\subsection{Auxiliary permutation and moment bounds}\label{subsec:null-canonical-permutation-bounds}
\subsubsection{Canonical centring}
\begin{lmm}
\label{lem:null-canonical-centering-bounds}
Assume the hypotheses of
Lemma~\ref{lem:deterministic-kernel-weight-bounds}, and suppose
\(n\ge3\). Let \(\pi\) be a uniform random permutation of
\(\{1,\ldots,n\}\), and let all expectations in this lemma be taken
with respect to \(\pi\). Under the Gaussian null model, this
expectation agrees with \(\ee_0\) for all rank-measurable quantities
appearing below. With \(b_{ij}\), \(D_n\), \(T_n\), and
\(b_{n0}\) as defined in the null decomposition, set
\[
        x_i:=L_i-\bar L,
        \qquad
        y_k:=\bar b_k-\bar b,
\]
and
\[
        d_i
        :=
        a_{ii}-\frac{1}{n}\sum_{r=1}^na_{rr},
        \qquad
        z_k
        :=
        B_{kk}-\frac{1}{n}\sum_{r=1}^nB_{rr}.
\]
For \(i\ne j\), define
\begin{equation*} 
\begin{aligned}
        b_{ij}^{\mathrm c}
        &:=
        b_{ij}-\frac{y_i+y_j}{n-2},\\
        a_{ij}^{\mathrm c}
        &:=
        a_{ij}
        -\frac{L_i+L_j}{n-2}
        +\frac{n\bar L}{(n-1)(n-2)},
\end{aligned}
\end{equation*}
and set
\[
        a_{ii}^{\mathrm c}
        =
        b_{ii}^{\mathrm c}
        :=0.
\]
Let
\[
        A^{\mathrm c}:=(a_{ij}^{\mathrm c})_{i,j=1}^n,
        \qquad
        B^{\mathrm c}:=(b_{ij}^{\mathrm c})_{i,j=1}^n.
\]
Then \(A^{\mathrm c}\) and \(B^{\mathrm c}\) are symmetric,
zero-diagonal matrices with zero row sums:
\begin{equation}\label{eq:null-canonical-row-sums}
        \sum_{j\ne i}a_{ij}^{\mathrm c}
        =
        \sum_{j\ne i}b_{ij}^{\mathrm c}
        =0,
        \qquad 1\le i\le n.
\end{equation}
Define
\begin{equation*}
        \widetilde D_n
        :=
        \frac{1}{n^2h_1}
        \sum_{i\ne j}
        a_{ij}^{\mathrm c}
        b_{\pi_i\pi_j}^{\mathrm c}.
\end{equation*}
The centred linear components admit the representations
\begin{equation}\label{eq:canonical-linear-representations}
\begin{aligned}
        T_n-\ee T_n
        &=
        \frac{1}{n^2h_1}
        \sum_{i=1}^n x_i y_{\pi_i},\\
        b_{n0}-\ee b_{n0}
        &=
        \frac{1}{n^2h_1}
        \sum_{i=1}^n d_i z_{\pi_i}.
\end{aligned}
\end{equation}
Moreover,
\begin{equation}\label{eq:D-mean-zero-canonical}
        \ee D_n=0,
\end{equation}
and the following exact decompositions hold:
\begin{align}
        D_n
        &=
        \widetilde D_n
        +
        \frac{2}{n-2}(T_n-\ee T_n),
        \label{eq:D-canonical-decomp}\\
        U_n
        &:=
        \hat\tau_n^2-\ee[\hat\tau_n^2]
        \notag\\
        &=
        \widetilde D_n
        +
        \frac{2(n-1)}{n-2}(T_n-\ee T_n)
        +
        (b_{n0}-\ee b_{n0}).
        \label{eq:U-canonical-decomp}
\end{align}
The canonically centred kernel weights satisfy, uniformly in
\(i\ne j\),
\begin{equation}\label{eq:a-can-pointwise}
        |a_{ij}^{\mathrm c}|
        \le
        C\left\{
        \bone\bigl(|i-j|\le2nh_1+2\bigr)+h_1
        \right\},
        \qquad
        |b_{ij}^{\mathrm c}|\le C.
\end{equation}
Consequently,
\begin{align}
        \|A^{\mathrm c}\|_{\mathrm{op}}
        &=
        O(nh_1),
        \label{lochd12}\\
        \|A^{\mathrm c}\|_{\mathrm F}^2
        &=
        O(n^2h_1),
        \label{lochd13}\\
        \|B^{\mathrm c}\|_{\mathrm F}
        &=
        O(n),
        \qquad
        \|B^{\mathrm c}\|_{\mathrm{op}}
        =
        O(n).
        \label{lochd14}
\end{align}
The centred row-sum vectors satisfy
\begin{equation}\label{eq:xy-norms-higher}
\begin{gathered}
        \sum_{i=1}^nx_i=0,
        \qquad
        \max_i|x_i|=O(nh_1),
        \qquad
        \|x\|_2
        =
        O(n^{3/2}h_1^{3/2}),\\
        \sum_{k=1}^ny_k=0,
        \qquad
        \max_k|y_k|\le1,
        \qquad
        \|y\|_2
        =
        O(n^{1/2}).
\end{gathered}
\end{equation}
The centred diagonal vectors satisfy
\begin{equation}\label{lochd16a}
        \sum_{i=1}^nd_i=0,
        \qquad
        \max_i|d_i|=O(1),
        \qquad
        \|d\|_2
        =
        O\bigl((nh_1)^{1/2}\bigr),
\end{equation}
and
\begin{equation}\label{lochd16b}
        \sum_{k=1}^nz_k=0,
        \qquad
        \max_k|z_k|\le1,
        \qquad
        \|z\|_2
        =
        O(n^{1/2}).
\end{equation}
\end{lmm}

\begin{proof}
We first prove the row-sum identities. From
\[
        b_{ij}
        =
        B_{ij}-\bar b_i-\bar b_j+\bar b
\]
and
\[
        \sum_{j\ne i}B_{ij}
        =
        (n-1)\bar b_i,
        \qquad
        \sum_{j=1}^n\bar b_j
        =
        n\bar b,
\]
we obtain
\begin{align*}
        \sum_{j\ne i}b_{ij}
        =
        (n-1)\bar b_i
        -(n-1)\bar b_i
        -\sum_{j\ne i}\bar b_j
        +(n-1)\bar b
        =
        \bar b_i-\bar b
        =
        y_i.
\end{align*}
Since \(\sum_jy_j=0\),
\begin{align*}
        \sum_{j\ne i}b_{ij}^{\mathrm c}
        =    y_i
        - \frac{1}{n-2}
        \biggl\{
        (n-1)y_i+\sum_{j\ne i}y_j
        \biggr\}
        =
        y_i-\frac{(n-2)y_i}{n-2}
        =0.
\end{align*}
Also, \(\sum_{j\ne i}a_{ij}=L_i\) and
\(\sum_jL_j=n\bar L\). Hence
\begin{align*}
        \sum_{j\ne i}a_{ij}^{\mathrm c}
        &=
        L_i    -  \frac{1}{n-2}
        \biggl\{
        (n-1)L_i+\sum_{j\ne i}L_j    \biggr\}
        +   \frac{n\bar L}{n-2}\\
        &=  L_i  -\frac{(n-2)L_i+n\bar L}{n-2}  +
        \frac{n\bar L}{n-2}
        =0.
\end{align*}
This proves \eqref{eq:null-canonical-row-sums}.

Using the definitions of \(T_n\) and \(\ee T_n\),
\begin{align*}
        T_n-\ee T_n
        =
        \frac{1}{n^2h_1}
        \sum_{i=1}^n
        L_i(\bar b_{\pi_i}-\bar b)
        =
        \frac{1}{n^2h_1}
        \sum_{i=1}^nL_i y_{\pi_i}.
\end{align*}
Since \(\sum_i y_{\pi_i}=0\), we may replace \(L_i\) by
\(x_i=L_i-\bar L\), proving the first identity in
\eqref{eq:canonical-linear-representations}.

Under the uniform permutation law,
\[
        \ee[b_{n0}]
        =
        \frac{1}{n^2h_1}
        \frac{1}{n}
        \left(\sum_{i=1}^na_{ii}\right)
        \left(\sum_{k=1}^nB_{kk}\right).
\]
Consequently,
\begin{align*}
        b_{n0}-\ee[b_{n0}]
        =
        \frac{1}{n^2h_1}
        \sum_{i=1}^n
        \left(
        a_{ii}-\frac{1}{n}\sum_{r=1}^na_{rr}
        \right)
        \left(
        B_{\pi_i\pi_i}-\frac{1}n\sum_{r=1}^nB_{rr}
        \right)
        =
        \frac{1}{n^2h_1}
        \sum_{i=1}^nd_i z_{\pi_i},
\end{align*}
which proves the second identity.

Furthermore,
\[
        \sum_{k\ne\ell}b_{k\ell}
        =
        \sum_{k=1}^ny_k
        =0.
\]
Since \((\pi_i,\pi_j)\) is uniform over the ordered pairs of distinct
indices whenever \(i\ne j\), it follows that \(\ee D_n=0\), proving
\eqref{eq:D-mean-zero-canonical}.

Because \(B^{\mathrm c}\) has zero row sums,
\[
        \sum_{i\ne j}
        a_{ij}b_{\pi_i\pi_j}^{\mathrm c}
        =
        \sum_{i\ne j}
        a_{ij}^{\mathrm c}b_{\pi_i\pi_j}^{\mathrm c}.
\]
Indeed, \(a_{ij}-a_{ij}^{\mathrm c}\) is the sum of a term depending
only on \(i\), a term depending only on \(j\), and a constant; each
such term vanishes when contracted with a zero-row-sum matrix.
Therefore,
\begin{align*}
        D_n
        &=
        \frac{1}{n^2h_1}
        \sum_{i\ne j}
        a_{ij}
        \left\{
        b_{\pi_i\pi_j}^{\mathrm c}
        +
        \frac{y_{\pi_i}+y_{\pi_j}}{n-2}
        \right\}\\
        &=
        \widetilde D_n
        +
        \frac{2}{n^2h_1(n-2)}
        \sum_{i=1}^nL_i y_{\pi_i}\\
        &=
        \widetilde D_n
        +
        \frac{2}{n-2}(T_n-\ee T_n),
\end{align*}
which proves \eqref{eq:D-canonical-decomp}.

The null decomposition gives
\[
        \hat\tau_n^2-\ee[\hat\tau_n^2]
        =
        D_n+2(T_n-\ee T_n)+(b_{n0}-\ee b_{n0}).
\]
Combining this identity with \eqref{eq:D-canonical-decomp} yields
\eqref{eq:U-canonical-decomp}.

We next prove the deterministic bounds. By
\eqref{eq:kernel-weight-pointwise} and
\eqref{eq:L-basic-bounds},
\[
        |a_{ij}|
        \le
        C\bone\{|i-j|\le2nh_1+2\},
        \qquad
        \max_i|L_i|+|\bar L|
        =
        O(nh_1).
\]
Substitution into the definition of \(a_{ij}^{\mathrm c}\) gives
\[
        |a_{ij}^{\mathrm c}|
        \le
        C\left\{
        \bone\bigl(|i-j|\le2nh_1+2\bigr)+h_1
        \right\}.
\]
Also, \(0\le B_{ij},\bar b_i,\bar b\le1\), so the entries
\(b_{ij}\) are uniformly bounded. Since \(|y_i|\le1\), the entries
\(b_{ij}^{\mathrm c}\) are uniformly bounded as well. This proves
\eqref{eq:a-can-pointwise}.

For every fixed row of \(A^{\mathrm c}\), there are \(O(nh_1)\)
entries in the local part of \eqref{eq:a-can-pointwise}, while the
remaining \(O(n)\) entries are \(O(h_1)\). Hence
\[
        \max_i\sum_{j=1}^n|a_{ij}^{\mathrm c}|
        =
        O(nh_1),
\]
which proves \eqref{lochd12}. Similarly,
\begin{align*}
        \|A^{\mathrm c}\|_{\mathrm F}^2
        &\le
        C\sum_{i\ne j}
        \left\{
        \bone\bigl(|i-j|\le2nh_1+2\bigr)+h_1^2
        \right\}
        =
        O(n^2h_1)+O(n^2h_1^2)
        =
        O(n^2h_1),
\end{align*}
proving \eqref{lochd13}. The estimates in \eqref{lochd14} follow
immediately from the uniform boundedness of \(b_{ij}^{\mathrm c}\).
The estimates for \(x\) and \(y\) in
\eqref{eq:xy-norms-higher} are exactly
\eqref{eq:L-centered-bounds} and
\eqref{eq:B-row-centered-bounds}.

It remains to prove the diagonal-vector estimates. A change of
variables gives
\[
        a_{ii}
        =
        \int_{-i/(nh_1)}^{(n-i)/(nh_1)}K(t)^2\,dt.
\]
Let
\[
        \kappa_2:=\int_{-1}^1K(t)^2\,dt.
\]
If \(i/n\in[h_1,1-h_1]\), then \(a_{ii}=\kappa_2\). There are only
\(O(nh_1)\) indices outside this interval, and all \(a_{ii}\) are
uniformly bounded. Hence
\[
        \frac{1}{n}\sum_{i=1}^na_{ii}
        =
        \kappa_2+O(h_1).
\]
It follows that \(d_i=O(h_1)\) for the interior indices and
\(d_i=O(1)\) for the \(O(nh_1)\) boundary indices. Therefore
\[
        \sum_{i=1}^nd_i^2
        =
        O(nh_1^2)+O(nh_1)
        =
        O(nh_1),
\]
which proves \eqref{lochd16a}. Finally, \(0\le B_{kk}\le1\) implies
\[
        |z_k|\le1,
        \qquad
        \sum_{k=1}^nz_k^2\le n.
\]
Together with the definitions of \(d_i\) and \(z_k\), which give
their zero-sum properties, this proves \eqref{lochd16b}.
\end{proof}

\subsubsection{Permutation fourth-moment bounds}
\begin{lmm}\label{lem:linear-permutation-fourth}
Let $n\geq 4$ and \(\pi\) be a uniform random permutation of \(\{1,\ldots,n\}\).
Let \(a_1,\ldots,a_n\) and \(c_1,\ldots,c_n\) be deterministic
numbers satisfying
\[
        \sum_{i=1}^n a_i=0,
        \qquad
        \sum_{k=1}^n c_k=0,
        \qquad
        \max_{1\le k\le n}|c_k|\le M.
\]
Then
\[
        \ee_\pi\left[
        \left(
        \sum_{i=1}^n a_i c_{\pi_i}
        \right)^4
        \right]
        \le
        C M^4
        \left(\sum_{i=1}^n a_i^2\right)^2,
\]
where $C$ is a universal constant.
\end{lmm}

\begin{proof}
Write
\[
        A_2:=\sum_{i=1}^n a_i^2,
        \qquad
        C_r:=\sum_{k=1}^n c_k^r.
\]
Since \(\sum_k c_k=0\) and \(\max_k|c_k|\le M\), we have
\[
        |C_2|\le nM^2,
        \qquad
        |C_4|\le nM^4.
\]
Let
\[
        S:=\sum_{i=1}^n a_i c_{\pi_i}.
\]
We expand \(S^4\) according to equality patterns among the four
indices.

First, since $\ee_\pi[c_{\pi_i}^4] = \frac{C_4}{n}
=O(M^4)$ for each fixed index \(i\), we have  
\[
        \Big |
        \sum_i a_i^4\ee_\pi[c_{\pi_i}^4]
        \Big |
        \le
        C M^4\sum_i a_i^4
        \le
        C M^4 A_2^2.
\]
Next, for \(i\ne j\),
\[
        \ee_\pi[c_{\pi_i}^3c_{\pi_j}]
        =
        \frac1{n(n-1)}
        \sum_{k\ne \ell}c_k^3c_\ell
        =
        -\frac{C_4}{n(n-1)}
        =
        O(M^4/n).
\]
Also, since \(\sum_i a_i=0\),
\[
        \sum_{i\ne j}a_i^3a_j
        =
        \sum_i a_i^3\sum_{j\ne i}a_j
        =
        -\sum_i a_i^4.
\]
Therefore the contribution of the \(3+1\) equality pattern is
\(O(M^4A_2^2)\).

For the \(2+2\) pattern, if \(i\ne j\), then
\[
        \ee_\pi[c_{\pi_i}^2c_{\pi_j}^2]
        =
        \frac{C_2^2-C_4}{n(n-1)}
        =
        O(M^4).
\]
Moreover,
\[
        \sum_{i<j}a_i^2a_j^2
        \le
        \frac12 A_2^2.
\]
Hence the \(2+2\) contribution is also \(O(M^4A_2^2)\).

For the \(2+1+1\) pattern, let \(i,j,k\) be distinct. Then
\[
\begin{aligned}
        \ee_\pi[c_{\pi_i}^2c_{\pi_j}c_{\pi_k}]
        &=
        \frac1{n(n-1)(n-2)}
        \sum_{\substack{u,v,w\\\mathrm{all\ distinct}}}
        c_u^2c_vc_w.
\end{aligned}
\]
For fixed \(u\),
\[
        \sum_{\substack{v,w\\v,w\ne u,\ v\ne w}}
        c_vc_w
        =
        \left(\sum_{v\ne u}c_v\right)^2
        -
        \sum_{v\ne u}c_v^2
        =
        2c_u^2-C_2.
\]
Thus
\[
        \sum_{\substack{u,v,w\\\mathrm{all\ distinct}}}
        c_u^2c_vc_w
        =
        2C_4-C_2^2,
\]
and so
\[
        \ee_\pi[c_{\pi_i}^2c_{\pi_j}c_{\pi_k}]
        =
        O(M^4/n).
\]
On the \(a\)-side,
\[
\begin{aligned}
\sum_i a_i^2
\sum_{\substack{j<k\\j,k\ne i}}a_ja_k
&=
\frac12
\sum_i a_i^2
\left\{
        \left(\sum_{j\ne i}a_j\right)^2
        -
        \sum_{j\ne i}a_j^2
\right\} \\
&=
\frac12
\sum_i a_i^2
\left\{
        a_i^2-(A_2-a_i^2)
\right\} \\
&=
\sum_i a_i^4-\frac12A_2^2.
\end{aligned}
\]
This is \(O(A_2^2)\). Hence the \(2+1+1\) contribution is
\(O(M^4A_2^2)\).

Finally, for four distinct indices \(i,j,k,\ell\),
\[
        \ee_\pi[c_{\pi_i}c_{\pi_j}c_{\pi_k}c_{\pi_\ell}]
        =
        \frac{
        \sum_{u,v,w,r}^{*}c_uc_vc_wc_r
        }{(n)_4},
\]
where the star denotes summation over distinct indices. Since
\(\sum_kc_k=0\), Newton's identities give
\[
        \sum_{u,v,w,r}^{*}c_uc_vc_wc_r
        =
        3C_2^2-6C_4.
\]
Therefore
\[
        \ee_\pi[c_{\pi_i}c_{\pi_j}c_{\pi_k}c_{\pi_\ell}]
        =
        O(M^4/n^2).
\]
Similarly,
\[
        \left|
        \sum_{i,j,k,\ell}^{*}a_ia_ja_ka_\ell
        \right|
        =
        |3A_2^2-6\sum_i a_i^4|
        \le
        C A_2^2.
\]
Thus the all-distinct contribution is \(O(M^4A_2^2)\).

Combining the five equality patterns gives
\[
        \ee_\pi S^4
        \le
        C M^4 A_2^2
        =
        C M^4
        \left(\sum_i a_i^2\right)^2.
\]
\end{proof}

\begin{lmm}
\label{lem:quadratic-permutation-fourth}
Let \(n\ge8\), and let \(\pi\) be a uniform random permutation of
\(\{1,\ldots,n\}\). Let \(A=(A_{ij})\) and \(H=(H_{ij})\) be
deterministic symmetric \(n\times n\) matrices such that
\[
        A_{ii}=H_{ii}=0,
        \qquad
        \sum_{j=1}^nA_{ij}=0,
        \quad 1\le i\le n.
\]
Set
\[
        Q(A,H;\pi)
        :=
        \sum_{i\ne j}A_{ij}H_{\pi_i\pi_j},
        \qquad
        \|H\|_{\max}
        :=
        \max_{k\ne\ell}|H_{k\ell}|.
\]
Then
\begin{equation*}
        \ee_\pi\left[
        Q(A,H;\pi)^4
        \right]
        \le
        C_0
        \|H\|_{\max}^4
        \|A\|_{\mathrm F}^4,
\end{equation*}
where \(C_0\) is a universal constant.
\end{lmm}

\begin{proof}
Write
\[
        Q(A,H;\pi)^4
        =
        \sum_{\substack{i_1\ne j_1,\ldots,i_4\ne j_4}}
        \prod_{e=1}^4
        A_{i_ej_e}H_{\pi_{i_e}\pi_{j_e}}.
\]
We classify the summands according to the equality pattern among the
eight endpoints
\[
        \mathcal H_4
        :=
        \{1^-,1^+,\ldots,4^-,4^+\}.
\]
Let \(\mathfrak P_4\) denote the finite collection of partitions
\(\alpha\) of \(\mathcal H_4\) such that \(e^-\) and \(e^+\) do not
belong to the same block for any \(1\le e\le4\). Let \(v(\alpha)\)
be the number of blocks of \(\alpha\).

For \(\alpha\in\mathfrak P_4\), define
\[
        W_\alpha^*(A)
        :=
        \sum_{\substack{x_B,\ B\in\alpha\\
        \text{all distinct}}}
        \prod_{e=1}^4
        A_{x_{B(e^-)}x_{B(e^+)}},
\]
where \(B(h)\) denotes the block of \(\alpha\) containing \(h\).
Define \(W_\alpha^*(H)\) analogously.

If the domain indices have equality pattern \(\alpha\), their images
under \(\pi\) form a uniform ordered \(v(\alpha)\)-tuple of distinct
indices. Consequently,
\begin{equation}\label{eq:quadratic-fourth-pattern-expansion}
        \ee_\pi[Q(A,H;\pi)^4]
        =
        \sum_{\alpha\in\mathfrak P_4}
        \frac{
        W_\alpha^*(A)W_\alpha^*(H)
        }{
        (n)_{v(\alpha)}
        },
\end{equation}
where
\[
        (n)_v:=n(n-1)\cdots(n-v+1).
\]

We first bound \(W_\alpha^*(A)\). By M\"obius inversion on the
partition lattice,
\[
        W_\alpha^*(A)
        =
        \sum_{\beta\succeq\alpha}
        \mu(\alpha,\beta)W_\beta(A),
\]
where
\[
        W_\beta(A)
        :=
        \sum_{x_B,\ B\in\beta}
        \prod_{e=1}^4
        A_{x_{B(e^-)}x_{B(e^+)}}.
\]
The number of terms and all M\"obius coefficients are bounded by
universal constants, because only partitions of eight points are
involved.

Associate with \(\beta\) the quotient multigraph whose vertices are
the blocks of \(\beta\) and whose four edges are induced by
\((e^-,e^+)\), \(1\le e\le4\). If an edge becomes a loop, the
corresponding graph sum vanishes because \(A_{ii}=0\). If the quotient
graph has a vertex of degree one, the graph sum also vanishes: summing
first over the index corresponding to that vertex gives a row sum of
\(A\).

It remains to consider quotient multigraphs with four edges, no loops,
and minimum degree at least two. Such a graph has at most four
vertices. We claim that every corresponding unrestricted graph sum is
bounded in absolute value by $\|A\|_{\mathrm F}^4$.

If there are two vertices, the graph sum is bounded by
\[
        \sum_{i,j}|A_{ij}|^4
        \le
        \left(\sum_{i,j}A_{ij}^2\right)^2.
\]

If there are three vertices, then, up to relabeling, the edge
multiplicities are either \((2,2,0)\) or \((2,1,1)\). In the first
case,
\begin{align*}
        \sum_{i,j,k}|A_{ij}|^2|A_{ik}|^2
        =
        \sum_i
        \left(\sum_jA_{ij}^2\right)^2
        \le
        \left(\sum_{i,j}A_{ij}^2\right)^2.
\end{align*}
For the second case, put
\[
        r_i
        :=
        \left(\sum_jA_{ij}^2\right)^{1/2}.
\]
Then
\begin{align*}
\sum_{i,j,k}
        |A_{ij}|^2|A_{ik}||A_{jk}|
\le
        \sum_{i,j}|A_{ij}|^2r_ir_j
\le
        \frac{1}{2}
        \sum_{i,j}|A_{ij}|^2(r_i^2+r_j^2)
=
        \sum_i r_i^4
        \le
        \left(\sum_i r_i^2\right)^2
        =
        \|A\|_{\mathrm F}^4.
\end{align*}
If there are four vertices, every vertex has degree two. The graph is
therefore either the union of two doubled edges or a four-cycle. The
first case is bounded by
\[
        \left(\sum_{i,j}A_{ij}^2\right)^2.
\]
For a four-cycle, let \(A_{+}=(|A_{ij}|)\). Since \(A_{+}\) is symmetric,
\begin{align*}
        \sum_{i,j,k,\ell}
        |A_{ij}A_{jk}A_{k\ell}A_{\ell i}|
        =
        \operatorname{tr}(A_{+}^4)
        \le
        \{\operatorname{tr}(A_{+}^2)\}^2
        =
        \|A\|_{\mathrm F}^4.
\end{align*}
It follows that
\begin{equation}\label{eq:A-restricted-fourth-graph}
        |W_\alpha^*(A)|
        \le
        C\|A\|_{\mathrm F}^4
\end{equation}
uniformly over \(\alpha\in\mathfrak P_4\).

On the other hand, \(W_\alpha^*(H)\) contains at most
\((n)_{v(\alpha)}\) summands, each bounded in absolute value by
\(\|H\|_{\max}^4\). Hence
\[
        |W_\alpha^*(H)|
        \le
        (n)_{v(\alpha)}
        \|H\|_{\max}^4.
\]
Combining this estimate with
\eqref{eq:quadratic-fourth-pattern-expansion} and
\eqref{eq:A-restricted-fourth-graph}, and using the fact that
\(\mathfrak P_4\) is finite, gives
\[
        \ee_\pi[Q(A,H;\pi)^4]
        \le
        C_0
        \|H\|_{\max}^4
        \|A\|_{\mathrm F}^4.
\]
\end{proof}

\begin{lmm}
\label{lem:null-component-fourth-moments}
Assume the hypotheses of
Lemma~\ref{lem:null-canonical-centering-bounds}. For a random variable
\(W\), write
\[
        \|W\|_p
        :=
        \{\ee|W|^p\}^{1/p}.
\]
Then
\begin{align}
        \|\widetilde D_n\|_4
        &=
        O\left(\frac{1}{nh_1^{1/2}}\right),
        \label{eq:Dtilde-fourth}\\
        \|D_n\|_4
        &=
        O\left(\frac{1}{nh_1^{1/2}}\right),
        \label{eq:D-fourth}\\
        \|T_n-\ee[T_n]\|_4
        &=
        O\left(\frac{h_1^{1/2}}{n^{1/2}}\right),
        \label{eq:T-fourth}\\
        \|b_{n0}-\ee[b_{n0}]\|_4
        &=
        O\left(\frac{1}{n^{3/2}h_1^{1/2}}\right).
        \label{eq:b-sharp-fourth}
\end{align}
If, in addition, \(nh_1^2\to\infty\), then
\begin{equation}\label{eq:U-fourth}
        \|U_n\|_4
        =
        O\left(\frac{h_1^{1/2}}{n^{1/2}}\right).
\end{equation}
\end{lmm}

\begin{proof}
Under \(\rho=0\), after ordering the observations according to their
\(X\)-ranks, the corresponding \(Y\)-rank vector
\(\pi=(\pi_1,\ldots,\pi_n)\) is a uniform random permutation.

Let
\[
        A^{\mathrm c}
        =
        (a_{ij}^{\mathrm c})_{i,j=1}^n,
        \qquad
        B^{\mathrm c}
        =
        (b_{ij}^{\mathrm c})_{i,j=1}^n
\]
be the canonically centred matrices from
Lemma~\ref{lem:null-canonical-centering-bounds}. Then
\[
        \widetilde D_n
        =
        \frac{1}{n^2h_1}
        Q(A^{\mathrm c},B^{\mathrm c};\pi).
\]
The matrices have zero diagonal and zero row sums. Moreover, by
\eqref{eq:a-can-pointwise} and \eqref{lochd13},
\[
        \max_{k\ne\ell}|b_{k\ell}^{\mathrm c}|=O(1),
        \qquad
        \|A^{\mathrm c}\|_{\mathrm F}^2
        =
        O(n^2h_1).
\]
Lemma~\ref{lem:quadratic-permutation-fourth} therefore gives
\begin{align*}
        \ee_0|\widetilde D_n|^4
        \le
        \frac{C}{n^8h_1^4}
        \|A^{\mathrm c}\|_{\mathrm F}^4
        =
        O\left(
        \frac{n^4h_1^2}{n^8h_1^4}
        \right)
        =
        O\left(\frac{1}{n^4h_1^2}\right).
\end{align*}
Taking fourth roots proves \eqref{eq:Dtilde-fourth}.

Next, by \eqref{eq:canonical-linear-representations},
\[
        T_n-\ee[T_n]
        =
        \frac{1}{n^2h_1}
        \sum_{i=1}^n x_i y_{\pi_i}.
\]
The vectors \(x\) and \(y\) satisfy
\[
        \sum_i x_i=\sum_k y_k=0,
        \qquad
        \max_k|y_k|\le1,
        \qquad
        \sum_{i=1}^nx_i^2
        =
        O(n^3h_1^3).
\]
Applying Lemma~\ref{lem:linear-permutation-fourth}, we obtain
\begin{align*}
        \ee|T_n-\ee[T_n]|^4
        \le
        \frac{C}{n^8h_1^4}
        \left(\sum_{i=1}^nx_i^2\right)^2
        =
        O\left(
        \frac{n^6h_1^6}{n^8h_1^4}
        \right)
        =
        O\left(\frac{h_1^2}{n^2}\right).
\end{align*}
This proves \eqref{eq:T-fourth}. Similarly,
\[
        b_{n0}-\ee[b_{n0}]
        =
        \frac{1}{n^2h_1}
        \sum_{i=1}^n d_i z_{\pi_i},
\]
where
\[
        \sum_i d_i=\sum_k z_k=0,
        \qquad
        \max_k|z_k|\le1,
        \qquad
        \sum_{i=1}^nd_i^2
        =
        O(nh_1).
\]
Another application of
Lemma~\ref{lem:linear-permutation-fourth} gives
\begin{align*}
        \ee|b_{n0}-\ee[b_{n0}]|^4
        \le
        \frac{C}{n^8h_1^4}
        \left(\sum_{i=1}^nd_i^2\right)^2
        =
        O\left(
        \frac{n^2h_1^2}{n^8h_1^4}
        \right)
        =
        O\left(\frac{1}{n^6h_1^2}\right).
\end{align*}
Taking fourth roots proves \eqref{eq:b-sharp-fourth}.

By the canonical decomposition
\eqref{eq:D-canonical-decomp},
\[
        D_n
        =
        \widetilde D_n
        +
        \frac{2}{n-2}(T_n-\ee[T_n]).
\]
Minkowski's inequality and the preceding estimates yield
\begin{align*}
        \|D_n\|_4
        \le
        \|\widetilde D_n\|_4
        +
        \frac{2}{n-2}\|T_n-\ee[T_n]\|_4
        =
        O\left(\frac{1}{nh_1^{1/2}}\right)
        +
        O\left(\frac{h_1^{1/2}}{n^{3/2}}\right)
        =
        O\left(\frac{1}{nh_1^{1/2}}\right),
\end{align*}
which proves \eqref{eq:D-fourth}.

Finally, by \eqref{eq:U-canonical-decomp},
\[
        U_n
        =
        \widetilde D_n
        +
        \frac{2(n-1)}{n-2}(T_n-\ee[T_n])
        +
        (b_{n0}-\ee[b_{n0}]).
\]
Hence
\[
\begin{aligned}
        \|U_n\|_4
        &\le
        O\left(\frac{1}{nh_1^{1/2}}\right)
        +
        O\left(\frac{h_1^{1/2}}{n^{1/2}}\right)
        +
        O\left(\frac{1}{n^{3/2}h_1^{1/2}}\right).
\end{aligned}
\]
Finally, since \(nh_1^2\to\infty\),
\[
        \frac{1/(nh_1^{1/2})}
        {h_1^{1/2}/n^{1/2}}
        =
        \frac{1}{n^{1/2}h_1}
        =
        o(1),
\]
and
\[
        \frac{1/(n^{3/2}h_1^{1/2})}
        {h_1^{1/2}/n^{1/2}}
        =
        \frac{1}{nh_1}
        =
        o(1).
\]
This proves \eqref{eq:U-fourth}.
\end{proof}

\subsection{Proof of the null central limit theorem}\label{sec:proof-null-clt}
To prepare for the proof of Theorem~\ref{thm:taunCLT}, recall that $\hat{r}_n = 6\hat{\tau}_n^2 - 2$. Therefore, it suffices to establish the asymptotic normality of $\hat{\tau}_n^2$, which is the statement of the following result.

\begin{thm}\label{thm:BnCLT}
  Under the assumptions of Theorem \ref{thm:taunCLT} we have 
    \begin{align*}
        \sqrt{\frac{n}{h_1}}(\hat{\tau}^2_n - \hat{b}_n)\stackrel{\mathcal{D}}{\rightarrow} \mathcal{N}(0, 4\sigma_0^2),
    \end{align*}
    where 
    \begin{align*}
    \hat{b}_n :=& \frac{1}{n^2h_1^2} \sum_{i\neq j}\int_{0}^{1}K\Bigl(\frac{u - i/n}{h_1}\Bigr)K\Bigl(\frac{u - j/n}{h_1}\Bigr)\,du  \\
    & ~~~~~~~~~~~\times \frac{1}{n(n - 1)} \sum_{i\neq j}\int_{0}^{1}\bar{K}\Bigl(\frac{v - i/n}{h_2}\Bigr)\bar{K}\Bigl(\frac{v - j/n}{h_2}\Bigr)\,dv  \\
    & ~~~~~~~~~~~ + \frac{1}{n^2h_1^2}\Bigl(\sum_{i = 1}^n \int_{0}^1 K^2\Bigl(\frac{u - i/n}{h_1}\Bigr)\,du\Bigr)\Bigl(\frac{1}{n}\sum_{i = 1}^n \int_{0}^1\bar{K}^2\Bigl(\frac{v - i/n}{h_2}\Bigr)\,dv\Bigr),
    \end{align*}
       and $\sigma_0^2$ is defined in \eqref{eq:sigma0}.
\end{thm}

\subsubsection{The linear permutation component}
\begin{lmm}\label{lmm:Tn}
    For continuous and independent $X$ and $Y$ we have 
    \begin{align*}
        \sqrt{\frac{n}{h_1}}\left(T_n - \ee[T_n]\right)\stackrel{\mathcal D}{\rightarrow}\mathcal{N}(0, \sigma_0^2),
    \end{align*}
    where 
    \begin{align*}
        \lefteqn{\ee[T_n] = \frac{1}{n^2h_1^2}\sum_{i\neq j}\int_{0}^{1}K\Bigl(\frac{u - i/n}{h_1}\Bigr)K\Bigl(\frac{u - j/n}{h_1}\Bigr)\,du\times}\\
        &\qquad\qquad\qquad\qquad\frac{1}{n(n  -1)}\sum_{i\neq j}\int_{0}^{1}\bar{K}\Bigl(\frac{v - i/n}{h_2}\Bigr)\bar{K}\Bigl(\frac{v - j/n}{h_2}\Bigr)\,dv,
    \end{align*}
and $\sigma_0^2$ is defined in \eqref{eq:sigma0}.
\end{lmm}

\begin{proof} [Proof of Lemma~\ref{lmm:Tn}]
Under independence, \(\pi\) is a uniform random permutation of
\(\{1,\ldots,n\}\). Set
\[
        x_i:=L_i-\bar L,
        \qquad
        y_k:=\bar b_k-\bar b.
\]
Since \(\sum_i x_i=\sum_k y_k=0\),
\begin{equation}\label{eq:Tn-centered-linear}
        T_n-\ee[T_n]
        =
        \frac{1}{n^2h_1}
        \sum_{i=1}^n L_i y_{\pi_i}
        =
        \frac{1}{n^2h_1}
        \sum_{i=1}^n x_i y_{\pi_i}.
\end{equation}
Moreover, \(\ee[\bar b_{\pi_i}]=\bar b\), and hence
\[
        \ee[T_n]
        =
        \frac{\bar b}{n^2h_1}
        \sum_{i=1}^nL_i.
\]
Substituting the definitions of \(L_i\), \(\bar b\), \(a_{ij}\), and
\(B_{k\ell}\) gives the expression for \(\ee[T_n]\) stated in the
lemma.

We first determine the two quadratic quantities appearing in the
variance of the linear permutation statistic in
\eqref{eq:Tn-centered-linear}. Since \(K\) is supported on
\([-1,1]\), uniformly in \(k\ne\ell\),
\[
        B_{k\ell}
        =
        1-\frac{\max\{k,\ell\}}{n}
        +O(h_2).
\]
Indeed, the product defining \(B_{k\ell}\) differs from
\(\bone\{v\ge \max(k,\ell)/n\}\) only on a set of Lebesgue measure
\(O(h_2)\). Consequently, uniformly in \(k\),
\begin{align*}
        \bar b_k
        &=
        \frac{1}{n-1}
        \left\{
        (k-1)\left(1-\frac{k}{n}\right)
        +
        \sum_{\ell=k+1}^n
        \left(1-\frac{\ell}{n}\right)
        \right\}
        +O(h_2)\\
        &=
        \frac{1-(k/n)^2}{2}
        +O(h_2)+O(1/n).
\end{align*}
It follows that
\[
        \bar b=\frac{1}{3}+o(1)
\]
and, by a Riemann-sum approximation,
\begin{equation}\label{eq:y-square-asymptotic-Tn}
\begin{aligned}
        \frac{1}{n}\sum_{k=1}^ny_k^2
        &\rightarrow
        \int_0^1
        \left\{
        \frac{1-t^2}{2}-\frac{1}{3}
        \right\}^2dt
        =\frac{1}{45}.
\end{aligned}
\end{equation}
We next obtain the sharp asymptotic order of
\(\sum_i x_i^2\). Define, for \(t\ge0\),
\[
        \Psi_K(t)
        :=
        \int_{-t}^1K(v)\bar K(v+t)\,dv.
\]
Since \(K\) is supported on \([-1,1]\), we have
\(\Psi_K(t)=1\) for every \(t\ge2\). Let
\[
        m_n:=\lceil2nh_1\rceil,
        \qquad
        I_n:=\{m_n+1,\ldots,n-m_n\},
\]
and let \(\ell_n\) be as in
Lemma~\ref{lem:deterministic-kernel-weight-bounds}. We claim that,
uniformly for \(1\le i\le m_n\),
\begin{equation}\label{eq:left-boundary-profile-Tn}
        \frac{L_i}{nh_1}
        =
        \Psi_K\left(\frac{i}{nh_1}\right)+o(1).
\end{equation}
To see this, put \(t_i=i/(nh_1)\) and
\[
        H_t(s)
        :=
        \int_{-t}^1K(v)K(v+s)\,dv,
        \qquad 0\le t\le3.
\]
For \(i\le m_n\), the change of variables used in
\eqref{eq:aij-convolution-form} gives, for all sufficiently large
\(n\),
\[
        L_i
        =
        \sum_{j\ne i}
        H_{t_i}\left(\frac{i-j}{nh_1}\right).
\]
The family \(\{H_t:0\le t\le3\}\) is uniformly bounded, supported on
\([-2,2]\), and equicontinuous, since
\[
        |H_t(s)-H_t(r)|
        \le
        \|K\|_2
        \|K(\,\cdot+s)-K(\,\cdot+r)\|_2.
\]
Hence the corresponding Riemann-sum approximation is uniform in
\(t\in[0,3]\). Since excluding the term \(j=i\) changes the
normalized sum by \(O((nh_1)^{-1})\),
\begin{align*}
        \frac{L_i}{nh_1}
        &=
        \int_{-\infty}^{t_i}H_{t_i}(s)\,ds+o(1)\\
        &=
        \int_{-t_i}^1
        K(v)
        \left\{
        \int_{-\infty}^{t_i}K(v+s)\,ds
        \right\}dv+o(1)\\
        &=
        \int_{-t_i}^1K(v)\bar K(v+t_i)\,dv+o(1),
\end{align*}
which proves \eqref{eq:left-boundary-profile-Tn}. An analogous
calculation at the right boundary, using the symmetry of \(K\), gives,
uniformly for \(n-m_n+1\le i\le n\),
\begin{equation}\label{eq:right-boundary-profile-Tn}
        \frac{L_i}{nh_1}
        =
        \Psi_K\left(\frac{n-i}{nh_1}\right)+o(1).
\end{equation}

By Lemma~\ref{lem:deterministic-kernel-weight-bounds},
\(L_i=\ell_n\) for \(i\in I_n\) and
\(\ell_n/(nh_1)\to1\). Moreover, by
\eqref{eq:L-basic-bounds},
\[
        |\bar L-\ell_n|
        \le
        \frac{1}{n}\sum_{i\notin I_n}|L_i-\ell_n|
        =O(nh_1^2).
\]
Therefore
\begin{align*}
        \sum_{i=1}^n(L_i-\bar L)^2
        =
        \sum_{i=1}^n(L_i-\ell_n)^2
        -n(\bar L-\ell_n)^2
        =
        \sum_{i\notin I_n}(L_i-\ell_n)^2
        +o(n^3h_1^3).
\end{align*}
Applying \eqref{eq:left-boundary-profile-Tn} and
\eqref{eq:right-boundary-profile-Tn}, followed by a Riemann-sum
approximation, yields
\begin{equation}\label{eq:x-square-asymptotic-Tn}
        \frac{1}{n^3h_1^3}
        \sum_{i=1}^n x_i^2
        \rightarrow
        2\int_0^2
        \left\{
        1-\Psi_K(t)
        \right\}^2dt.
\end{equation}
The function \(\Psi_K\) is continuous by dominated convergence, and
the limit on the right-hand side is strictly positive. Indeed, symmetry
of \(K\) gives \(\bar K(0)=1/2\), and hence
\[
        \Psi_K(0)
        =
        \int_0^1K(v)\bar K(v)\,dv
        =
        \frac{1}{2}\left\{\bar K(1)^2-\bar K(0)^2\right\}
        =\frac38.
\]

For centred deterministic sequences \((x_i)\) and \((y_k)\), the
variance formula for a uniform random permutation gives
\begin{equation}\label{eq:Tn-variance-centered}
        \var(T_n)
        =
        \frac{
        \left(\sum_{i=1}^nx_i^2\right)
        \left(\sum_{k=1}^ny_k^2\right)
        }{
        n^4h_1^2(n-1)
        }.
\end{equation}
Combining \eqref{eq:y-square-asymptotic-Tn},
\eqref{eq:x-square-asymptotic-Tn}, and
\eqref{eq:Tn-variance-centered}, we obtain
\begin{equation}\label{eq:Tn-variance-limit}
\begin{aligned}
        \frac{n}{h_1}\var(T_n)
        &\rightarrow
        \frac{2}{45}
        \int_0^2
        \left(
        1-
        \int_{-t}^1K(v)\bar K(v+t)\,dv
        \right)^2dt
        =\sigma_0^2.
\end{aligned}
\end{equation}

It remains to prove asymptotic normality. By
\eqref{eq:x-square-asymptotic-Tn},
\eqref{eq:y-square-asymptotic-Tn}, and the bounds
\eqref{eq:L-centered-bounds}--\eqref{eq:B-row-centered-bounds},
\begin{align*}
n
\frac{\max_{1\le i\le n}x_i^2}{\sum_{i=1}^nx_i^2}
\frac{\max_{1\le k\le n}y_k^2}{\sum_{k=1}^ny_k^2}
=
O\left(\frac{1}{nh_1}\right)
\rightarrow0.
\end{align*}
Therefore Theorem~4 of \citet{hoeffding1951combinatorial}, applied to
\(\sum_i x_i y_{\pi_i}\), gives
\[
        \frac{T_n-\ee[T_n]}{\sqrt{\var(T_n)}}
        \stackrel{\mathcal D}{\rightarrow}
        \mathcal{N}(0,1).
\]
The assertion now follows from \eqref{eq:Tn-variance-limit} and
Slutsky's theorem.
\end{proof}

\subsubsection{The degenerate quadratic component}
\begin{lmm}\label{lmm:Dn}
    For continuous and independent $X$ and $Y$ we have 
    \begin{align*}
        \ee[D_n] = 0, \qquad\var(D_n  ) = O\Bigl(\frac{1}{n^2h_1}\Bigr).
    \end{align*}
\end{lmm}

\begin{proof}[Proof of Lemma~\ref{lmm:Dn}]
Since $    \sum_{k\ne\ell}b_{k\ell}=0$
and \((\pi_i,\pi_j)\) is uniform over the ordered pairs of distinct
indices, we have $\ee[D_n]=0$.
Moreover, Lemma~\ref{lem:null-component-fourth-moments} gives
\[
        \|D_n\|_4
        =
        O\left(\frac{1}{nh_1^{1/2}}\right).
\]
Consequently,
\[
        \var(D_n)
        =
        \|D_n\|_2^2
        \le
        \|D_n\|_4^2
        =
        O\left(\frac{1}{n^2h_1}\right).
\]
\end{proof}

\subsubsection{Completion of the proof}
\begin{proof}[Proof of Theorem~\ref{thm:BnCLT}]
Using Lemma~\ref{lmm:Tn} and Lemma~\ref{lmm:Dn} we can complete the proof of Theorem \ref{thm:BnCLT}. More specifically, we  have
\begin{align*}
    \ee[S_n] = 2\ee[T_n] - C_n = \ee[T_n],
\end{align*}
and
\begin{align*}
    \sqrt{\frac{n}{h_1}}(S_n - \ee[S_n]) = \sqrt{\frac{n}{h_1}}D_n + 2\sqrt{\frac{n}{h_1}}(T_n - \ee[T_n]).
\end{align*}
Since $\var(D_n) = O(1/(n^2h_1))$, and $nh_1^2\to\infty$, we have $\sqrt{n/h_1}D_n \rightarrow 0$ in probability, which gives 
\begin{align*}
    \sqrt{\frac{n}{h_1}}(S_n - \ee[S_n]) \stackrel{\mathcal{D}}{\rightarrow}\mathcal{N}(0, 4\sigma_0^2)
\end{align*}
by Lemma~\ref{lmm:Dn} and Lemma~\ref{lmm:Tn}. Since
\[
        \hat\tau_n^2-\hat b_n
        =
        \{S_n-\ee S_n\}
        +
        \{b_{n0}-\ee[b_{n0}]\},
\]
and 
\[
        b_n = 6\hat b_n - 2,
\]
it remains to show that the centred diagonal term is negligible. By \eqref{eq:b-sharp-fourth}
\[
\var(b_{n0}) = O\left(\frac{1}{n^3h_1^2}\right).
\]
Therefore the assertion follows from the preceding limit for \(S_n\)
and Slutsky's theorem.  Moreover, the preceding bounds imply
\begin{equation}\label{eq:null-variance-order-centered}
        \var_0(\hat\tau_n^2)
        =
        O\left(\frac{h_1}{n}\right).
\end{equation}
\end{proof}

\begin{proof}[Proof of Theorem~\ref{thm:taunCLT}]
Since $\hat r_n = 6\hat\tau_n^2 - 2$, Theorem~\ref{thm:BnCLT} implies
\[
\sqrt{\frac{n}{h_1}} \bigl(\hat r_n - b_n\bigr)
    = 6 \sqrt{\frac{n}{h_1}}(\hat\tau_n^2 - \hat{b}_n)
    \stackrel{\mathcal{D}}{\rightarrow} \mathcal{N}(0, 144\sigma_0^2).
\]
This completes the proof.
\end{proof}

\section{Proofs of the results in Section \ref{sec3}}
\label{sec:proofs-gaussian-local}

In this section we prove the results underlying the local power analysis under Gaussian copula alternatives. We first collect the required Gaussian analytic preliminaries in Section~\ref{subsec:gaussian-preliminaries}. We then establish the mean expansion and variance stability in Sections~\ref{subsec:mean-expansion} and~\ref{subsec:variance-stability}, respectively, before proving the central limit theorem under local alternatives in Section~\ref{subsec:gaussian-local-clt}. Together, these results yield the limiting local power stated in Section~\ref{sec3}. Finally, Section~\ref{subsec:superlocal-power} proves consistency of the test for alternatives in the local range specified in Proposition~\ref{prop:Gaussian-superlocal-power}.

For $Z\sim\mathcal{N}(0,1)$, let $\Phi(z) = \pp(Z\leq z)$ be the cumulative distribution function of $Z$ and set $q(u):=\Phi^{-1}(u)$ to be the quantile function. For \(1\le k\le n\), set
\[
        q_{k,n}:=
        q\left(\frac{k}{n+1}\right).
\]
\subsection{Gaussian analytic preliminaries}\label{subsec:gaussian-preliminaries}
\subsubsection{Gaussian rank-score identities}
\begin{lmm}
\label{lem:Gaussian-rank-score-identities}
Let \(Z_1,\ldots,Z_n\) be i.i.d. $\mathcal{N}(0,1)$, and let
\[
    Z_{1:n}\le\cdots\le Z_{n:n}
\]
be the corresponding order statistics. Define
\begin{equation}\label{eq:rank-scores-centered}
    \mu_k:=\ee Z_{k:n},
    \qquad
    \beta_k:=\ee(Z_{k:n}^2-1),
    \qquad
    m_{k\ell}:=\ee(Z_{k:n}Z_{\ell:n}).
\end{equation}
Then, for every \(i,k\in\{1,\ldots,n\}\),
\begin{equation}\label{eq:score-identities-centered}
\begin{gathered}
    \sum_{k=1}^n\mu_k=0,
    \qquad
    \sum_{k=1}^n\beta_k=0,
    \qquad
    \sum_{j=1}^n m_{ij}=1,
    \qquad
    \sum_{\substack{1\le i,j\le n\\ i\ne j}}m_{ij}=0,
    \qquad
    \sum_{\ell\ne k}m_{k\ell}=-\beta_k,
    \\[0.4em]
    \sum_{k=1}^n\mu_k^2=O(n),
    \qquad
    \sum_{k=1}^n\beta_k^2=O(n),
    \qquad
    \sum_{k,\ell=1}^n m_{k\ell}^2=O(n^2),
    \\[0.4em]
    \sum_{i=1}^n|\mu_i|=O(n),
    \qquad
    \sum_{i=1}^n|\beta_i|=O(n).
\end{gathered}
\end{equation}
More precisely,
\[
    \sum_{k=1}^n\mu_k^2\le n,
    \qquad
    \sum_{k=1}^n\beta_k^2\le 2n,
    \qquad
    \sum_{k,\ell=1}^n m_{k\ell}^2\le n^2,
\]
and consequently
\[
    \sum_{k=1}^n|\mu_k|\le n,
    \qquad
    \sum_{k=1}^n|\beta_k|\le\sqrt{2}\,n.
\]
\end{lmm}

\begin{proof}[Proof of Lemma~\ref{lem:Gaussian-rank-score-identities}]
Since ordering only permutes the observations, we have the pointwise
identities
\[
    \sum_{k=1}^n Z_{k:n}
    =
    \sum_{r=1}^n Z_r,
    \qquad
    \sum_{k=1}^n Z_{k:n}^2
    =
    \sum_{r=1}^n Z_r^2.
\]
Taking expectations in the first identity gives
\[
    \sum_{k=1}^n\mu_k
    =
    \ee\left[\sum_{r=1}^nZ_r\right]
    =
    0.
\]
Similarly,
\begin{align*}
    \sum_{k=1}^n\beta_k
    =
    \ee\left[ \sum_{k=1}^n(Z_{k:n}^2-1)
    \right] = \ee\left[\sum_{r=1}^n(Z_r^2-1)\right] = 0.
\end{align*}

We next prove the row-sum identity for \(m_{ij}\). Let $\bar Z:=n^{-1}\sum_{r=1}^nZ_r$ and note that $Z$ and $\bar Z$ are jointly Gaussian.

For every \(r\),
\begin{align*}
    \cov(\bar Z,Z_r-\bar Z)
    =
    \cov(\bar Z,Z_r)
    -
    \var(\bar Z)
    =
    \frac{1}{n}-\frac{1}{n}
    =
    0.
\end{align*}
Hence \(\bar Z\) is independent of \(Z - \bar Z\).

For each fixed \(i\),
\begin{align*}
    \sum_{j=1}^n m_{ij}
    =
    \ee\left[
        Z_{i:n}\sum_{j=1}^nZ_{j:n}
    \right] 
    =
    n\ee\left[
        (\bar Z+Z_{i:n} - \bar Z)\bar Z
    \right] 
    =
    n\var(\bar Z)
    =
    1,
\end{align*}
where we have used independence of \(Z_{i:n} - \bar Z\) and \(\bar Z\).

By definition,
\[
    m_{kk}
    =
    \ee Z_{k:n}^2
    =
    \beta_k+1.
\]
It follows from the row-sum identity that
\begin{align*}
    \sum_{\ell\ne k}m_{k\ell}
    =
    \sum_{\ell=1}^nm_{k\ell}-m_{kk} 
    =
    1-(\beta_k+1)
    =
    -\beta_k.
\end{align*}
Summing this identity over \(k\) and using
\(\sum_k\beta_k=0\), we obtain
\[
    \sum_{\substack{1\le k,\ell\le n\\k\ne\ell}}
    m_{k\ell}
    =
    -\sum_{k=1}^n\beta_k
    =
    0.
\]

It remains to prove the norm bounds. Jensen's inequality gives
\[
    \mu_k^2
    =
    (\ee Z_{k:n})^2
    \le
    \ee Z_{k:n}^2.
\]
Therefore,
\begin{align*}
    \sum_{k=1}^n\mu_k^2
    \le
    \ee\left[
        \sum_{k=1}^nZ_{k:n}^2
    \right] 
    =
    \ee\left[
        \sum_{r=1}^nZ_r^2
    \right]
    =
    n.
\end{align*}
Applying Jensen's inequality once more,
\[
    \beta_k^2
    =
    \{\ee(Z_{k:n}^2-1)\}^2
    \le
    \ee(Z_{k:n}^2-1)^2.
\]
Since ordering only permutes the sample,
\begin{align*}
    \sum_{k=1}^n\beta_k^2
    \le
    \ee\left[
        \sum_{k=1}^n(Z_{k:n}^2-1)^2
    \right] 
    =
    \ee\left[
        \sum_{r=1}^n(Z_r^2-1)^2
    \right] 
    =
    n\ee(Z_1^2-1)^2
    =
    2n.
\end{align*}
For \(m_{k\ell}\), the Cauchy--Schwarz inequality gives
\[
    m_{k\ell}^2
    =
    \{\ee(Z_{k:n}Z_{\ell:n})\}^2
    \le
    \ee Z_{k:n}^2\,
    \ee Z_{\ell:n}^2.
\]
Hence
\begin{align*}
    \sum_{k,\ell=1}^nm_{k\ell}^2
    \le
    \left(
        \sum_{k=1}^n\ee Z_{k:n}^2
    \right)^2 
    =
    \left(
        \ee\sum_{r=1}^nZ_r^2
    \right)^2
    =
    n^2.
\end{align*}
Finally, by the Cauchy--Schwarz inequality,
\begin{align*}
    \sum_{k=1}^n|\mu_k|
    \le
    \sqrt n
    \left(\sum_{k=1}^n\mu_k^2\right)^{1/2}
    \le n,
\qquad
    \sum_{k=1}^n|\beta_k|
    \le
    \sqrt n
    \left(\sum_{k=1}^n\beta_k^2\right)^{1/2}
    \le
    \sqrt{2}\,n.
\end{align*}
This proves all the assertions.
\end{proof}
\subsubsection{Hermite chaos and Mehler's identity}
\label{sec:hermite_mehler}

In this section we collect the basic facts on Hermite chaos and Mehler's formula that are used in the proofs. We formulate the main results
directly for functions taking values in a finite-dimensional Hilbert space.
For general accounts of Gaussian Hilbert spaces and Wiener chaos, we refer
to \citet{Janson1997} and \citet{NourdinPeccati2012}.

Let $\gamma_d$ denote the standard Gaussian measure on $\mathbb{R}^d$ and
let $Z=(Z_1,\ldots,Z_d)\sim \mathcal{N}(0,I_d)$ denote a $d$-dimensional standard normal distributed random vector.
For $m\geq 0$, let
\[
H_m(x)
=
(-1)^m e^{x^2/2}
\frac{d^m}{dx^m}e^{-x^2/2}
\]
denote the $m$th probabilists' Hermite polynomial; thus,
\[
H_0(x)=1,\qquad H_1(x)=x,\qquad H_2(x)=x^2-1.
\]
For a multi-index
$\alpha=(\alpha_1,\ldots,\alpha_d)\in\mathbb{N}_0^d$, define
\[
H_\alpha(x)
=
\prod_{j=1}^d H_{\alpha_j}(x_j),
\qquad
|\alpha|
=
\sum_{j=1}^d \alpha_j,
\qquad
\alpha!
=
\prod_{j=1}^d \alpha_j!.
\]
Independence of the coordinates of $Z$ gives
\begin{align} \label{lochd18}
\mathbb{E}\bigl[H_\alpha(Z)H_\beta(Z)\bigr]
=
\begin{cases}
\alpha!, & \alpha=\beta,\\
0, & \alpha\neq\beta.
\end{cases}
\end{align}
Let $\mathcal H$ be a finite-dimensional real Hilbert space with inner
product $\langle\cdot,\cdot\rangle_{\mathcal H}$ and norm
$\|\cdot\|_{\mathcal H}$. We consider the Hilbert space of $\mathcal H$-valued square-integrable functions
\[
L^2(\gamma_d;\mathcal H)
=
\left\{
F:\mathbb{R}^d\to\mathcal H:
\mathbb{E}\|F(Z)\|_{\mathcal H}^2<\infty
\right\},
\]
equipped with the inner product
\[
\langle F,G\rangle_{L^2(\gamma_d;\mathcal H)}
=
\mathbb{E}\bigl[
\langle F(Z),G(Z)\rangle_{\mathcal H}
\bigr].
\]
Let $e_1,\ldots,e_D$ be an orthonormal basis of $\mathcal H$, and write
\[
F_j(x)=\langle F(x),e_j\rangle_{\mathcal H},
\qquad j=1,\ldots,D.
\]
If $P_q$ denotes the orthogonal projection in $L^2(\gamma_d)$ onto the
$q$th homogeneous Hermite chaos
\[
\operatorname{span}
\bigl\{
H_\alpha:|\alpha|=q
\bigr\},
\]
we define the corresponding projection for $\mathcal H$-valued functions by
\[
P_qF
=
\sum_{j=1}^D (P_qF_j)e_j.
\]
This definition is independent of the choice of the orthonormal basis.
Equivalently, under the canonical identification
\[
L^2(\gamma_d;\mathcal H)
\cong
L^2(\gamma_d)\,{\otimes}\,\mathcal H,
\]
the projection is given by $P_q\otimes I_{\mathcal H}$. Every $F\in L^2(\gamma_d;\mathcal H)$ admits the orthogonal decomposition
\[
F
=
\sum_{q=0}^\infty P_qF
\qquad
\text{in }L^2(\gamma_d;\mathcal H),
\]
and Parseval's identity gives
\[
\|F\|_{L^2(\gamma_d;\mathcal H)}^2
=
\sum_{q=0}^\infty
\|P_qF\|_{L^2(\gamma_d;\mathcal H)}^2.
\]

We next record the Hilbert-space-valued version of Mehler's identity,
including the derivative estimate used throughout this Appendix.

\begin{prop}
\label{prop:hilbert_mehler}
Let $\mathcal H$ be a finite-dimensional real Hilbert space and let
\[
F,G\in L^2(\gamma_d;\mathcal H).
\]
Suppose that $X,Y\sim \mathcal{N}(0,I_d)$ are jointly Gaussian and satisfy $
\mathbb{E}[X_iY_j]
= \rho\,\bone\{i=j\}.
$
Then, for $|\rho|<1$,
\begin{align}
  \label{lochd19}  
\Psi_{F,G}(\rho)
:=
\mathbb{E}_\rho
\bigl[
\langle F(X),G(Y)\rangle_{\mathcal H}
\bigr]
=\sum_{q=0}^\infty \rho^q
\langle P_qF,P_qG\rangle_{L^2(\gamma_d;\mathcal H)}.
\end{align}
Consequently, for every integer $r\geq 0$,
\begin{align}
    \label{lochd18a}
\Psi_{F,G}^{(r)}(0)
=
r!\,
\langle P_rF,P_rG\rangle_{L^2(\gamma_d;\mathcal H)}.
\end{align}
In particular,
\[
\big|
\Psi_{F,G}^{(r)}(0)
\big|
\leq
r!\,
\|F\|_{L^2(\gamma_d;\mathcal H)}
\|G\|_{L^2(\gamma_d;\mathcal H)}.
\]
Moreover, for every $a\in(0,1)$,
\[
\big|
\Psi_{F,G}^{(r)}(\rho)
-
\Psi_{F,G}^{(r)}(0)
\big|
\leq
C_{r,a}|\rho|\,
\|F\|_{L^2(\gamma_d;\mathcal H)}
\|G\|_{L^2(\gamma_d;\mathcal H)},
\qquad
|\rho|\leq a.
\]
\end{prop}

\begin{proof}[Proof of Proposition~\ref{prop:hilbert_mehler}]
Let $e_1,\ldots,e_D$ be an orthonormal basis of $\mathcal H$ and write
\[
F_j(x)=\langle F(x),e_j\rangle_{\mathcal H},
\qquad
G_j(x)=\langle G(x),e_j\rangle_{\mathcal H},
\qquad j=1,\ldots,D.
\]
For each $j$, the scalar-valued functions $F_j$ and $G_j$ admit the
Hermite expansions in $L^2(\gamma_d)$
\[
F_j
=
\sum_{\alpha\in\mathbb N_0^d}
f_{j,\alpha}H_\alpha,
\qquad
G_j
=
\sum_{\alpha\in\mathbb N_0^d}
g_{j,\alpha}H_\alpha,
\]
where
\[
f_{j,\alpha}
=
\frac{\mathbb{E} [F_j(Z)H_\alpha(Z)]}{\alpha!},
\qquad
g_{j,\alpha}
=
\frac{\mathbb{E} [G_j(Z)H_\alpha(Z)]}{\alpha!}.
\]

If  $f_\rho$ denotes the density of  the bivariate centred standard normal distribution with variances $1$ and correlation $\rho$, we have 
\begin{align}
    \label{mehler}
\frac{f_\rho(x,y)}{\varphi(x)\varphi(y)}
=
\sum_{q=0}^{\infty}
\frac{\rho^q}{q!}H_q(x)H_q(y),
\end{align}
where $\varphi$ is the density of the standard normal distribution and 
$H_q$ denotes the $q$th Hermite polynomial \citep[see][p.~380]{Szego1975}. Consequently, using the orthogonality relation \eqref{lochd18}, it follows that
\[
\ee_\rho
\bigl[
H_k(X_i)H_\ell(Y_i)
\bigr]
=
\bone\{k=\ell\}\rho^k k!.
\]
Since the pairs $(X_1,Y_1), \ldots , (X_d,Y_d)$ are independent and identically distributed, we obtain, for multi-indices $\alpha,\beta\in\mathbb N_0^d$,
\begin{align*}
\ee_\rho
\bigl[ H_\alpha(X)H_\beta(Y)
\bigr] &= \prod_{i=1}^d \ee_\rho
\bigl[
 H_{\alpha_i}(X_i)H_{\beta_i}(Y_i)
\bigr] = \bone\{\alpha=\beta\}
\rho^{|\alpha|}\alpha!.
\end{align*}
We first assume that $F$ and $G$ are finite linear combinations of
Hermite polynomials. Using the orthonormal basis of $\mathcal H$ and the
preceding identity, we obtain
\begin{align*}
\mathbb{E} _\rho
\bigl[
\langle F(X),G(Y)\rangle_{\mathcal H}
\bigr]
&= \sum_{j=1}^D \mathbb{E} _\rho \bigl[ F_j(X)G_j(Y) \bigr]
= \sum_{j=1}^D \sum_{\alpha,\beta} f_{j,\alpha}g_{j,\beta}
\mathbb{E} _\rho
\bigl[
H_\alpha(X)H_\beta(Y)
\bigr]
\\
&=
\sum_{j=1}^D
\sum_{\alpha}
\rho^{|\alpha|}
\alpha! f_{j,\alpha}g_{j,\alpha} =
\sum_{q=0}^\infty
\rho^q
\langle P_qF,P_qG\rangle_{L^2(\gamma_d;\mathcal H)},
\end{align*}
where the last identity follows by grouping the terms according to their total Hermite degree. This proves \eqref{lochd19} for  finite linear combinations of Hermite polynomials. 

For general $F,G\in L^2(\gamma_d;\mathcal H)$, define the truncated
chaos expansions
\[
F^{(N)} = \sum_{q=0}^N P_qF,
\qquad
G^{(N)} =
\sum_{q=0}^N P_qG.
\]
By Parseval's identity, $
F^{(N)}$  and  $G^{(N)}$ converge to $F$ and $G$, respectively, in $L^2(\gamma_d;\mathcal H).$ Therefore, the 
Cauchy--Schwarz inequality gives
\begin{align*}
&
\left|
\mathbb{E} _\rho \bigl[
\langle F(X),G(Y)\rangle_{\mathcal H} -
\langle F^{(N)}(X),G^{(N)}(Y)\rangle_{\mathcal H}
\bigr]
\right|
\\
&\qquad\leq
\|F-F^{(N)}\|_{L^2(\gamma_d;\mathcal H)}
\|G\|_{L^2(\gamma_d;\mathcal H)}
\\
&\qquad\quad+
\|F^{(N)}\|_{L^2(\gamma_d;\mathcal H)}
\|G-G^{(N)}\|_{L^2(\gamma_d;\mathcal H)}
\rightarrow 0.
\end{align*}
Moreover,
\[
\sum_{q=0}^\infty
\left|
\langle P_qF,P_qG\rangle_{L^2(\gamma_d;\mathcal H)}
\right|
\leq
\|F\|_{L^2(\gamma_d;\mathcal H)}
\|G\|_{L^2(\gamma_d;\mathcal H)}
\]
by the Cauchy--Schwarz inequality and Parseval's identity. Hence, passing
to the limit in the identity for $F^{(N)}$ and $G^{(N)}$, we obtain
\[
\Psi_{F,G}(\rho)
=
\sum_{q=0}^\infty
\rho^q
\langle P_qF,P_qG\rangle_{L^2(\gamma_d;\mathcal H)}, 
\]
which proves \eqref{lochd19} in the general case. 

Setting  $
c_q =
\langle P_qF,P_qG\rangle_{L^2(\gamma_d;\mathcal H)}$  and  observing 
\[
\sum_{q=0}^\infty |c_q|
\leq
\|F\|_{L^2(\gamma_d;\mathcal H)}
\|G\|_{L^2(\gamma_d;\mathcal H)},
\]
it follows that 
the series may be differentiated termwise on every compact subinterval
of $(-1,1)$. Thus
\[
\Psi_{F,G}^{(r)}(\rho)
=
\sum_{q=r}^\infty
(q)_r\rho^{q-r}c_q,
\qquad
(q)_r=q(q-1)\cdots(q-r+1).
\]
At $\rho=0$, only the term $q=r$ remains, which proves \eqref{lochd18a}. The stated bound at $\rho=0$ follows from the contraction property of
orthogonal projections and the Cauchy--Schwarz inequality. Finally, for $|\rho|\leq a<1$,
\begin{align*}
\big|
\Psi_{F,G}^{(r)}(\rho)
-
\Psi_{F,G}^{(r)}(0)
\big|
&\leq
\sum_{q=r+1}^\infty
(q)_r|\rho|^{q-r}|c_q| \\
&\leq
|\rho|
\sup_{q\geq r+1}
\bigl\{
(q)_r a^{q-r-1}
\bigr\}
\sum_{q=r+1}^\infty |c_q| \\
&\leq
C_{r,a}|\rho|\,
\|F\|_{L^2(\gamma_d;\mathcal H)}
\|G\|_{L^2(\gamma_d;\mathcal H)}.
\end{align*}
\end{proof}

\subsubsection{Likelihood derivatives at independence}
For later use, we relate this identity to derivatives of the likelihood
ratio. Consider the particular case $d=n$ and write
\[
L_\rho
=
\frac{d \mathbb P_\rho}{d \mathbb P_0},
\]
where $\mathbb  P_\rho$ denotes the joint distribution of the $2n$-dimensional centred Gaussian vector $(X,Y)$ with  $\operatorname{Cov}(X) = \operatorname{Cov}(Y) = I_n$ and 
$\operatorname{Cov}(X,Y)=\rho I_n$. For $r\geq 1$, define
\[
L_{n,0}^{(r)}
=
\left.
\frac{\partial^r}{\partial\rho^r}L_\rho
\right|_{\rho=0},
\qquad
L_{n,0}^{(0)}=1.
\]
Since
\[
\mathbb{E}_\rho
\bigl[
\langle F(X),G(Y)\rangle_{\mathcal H}
\bigr]
=
\mathbb{E}_0
\bigl[
\langle F(X),G(Y)\rangle_{\mathcal H}L_\rho
\bigr],
\]
Proposition~\ref{prop:hilbert_mehler} yields
\[
\mathbb{E}_0
\left[
\langle F(X),G(Y)\rangle_{\mathcal H}
L_{n,0}^{(r)}
\right]
=
r!\,
\langle P_rF,P_rG\rangle_{L^2(\gamma_n;\mathcal H)}.
\]
This is the form of Mehler's identity that will be used repeatedly in the
proofs below.

We finally give the scalar version and the resulting explicit likelihood
coefficients. If $\mathcal H=\mathbb{R}$ and $d=1$, the projection onto the
$q$th Hermite chaos is
\[
P_qF
=
\frac{\mathbb{E}[F(Z)H_q(Z)]}{q!}\,H_q,
\]
and Proposition~\ref{prop:hilbert_mehler} reduces to
\[
\mathbb{E}_\rho[F(X)G(Y)]
=
\sum_{q=0}^\infty
\frac{\rho^q}{q!}
\mathbb{E}[F(Z)H_q(Z)]
\mathbb{E}[G(Z)H_q(Z)].
\]
Applying the density expansion \eqref{mehler} independently to the $n$ coordinate pairs gives the
following explicit form of the likelihood derivatives.

\begin{lmm}
\label{lem:likelihood-derivatives}
Define
\begin{equation*}
    L_{\rho}
    :=
    \frac{d\mathbb P_\rho}{d\mathbb P_0}
    =
    \prod_{i=1}^n
    \frac{f_\rho(X_i,Y_i)}
         {\varphi(X_i)\varphi(Y_i)}.
\end{equation*}
For \(r\ge1\), let
\begin{equation}\label{eq:likelihood-derivative-def}
    L_{n, 0}^{(r)}
    :=
    \left.
    \frac{\partial^r}{\partial\rho^r}
    L_{\rho}
    \right|_{\rho=0},
\end{equation}
and set \(L_{n, 0}^{(0)}:=1\).

Then
\begin{equation}\label{eq:general-likelihood-derivative}
\begin{aligned}
    L_{n, 0}^{(r)}
    &=
    r!
    \sum_{s=1}^r\frac{1}{s!}
    \sum_{\substack{
        \alpha_1,\ldots,\alpha_s\ge1\\
        \alpha_1+\cdots+\alpha_s=r
    }}
    \frac{1}{\alpha_1!\cdots\alpha_s!} 
    \times
    \sum_{i_1,\ldots,i_s}^{*}
    \prod_{u=1}^s
    H_{\alpha_u}(X_{i_u})
    H_{\alpha_u}(Y_{i_u}),
\end{aligned}
\end{equation}
where the star denotes summation over pairwise distinct indices.

Moreover, for every bounded measurable random variable \(A\),
\begin{equation*}
    \left.
    \frac{d^r}{d\rho^r}
    \ee_\rho A
    \right|_{\rho=0}
    =
    \ee_0[AL_{n, 0}^{(r)}].
\end{equation*}
Finally,
\begin{equation*}
    \ee_0[L_{n, 0}^{(r)}]=0,
    \qquad r\ge1.
\end{equation*}
\end{lmm}

\begin{proof}[Proof of Lemma~\ref{lem:likelihood-derivatives}]
By the one-dimensional Mehler formula,
\[
L_\rho
=
\prod_{i=1}^n
\left\{
\sum_{q=0}^\infty
\frac{\rho^q}{q!}
H_q(X_i)H_q(Y_i)
\right\}.
\]
To obtain a term of total order $r$, choose $s$ distinct coordinate pairs
and assign positive Hermite degrees
$\alpha_1,\ldots,\alpha_s$ satisfying
\[
\alpha_1+\cdots+\alpha_s=r.
\]
Writing the distinct indices as an ordered tuple counts each selection
$s!$ times, which gives the asserted formula after multiplication by $r!$.

The derivative identity follows from
\[
\mathbb{E}_\rho[A]
=
\mathbb{E}_0[AL_\rho],
\]
while the last assertion follows by taking $A=1$ and using
$\mathbb{E}_0[L_\rho]=1$.
\end{proof}

\begin{cor}
\label{cor:first_likelihood_coefficients}
Define
\begin{align}
    S_{n,1}
    &:=
    \sum_{i=1}^n X_iY_i,
    \label{eq:S1-def-centered}\\
    S_{n,2}
    &:=
    \frac{1}{2}
    \sum_{i=1}^n
    (X_i^2-1)(Y_i^2-1)
    +
    \sum_{1\le i<j\le n}
    X_iX_jY_iY_j.
    \label{eq:S2-def-centered}
\end{align}
Then
\begin{equation*}
    L_{n,0}^{(1)}
    =
    S_{n,1},
    \qquad
    L_{n,0}^{(2)}
    =
    2S_{n,2}.
\end{equation*}
Consequently, for every bounded measurable random variable \(A\),
\begin{equation}\label{eq:first-derivative-identity-centered}
    \left.
    \frac{d}{d\rho}
    \ee_\rho A
    \right|_{\rho=0}
    =
    \ee_0(AS_{n,1}),
\end{equation}
and
\begin{equation}\label{eq:second-derivative-identity-centered}
    \left.
    \frac{d^2}{d\rho^2}
    \ee_\rho A
    \right|_{\rho=0}
    =
    2\ee_0(AS_{n,2}).
\end{equation}
Moreover,
\begin{equation*}
    \ee_0S_{n,1}=0,
    \qquad
    \ee_0S_{n,2}=0.
\end{equation*}
In particular, the likelihood ratio admits the expansion
\[
L_\rho
=
1+\rho S_{n,1}+\rho^2S_{n,2}
+\rho^3\widetilde R_{n,\rho}
\]
for $\rho$ in a neighbourhood of zero.
\end{cor}

\begin{proof}[Proof of Corollary~\ref{cor:first_likelihood_coefficients}]
Taking $r=1$ in Lemma~\ref{lem:likelihood-derivatives} gives
\[
L_{n,0}^{(1)}
=
\sum_{i=1}^n
H_1(X_i)H_1(Y_i)
=
\sum_{i=1}^n X_iY_i
=
S_{n,1}.
\]
For $r=2$,
\[
L_{n,0}^{(2)}
=
\sum_{i=1}^n
H_2(X_i)H_2(Y_i)
+
2\sum_{1\leq i<j\leq n}
X_iX_jY_iY_j
=
2S_{n,2},
\]
because $H_2(x)=x^2-1$. The remaining assertions follow directly from
Lemma~\ref{lem:likelihood-derivatives}.
\end{proof}

\subsection{Mean expansion}\label{subsec:mean-expansion}
\subsubsection{Proof of Theorem~\ref{thm:mean-tau-square}}
We decompose $\ee_\rho[\hat\tau_n^2]$ into its off-diagonal and diagonal parts,
\begin{align*}
\ee_\rho[\hat\tau_n^2]
=
T_{n,\rho}^{\mathrm{off}}
+
T_{n,\rho}^{\mathrm{diag}},
\end{align*}
where 
\begin{align*}
T_{n,\rho}^{\mathrm{off}}
:=
\frac{n(n-1)}{n^2h_1}
\ee_\rho
\left[
a_{R_1^x R_2^x}
B_{R_1^y R_2^y}
\right],
\qquad
T_{n,\rho}^{\mathrm{diag}}
:=
\frac{n}{n^2h_1}
\ee_\rho
\left[
a_{R_1^x R_1^x}
B_{R_1^y R_1^y}
\right].
\end{align*}
Set
\[
        \delta_n
        :=
        \Bigl(\frac{h_1}{n}\Bigr)^{1/4}.
\]
Since \(|\rho_n|=O(\delta_n)\), Lemma~\ref{lmm:Tdiag} gives
\begin{align*}
\left|
T_{n,\rho_n}^{\mathrm{diag}}
-
T_{n,0}^{\mathrm{diag}}
\right|
&=
O\left(
        \frac{\delta_n}{n^2h_1^{3/2}}
\right)
+
O\left(
        \frac{\delta_n^2}{nh_1}
\right).
\end{align*}
Moreover,
\[
        \frac{1}{n^2h_1^{3/2}\delta_n}
        =
        \frac{1}{(nh_1)^{7/4}}
        \longrightarrow0,
        \qquad
        \frac{1}{nh_1}\longrightarrow0.
\]
Hence
\[
T_{n,\rho_n}^{\mathrm{diag}}
-
T_{n,0}^{\mathrm{diag}}
=
o(\delta_n^2)
=
o\left\{
        \Bigl(\frac{h_1}{n}\Bigr)^{1/2}
\right\}.
\]

We now consider the off-diagonal part. By exchangeability,
\[
T_{n,\rho}^{\mathrm{off}}
=
\frac{n(n-1)}{n^2h_1}
\sum_{\substack{k_1\ne k_2\\ \ell_1\ne \ell_2}}
a_{k_1k_2}
B_{\ell_1\ell_2}
p_{n}(k_1,k_2,\ell_1,\ell_2; \rho),
\]
where \(p_n(k_1,k_2,\ell_1,\ell_2;\rho)\) is defined in
Proposition~\ref{prop:GaussianRankProbExp}.
By Proposition~\ref{prop:GaussianRankProbExp}, the coefficients of
$1,\rho,\rho^2$ in the expansion of $T_{n,\rho}^{\mathrm{off}}$ are $T_{n0}^{\mathrm{off}},T_{n1}^{\mathrm{off}},T_{n2}^{\mathrm{off}}$.
Moreover, by Lemma~\ref{lmm:aggregate-third-order-hermite},
\begin{align*}
T_{n,\rho}^{\mathrm{off}}
=
T_{n0}^{\mathrm{off}}
+\rho T_{n1}^{\mathrm{off}}
+\rho^2T_{n2}^{\mathrm{off}}
+\rho^3\mathcal R_n^{\mathrm{off}}(\rho),
\end{align*}
where
\begin{align}\label{eq:Tnoff}
T_{nj}^{\mathrm{off}}
:=
\frac{n(n-1)}{n^2h_1}
\sum_{\substack{k_1\ne k_2\\ \ell_1\ne \ell_2}}
a_{k_1k_2}
B_{\ell_1\ell_2}
p_{j,n}(k_1,k_2,\ell_1,\ell_2),
\qquad j=0,1,2,
\end{align}
$p_{0,n}, p_{1,n}, p_{2,n}$ are defined in Proposition~\ref{prop:GaussianRankProbExp} and $
\rho_n\mathcal R_n^{\mathrm{off}}(\rho_n)\to 0 $. Consequently,
\[
\rho_n^3\mathcal R_n^{\mathrm{off}}(\rho_n)
=
o\left(\Bigl(\frac{h_1}{n}\Bigr)^{1/2}\right).
\]
Therefore
\[
T_{n,\rho_n}^{\mathrm{off}}
=
T_{n0}^{\mathrm{off}}
+\rho_n T_{n1}^{\mathrm{off}}
+\rho_n^2T_{n2}^{\mathrm{off}}
+o\left(\Bigl(\frac{h_1}{n}\Bigr)^{1/2}\right).
\]
Recall that 
\[
\ee_0[\hat{\tau}_n^2] = \hat{b}_n = \ee_0[T_n] + \ee_0[b_{n0}]= \hat{b}_n^{\mathrm{off}} + \hat{b}_n^\mathrm{diag},
\]
where
\begin{align}\label{eq:bnhatoff}
   \hat{b}_n^{\mathrm{off}} =\frac{1}{n^3(n-1)h_1^2} &\left(\sum_{i\neq j}\int_{0}^{1}K\Bigl(\frac{u - i/n}{h_1}\Bigr)K\Bigl(\frac{u - j/n}{h_1}\Bigr)\,du\right)\times\nonumber \\
   &\qquad\qquad\left(\sum_{i\neq j}\int_{0}^{1}\bar{K}\Bigl(\frac{v - i/n}{h_2}\Bigr)\bar{K}\Bigl(\frac{v - j/n}{h_2}\Bigr)\,dv\right),
\end{align}
and
\begin{align*}
    \hat{b}_n^\mathrm{diag} =  \frac{1}{n^3h_1^2}\Bigl(\sum_{i = 1}^n \int_{0}^1 K^2\Bigl(\frac{u - i/n}{h_1}\Bigr)\,du\Bigr)\Bigl(\sum_{i = 1}^n \int_{0}^1\bar{K}^2\Bigl(\frac{v - i/n}{h_2}\Bigr)\,dv\Bigr).
\end{align*}
By Lemma~\ref{lmm:Toff1} and
\[
        \rho_n
        =
        O\left\{
        \Bigl(\frac{h_1}{n}\Bigr)^{1/4}
        \right\},
\]
we have
\begin{align}
\left|
\rho_nT_{n1}^{\mathrm{off}}
\right|
&=
O\left\{
\frac{(h_1/n)^{1/4}}{nh_1}
\right\}
=
o\left\{
\Bigl(\frac{h_1}{n}\Bigr)^{1/2}
\right\}.
\label{eq:rhoTn1}
\end{align}

Combining Lemmas~\ref{lmm:Toff0}, \ref{lmm:Toff2},
\eqref{eq:rhoTn1}, and the remainder bound above, we obtain
\[
T_{n,\rho_n}^{\mathrm{off}}
=
T_{n0}^{\mathrm{off}}
+
\frac{\rho_n^2}{2\pi\sqrt3}
+
o\left\{
\Bigl(\frac{h_1}{n}\Bigr)^{1/2}
\right\}.
\]

Combining the off-diagonal expansion with the diagonal estimate gives
\begin{align*}
\ee_{\rho_n}[\hat\tau_n^2]
-
\ee_0[\hat\tau_n^2]
=
\left(
T_{n,\rho_n}^{\mathrm{off}}
-
T_{n0}^{\mathrm{off}}
\right)
+
\left(
T_{n,\rho_n}^{\mathrm{diag}}
-
T_{n,0}^{\mathrm{diag}}
\right)
=
\frac{\rho_n^2}{2\pi\sqrt3}
+
o\left\{
\Bigl(\frac{h_1}{n}\Bigr)^{1/2}
\right\}.
\end{align*}
This proves the theorem.
\subsubsection{Auxiliary results for the mean expansion}
\begin{prop}\label{prop:GaussianRankProbExp}
Let $n\geq 4$, and let $(X_1,Y_1),\ldots,(X_n,Y_n)$ be i.i.d. bivariate normal random
vectors with standard normal marginals and correlation coefficient $\rho$.
Let $R_i^x$ and $R_i^y$ denote the ranks of $X_i$ and $Y_i$, among $X_1, \ldots , X_n$ and $Y_1, \ldots ,Y_n$,
respectively. For $k_1\neq k_2$ and $\ell_1\neq \ell_2$, define
\[
p_{n}(k_1,k_2,\ell_1,\ell_2; \rho)
:=
\pp_\rho\bigl(
R_1^x=k_1,\ R_2^x=k_2,\ R_1^y=\ell_1,\ R_2^y=\ell_2
\bigr).
\]

The coefficients of $1,\rho,\rho^2$ in the expansion of
$p_{n}(k_1,k_2,\ell_1,\ell_2; \rho)$ with respect to $\rho$ are denoted by $p_{0,n},p_{1,n},p_{2,n}$
and given by 
\begin{align} \nonumber
p_{0,n}(k_1,k_2,\ell_1,\ell_2)
& =
\frac{1}{n^2(n-1)^2}, \\
p_{1,n}(k_1,k_2,\ell_1,\ell_2)
& =
\frac{\mu_{k_1}\mu_{\ell_1}+\mu_{k_2}\mu_{\ell_2}}
{n^2(n-1)^2}
+
\frac{(\mu_{k_1}+\mu_{k_2})(\mu_{\ell_1}+\mu_{\ell_2})}
{n^2(n-1)^2(n-2)},
\label{lochol2}
\end{align}
and 
\[
p_{2,n}=p^{H_2}_{2,n}+p^{\mathrm{pair}}_{2,n},
\]
where
\begin{align}\label{eq:p2-H2}
p^{H_2}_{2,n}
=
\frac{1}{2}
\left[
\frac{\beta_{k_1}\beta_{\ell_1}+\beta_{k_2}\beta_{\ell_2}}
{n^2(n-1)^2}
+
\frac{(\beta_{k_1}+\beta_{k_2})(\beta_{\ell_1}+\beta_{\ell_2})}
{n^2(n-1)^2(n-2)}
\right], 
\end{align}
\begin{equation}\label{eq:p2-pair}
\begin{split}
p^{\mathrm{pair}}_{2,n}
& =
\frac{m_{k_1k_2}m_{\ell_1\ell_2}}
{n^2(n-1)^2}
\\[0.4em]
&+
\frac{
\left(\sum_{r\notin\{k_1,k_2\}}m_{k_1r}\right)
\left(\sum_{s\notin\{\ell_1,\ell_2\}}m_{\ell_1s}\right)}
{n^2(n-1)^2(n-2)}
\\[0.4em]
&+
\frac{
\left(\sum_{r\notin\{k_1,k_2\}}m_{k_2r}\right)
\left(\sum_{s\notin\{\ell_1,\ell_2\}}m_{\ell_2s}\right)}
{n^2(n-1)^2(n-2)}
\\[0.4em]
&+
\frac{
\left(
\sum_{\substack{r\neq s\\ r,s\notin\{k_1,k_2\}}}
m_{rs}
\right)
\left(
\sum_{\substack{r\neq s\\ r,s\notin\{\ell_1,\ell_2\}}}
m_{rs}
\right)}
{2n^2(n-1)^2(n-2)(n-3)},
\end{split}
\end{equation}
and $\mu_i$, $\beta_i$, and $m_{ij}$ are defined as in~\eqref{eq:rank-scores-centered}.
\end{prop}

\begin{proof}
Let
\[
A_x:=\{R_1^x=k_1,\ R_2^x=k_2\},
\qquad
A_y:=\{R_1^y=\ell_1,\ R_2^y=\ell_2\}.
\]
When $\rho=0$, the $X$- and $Y$-samples are independent. Since
$k_1\neq k_2$ and $\ell_1\neq \ell_2$, we have
\[
p_{n}(k_1,k_2,\ell_1,\ell_2; 0) = \pp_0(A_x)\pp_0(A_y) =
\frac{1}{n^2(n-1)^2}.
\]
By Corollary~\ref{cor:first_likelihood_coefficients},
the likelihood ratio has the expansion
\[
L_{\rho}
=
1+\rho S_{n,1}+\rho^2S_{n,2}
+\rho^3\widetilde R_{n,\rho}.
\]
Consequently,
\[
p_{n}(k_1,k_2,\ell_1,\ell_2; \rho)
=
\ee_0[\bone_{A_x}\bone_{A_y}L_\rho].
\]
Since under $\pp_0$ the $X$- and $Y$-samples are independent, all
coefficients factorize into an $X$-rank expectation and a $Y$-rank expectation.

Define
\[
\alpha_i^x:=\ee_0[\bone_{A_x}X_i],
\qquad
\alpha_i^y:=\ee_0[\bone_{A_y}Y_i],
\qquad
\gamma_i^x:=\ee_0[\bone_{A_x}H_2(X_i)],
\qquad
\gamma_i^y:=\ee_0[\bone_{A_y}H_2(Y_i)],
\]
and
\[
\delta_{ij}^x:=\ee_0[\bone_{A_x}X_iX_j],
\qquad
\delta_{ij}^y:=\ee_0[\bone_{A_y}Y_iY_j].
\]
Then
\[
p_{1,n}=\sum_{i=1}^n \alpha_i^x\alpha_i^y
\]
and
\[
p_{2,n}
=
\frac{1}{2}\sum_{i=1}^n \gamma_i^x\gamma_i^y
+
\sum_{1\leq i<j\leq n}\delta_{ij}^x\delta_{ij}^y.
\]

We first compute the first-order coefficient. On $A_x$, $X_1$ is the
$k_1$-th order statistic and $X_2$ is the $k_2$-th order statistic. Hence
\[
\alpha_1^x=\frac{\mu_{k_1}}{n(n-1)},
\qquad
\alpha_2^x=\frac{\mu_{k_2}}{n(n-1)}.
\]
For $i\geq 3$, conditional on $A_x$, the rank of $X_i$ is uniformly
distributed over the remaining ranks  $\{1,\ldots,n\}\setminus\{k_1,k_2\}.$
Since $\sum_{r=1}^n\mu_r = 0$, we obtain
\[
\alpha_i^x
=
-\frac{\mu_{k_1}+\mu_{k_2}}{n(n-1)(n-2)},
\qquad i\geq 3.
\]
Similarly,
\[
\alpha_1^y=\frac{\mu_{\ell_1}}{n(n-1)},
\qquad
\alpha_2^y=\frac{\mu_{\ell_2}}{n(n-1)},
\]
and
\[
\alpha_i^y
=
-\frac{\mu_{\ell_1}+\mu_{\ell_2}}{n(n-1)(n-2)},
\qquad i\geq 3.
\]
Therefore
\begin{align*}
p_{1,n} =
\alpha_1^x\alpha_1^y
+
\alpha_2^x\alpha_2^y
+
\sum_{i=3}^n \alpha_i^x\alpha_i^y =
\frac{\mu_{k_1}\mu_{\ell_1}+\mu_{k_2}\mu_{\ell_2}}
{n^2(n-1)^2}
+
\frac{(\mu_{k_1}+\mu_{k_2})(\mu_{\ell_1}+\mu_{\ell_2})}
{n^2(n-1)^2(n-2)}, 
\end{align*}
which proves \eqref{lochol2}.

Next we compute the $H_2$-part of the second-order coefficient. With the notation  $\beta_k$ defined as in~\eqref{eq:rank-scores-centered}
we have
\[
\gamma_1^x=\frac{\beta_{k_1}}{n(n-1)},
\qquad
\gamma_2^x=\frac{\beta_{k_2}}{n(n-1)}.
\]
By~\eqref{eq:score-identities-centered} we have $\sum_{r=1}^n\beta_r = 0$; thus, for $i\geq 3$,
\[
\gamma_i^x
=
-\frac{\beta_{k_1}+\beta_{k_2}}{n(n-1)(n-2)}.
\]
The corresponding $Y$-side identities are obtained by replacing
$k_1,k_2$ by $\ell_1,\ell_2$. Hence
\[
\frac{1}{2}\sum_{i=1}^n\gamma_i^x\gamma_i^y
=
\frac{1}{2}
\left[
\frac{\beta_{k_1}\beta_{\ell_1}+\beta_{k_2}\beta_{\ell_2}}
{n^2(n-1)^2}
+
\frac{(\beta_{k_1}+\beta_{k_2})(\beta_{\ell_1}+\beta_{\ell_2})}
{n^2(n-1)^2(n-2)}
\right], 
\]
which  gives the term $p^{H_2}_{2,n}$.

It remains to compute the pair contribution
\[
\sum_{1\leq i<j\leq n}\delta_{ij}^x\delta_{ij}^y.
\]
For the pair $(1,2)$,
\[
\delta_{12}^x=\frac{m_{k_1k_2}}{n(n-1)},
\qquad
\delta_{12}^y=\frac{m_{\ell_1\ell_2}}{n(n-1)}.
\]
Thus this pair contributes
\[
\frac{m_{k_1k_2}m_{\ell_1\ell_2}}{n^2(n-1)^2}.
\]

For pairs $(1,j)$ with $j\geq 3$, conditional on $A_x$, the rank of
$X_j$ is uniformly distributed over
$\{1,\ldots,n\}\setminus\{k_1,k_2\}$. Therefore
\[
\delta_{1j}^x
=
\frac{1}{n(n-1)(n-2)}
\sum_{r\notin\{k_1,k_2\}}m_{k_1r}.
\]
Similarly,
\[
\delta_{1j}^y
=
\frac{1}{n(n-1)(n-2)}
\sum_{s\notin\{\ell_1,\ell_2\}}m_{\ell_1s}.
\]
There are $n-2$ such pairs, so their total contribution is
\[
\frac{
\left(\sum_{r\notin\{k_1,k_2\}}m_{k_1r}\right)
\left(\sum_{s\notin\{\ell_1,\ell_2\}}m_{\ell_1s}\right)}
{n^2(n-1)^2(n-2)}.
\]

The same argument for pairs $(2,j)$, $j\geq 3$, yields the contribution
\[
\frac{
\left(\sum_{r\notin\{k_1,k_2\}}m_{k_2r}\right)
\left(\sum_{s\notin\{\ell_1,\ell_2\}}m_{\ell_2s}\right)}
{n^2(n-1)^2(n-2)}.
\]

Finally consider pairs $(i,j)$ with $3\leq i<j\leq n$. Conditional on
$A_x$, the pair of ranks of $(X_i,X_j)$ is uniformly distributed
over all pairs of distinct ranks in
\[
\{1,\ldots,n\}\setminus\{k_1,k_2\}.
\]
Hence
\[
\delta_{ij}^x
=
\frac{1}{n(n-1)(n-2)(n-3)}
\sum_{\substack{r\neq s\\ r,s\notin\{k_1,k_2\}}}m_{rs}.
\]
Similarly,
\[
\delta_{ij}^y
=
\frac{1}{n(n-1)(n-2)(n-3)}
\sum_{\substack{r\neq s\\ r,s\notin\{\ell_1,\ell_2\}}}m_{rs}.
\]
There are $\binom{n-2}{2}$ such unordered index pairs $(i,j)$. Therefore their total contribution is
\[
\frac{
\left(
\sum_{\substack{r\neq s\\ r,s\notin\{k_1,k_2\}}}
m_{rs}
\right)
\left(
\sum_{\substack{r\neq s\\ r,s\notin\{\ell_1,\ell_2\}}}
m_{rs}
\right)}
{2n^2(n-1)^2(n-2)(n-3)}.
\]
Combining the four pair contributions gives $p^{\mathrm{pair}}_{2,n}$.
\end{proof}

\begin{prop}\label{prop:ApproxWeightm}
Let $Z_{1:n}\leq \cdots \leq Z_{n:n}$ denote the order statistics of an
i.i.d. standard normal sample. Suppose that Assumptions~\ref{asmp:kernel} and~\ref{asmp:band} hold. Let $a_{ij}$ and $m_{ij}$ be defined as in \eqref{eq:aij} and \eqref{eq:rank-scores-centered} respectively. Then we have
\[
\frac{1}{n^2h_1}
\sum_{i\neq j}
a_{ij}
\left\{
m_{ij}-q_{i,n}q_{j,n}
\right\}
\rightarrow 0.
\]
Consequently,
\[
\frac{1}{n^2h_1}
\sum_{i\neq j}
a_{ij}m_{ij}
\rightarrow
\int_0^1 q(u)^2\,du
=
1.
\]
Equivalently, in weighted sums of the above form one may replace
$m_{ij}$ by $q(i/n)q(j/n)$ after the usual truncation away from the
endpoints. The use of $q_{i,n}=q(i/(n+1))$ only avoids the singularity of
$q(1)$ at the endpoint.
\end{prop}

\begin{proof}
By Lemma~\ref{lem:deterministic-kernel-weight-bounds},
\begin{equation}\label{eq:ApproxWeightm-a-bounds}
        0\le a_{ij}
        \le
        C\bone\{|i-j|\le 2nh_1+2\},
        \qquad
        \max_{1\le i\le n}\sum_{j\ne i}a_{ij}
        =
        O(nh_1).
\end{equation}

We first prove
\begin{equation}\label{eq:ApproxWeightm-first}
\frac{1}{n^2h_1}
\sum_{i\ne j}
a_{ij}
\left|
m_{ij} - q_{i,n}q_{j,n}
\right|
\rightarrow0.
\end{equation}
For brevity, write
\[
        d_i
        :=
        \left\{
        \ee\left(Z_{i:n}-q_{i,n}\right)^2
        \right\}^{1/2},
        \qquad
        s_i
        :=
        \left\{\ee [Z_{i:n}^2]\right\}^{1/2}.
\]
Since
\[
m_{ij}-q_{i,n}q_{j,n}
=
\ee\left[(Z_{i:n}-q_{i,n})Z_{j:n}\right]
+
q_{i,n}\ee\left[Z_{j:n}-q_{j,n}\right],
\]
the Cauchy--Schwarz inequality gives
\begin{equation*}
        |m_{ij}-q_{i,n}q_{j,n}|
        \le
        d_i s_j+|q_{i,n}|d_j.
\end{equation*}
We claim that
\begin{equation}\label{eq:ApproxWeightm-di-average}
        \frac{1}{n}\sum_{i=1}^n d_i^2 \rightarrow 0.
\end{equation}
Let \(\widehat q_n\) denote the empirical quantile function
of \(Z_1,\ldots,Z_n\), so that
\[
        \widehat q_n(u)=Z_{i:n},
        \qquad
        \frac{i-1}{n}<u\le\frac{i}{n},
\]
and define the deterministic step function
\[
        \widetilde q_n(u) = q_{i,n},
        \qquad
        \frac{i-1}{n}<u\le\frac{i}{n}.
\]
Then
\[
\frac1n\sum_{i=1}^nd_i^2
=
\ee\int_0^1
\left|\widehat q_n(u)-\widetilde q_n(u)\right|^2\,du.
\]

By the standard \(L^2\)-consistency of the empirical quantile function,
\[
        \ee\int_0^1
        |\widehat q_n(u)-q(u)|^2\,du
        \rightarrow0.
\]
We next show that
\begin{equation}\label{eq:ApproxWeightm-qtilde-L2}
        \int_0^1
        |\widetilde q_n(u)-q(u)|^2\,du
        \rightarrow0.
\end{equation}
Fix \(\varepsilon\in(0,1/4)\). Since \(q\) is uniformly
continuous on \([\varepsilon/2,1-\varepsilon/2]\), and
\[
        \sup_{\substack{(i-1)/n<u\le i/n}}
        \left|
        u-\frac{i}{n+1}
        \right|
        \le \frac{2}{n},
\]
we have
\[
        \sup_{u\in[\varepsilon,1-\varepsilon]}
        |\widetilde q_n(u)-q(u)|
        \longrightarrow0.
\]
Consequently,
\begin{equation}\label{eq:ApproxWeightm-qtilde-interior}
        \int_{\varepsilon}^{1-\varepsilon}
        |\widetilde q_n(u)-q(u)|^2\,du
        \longrightarrow0.
\end{equation}

It remains to control the endpoints. Since \(q(u)^2\) is decreasing
on \((0,1/2]\), for \(i/(n+1)\le1/2\),
\[
\frac1n q_{i,n}^2
\le
\frac{n+1}{n}
\int_{(i-1)/(n+1)}^{i/(n+1)}
q(u)^2\,du.
\]
Similarly, since \(q(u)^2\) is increasing on \([1/2,1)\), for
\(i/(n+1)\ge1/2\),
\[
\frac1n q_{i,n}^2
\le
\frac{n+1}{n}
\int_{i/(n+1)}^{(i+1)/(n+1)}
q(u)^2\,du.
\]
It follows that, for every fixed \(\varepsilon\in(0,1/4)\),
\begin{align}
\limsup_{n\to\infty}
\int_{[0,\varepsilon]\cup[1-\varepsilon,1]}
\widetilde q_n(u)^2\,du
&\le
\int_{[0,2\varepsilon]\cup[1-2\varepsilon,1]}
q(u)^2\,du.
\label{eq:ApproxWeightm-qtilde-tail}
\end{align}
Using
\[
        |\widetilde q_n-q|^2
        \le
        2\widetilde q_n^2+2q^2,
\]
equations \eqref{eq:ApproxWeightm-qtilde-interior} and
\eqref{eq:ApproxWeightm-qtilde-tail}, followed by
\(\varepsilon\downarrow0\), prove
\eqref{eq:ApproxWeightm-qtilde-L2}, because
\(q\in L^2(0,1)\).

Therefore,
\[
\frac1n\sum_{i=1}^nd_i^2
=
\ee\int_0^1
\left|\widehat q_n(u)-\widetilde q_n(u)\right|^2\,du
\longrightarrow0,
\]
by
\[
|\widehat q_n-\widetilde q_n|^2
\le
2|\widehat q_n-q|^2
+
2|\widetilde q_n-q|^2.
\]
This proves \eqref{eq:ApproxWeightm-di-average}.

Moreover, \eqref{eq:ApproxWeightm-qtilde-L2} implies
\[
        \|\widetilde q_n-q\|_{L^2(0,1)}
        \rightarrow0.
\]
Hence
\[
        \|\widetilde q_n\|_{L^2(0,1)}
        \longrightarrow
        \|q\|_{L^2(0,1)},
\]
and therefore
\begin{equation}\label{eq:ApproxWeightm-qi-square}
        \frac{1}{n}\sum_{i=1}^n q_{i,n}^2
        =
        \int_0^1\widetilde q_n(u)^2\,du
        \rightarrow
        \int_0^1q(u)^2\,du
        =1.
\end{equation}

This proves \eqref{eq:ApproxWeightm-first}, and hence
\begin{equation}\label{eq:ApproxWeightm-first-signed}
\frac{1}{n^2h_1}
\sum_{i\ne j}
a_{ij}
\left\{
m_{ij} - q_{i,n}q_{j,n}
\right\}
\rightarrow0.
\end{equation}
It remains to evaluate
\[
        Q_n
        :=
        \frac{1}{n^2h_1}
        \sum_{i\ne j}a_{ij}q_{i,n}q_{j,n}.
\]
Since \(a_{ij} = a_{ji}\), the identity
\[
        2q_{i,n}q_{j,n}
        =
        q_{i,n}^2 + q_{j,n}^2 - (q_{i,n} - q_{j,n})^2
\]
gives
\begin{equation}\label{eq:ApproxWeightm-Q-decomp}
        Q_n = A_n - R_n,
\end{equation}
where
\[
        A_n
        :=
        \frac{1}{n^2h_1}
        \sum_{i=1}^n L_i q_{i,n}^2,
        \qquad
        R_n
        :=
        \frac{1}{2n^2h_1}
        \sum_{i\ne j}
        a_{ij}(q_{i,n} - q_{j,n})^2.
\]

We first show that
\begin{equation}\label{eq:ApproxWeightm-An}
        A_n\rightarrow1.
\end{equation}
Let \(m_n\), \(I_n\), and \(\ell_n\) be as in
Lemma~\ref{lem:deterministic-kernel-weight-bounds}. Thus
\[
        I_n=\{m_n+1,\ldots,n-m_n\},
        \qquad
        m_n=\lceil2nh_1\rceil,
\]
and
\[
        L_i=\ell_n
        \quad\text{for }i\in I_n,
        \qquad
        \frac{\ell_n}{nh_1}\rightarrow 1,
        \qquad
        \max_iL_i = O(nh_1).
\]
Therefore
\begin{align*}
A_n
&=
\frac{\ell_n}{nh_1}
\frac{1}{n}\sum_{i\in I_n}q_{i,n}^2
+
\frac{1}{n^2h_1}
\sum_{i\notin I_n}L_iq_{i,n}^2.
\end{align*}
Since \(m_n/n\to0\), the endpoint bound
\eqref{eq:ApproxWeightm-qtilde-tail} and the fact that \(q\in L^2(0,1)\) give
\begin{equation*}
        \frac{1}{n}\sum_{i\notin I_n}q_{i,n}^2
        \rightarrow 0.
\end{equation*}
Together with \eqref{eq:ApproxWeightm-qi-square}, this yields
\[
        \frac{1}{n}\sum_{i\in I_n}q_{i,n}^2\rightarrow1.
\]
Furthermore,
\[
\left|
        \frac{1}{n^2h_1}
        \sum_{i\notin I_n}L_i q_{i,n}^2
\right|
\le
\frac{C}{n}
\sum_{i\notin I_n}q_{i,n}^2
\rightarrow0.
\]
Hence \eqref{eq:ApproxWeightm-An} follows.

We next prove
\begin{equation}\label{eq:ApproxWeightm-Rn}
        R_n\rightarrow0.
\end{equation}
Fix \(\varepsilon\in(0,1/4)\) and set
\[
        I_\varepsilon
        :=
        \{i:\varepsilon n\le i\le(1-\varepsilon)n\},
        \qquad
        B_\varepsilon
        :=
        \{1,\ldots,n\}\setminus I_\varepsilon.
\]
Consider first pairs with \(i,j\in I_\varepsilon\). By
\eqref{eq:ApproxWeightm-a-bounds}, \(a_{ij}\ne0\) implies
\[
        |i-j|\leq 2nh_1 + 2.
\]
For all sufficiently large \(n\), the arguments \(i/(n+1)\) and
\(j/(n+1)\) then remain in a fixed compact subinterval of \((0,1)\).
Since \(q'\) is bounded on that interval and \(nh_1\to\infty\),
\[
\max_{\substack{i,j\in I_\varepsilon\\a_{ij}\ne0}}
        |q_{i,n} - q_{j,n}|
\le
C_\varepsilon
\frac{|i-j|}{n+1}
=
O_\varepsilon(h_1).
\]
Consequently,
\begin{align}
R_{n,\varepsilon}^{\mathrm{int}}
:=
\frac{1}{2n^2h_1}
\sum_{\substack{i\ne j\\i,j\in I_\varepsilon}}
a_{ij}(q_{i,n} - q_{j,n})^2
\le
\frac{C_\varepsilon h_1^2}{n^2h_1}
\sum_{i\ne j}a_{ij}
=
O_\varepsilon(h_1^2)
\rightarrow 0.
\label{eq:ApproxWeightm-R-interior}
\end{align}

It remains to control the boundary pairs. The support property in
\eqref{eq:ApproxWeightm-a-bounds}, together with \(h_1\to0\), implies
that, for all sufficiently large \(n\),
\begin{equation}\label{eq:ApproxWeightm-boundary-enlargement}
        a_{ij}\ne0,\quad
        \{i,j\}\cap B_\varepsilon\ne\varnothing
        \quad\Longrightarrow\quad
        i,j\in B_{2\varepsilon}.
\end{equation}
Indeed, if for example \(i\le\varepsilon n\), then
\[
        j \leq i + 2nh_1 + 2 \leq 2\varepsilon n
\]
for all sufficiently large \(n\), and the other three boundary
cases are identical.

Using $(q_{i,n} - q_{j,n})^2 \le 2q_{i,n}^2 + 2q_{j,n}^2$, symmetry of \(a_{ij}\), and the row-sum bound, we obtain from
\eqref{eq:ApproxWeightm-boundary-enlargement}
\begin{align}
R_{n,\varepsilon}^{\mathrm{bdry}}
:=
\frac{1}{2n^2h_1}
\sum_{\substack{i\ne j\\
\{i,j\}\cap B_\varepsilon\ne\varnothing}}
a_{ij}(q_{i,n} - q_{j,n})^2
\le
\frac{C}{n^2h_1}
\sum_{i\in B_{2\varepsilon}}
q_{i,n}^2
\sum_{j\ne i}a_{ij}
\le
\frac{C}{n}
\sum_{i\in B_{2\varepsilon}}q_{i,n}^2.
\label{eq:ApproxWeightm-R-boundary}
\end{align}
To control the last term, let
\[
        r_n:=\lceil 2\varepsilon n\rceil-1,
        \qquad
        s_n:=\lfloor(1-2\varepsilon)n\rfloor+1.
\]
Then
\[
        B_{2\varepsilon}
        \subseteq
        \{1,\ldots,r_n\}
        \cup
        \{s_n,\ldots,n\}.
\]
Since \(q(u)^2\) is decreasing on \((0,1/2]\), for
\(i\le r_n\) and all sufficiently large \(n\),
\[
        \frac{1}{n} q_{i,n}^2
        \le
        \frac{n+1}{n}
        \int_{(i-1)/(n+1)}^{i/(n+1)}
        q(u)^2\,du.
\]
Therefore
\[
        \frac{1}{n}\sum_{i=1}^{r_n}q_{i,n}^2
        \le
        \frac{n+1}{n}
        \int_0^{r_n/(n+1)}q(u)^2\,du.
\]
Similarly, since \(q(u)^2\) is increasing on \([1/2,1)\),
\[
        \frac{1}{n}\sum_{i=s_n}^{n}q_{i,n}^2
        \le
        \frac{n+1}{n}
        \int_{s_n/(n+1)}^1q(u)^2\,du.
\]
Since
\[
        \frac{r_n}{n+1}\rightarrow2\varepsilon,
        \qquad
        \frac{s_n}{n+1}\rightarrow1-2\varepsilon,
\]
and \(q^2\in L^1(0,1)\), it follows that
\[
\limsup_{n\to\infty}
\frac{1}{n}\sum_{i\in B_{2\varepsilon}}q_{i,n}^2
\le
\int_0^{2\varepsilon}q(u)^2\,du
+
\int_{1-2\varepsilon}^1q(u)^2\,du.
\]
The right-hand side tends to zero as
\(\varepsilon\downarrow0\).
Combining
\eqref{eq:ApproxWeightm-R-interior} and
\eqref{eq:ApproxWeightm-R-boundary}, and then letting
\(\varepsilon\downarrow0\), proves
\eqref{eq:ApproxWeightm-Rn}.

It now follows from \eqref{eq:ApproxWeightm-Q-decomp},
\eqref{eq:ApproxWeightm-An}, and
\eqref{eq:ApproxWeightm-Rn} that $Q_n\rightarrow 1$. Equivalently,
\[
\frac{1}{n^2h_1}
\sum_{i\ne j}a_{ij}q_{i,n}q_{j,n}
\rightarrow
\int_0^1q(u)^2\,du
=
1.
\]
Combining this with
\eqref{eq:ApproxWeightm-first-signed} gives
\[
\frac{1}{n^2h_1}
\sum_{i\ne j}a_{ij}m_{ij}
\rightarrow1,
\]
which proves the proposition.
\end{proof}

\begin{lmm}\label{lem:Lmu-bound}
Suppose that Assumptions~\ref{asmp:kernel} and~\ref{asmp:band} hold. Then
\[
    |\sum_{i=1}^n (L_i-\bar L)\mu_i| = O(n).
\]
\end{lmm}

\begin{proof}
Since $\sum_{i=1}^n\mu_i = 0$, we have
\[
        \sum_{i=1}^n(L_i-\bar L)\mu_i
        =
        \sum_{i=1}^nL_i\mu_i .
\]
Moreover, by symmetry of the normal distribution, we have $\mu_{n+1-i}=-\mu_i$. Therefore
\begin{align*}
        \sum_{i=1}^nL_i\mu_i
        &=
        \frac{1}{2}
        \sum_{i=1}^n
        \mu_i(L_i-L_{n+1-i}).
\end{align*}
First we show that $\max_{1\le i\le n}|L_i-L_{n+1-i}|=O(1)$. Set $x_j:=j/n$, $c_n:=1+1/n$, and $J(k):=n+1-k$. Then
\[
x_{J(j)}
=
\frac{n+1-j}{n}
=
1+\frac{1}{n}-\frac{j}{n}
=
c_n-x_j.
\]
Using the symmetry of the kernel, \(K(t)=K(-t)\), and the change of variables $v=c_n-u$, we obtain
\begin{align*}
a_{J(k),J(j)} &=
\frac{1}{h_1}\int_{1/n}^{1+1/n}
K\left(\frac{x_k-v}{h_1}\right)
K\left(\frac{x_j-v}{h_1}\right)\,dv \\
&=
\frac{1}{h_1}\int_{1/n}^{1+1/n}
K\left(\frac{v-x_k}{h_1}\right)
K\left(\frac{v-x_j}{h_1}\right)\,dv.
\end{align*}
Consequently, after relabelling $j\mapsto J(j)$,
\begin{align*}
L_{J(k)} =
\sum_{j\ne k}a_{J(k)J(j)}
=
\frac{1}{h_1}\sum_{j\ne k}
\int_{1/n}^{1+1/n}
K\left(\frac{v-x_k}{h_1}\right)
K\left(\frac{v-x_j}{h_1}\right)\,dv.
\end{align*}
On the other hand,
\[
L_k
=
\frac{1}{h_1}\sum_{j\ne k}
\int_0^1
K\left(\frac{v-x_k}{h_1}\right)
K\left(\frac{v-x_j}{h_1}\right)\,dv, 
\]
which gives 
\begin{align*}
L_{J(k)} - L_k =&
\frac{1}{h_1}\sum_{j\ne k}
\int_1^{1+1/n}
K\left(\frac{v-x_k}{h_1}\right)
K\left(\frac{v-x_j}{h_1}\right)\,dv
 \\
&- \frac{1}{h_1}\sum_{j\ne k}
\int_0^{1/n}
K\left(\frac{v-x_k}{h_1}\right)
K\left(\frac{v-x_j}{h_1}\right)\,dv.
\end{align*}
Thus the difference between \(L_k\) and \(L_{J(k)}\) is entirely a boundary effect. It remains to bound the two boundary integrals.

For fixed \(v\), the number of indices \(j\in\{1,\ldots,n\}\) satisfying $|v-j/n|\le h_1 $
is at most $
2nh_1+3. $
Therefore, uniformly in $v$,
\[
\sum_{j\ne k}
\left|
K\left(\frac{v-x_j}{h_1}\right)
\right|
\le
\|K\|_\infty(2nh_1+3), 
\]
and it 
follows that
\begin{align*}
\left|L_{J(k)} - L_k\right|
&\le
\frac{1}{h_1}\int_1^{1+1/n}
\left|
K\left(\frac{v-x_k}{h_1}\right)
\right|
\sum_{j\ne k}
\left|
K\left(\frac{v-x_j}{h_1}\right)
\right|\,dv
\\
&\quad+
\frac{1}{h_1}\int_0^{1/n}
\left|
K\left(\frac{v-x_k}{h_1}\right)
\right|
\sum_{j\ne k}
\left|
K\left(\frac{v-x_j}{h_1}\right)
\right|\,dv
\\
&\le
2\cdot \frac{1}{n} \cdot \|K\|_\infty^2(2n+\frac{3}{h_1})
\\
&=
4\|K\|_\infty^2+\frac{6\|K\|_\infty^2}{nh_1}.
\end{align*}
Hence
\[
\max_{1\le k\le n}
\left|L_k-L_{J(k)}\right|
=
O(1+\frac{1}{nh_1}).
\]
Consequently
\[
        \left|\sum_{i=1}^n(L_i - \bar{L})\mu_i\right|
        \le
        \frac{1}{2}
        \max_i|L_i-L_{n+1-i}|
        \sum_{i=1}^n|\mu_i|
        =
        O\left(\sum_{i=1}^n|\mu_i|\right) = O(n),
\]
where the last equality is by~\eqref{eq:score-identities-centered}.
\end{proof}

\begin{lmm}\label{lmm:order-stat-average}
Let $\mu_k$ and $m_{k\ell}$ be defined as in \eqref{eq:rank-scores-centered}. Then
\[
\frac{1}{n}\sum_{k=1}^n |\mu_{k}-q_{k,n}|\to0,
\qquad
\frac{1}{n^2}\sum_{k,\ell=1}^n
\left|
m_{k\ell}-\mu_{k}\mu_{\ell}
\right|
\to0.
\]
Consequently, if \(0\le C_{k\ell,n}\le 1\), then
\[
\frac{1}{n^2}\sum_{k,\ell=1}^n
C_{k\ell,n}
\left\{
m_{k\ell}-q_{k,n}q_{\ell,n}
\right\}
=o(1).
\]
\end{lmm}

\begin{proof}
For the first claim, let \(\widehat q_n\) denote the empirical quantile
function based on \(Z_1,\ldots,Z_n\), and define the deterministic step
function
\[
    \widetilde q_n(u)
    :=
    q_{k,n},
    \qquad
    \frac{k-1}{n}<u\le \frac{k}{n},
    \quad
    k=1,\ldots,n.
\]
Since
\[
    \widehat q_n(u)=Z_{k:n},
    \qquad
    \frac{k-1}{n}<u\le \frac{k}{n},
\]
Jensen's inequality gives
\begin{align*}
\frac1n\sum_{k=1}^n
|\mu_k-q_{k,n}|
&=
\int_0^1
\left|
\ee\widehat q_n(u)-\widetilde q_n(u)
\right|\,du
\\
&\le
\ee\int_0^1
\left|
\widehat q_n(u)-\widetilde q_n(u)
\right|\,du
\\
&\le
\ee\int_0^1
\left|
\widehat q_n(u)-q(u)
\right|\,du
+
\int_0^1
\left|
\widetilde q_n(u)-q(u)
\right|\,du.
\end{align*}
The first term tends to zero by the standard \(L^1\)-consistency of
the empirical normal quantile function. The second term also tends to
zero. Indeed, \(q\) is monotone and belongs to \(L^1(0,1)\), and
\(k/(n+1)\in((k-1)/n,k/n]\); hence \(\widetilde q_n\) is a tagged
step-function approximation to \(q\). Therefore
\[
    \int_0^1
    |\widetilde q_n(u)-q(u)|\,du
    \rightarrow 0.
\]
Consequently,
\[
    \frac{1}{n}\sum_{k=1}^n|\mu_k-q_{k,n}| \rightarrow 0.
\]

For the second claim, note that
\[
m_{k\ell}-\mu_{k}\mu_{\ell}
=
\cov(Z_{k:n},Z_{\ell:n}).
\]
For $j\in\{1, \ldots, n\}$, $Z_{j:n}$ is a coordinatewise increasing function of $(Z_1, \ldots, Z_n)$. Hence, by the Harris–FKG inequality,
\[
\cov(Z_{k:n},Z_{\ell:n})\geq 0.
\]
Therefore
\[
\sum_{k,\ell=1}^n
\left|
\cov(Z_{k:n},Z_{\ell:n})
\right|
=
\sum_{k,\ell=1}^n
\cov(Z_{k:n},Z_{\ell:n})
=
\var\left(\sum_{k=1}^n Z_{k:n}\right)
=
\var\left(\sum_{k=1}^n Z_k\right)
=
n.
\]
Thus
\[
\frac{1}{n^2}
\sum_{k,\ell=1}^n
\left|
m_{k\ell}-\mu_{k}\mu_{\ell}
\right|
=
\frac{1}{n}\to0.
\]

Finally,
\begin{align*}
&\frac{1}{n^2}
\sum_{k,\ell=1}^n
C_{k\ell,n}
\left|
\mu_{k}\mu_{\ell}-q_{k,n}q_{\ell,n}
\right|\le \\
&\qquad\qquad\left(\frac{1}{n}\sum_{k=1}^n|\mu_{k}-q_{k,n}|\right)
\left(\frac{1}{n}\sum_{\ell=1}^n|\mu_{\ell}|\right)+
\left(\frac{1}{n}\sum_{\ell=1}^n|\mu_{\ell}-q_{\ell,n}|\right)
\left(\frac{1}{n}\sum_{k=1}^n|q_{k,n}|\right)
=o(1),
\end{align*}
Here
\[
\frac{1}{n}\sum_{k=1}^n|q_{k,n}|
=
\int_0^1|\widetilde q_n(u)|\,du
=
O(1),
\]
since
\[
\int_0^1
|\widetilde q_n(u)-q(u)|\,du
\rightarrow0
\]
and \(q\in L^1(0,1)\). Moreover, by the first claim,
\[
\frac{1}{n}\sum_{k=1}^n|\mu_k|
\le
\frac{1}{n}\sum_{k=1}^n|\mu_k-q_{k,n}|
+
\frac{1}{n}\sum_{k=1}^n|q_{k,n}|
=
O(1).
\]
Combining this with the covariance bound gives the result.
\end{proof}

\begin{lmm}\label{lmm:B-weight-approx}
For $B_{k\ell}$ defined as in \eqref{eq:Bij} and $h_2\to0$, we have uniformly in $k,\ell$,
\[
B_{k\ell}
=
H_{k\ell}
+
O(h_2),
\]
where 
\[
H_{k\ell}:=1-\max\left(\frac{k}{n},\frac{\ell}{n}\right).
\]
\end{lmm}

\begin{proof}
Since \(K\) is compactly supported, there exists \(M<\infty\) such that
\[
\bar K(t)=0\quad\text{for }t\le -M,
\qquad
\bar K(t)=1\quad\text{for }t\ge M.
\]
Hence
\[
\bar K\left(\frac{v-k/n}{h_2}\right)
=
\bone\{v\ge k/n\}
\]
except possibly when
\[
|v-k/n|\le Mh_2.
\]
Similarly,
\[
\bar K\left(\frac{v-\ell/n}{h_2}\right)
=
\bone\{v\ge \ell/n\}
\]
except possibly when
\[
|v-\ell/n|\le Mh_2.
\]
Therefore the product differs from
\[
\bone\{v\ge k/n\}\bone\{v\ge \ell/n\}
=
\bone\left\{v\ge \max(k/n,\ell/n)\right\}
\]
only on a set of Lebesgue measure $O(h_2)$. Since all functions are bounded
by $1$,
\[
\left|
B_{k\ell}
-
\int_0^1
\bone\left\{v\ge \max(k/n,\ell/n)\right\}\,dv
\right|
\le Ch_2.
\]
But
\[
\int_0^1
\bone\left\{v\ge \max(k/n,\ell/n)\right\}\,dv
=
1-\max(k/n,\ell/n).
\]
The result follows.
\end{proof}

\begin{lmm}\label{lmm:B-weighted-m-limit}
For $B_{k\ell}$ defined as in \eqref{eq:Bij}, $m_{k\ell}$ defined as in~\eqref{eq:rank-scores-centered}, and $h_2\to0$,
\[
\frac{1}{n^2}\sum_{k\ne \ell}B_{k\ell}m_{k\ell}
\to
\frac{1}{2\pi\sqrt3}.
\]
\end{lmm}

\begin{proof}
By Lemma~\ref{lmm:order-stat-average}, since $0\le B_{k\ell}\leq 1$,
\[
\frac{1}{n^2}\sum_{k,\ell=1}^n
B_{k\ell}\{m_{k\ell}-q_{k,n}q_{\ell,n}\}
=o(1).
\]
The diagonal terms are negligible because
\[
\frac{1}{n^2}\sum_{k=1}^n B_{kk}|m_{kk}|
\le
\frac{1}{n^2}\sum_{k=1}^n \ee Z_{k:n}^2
=
\frac{1}{n^2}\ee\sum_{k=1}^n Z_{k:n}^2
=
\frac{1}{n}
=o(1).
\]
Thus it suffices to show
\[
\frac{1}{n^2}\sum_{k,\ell=1}^n
B_{k\ell}q_{k,n}q_{\ell,n}
\to
\frac{1}{2\pi\sqrt3}.
\]

By Lemma~\ref{lmm:B-weight-approx}, $B_{k\ell} = H_{k\ell}+O(h_2)$. Since $n^{-1}\sum_{k=1}^n |q_{k,n}|=O(1)$, we have
\[
\left|
\frac{1}{n^2}\sum_{k,\ell=1}^n
(B_{k\ell}-H_{k\ell})q_{k,n}q_{\ell,n}
\right|
\le
Ch_2
\left(\frac{1}{n}\sum_{k=1}^n |q_{k,n}|\right)^2
=o(1).
\]
Therefore
\[
\frac{1}{n^2}\sum_{k,\ell=1}^n
B_{k\ell}q_{k,n}q_{\ell,n}
=
\frac{1}{n^2}\sum_{k,\ell=1}^n
H_{k\ell}q_{k,n}q_{\ell,n}
+o(1).
\]
We now pass to the limit in the Riemann sum. Since
\[
q\in L^1(0,1)
\]
and $0\le 1-\max(x,y)\le1,$
an endpoint truncation argument gives
\[
\frac{1}{n^2}\sum_{k,\ell=1}^n
H_{k\ell}q_{k,n}q_{\ell,n}
\to
\int_0^1\int_0^1
\{1-\max(x,y)\}q(x)q(y)\,dx\,dy.
\]
Indeed, on \([\varepsilon,1-\varepsilon]^2\) this is ordinary Riemann
convergence, while the contribution of the complement is bounded by
\[
C
\left(
\int_0^\varepsilon |q(u)|\,du
+
\int_{1-\varepsilon}^1 |q(u)|\,du
\right),
\]
which tends to zero as $\varepsilon\downarrow0$.
Finally, the identity
\[
1-\max(x,y) = \int_0^1 \bone\{x\le t,\ y\le t\}\,dt,
\]
yields
\begin{align*}
\int_0^1\int_0^1
\{1-\max(x,y)\}q(x)q(y)\,dx\,dy
&=
\int_0^1
\left(\int_0^t q(x)\,dx\right)^2dt.
\end{align*}

With $z=q(x)$, $dx=\varphi(z)\,dz$, so
\[
\int_0^t q(x)\,dx = \int_{-\infty}^{q(t)}z\varphi(z)\,dz = -\varphi(q(t)).
\]
Therefore
\[
\int_0^1\left(\int_0^t q(x)\,dx\right)^2dt= \int_0^1\varphi(q(t))^2\,dt.
\]
Changing variables again, $t=\Phi(z)$, gives
\[
\int_0^1\varphi(q(t))^2\,dt
=
\int_{-\infty}^{\infty}\varphi(z)^3\,dz
=
\frac{1}{2\pi\sqrt3}.
\]
This proves the result.
\end{proof}

\begin{lmm}\label{lmm:Toff2-aux-bounds}
Suppose that Assumptions~\ref{asmp:kernel} and~\ref{asmp:band} hold. With $\beta_i$ as defined in \eqref{eq:rank-scores-centered} we have
\[
|\beta_i|
\le
C\left\{1+\log\left(\frac{n}{i\wedge(n+1-i)}\right)\right\}.
\]
Consequently,
\[
\sum_{i=1}^n\beta_iL_i=o(n^2h_1).
\]
\end{lmm}

\begin{proof}
For $Z_{1:n} \leq\cdots\leq Z_{n:n}$ the order statistics of $n$ i.i.d. standard normal random variables, we have
\[
\ee Z_{i:n}^2
\le
C\left\{1+\log\left(\frac{n}{i\wedge(n+1-i)}\right)\right\}.
\]
Since
\[
        |\beta_i|
        =
        |\ee(Z_{i:n}^2-1)|
        \le
        1+\ee Z_{i:n}^2,
\]
the first inequality holds. 
Since $\sum_i\beta_i=0$,
\[
\sum_{i=1}^n\beta_iL_i
=
\sum_{i=1}^n\beta_i(L_i-nh_1).
\]
On the interior,
\[
L_i-nh_1=o(nh_1)
\]
uniformly; therefore the interior contribution is $o(n^2h_1)$. On the boundary, there are
$O(nh_1)$ indices on each side, and
\[
|L_i-nh_1|\le Cnh_1.
\]
Using the logarithmic bound on $\beta_i$, the boundary contribution is
\[
2\sum_{i = 1}^m (1 + \log(\frac{n}{i}))Cnh_1 = O(n^2h_1^2(1 + \log(\frac{1}{h_1}))) = o(n^2h_1),
\]
where $m = O(nh_1)$. Hence
\[
\sum_{i=1}^n\beta_iL_i=o(n^2h_1).
\]
\end{proof}

\begin{lmm}\label{lmm:m-row-identities}
Let $m_{rs}$ be defined as in \eqref{eq:rank-scores-centered}. Then,
\[
\max_{r,s}|m_{rs}|\le C(1+\log n).
\]
Consequently, uniformly over \(i\ne j\),
\[
\abs{\sum_{r\notin\{i,j\}}m_{ir}}\le C(1+\log n),
\qquad
\abs{\sum_{\substack{r\ne s\\ r,s\notin\{i,j\}}}m_{rs}}\le C(1+\log n).
\]
\end{lmm}

\begin{proof} 
Note that
\[
|m_{rs}|
\le
\ee[Z_{r:n}^2]^{1/2}\ee[Z_{s:n}^2]^{1/2}
\le
\ee\max_{1\le a\le n}Z_a^2
\le
C(1+\log n).
\]
Since $\sum_{r\ne s}m_{rs}=0$ and $\sum_{s=1}^n m_{rs}=1$ by \eqref{eq:score-identities-centered}, the two displayed consequences follow immediately from
\[
\sum_{r\notin\{i,j\}}m_{ir}
=
1-m_{ii}-m_{ij},\qquad
\sum_{\substack{r\ne s\\ r,s\notin\{i,j\}}}m_{rs} = -\sum_{\substack{r\ne s\\ \{r,s\}\cap\{i,j\}\neq\varnothing}}m_{rs},
\]
since only row sums and the finitely many corrections $m_{ii},m_{jj},m_{ij},m_{ji}$ appear, each of size $O(1+\log n)$.
\end{proof}

\begin{lmm}\label{lmm:aggregate-third-order-hermite}
Let $a_{ij}$ and $B_{ij}$ be as defined in~\eqref{eq:aij} and~\eqref{eq:Bij} and define
\[
F_n(X)
:=
\sum_{i\ne j}a_{ij}\bone\{R_1^x=i,\ R_2^x=j\},
\qquad
G_n(Y)
:=
\sum_{i\ne j}B_{i j}\bone\{R_1^y=i,\ R_2^y=j\}.
\]
Let $P_m$ denote the orthogonal projection in $L^2(\gamma_n)$ onto the
$m$-th Hermite chaos. Then
\[
\ee_\rho[F_n(X)G_n(Y)]
=
\sum_{m=0}^\infty
\rho^m
\langle P_mF_n,P_mG_n\rangle_{L^2(\gamma_n)}.
\]
Consequently,
\[
\mathcal{R}_n^{\mathrm{off}}(\rho)
=
\frac{n(n-1)}{n^2h_1}
\sum_{m=3}^\infty
\rho^{m-3}
\langle P_mF_n,P_mG_n\rangle .
\]
Moreover, uniformly for $|\rho| < 1$,
\[
\left|
\rho\mathcal R_n^{\mathrm{off}}(\rho)
\right|
\le
C\frac{|\rho|}{h_1^{1/2}}.
\]
In particular, if
\[
        \rho_n
        =
        O\left\{
        \Bigl(\frac{h_1}{n}\Bigr)^{1/4}
        \right\},
\]
then
\[
        |\rho_n^3
        \mathcal R_n^{\mathrm{off}}(\rho_n)|
        =
        o\left\{
        \Bigl(\frac{h_1}{n}\Bigr)^{1/2}
        \right\}.
\]
\end{lmm}

\begin{proof}
The first two assertions follow from
Proposition~\ref{prop:hilbert_mehler}, applied with
\(\mathcal{H}=\rr\).
It remains to prove the uniform bound.

By the Cauchy--Schwarz inequality, 
\[
\begin{aligned}
\left|
\rho\mathcal R_n^{\mathrm{off}}(\rho)
\right|
&\le
\frac{n(n-1)}{n^2h_1}
\sum_{m=3}^\infty
|\rho|^{m-2}
\|P_mF_n\|_2\|P_mG_n\|_2
\\
&\le
\frac{n(n-1)}{n^2h_1}
|\rho|
\left(
\sum_{m=3}^\infty \|P_mF_n\|_2^2
\right)^{1/2}
\left(
\sum_{m=3}^\infty |\rho|^{2(m-3)}
\|P_mG_n\|_2^2
\right)^{1/2}
\\
& \le
\frac{n(n-1)}{n^2h_1}
|\rho|\|F_n\|_2\|G_n\|_2.
\end{aligned}
\]
By \eqref{eq:kernel-weight-a-norms} and
\eqref{eq:kernel-weight-B-norms},
\[
        \|F_n\|_2^2=O(h_1),
        \qquad
        \|G_n\|_2^2=O(1).
\]

Combining these estimates,
\[
\left|
\rho\mathcal R_n^{\mathrm{off}}(\rho)
\right|
\le
C\frac{|\rho|}{h_1^{1/2}}.
\]
In particular,
\begin{align*}
\left|
\rho_n^3\mathcal R_n^{\mathrm{off}}(\rho_n)
\right|
=
\rho_n^2
\left|
\rho_n\mathcal R_n^{\mathrm{off}}(\rho_n)
\right|
\le
C\frac{|\rho_n|^3}{h_1^{1/2}}
=
o\left\{
\Bigl(\frac{h_1}{n}\Bigr)^{1/2}
\right\},
\end{align*}
since \(nh_1\to\infty\).
\end{proof}

\subsubsection{Diagonal and off-diagonal terms}

\begin{lmm}\label{lmm:Tdiag}
Suppose that Assumptions~\ref{asmp:kernel} and~\ref{asmp:band} hold. For
\[
F_n^{\mathrm{diag}}(X)
:=
a_{R_1^xR_1^x},
\qquad
G_n^{\mathrm{diag}}(Y)
:=
B_{R_1^yR_1^y},
\]
we have
\[
T_{n,\rho}^{\mathrm{diag}}
=
\frac{1}{nh_1}
\ee_\rho
\left[
F_n^{\mathrm{diag}}(X)
G_n^{\mathrm{diag}}(Y)
\right].
\]
Moreover, for every fixed $\rho_0\in(0,1)$, uniformly for
$|\rho|\le\rho_0$,
\[
T_{n,\rho}^{\mathrm{diag}}
-
T_{n,0}^{\mathrm{diag}}
=
O\left(
\frac{|\rho|}{n^2h_1^{3/2}}
\right)
+
O\left(
\frac{\rho^2}{nh_1}
\right).
\]
Consequently, if
\[
\rho_n\asymp
\Bigl(\frac{h_1}{n}\Bigr)^{1/4}
\]
and $nh_1\to\infty$, then
\[
T_{n,\rho_n}^{\mathrm{diag}}
-
T_{n,0}^{\mathrm{diag}}
=
o(\rho_n^2).
\]
\end{lmm}

\begin{proof}
By Proposition~\ref{prop:hilbert_mehler}, applied with
$\mathcal H = \rr$,
\[
nh_1T_{n,\rho}^{\mathrm{diag}}
=
\sum_{m=0}^{\infty}
\rho^m
\left\langle
P_mF_n^{\mathrm{diag}},
P_mG_n^{\mathrm{diag}}
\right\rangle .
\]
Hence
\begin{align}
T_{n,\rho}^{\mathrm{diag}}
-
T_{n,0}^{\mathrm{diag}}
&=
\frac{\rho}{nh_1}
\left\langle
P_1F_n^{\mathrm{diag}},
P_1G_n^{\mathrm{diag}}
\right\rangle
\nonumber\\
&\quad+
\frac{1}{nh_1}
\sum_{m=2}^{\infty}
\rho^m
\left\langle
P_mF_n^{\mathrm{diag}},
P_mG_n^{\mathrm{diag}}
\right\rangle .
\label{eq:Tdiag-chaos}
\end{align}

We first consider the linear term. Write
\[
A_n:=\sum_{k=1}^na_{kk}\mu_k,
\qquad
C_n:=\sum_{\ell=1}^nB_{\ell\ell}\mu_\ell,
\]
where $\mu_k=\ee Z_{k:n}$. As in the calculation of the first
Hermite projection,
\[
\ee_0[
F_n^{\mathrm{diag}}(X)X_1
]
=
\frac{A_n}{n},
\]
while, for $r\ne1$,
\[
\ee_0[
F_n^{\mathrm{diag}}(X)X_r
]
=
-\frac{A_n}{n(n-1)}.
\]
Similarly,
\[
\ee_0[
G_n^{\mathrm{diag}}(Y)Y_1
]
=
\frac{C_n}{n},
\qquad
\ee_0[
G_n^{\mathrm{diag}}(Y)Y_r
]
=
-\frac{C_n}{n(n-1)},
\qquad r\ne1.
\]
Therefore
\begin{equation}
\left\langle
P_1F_n^{\mathrm{diag}},
P_1G_n^{\mathrm{diag}}
\right\rangle
=
\frac{A_nC_n}{n(n-1)}.
\label{eq:Tdiag-first-chaos}
\end{equation}

We now sharpen the bound on $A_n$. Let
\[
J(k):=n+1-k.
\]
Since $K$ is symmetric, $K^2$ is even, and a change of variables gives
\[
a_{kk}
=
\int_{-k/(nh_1)}^{(n-k)/(nh_1)}
K^2(t)\,dt
\]
and
\[
a_{J(k)J(k)}
=
\int_{-(k-1)/(nh_1)}^{(n+1-k)/(nh_1)}
K^2(t)\,dt.
\]
Since $K$ is bounded,
\begin{equation}
\left|
a_{kk}-a_{J(k)J(k)}
\right|
\le
\frac{C}{nh_1}.
\label{eq:diag-reflection-difference}
\end{equation}
Moreover, because $K$ is supported on $[-1,1]$, the difference in
\eqref{eq:diag-reflection-difference} vanishes unless $k$ belongs to
one of the two boundary regions of width $O(nh_1)$. Thus there is a
set $\mathcal B_n\subset\{1,\ldots,n\}$ with
\[
|\mathcal B_n|=O(nh_1)
\]
such that
\[
a_{kk}=a_{J(k)J(k)},
\qquad k\notin\mathcal B_n.
\]

By symmetry of the Gaussian order statistics,
\[
\mu_{J(k)}=-\mu_k.
\]
Hence
\[
A_n
=
\frac12
\sum_{k=1}^n
\mu_k
\left(
a_{kk}-a_{J(k)J(k)}
\right),
\]
and therefore, using
$\sum_{k=1}^n\mu_k^2=O(n)$,
\begin{align*}
|A_n|
&\le
\frac{C}{nh_1}
\sum_{k\in\mathcal B_n}|\mu_k|
\\
&\le
\frac{C}{nh_1}
|\mathcal B_n|^{1/2}
\left(
\sum_{k=1}^n\mu_k^2
\right)^{1/2}
\\
&=
O\left(h_1^{-1/2}\right).
\end{align*}
On the other hand, since $0\le B_{\ell\ell}\le1$,
\[
|C_n|
\le
\sum_{\ell=1}^n|\mu_\ell|
=
O(n).
\]
It follows from \eqref{eq:Tdiag-first-chaos} that
\[
\left|
\left\langle
P_1F_n^{\mathrm{diag}},
P_1G_n^{\mathrm{diag}}
\right\rangle
\right|
=
O\left(
\frac{1}{nh_1^{1/2}}
\right),
\]
and hence
\begin{equation}
\frac{|\rho|}{nh_1}
\left|
\left\langle
P_1F_n^{\mathrm{diag}},
P_1G_n^{\mathrm{diag}}
\right\rangle
\right|
=
O\left(
\frac{|\rho|}{n^2h_1^{3/2}}
\right).
\label{eq:Tdiag-linear-bound}
\end{equation}

It remains to control the higher-order chaoses. By
Cauchy--Schwarz, orthogonality of the Hermite chaoses, and
$|\rho|\le\rho_0<1$,
\begin{align*}
&
\left|
\sum_{m=2}^{\infty}
\rho^m
\left\langle
P_mF_n^{\mathrm{diag}},
P_mG_n^{\mathrm{diag}}
\right\rangle
\right|
\\
&\qquad\le
\rho^2
\left(
\sum_{m=2}^{\infty}
\|P_mF_n^{\mathrm{diag}}\|_2^2
\right)^{1/2}
\left(
\sum_{m=2}^{\infty}
|\rho|^{2(m-2)}
\|P_mG_n^{\mathrm{diag}}\|_2^2
\right)^{1/2}
\\
&\qquad\le
C\rho^2
\|F_n^{\mathrm{diag}}\|_2
\|G_n^{\mathrm{diag}}\|_2.
\end{align*}
Since $a_{kk}=O(1)$ uniformly and $0\le B_{\ell\ell}\le1$,
\[
\|F_n^{\mathrm{diag}}\|_2^2
=
\frac1n\sum_{k=1}^na_{kk}^2
=
O(1),
\qquad
\|G_n^{\mathrm{diag}}\|_2^2
=
\frac1n\sum_{\ell=1}^nB_{\ell\ell}^2
\le1.
\]
Consequently,
\begin{equation}
\frac1{nh_1}
\left|
\sum_{m=2}^{\infty}
\rho^m
\left\langle
P_mF_n^{\mathrm{diag}},
P_mG_n^{\mathrm{diag}}
\right\rangle
\right|
=
O\left(
\frac{\rho^2}{nh_1}
\right).
\label{eq:Tdiag-higher-bound}
\end{equation}
Combining \eqref{eq:Tdiag-chaos},
\eqref{eq:Tdiag-linear-bound}, and
\eqref{eq:Tdiag-higher-bound} proves
\[
T_{n,\rho}^{\mathrm{diag}}
-
T_{n,0}^{\mathrm{diag}}
=
O\left(
\frac{|\rho|}{n^2h_1^{3/2}}
\right)
+
O\left(
\frac{\rho^2}{nh_1}
\right).
\]

Finally, if
\[
\rho_n\asymp
\Bigl(\frac{h_1}{n}\Bigr)^{1/4},
\]
then
\begin{align*}
\frac{|\rho_n|/(n^2h_1^{3/2})}{\rho_n^2}
&=
O\left(
\frac{1}{n^2h_1^{3/2}|\rho_n|}
\right)
=
O\left(
(nh_1)^{-7/4}
\right)
=o(1),
\\
\frac{\rho_n^2/(nh_1)}{\rho_n^2}
&=
\frac1{nh_1}
=o(1).
\end{align*}
Therefore
\[
T_{n,\rho_n}^{\mathrm{diag}}
-
T_{n,0}^{\mathrm{diag}}
=
o(\rho_n^2).
\]
\end{proof}

\begin{lmm}\label{lmm:Toff0}
Let $T_{n0}^{\mathrm{off}}$ and $\hat{b}_n^{\mathrm{off}}$ be defined as in \eqref{eq:Tnoff} and \eqref{eq:bnhatoff}. Under independence of $X$ and $Y$, that is, when $\rho=0$,
\[
T_{n0}^{\mathrm{off}} = \hat{b}_n^{\mathrm{off}}.
\]
\end{lmm}

\begin{proof}
Substituting
\[
        p_{0,n}(k_1,k_2,\ell_1,\ell_2)
        =
        \frac{1}{n^2(n-1)^2}
\]
from Proposition~\ref{prop:GaussianRankProbExp} into
\eqref{eq:Tnoff}, and using \eqref{eq:aij} and \eqref{eq:Bij},
gives
\begin{align*}
T_{n0}^{\mathrm{off}}
&=
\frac{n(n-1)}{n^2h_1}
\frac{1}{n^2(n-1)^2}
\left(
    \sum_{k_1\ne k_2}a_{k_1k_2}
\right)
\left(
    \sum_{\ell_1\ne\ell_2}B_{\ell_1\ell_2}
\right)\\
&=
\frac{1}{n^3(n-1)h_1}
\left(
    \sum_{k_1\ne k_2}a_{k_1k_2}
\right)
\left(
    \sum_{\ell_1\ne\ell_2}B_{\ell_1\ell_2}
\right)\\
&=
\hat b_n^{\mathrm{off}}.
\end{align*}
This proves the result.
\end{proof}

\begin{lmm}\label{lmm:Toff1}
Suppose that Assumptions~\ref{asmp:kernel} and~\ref{asmp:band} hold.
Let $T_{n1}^{\mathrm{off}}$ be as defined in~\eqref{eq:Tnoff}. Then
\[
T_{n1}^{\mathrm{off}} = O\left(\frac{1}{nh_1}\right).
\]
\end{lmm}

\begin{proof}
Substituting \eqref{lochol2} into \eqref{eq:Tnoff}, and using the
definitions of \(L_k\) and \(\bar b_\ell\) from
Lemma~\ref{lem:deterministic-kernel-weight-bounds}, set
\[
        M_a:=\sum_{k=1}^n\mu_kL_k,
        \qquad
        M_B:=(n-1)\sum_{\ell=1}^n\mu_\ell\bar b_\ell.
\]
Since $a_{k_1k_2} = a_{k_2k_1}$ and
$B_{\ell_1\ell_2} = B_{\ell_2\ell_1}$, we obtain
\begin{align*}
\sum_{\substack{k_1\ne k_2\\ \ell_1\ne \ell_2}}
a_{k_1k_2}
B_{\ell_1\ell_2}
\left(
\mu_{k_1}\mu_{\ell_1}
+
\mu_{k_2}\mu_{\ell_2}
\right)
=2\left(
\sum_{k_1\ne k_2} a_{k_1k_2}\mu_{k_1}
\right)
\left(
\sum_{\ell_1\ne \ell_2} B_{\ell_1\ell_2}\mu_{\ell_1}
\right) = 2 M_a M_B .
\end{align*}
Similarly,
\begin{align*}
\sum_{\substack{k_1\ne k_2\\ \ell_1\ne \ell_2}}
a_{k_1k_2}
B_{\ell_1\ell_2}
(\mu_{k_1}+\mu_{k_2})(\mu_{\ell_1}+\mu_{\ell_2})
=
4M_a M_B .
\end{align*}
Therefore
\begin{align}
T_{n1}^{\mathrm{off}}=
\frac{n(n-1)}{n^2h_1}
\left[
\frac{2M_a M_B}{n^2(n-1)^2}
+
\frac{4M_a M_B}{n^2(n-1)^2(n-2)}
\right]=
\frac{2M_a M_B}{n^3(n-1)h_1}
\left(1+\frac{2}{n-2}\right).
\label{eq:Tn1-factorized}
\end{align}
Since $\sum_{i = 1}^n \mu_i = 0$, 
\[
M_a = \sum_{i = 1}^n \mu_iL_i = \sum_{i = 1}^n \mu_i(L_i - \bar{L}),
\]
so Lemma~\ref{lem:Lmu-bound} gives $M_a = O(n)$.

As we have $\abs{(n - 1)\bar{b}_\ell}
=
\left|
\sum_{j\ne \ell}B_{\ell j}
\right|
\le n,  $ it follows that 
\begin{align*}
\abs{M_B} =
\left|
\sum_{\ell=1}^n \mu_\ell (n - 1)\bar{b}_\ell
\right| \le
\sum_{\ell=1}^n \abs{\mu_\ell}\, \abs{(n - 1)\bar{b}_\ell}\le
n\sum_{\ell=1}^n \abs{\mu_\ell} =
O(n^2).
\end{align*}
Substituting these estimates into \eqref{eq:Tn1-factorized} gives
\[
T_{n1}^{\mathrm{off}} = O\left(\frac{1}{nh_1}\right).
\]
\end{proof}

\begin{lmm}\label{lmm:Toff2}
Suppose that Assumptions~\ref{asmp:kernel} and~\ref{asmp:band} hold.
Let $T_{n2}^{\mathrm{off}}$ be as defined in~\eqref{eq:Tnoff}. Then
\[
T_{n2}^{\mathrm{off}} = \frac{1}{2\pi\sqrt{3}} + o(1).
\]
\end{lmm}

\begin{proof}
From Proposition~\ref{prop:GaussianRankProbExp}
\[
p_{2,n}:= p_{2,n}(k_1, k_2, \ell_1, \ell_2)
=
p_{2,n}^{H_2}
+
p_{2,n}^{\mathrm{pair}},
\]
where $p_{2,n}^{H_2}$ is given by \eqref{eq:p2-H2}. Using the display for $p_{2,n}^{\mathrm{pair}}$ in \eqref{eq:p2-pair}, we let $p_{2,n}^{\mathrm{pair}} = \sum_{i = 1}^4 p_{2,n}^{\mathrm{pair}, i}$ where 
\begin{align*}
    p_{2,n}^{\mathrm{pair}, 1} &:= \frac{m_{k_1 k_2} m_{\ell_1 \ell_2}}{n^2 (n-1)^2}, \\
    p_{2,n}^{\mathrm{pair}, 2} &:=\frac{
\left( \sum_{r \notin \{k_1,k_2\}} m_{k_1 r} \right)
\left( \sum_{s \notin \{\ell_1,\ell_2\}} m_{\ell_1 s} \right)
}{n^2 (n-1)^2 (n-2)}, \\
p_{2,n}^{\mathrm{pair}, 3} &:=
\frac{
\left( \sum_{r \notin \{k_1,k_2\}} m_{k_2 r} \right)
\left( \sum_{s \notin \{\ell_1,\ell_2\}} m_{\ell_2 s} \right)
}{n^2 (n-1)^2 (n-2)}, \\
p_{2,n}^{\mathrm{pair}, 4} &:=
\frac{
\left( \sum_{\substack{r \neq s \\ r,s \notin \{k_1,k_2\}}} m_{rs} \right)
\left( \sum_{\substack{r \neq s \\ r,s \notin \{\ell_1,\ell_2\}}} m_{rs} \right)
}{2\, n^2 (n-1)^2 (n-2)(n-3)}.
\end{align*}

We first treat the part of $T_{n2}^{\mathrm{off}}$ that involves $p_{2,n}^{H_2}$. Recall $a_{ij}$ and $B_{ij}$ as defined in~\eqref{eq:aij} and~\eqref{eq:Bij} respectively. 
Let
\[
T_{n,2}^{H_2}
:=
\frac{n(n-1)}{n^2h_1}
\sum_{i\ne j}\sum_{k\ne \ell}
a_{ij}B_{k\ell}
p_{2,n}^{H_2}(i,j,k,\ell).
\]
We split this into two terms,
\[
T_{n,2}^{H_2}=T_{n,2}^{H_2,1}+T_{n,2}^{H_2,2}.
\]
First consider
\[
T_{n,2}^{H_2,1}
:= \frac{1}{2n^3(n-1)h_1} \sum_{i\ne j}\sum_{k\ne \ell}
a_{ij}B_{k\ell} (\beta_i\beta_k+\beta_j\beta_\ell).
\]
Using the symmetry \(a_{ij}=a_{ji}\) and \(B_{k\ell}=B_{\ell k}\), the sum factorizes as
\[
\sum_{i\ne j}\sum_{k\ne \ell}
a_{ij}B_{k\ell}
(\beta_i\beta_k+\beta_j\beta_\ell)
=
2
\left(\sum_{i=1}^n \beta_i L_i\right)
\left(\sum_{k=1}^n \beta_k (n - 1)\bar{b}_k\right).
\]
Therefore
\[
T_{n,2}^{H_2,1}
=
\frac{1}{n^3(n-1)h_1} \left(\sum_{i=1}^n \beta_i L_i\right)
\left(\sum_{k=1}^n \beta_k (n - 1)\bar{b}_k\right).
\]
By Lemma~\ref{lmm:Toff2-aux-bounds},
\[
    \sum_{i=1}^n\beta_iL_i=o(n^2h_1).
\]
Moreover, since \(0\le B_{k\ell}\le1\),
\[
\left|
    \sum_{k=1}^n\beta_k(n-1)\bar b_k
\right|
\le
n\sum_{k=1}^n|\beta_k|
=
O(n^2).
\]
Combining the two estimates gives
\[
T_{n,2}^{H_2,1} = \frac{1}{n^3(n-1)h_1} \,o(n^2h_1)\,O(n^2) = o(1).
\]
It remains to control
\[
T_{n,2}^{H_2,2} := \frac{1}{2n^3(n-1)(n-2)h_1}
\sum_{i\ne j}\sum_{k\ne \ell}
a_{ij}B_{k\ell}
(\beta_i+\beta_j)(\beta_k+\beta_\ell).
\]
Again using symmetry,
\[
\sum_{i\ne j}\sum_{k\ne \ell}
a_{ij}B_{k\ell}
(\beta_i+\beta_j)(\beta_k+\beta_\ell)
=
4
\left(\sum_{i=1}^n \beta_i L_i\right)
\left(\sum_{k=1}^n \beta_k (n - 1)\bar{b}_k\right).
\]
Hence
\[
T_{n,2}^{H_2,2} = \frac{2}{n^3(n-1)(n-2)h_1}
\left(\sum_{i=1}^n \beta_i L_i\right)
\left(\sum_{k=1}^n \beta_k (n - 1)\bar{b}_k\right) = o(1),
\]
using the bounds already proved. Therefore
\[
T_{n,2}^{H_2} = T_{n,2}^{H_2,1} + T_{n,2}^{H_2,2} = o(1).
\]
We now focus on the contribution of terms involving $p_{2,n}^{\mathrm{pair}}$. Let
\[
T_{n,2}^{\mathrm{pair}}
=
\frac{(n-1)}{nh_1}
\sum_{i\ne j}\sum_{k\ne \ell}
a_{ij}B_{k\ell}
p_{2,n}^{\mathrm{pair}}(i,j,k,\ell),
\]
which decomposes into four terms as
\[
T_{n,2}^{\mathrm{pair}} = T_{n,2}^{(1)} + T_{n,2}^{(2)} + T_{n,2}^{(3)} + T_{n,2}^{(4)},
\]
corresponding to $p_{2,n}^{\mathrm{pair}, i}$ for $i = 1, \ldots, 4$. First consider the leading term
\[
T_{n,2}^{(1)} := \frac{1}{n^3(n-1)h_1}
\left(\sum_{i\ne j}a_{ij}m_{ij}\right)
\left(\sum_{k\ne \ell}B_{k\ell}m_{k\ell}\right).
\]
By Proposition~\ref{prop:ApproxWeightm},
\[
    \frac{1}{n^2h_1}\sum_{i\ne j}a_{ij}m_{ij}\rightarrow 1,
\]
while Lemma~\ref{lmm:B-weighted-m-limit} gives
\[
    \frac{1}{n^2}\sum_{k\ne\ell}B_{k\ell}m_{k\ell}
    \rightarrow\frac{1}{2\pi\sqrt3}.
\]
Therefore
\[
    T_{n,2}^{(1)}
    =
    \frac{1}{2\pi\sqrt{3}}+o(1).
\]
It remains to show that
\[
T_{n,2}^{(2)}+T_{n,2}^{(3)}+T_{n,2}^{(4)}=o(1).
\]
We next focus on $T_{n,2}^{(2)}$; and the analysis of the term $T_{n,2}^{(3)}$ is identical. With the notation
\[
U_{ij}:=\sum_{r\notin\{i,j\}}m_{ir} = \sum_{r=1}^n m_{ir}-m_{ii}-m_{ij},
\qquad
V_{k\ell}:=\sum_{s\notin\{k,\ell\}}m_{ks}.
\]
we obtain the representation
\[
T_{n,2}^{(2)} = \frac{1}{n^3(n-1)(n-2)h_1}
\left(\sum_{i\ne j}a_{ij}U_{ij}\right)
\left(\sum_{k\ne \ell}B_{k\ell}V_{k\ell}\right).
\]
We claim that
\[
\sum_{i\ne j}a_{ij}U_{ij}
=
O(n^2h_1\log n),
\qquad
\sum_{k\ne \ell}B_{k\ell}V_{k\ell}
=
O(n^2\log n).
\]
By Lemma~\ref{lmm:m-row-identities},
\[
|U_{ij}|
=
\left|
\sum_{r\notin\{i,j\}}m_{ir}
\right|
\le
C(1+\log n),
\]
uniformly over $i\ne j$. Therefore
\[
\left|
\sum_{i\ne j}a_{ij}U_{ij}
\right|
\le
C(1+\log n)\sum_{i\ne j}|a_{ij}|
=
O(n^2h_1\log n).
\]
Similarly, we have 
\[
|V_{k\ell}|\le C(1+\log n),
\]
and since $0\le B_{k\ell}\leq 1$, it follows that 
\[
\left|
\sum_{k\ne\ell}B_{k\ell}V_{k\ell}
\right|
=
O(n^2\log n).
\]
Therefore
\[
T_{n,2}^{(2)} = O\left(\frac{1}{n^5h_1}
(n^2h_1\log n)(n^2\log n)\right) = O\left(\frac{\log^2 n}{n}\right) = o(1).
\]
The same argument gives
\[
T_{n,2}^{(3)}=o(1).
\]
Finally consider \(T_{n,2}^{(4)}\). Define
\[
W_{ij}:=
\sum_{\substack{r\ne s\\ r,s\notin\{i,j\}}}m_{rs}.
\]
Then
\[
T_{n,2}^{(4)} = \frac{1}{2n^3(n-1)(n-2)(n-3)h_1}
\left(\sum_{i\ne j}a_{ij}W_{ij}\right)
\left(\sum_{k\ne \ell}B_{k\ell}W_{k\ell}\right).
\]
We claim that
\[
|W_{ij}|\le C(1+\log n).
\]
By Lemma~\ref{lmm:m-row-identities},
\[
\abs{W_{ij}} = \abs{\sum_{\substack{r\ne s\\r,s\notin\{i,j\}}}m_{rs}}
\leq C(1+\log n).
\]
Consequently,
\[
\sum_{i\ne j}a_{ij}W_{ij}
=
O(n^2h_1\log n),
\qquad
\sum_{k\ne \ell}B_{k\ell}W_{k\ell}
=
O(n^2\log n).
\]
It follows that
\[
T_{n,2}^{(4)} = O\left(\frac{\log^2 n}{n^2}\right) = o(1).
\]
Combining the four terms, we obtain
\[
T_{n,2}^{\mathrm{pair}} =
T_{n,2}^{(1)} + T_{n,2}^{(2)} + T_{n,2}^{(3)} + T_{n,2}^{(4)} = \frac{1}{2\pi\sqrt{3}} + o(1).
\]
Combining the two parts gives
\[
T_{n2}^{\mathrm{off}}
=
T_{n,2}^{H_2}
+
T_{n,2}^{\mathrm{pair}}
=
\frac{1}{2\pi\sqrt3}
+
o(1).
\]
\end{proof}

\subsection{Variance stability}\label{subsec:variance-stability}
For \(\rho\in(-1,1)\), write
\begin{equation}
\label{eq:VnRho}
        V_n(\rho)
        :=
        \var_\rho(\hat\tau_n^2).
\end{equation}

\subsubsection{Proof of Theorem~\ref{thm:variance-stability2}}
Set
\[
        \delta_n
        :=
        \Bigl(\frac{h_1}{n}\Bigr)^{1/4}.
\]
By \eqref{condrho},
\[
        |\rho_n|=O(\delta_n).
\]

Taylor's formula at the origin gives
\begin{align}
V_n(\rho_n)-V_n(0)
&=
\rho_nV_n'(0)
+
\frac{\rho_n^2}{2}V_n''(0)
+
\frac{\rho_n^3}{6}V_n^{(3)}(0)
+
R_{n,V},
\label{eq:variance-Taylor-strong}
\end{align}
where
\begin{equation}
\label{eq:variance-Taylor-remainder-strong}
R_{n,V}
=
\frac{\rho_n^4}{6}
\int_0^1
(1-t)^3V_n^{(4)}(t\rho_n)\,dt.
\end{equation}

By Lemmas~\ref{lem:Vprime-bound-centered},
\ref{lem:Vsecond-bound-centered}, and
\ref{lem:Vthird-zero},
\[
        V_n'(0)
        =
        O\left(\frac{h_1^{1/2}}{n}\right),
        \qquad
        V_n''(0)
        =
        O\left(\frac1n\right),
        \qquad
        V_n^{(3)}(0)
        =
        o(\delta_n).
\]
Therefore,
\begin{align*}
|\rho_nV_n'(0)|
&=
O\left(
\delta_n\frac{h_1^{1/2}}{n}
\right)
=
\frac{h_1}{n}
O\left((nh_1)^{-1/4}\right)
=
o\left(\frac{h_1}{n}\right),
\\
|\rho_n|^2|V_n''(0)|
&=
O\left(\frac{\delta_n^2}{n}\right)
=
\frac{h_1}{n}
O\left((nh_1)^{-1/2}\right)
=
o\left(\frac{h_1}{n}\right),
\end{align*}
and
\[
        |\rho_n|^3|V_n^{(3)}(0)|
        =
        O(\delta_n^3)o(\delta_n)
        =
        o(\delta_n^4)
        =
        o\left(\frac{h_1}{n}\right).
\]
Here \(nh_1\to\infty\) follows from Assumptions~\ref{asmp:kernel} and~\ref{asmp:band}.

Finally, choose \(C_0>0\) such that
\[
        |\rho_n|\le C_0\delta_n
\]
for all sufficiently large \(n\). By
Lemma~\ref{lem:Vfourth-uniform},
\[
        \sup_{|\rho|\le C_0\delta_n}
        |V_n^{(4)}(\rho)|
        =
        o(1).
\]
Hence, from \eqref{eq:variance-Taylor-remainder-strong},
\begin{align*}
|R_{n,V}|
&\le
\frac{|\rho_n|^4}{24}
\sup_{|\rho|\le C_0\delta_n}
|V_n^{(4)}(\rho)|
\\
&=
O(\delta_n^4)o(1)
=
o\left(\frac{h_1}{n}\right).
\end{align*}

Substituting these bounds into
\eqref{eq:variance-Taylor-strong} yields
\[
        V_n(\rho_n)-V_n(0)
        =
        o\left(\frac{h_1}{n}\right),
\]
which is \eqref{eq:variance-stability-strong}.

For $n\times n$ symmetric zero-diagonal matrices $A$ and $B$ and permutation $\pi$ on $\{1, \ldots, n\}$, define
\begin{align}\label{eq:Q}
    Q(A, B;\pi) := \sum_{i\neq j} A_{ij}B_{\pi_i \pi_j}.
\end{align}
For \(u,v\in\rr^n\) such that $\sum_{i=1}^n u_i = \sum_{i=1}^n v_i =0$, write
\begin{align*}
    L(u,v;\pi)
    :=
    \sum_{i=1}^n u_i v_{\pi_i}.
\end{align*}
Additionally define 
\begin{align*}
    \Lambda(A,B;u)
    :=
    \sum_{i\ne j}(u_i+u_j)A_{ij}B_{ij}.
\end{align*}
For any $n\times n$ deterministic
matrix \(M=(M_{k\ell})\) and deterministic vector
\(u=(u_i)\) of length $n$, we denote by 
\begin{align*}
        M[\pi]
        :=
        \left(
        M_{\pi_i,\pi_j}
        \right)_{1\le i,j\le n}, 
        \qquad
        u[\pi]
        :=
        \bigl(u_{\pi_i}\bigr)_{i=1}^n,
\end{align*}
the matrix \(M\) and vector $u$ relabelled according to
the permutation $\pi$. Note that these relabelings preserve the relevant norms:
\begin{align*}
        \|M[\pi]\|_{\mathrm F}
        =
        \|M\|_{\mathrm F}, \qquad \|u[\pi]\|_2
        =
        \|u\|_2.
\end{align*}
where $\langle \cdot , \cdot \rangle_{\mathrm F}$ denotes the Frobenius inner product on $\rr^{n \times n}$.
For notational simplicity, let
\begin{align*}
    U_n:=\hat\tau_n^2-\ee_0[\hat\tau_n^2].
\end{align*}
We also use \(V_n(\rho)\) as defined in \eqref{eq:VnRho}.

\subsubsection{Conditional rank scores and Gaussian canonical centring}
Let
\[
R^x=(R^x_1,\ldots,R^x_n)
\quad\text{and}\quad
R^y=(R^y_1,\ldots,R^y_n)
\]
denote the corresponding rank vectors of $X_1,\ldots,X_n$ and
$Y_1,\ldots,Y_n$, respectively. Let $\cR$ be the $\sigma$-algebra
generated by $R^x$ and $R^y$.

\begin{lmm}\label{lmm:first-second-score-conditional-centered}
    Under $\rho=0$,
\begin{equation}\label{eq:first-score-conditional-centered}
        \ee_0[S_{n,1}\mid\cR] = \sum_{i=1}^n\mu_i\mu_{\pi_i}.
\end{equation}
Moreover,
\begin{equation}\label{eq:second-score-conditional-centered}
        \ee_0(S_{n,2}\mid\cR)
        =G_{n,2}^{H_2}(\pi)+G_{n,2}^{\rm pair}(\pi),
\end{equation}
where
\begin{align*}
        G_{n,2}^{H_2}(\pi):=\frac{1}{2}\sum_{i=1}^n\beta_i\beta_{\pi_i},
        \qquad
        G_{n,2}^{\rm pair}(\pi):=\sum_{1\le i<j\le n}m_{ij}m_{\pi_i\pi_j}.
\end{align*}
\end{lmm}

\begin{proof}
Conditional on the $X$-rank vector,
$\ee_0[X_r\mid \cR]=\mu_{R_r^x}$, and conditional on the $Y$-rank vector,
$\ee_0[Y_r\mid\cR]=\mu_{R_r^y}$.  Since the $X$-sample and the
$Y$-sample are independent under $\rho=0$,
\[
 \ee_0[X_rY_r\mid\cR]=\mu_{R_r^x}\mu_{R_r^y}.
\]
After ordering by \(X\)-ranks this gives \eqref{eq:first-score-conditional-centered}.
Similarly,
\[
 \ee_0[(X_r^2-1)(Y_r^2-1)\mid\cR]
 =\beta_{R_r^x}\beta_{R_r^y},
\]
and, for \(r\ne s\),
\[
 \ee_0[X_rX_sY_rY_s\mid\cR]
 = m_{R_r^xR_s^x}m_{R_r^yR_s^y}.
\]
Substituting these identities into \eqref{eq:S1-def-centered} and \eqref{eq:S2-def-centered} gives
\eqref{eq:second-score-conditional-centered}.
\end{proof}

\begin{lmm}
\label{lem:Gaussian-score-canonical-centering}
Let \(\beta_i\) and \(m_{ij}\) be defined in
\eqref{eq:rank-scores-centered}, suppose \(n\ge3\), and assume
\(nh_1\to\infty\). For \(i\ne j\), define
\begin{equation*}
        m_{ij}^{\mathrm c}
        :=
        m_{ij}+\frac{\beta_i+\beta_j}{n-2},
\end{equation*}
and set \(m_{ii}^{\mathrm c}:=0\). Let
\[
        M^{\mathrm c}:=(m_{ij}^{\mathrm c})_{i,j=1}^n.
\]
Then \(M^{\mathrm c}\) is symmetric, has zero diagonal, and has zero
row sums:
\begin{equation}\label{eq:Gaussian-score-canonical-row-sums}
        \sum_{j\ne i}m_{ij}^{\mathrm c}=0,
        \qquad 1\le i\le n.
\end{equation}
For a permutation \(\pi\) of \(\{1,\ldots,n\}\), define
\begin{equation}\label{lochd11}
        \widetilde G_{n,2}
        :=
        \sum_{1\le i<j\le n}
        m_{ij}^{\mathrm c}m_{\pi_i\pi_j}^{\mathrm c}.
\end{equation}
Then
\begin{equation}\label{eq:Gpair-canonical-decomp}
        G_{n,2}^{\mathrm{pair}}
        =
        \widetilde G_{n,2}
        +
        \frac{2}{n-2}G_{n,2}^{H_2}.
\end{equation}
Moreover,
\begin{equation}\label{eq:m-can-global-local}
        \sum_{i\ne j}(m_{ij}^{\mathrm c})^2
        =
        O(n^2),
        \qquad
        \sum_{i\ne j}|m_{ij}^{\mathrm c}|
        =
        O(n^2),
\end{equation}
and, for every fixed \(C_0>0\),
\begin{equation}\label{eq:m-can-local}
        \sum_{\substack{i\ne j\\
        |i-j|\le C_0nh_1+C_0}}
        |m_{ij}^{\mathrm c}|
        =
        O(n^2h_1).
\end{equation}
\end{lmm}

\begin{proof}
By Lemma~\ref{lem:Gaussian-rank-score-identities},
\[
        \sum_{j\ne i}m_{ij}=-\beta_i,
        \qquad
        \sum_{j=1}^n\beta_j=0.
\]
Therefore
\begin{align*}
        \sum_{j\ne i}m_{ij}^{\mathrm c}
        &=
        -\beta_i
        +
        \frac{1}{n-2}
        \left\{
        (n-1)\beta_i+\sum_{j\ne i}\beta_j
        \right\}\\
        &=
        -\beta_i
        +
        \frac{(n-2)\beta_i}{n-2}
        =0,
\end{align*}
which proves \eqref{eq:Gaussian-score-canonical-row-sums}.

For the pair-score decomposition, write
\[
        c_{ij}:=\frac{\beta_i+\beta_j}{n-2},
        \qquad i\ne j,
\]
so that \(m_{ij}=m_{ij}^{\mathrm c}-c_{ij}\). The two cross terms in
the expansion of \(G_{n,2}^{\mathrm{pair}}\) vanish because
\(M^{\mathrm c}\) has zero row sums. Moreover, using
\(\sum_i\beta_i=0\),
\[
        \sum_{i<j}
        (\beta_i+\beta_j)
        (\beta_{\pi_i}+\beta_{\pi_j})
        =
        (n-2)\sum_{i=1}^n\beta_i\beta_{\pi_i}.
\]
It follows that
\begin{align*}
        G_{n,2}^{\mathrm{pair}}
        &=
        \widetilde G_{n,2}
        +
        \frac{1}{(n-2)^2}
        \sum_{i<j}
        (\beta_i+\beta_j)
        (\beta_{\pi_i}+\beta_{\pi_j})\\
        &=
        \widetilde G_{n,2}
        +
        \frac{1}{n-2}
        \sum_{i=1}^n\beta_i\beta_{\pi_i}\\
        &=
        \widetilde G_{n,2}
        +
        \frac2{n-2}G_{n,2}^{H_2},
\end{align*}
which proves \eqref{eq:Gpair-canonical-decomp}.

For the global bounds, the inequality
\((u+v)^2\le2u^2+2v^2\) gives
\begin{align*}
        \sum_{i\ne j}(m_{ij}^{\mathrm c})^2
        &\le
        2\sum_{i\ne j}m_{ij}^2
        +
        \frac{2}{(n-2)^2}
        \sum_{i\ne j}(\beta_i+\beta_j)^2.
\end{align*}
Since \(\sum_i\beta_i=0\),
\[
        \sum_{i\ne j}(\beta_i+\beta_j)^2
        =
        2(n-2)\sum_{i=1}^n\beta_i^2.
\]
Lemma~\ref{lem:Gaussian-rank-score-identities} therefore yields
\[
        \sum_{i\ne j}(m_{ij}^{\mathrm c})^2
        =
        O(n^2).
\]
The global \(L^1\)-bound follows by Cauchy--Schwarz:
\[
        \sum_{i\ne j}|m_{ij}^{\mathrm c}|
        \le
        \{n(n-1)\}^{1/2}
        \left\{
        \sum_{i\ne j}(m_{ij}^{\mathrm c})^2
        \right\}^{1/2}
        =
        O(n^2).
\]
This proves \eqref{eq:m-can-global-local}.

To prove the local bound, first observe that
\[
        |m_{ij}|
        \le
        \{m_{ii}m_{jj}\}^{1/2}
        \le
        \frac{1}{2}(m_{ii}+m_{jj}),
\]
and
\[
        \sum_{i=1}^nm_{ii}
        =
        \sum_{i=1}^n\ee Z_{i:n}^2
        =n.
\]
Because \(nh_1\to\infty\), each index has \(O(nh_1)\) neighbours
satisfying \(|i-j|\le C_0nh_1+C_0\). Therefore
\[
        \sum_{\substack{i\ne j\\
        |i-j|\le C_0nh_1+C_0}}
        |m_{ij}|
        =
        O(n^2h_1).
\]
For the centring correction,
Lemma~\ref{lem:Gaussian-rank-score-identities} gives
\(\sum_i|\beta_i|=O(n)\), and hence
\begin{align*}
\frac{1}{n-2}
\sum_{\substack{i\ne j\\
|i-j|\le C_0nh_1+C_0}}
|\beta_i+\beta_j|
&\le
\frac{Cnh_1}{n-2}
\sum_{i=1}^n|\beta_i|\\
&=
O(nh_1)
=
O(n^2h_1).
\end{align*}
This proves \eqref{eq:m-can-local}.
\end{proof}

\begin{lmm}\label{lem:Gtilde-L2}
For $n\geq 4$ and \(\widetilde G_{n,2}\) defined in \eqref{lochd11},
\[
        \|\widetilde G_{n,2}\|_2=O(n),
\]
where the \(L^2\)-norm is under the uniform permutation law.
\end{lmm}

\begin{proof} 
By Lemma~\ref{lem:Gaussian-score-canonical-centering},
\(M^{\mathrm c}=(m_{ij}^{\mathrm c})\) is symmetric, has zero diagonal
and zero row sums, and
\[
        S:=
        \sum_{i\ne j}(m_{ij}^{\mathrm c})^2
        =
        O(n^2).
\]
Then
\[
        Q(M^{\mathrm c}, M^{\mathrm c}; \pi)
        =
        \sum_{i\ne j}
        m_{ij}^{\mathrm c}
        m_{\pi_i\pi_j}^{\mathrm c}
        =
        2\widetilde G_{n,2}(\pi).
\]
It is therefore enough to prove
\[
        \ee_\pi[Q(M^{\mathrm c}, M^{\mathrm c}; \pi)^2]=O(n^2).
\]
We classify the terms according to the overlap between the directed
pairs $(i,j)$ and $(k,\ell)$.

If the two directed pairs have the same unordered endpoints, namely
\[
        (k,\ell)=(i,j)
        \quad\text{or}\quad
        (k,\ell)=(j,i),
\]
then
\[
        \ee_\pi[(m_{\pi_i\pi_j}^{\mathrm{c}})^2]
        =
        \frac{S}{n(n-1)}.
\]
The domain-side sum over these two possibilities is $2S$. Hence
this contribution is 
\[
        \frac{2S^2}{n(n-1)}.
\]
Next consider the case where the two directed pairs share exactly one
endpoint. For one representative case, say \(k=i\) and
\(\ell\notin\{i,j\}\), we have
\[
\begin{aligned}
\sum_{\substack{i,j,\ell\\\mathrm{all\ distinct}}}
        m_{ij}^{\mathrm{c}}m_{i\ell}^{\mathrm{c}}
&=
\sum_i
\left\{
        \left(\sum_{j\ne i}m_{ij}^{\mathrm{c}}\right)^2
        -
        \sum_{j\ne i}(m_{ij}^{\mathrm{c}})^2
\right\} 
= -S,
\end{aligned}
\]
because the matrix $\{m_{ij}^{\mathrm{c}}\}_{i,j=1}^n$ has zero row sums. Similarly,
\[
        \ee_\pi[
        m_{\pi_i\pi_j}^{\mathrm{c}}m_{\pi_i\pi_\ell}^{\mathrm{c}}]
        =
        -\frac{S}{n(n-1)(n-2)}
\]
for distinct $i,j,\ell$. Hence this representative overlap pattern
contributes
\[
        \frac{S^2}{n(n-1)(n-2)}.
\]
There are four such one-overlap patterns. 
Finally suppose the two directed pairs have no common endpoint. For fixed $i\ne j$,
\[
\sum_{\substack{k\ne \ell\\ k,\ell\notin\{i,j\}}}
        m_{k\ell}^{\mathrm{c}}
=
2m_{ij}^{\mathrm{c}}.
\]
Indeed, the total directed off-diagonal sum of $M^{\mathrm c}$ is zero, and the
zero-row-sum condition gives this identity after excluding the two
rows and columns indexed by $i$ and $j$. Therefore
\[
\sum_{\substack{i,j,k,\ell\\\mathrm{all\ distinct}}}
        m_{ij}^{\mathrm{c}} m_{k\ell}^{\mathrm{c}}
=
2S.
\]
The same calculation on the image side gives
\[
        \ee_\pi[
        m_{\pi_i\pi_j}^{\mathrm{c}}m_{\pi_k\pi_\ell}^{\mathrm{c}}]
        =
        \frac{2S}{n(n-1)(n-2)(n-3)}
\]
for distinct $i,j,k,\ell$. Thus the no-overlap contribution is
\[
        \frac{4S^2}{n(n-1)(n-2)(n-3)}.
\]
Combining the three cases,
\begin{align*}
        \ee_\pi[Q(M^{\mathrm c}, M^{\mathrm c}; \pi)^2]
        &=
        S^2
        \left\{
        \frac{2}{n(n-1)}
        +
        \frac{4}{n(n-1)(n-2)}
        +
        \frac{4}{n(n-1)(n-2)(n-3)}
        \right\} \\
        &=
        O\left(\frac{S^2}{n^2}\right).
\end{align*}
Since $S=O(n^2)$, we get $\ee_\pi[ Q(M^{\mathrm c}, M^{\mathrm c}; \pi)^2]=O(n^2)$ which gives us $ \|\widetilde G_{n,2}\|_2= O(n). $
\end{proof}

\subsubsection{Finite-population contraction identities}

\begin{lmm}\label{lem:cubic-Q-contraction}
For $n\ge 6$, let $A$, $H$, $E$, $F$ be $n\times n$ symmetric, zero-diagonal matrices with zero row sums, and \(\pi\) a uniform random permutation on $\{1, \ldots, n\}$.
Then
\begin{align}
&\left|
\ee_\pi\left[
        Q(A,H;\pi)^2Q(E,F;\pi)
\right]
\right|
\nonumber\\
&\quad\le
C\left[
\frac{\mathcal E(A,E)\mathcal E(H,F)}{n^2}
+
\frac{
\{\mathcal E(A,E)+\mathcal T(A,E)\}
\{\mathcal E(H,F)+\mathcal T(H,F)\}
}{n^3}
\right],
\label{eq:cubic-Q-contraction}
\end{align}
where
\[
        \mathcal E(A,E)
        :=
        \sum_{i\ne j}|A_{ij}|^2|E_{ij}|
\]
and
\[
        \mathcal T(A,E)
        :=
        \sum_{\substack{i,j,k\\\mathrm{all\ distinct}}}
        |A_{ij}A_{ik}E_{jk}|.
\]
\end{lmm}
\begin{proof}
Write
\[
        M_1=M_2=A,\qquad M_3=E,
        \qquad
        N_1=N_2=H,\qquad N_3=F.
\]
Then
\[
Q(A,H;\pi)^2Q(E,F;\pi)
=
\sum_{\substack{i_1\ne j_1\\ i_2\ne j_2\\ i_3\ne j_3}}
\prod_{e=1}^3
M_e(i_e,j_e)
N_e(\pi_{i_e},\pi_{j_e}).
\]
We classify the terms according to the equality pattern among the six
indices
\[
        i_1,j_1,i_2,j_2,i_3,j_3.
\]

Let \(\mathcal H=\{1^-,1^+,2^-,2^+,3^-,3^+\}\), where \(e^-\) and
\(e^+\) denote the two endpoints of the \(e\)-th edge. Let
\(\mathfrak P\) be the finite set of partitions \(\alpha\) of
\(\mathcal H\) such that \(e^-\) and \(e^+\) are not in the same block
for \(e=1,2,3\). For \(\alpha\in\mathfrak P\), let \(v(\alpha)\) be the
number of blocks of \(\alpha\). The partition \(\alpha\) defines a
multigraph \(G_\alpha\) with vertex set equal to the blocks of
\(\alpha\), and with three labelled edges: the \(e\)-th edge joins the
blocks containing \(e^-\) and \(e^+\).

For such a partition \(\alpha\), define the restricted graph sum
\[
        W_\alpha^*(A,E)
        :=
        \sum_{\substack{x_B, B\in\alpha\\ \text{all distinct}}}
        \prod_{e=1}^3
        M_e\bigl(x_{B(e^-)},x_{B(e^+)}\bigr),
\]
where \(B(h)\) denotes the block of \(\alpha\) containing \(h\). Define
\(W_\alpha^*(H,F)\) analogously with \(M_e\) replaced by \(N_e\).

If the domain indices have equality pattern \(\alpha\), then their
images under the uniform permutation \(\pi\) are uniformly distributed
over ordered \(v(\alpha)\)-tuples of distinct labels. Hence
\begin{equation}\label{eq:graph-pattern-expansion}
\ee_\pi\left[
        Q(A,H;\pi)^2Q(E,F;\pi)
\right]
=
\sum_{\alpha\in\mathfrak P}
\frac{
        W_\alpha^*(A,E)W_\alpha^*(H,F)
}{
        (n)_{v(\alpha)}
},
\end{equation}
where $(n)_v:=n(n-1)\cdots(n-v+1)$ and $(n)_0=1.$

We next bound \(W_\alpha^*(A,E)\). By M{\"o}bius inversion on the finite poset of set partitions ordered
by refinement \citep[Proposition~3.7.2]{StanleyEC1}, we have
 \begin{equation}\label{eq:mobius-restriction}
         W_\alpha^*(A,E)
         =
         \sum_{\beta\succeq\alpha}
         \mu(\alpha,\beta)W_\beta(A,E),
 \end{equation}
where \(\beta\succeq\alpha\) means that \(\beta\) is coarser than
\(\alpha\), \(\mu(\alpha,\beta)\) is the M\"obius coefficient, and
\[
        W_\beta(A,E)
        :=
        \sum_{x_B, B\in\beta}
        \prod_{e=1}^3
        M_e\bigl(x_{B(e^-)},x_{B(e^+)}\bigr)
\]
is now an unrestricted graph sum.

If some edge of the quotient graph \(G_\beta\) is a loop, then
\(W_\beta(A,E)=0\), because all matrices have zero diagonal. If
\(G_\beta\) has a leaf, then \(W_\beta(A,E)=0\). Indeed, suppose a
leaf vertex \(u\) is incident to a single edge with matrix \(J\), where
\(J\in\{A,E\}\), and is adjacent to a vertex with label \(x_v\). Then
the sum over \(x_u\) gives
\[
        \sum_{x_u=1}^n J_{x_u x_v}=0,
\]
because \(J\) has zero row sums and is symmetric. Thus only quotient graphs with no loops and no leaves can contribute to
\eqref{eq:mobius-restriction}. Such a graph has three edges and minimum
degree at least two. Therefore it has either two or three vertices.

If it has two vertices, then all three edges join the same pair of
vertices. Since two of the edge matrices are \(A\) and the third is
\(E\), symmetry gives
\[
        |W_\beta(A,E)|
        \le
        \sum_{i\ne j}|A_{ij}|^2|E_{ij}|
        =
        \mathcal E(A,E).
\]
If it has three vertices, then the graph is a triangle. Up to relabeling
of the three vertices and using symmetry,
\[
        |W_\beta(A,E)|
        \le
        \sum_{\substack{i,j,k\\ \mathrm{all\ distinct}}}
        |A_{ij}A_{ik}E_{jk}|
        =
        \mathcal T(A,E).
\]
Since the partition lattice is taken over a six-element set, both the
number of partitions \(\beta\succeq\alpha\) and the corresponding
Möbius coefficients are bounded by absolute constants. Therefore
\eqref{eq:mobius-restriction} yields
\begin{equation}\label{eq:W-alpha-general-bound}
        |W_\alpha^*(A,E)|
        \le
        C\{\mathcal E(A,E)+\mathcal T(A,E)\}.
\end{equation}
The same argument gives
\begin{equation}\label{eq:W-alpha-general-bound-BD}
        |W_\alpha^*(H,F)|
        \le
        C\{\mathcal E(H,F)+\mathcal T(H,F)\}.
\end{equation}

We now return to \eqref{eq:graph-pattern-expansion}. First consider
patterns with \(v(\alpha)=2\). Then the three labelled edges all share
the same unordered pair of vertices, and hence
\[
        |W_\alpha^*(A,E)|\le C\mathcal E(A,E),
        \qquad
        |W_\alpha^*(H,F)|\le C\mathcal E(H,F).
\]
Since \((n)_2\asymp n^2\), the total contribution of all such patterns
is bounded by
\[
        C_0\frac{\mathcal E(A,E)\mathcal E(H,F)}{n^2}.
\]

For all remaining patterns, \(v(\alpha)\ge3\). Since
\(v(\alpha)\le6\), we have
\[
        (n)_{v(\alpha)}\ge c n^{v(\alpha)}\ge c n^3
\]
for all sufficiently large \(n\). Using
\eqref{eq:W-alpha-general-bound} and
\eqref{eq:W-alpha-general-bound-BD}, and again using that there are
only finitely many equality patterns, the total contribution of all
patterns with \(v(\alpha)\ge3\) is bounded by
\[
        C_0
        \frac{
        \{\mathcal E(A,E)+\mathcal T(A,E)\}
        \{\mathcal E(H,F)+\mathcal T(H,F)\}
        }{n^3}.
\]
Combining the \(v(\alpha)=2\) and \(v(\alpha)\ge3\) contributions gives \eqref{eq:cubic-Q-contraction}.
\end{proof}
\begin{lmm}
\label{lem:QLQ-identity}
Let $n\ge5$. Let $A, H , E$, and $F$ be $n\times n$ symmetric, zero-diagonal matrices
with zero row sums
\[
        \sum_{j=1}^n A_{ij}
        =
        \sum_{j=1}^n H_{ij}
        =
        \sum_{j=1}^n E_{ij}
        =
        \sum_{j=1}^n F_{ij}
        =
        0
        \qquad\text{for every }i.
\]
Let \(u,v\in\rr^n\) satisfy
\[
        \sum_{i=1}^n u_i=0,
        \qquad
        \sum_{i=1}^n v_i=0.
\]
Then
\[
\ee_\pi\left[
        Q(A,H;\pi)
        L(u,v;\pi)
        Q(E,F;\pi)
\right]
=
\frac{n-2}{(n-1)(n-3)(n-4)}
\Lambda(A,E;u)\Lambda(H,F;v).
\]
\end{lmm}

\begin{proof} 
For brevity write $\Lambda_A:=\Lambda(A,E;u)$, and $\Lambda_H:=\Lambda(H,F;v)$. We use the expansion 
\[
\ee_\pi[
        Q(A,H;\pi)L(u,v;\pi)Q(E,F;\pi)
]
= 
\sum_{i\ne j}\sum_{r=1}^n
A_{ij}u_r\,
\ee_\pi\left[
        H_{\pi_i\pi_j}v_{\pi_r}
        Q(E,F;\pi)
\right]
\]
and  split the sum according to whether \(r\in\{i,j\}\) or
\(r\notin\{i,j\}\).

First consider the case \(r=i\) or \(r=j\). We need the following
conditional expectation. If \(i\ne j\) and \(a\ne b\), then
\begin{equation}\label{eq:Q-cond-two-values}
        \ee_\pi\left[
        Q(E,F;\pi)
        \,\middle|\,
        \pi_i=a,\pi_j=b
        \right]
        =
        \frac{2(n-1)}{n-3}E_{ij}F_{ab}.
\end{equation}
Indeed, decompose the directed edge \((p,q)\) in \(Q(E,F;\pi)\)
according to its overlap with \(\{i,j\}\). The two directed edges
\((i,j)\) and \((j,i)\) give $
        2E_{ij}F_{ab}.$
The terms with exactly one endpoint in \(\{i,j\}\) give $
        \frac{4}{n-2}E_{ij}F_{ab},$
because, for example,
\[
        \sum_{q\notin\{i,j\}}E_{iq}=-E_{ij},
        \qquad
        \sum_{c\notin\{a,b\}}F_{ac}=-F_{ab}.
\]
The terms with no endpoint in \(\{i,j\}\) give $
        \frac{4}{(n-2)(n-3)}E_{ij}F_{ab},
$
because
\[
        \sum_{\substack{p\ne q\\p,q\notin\{i,j\}}}E_{pq}
        =
        2E_{ij},
        \qquad
        \sum_{\substack{c\ne d\\c,d\notin\{a,b\}}}F_{cd}
        =
        2F_{ab}.
\]
Combining these three contributions gives
\[
        2E_{ij}F_{ab}
        \left\{
        1+\frac{2}{n-2}
        +\frac{2}{(n-2)(n-3)}
        \right\}
        =
        \frac{2(n-1)}{n-3}E_{ij}F_{ab},
\]
which proves \eqref{eq:Q-cond-two-values}.

Therefore the contribution from $r=i$ is
\begin{align*}
\sum_{i\ne j}\sum_{a\ne b}
A_{ij}u_i H_{ab}v_a
\frac{1}{n(n-1)}
\frac{2(n-1)}{n-3}E_{ij}F_{ab} 
& =
\frac{2}{n(n-3)}
\left(\sum_{i\ne j}u_iA_{ij}E_{ij}\right)
\left(\sum_{a\ne b}v_a H_{ab}F_{ab}\right) \\
& = \frac{1}{2n(n-3)}\Lambda_A\Lambda_H,
\end{align*}
where the last equality follows from 
\[
        \sum_{i\ne j}u_iA_{ij}E_{ij}
        =
        \frac{1}{2}\Lambda_A,
        \qquad
        \sum_{a\ne b}v_a H_{ab} F_{ab}
        =
        \frac{1}{2}\Lambda_H.
\]
The \(r=j\) contribution is identical. Hence the total contribution
from \(r\in\{i,j\}\) is
\begin{equation}\label{eq:QLQ-same-r-contribution}
        \frac{1}{n(n-3)}\Lambda_A\Lambda_H,
\end{equation}
and it remains to consider \(r\notin\{i,j\}\). Fix distinct
\(i,j,r\) and distinct \(a,b,c\). We first compute
\[
        \ee_\pi\left[
        Q(E,F;\pi)
        \,\middle|\,
        \pi_i=a,\pi_j=b,\pi_r=c
        \right].
\]
Let
\[
        S_E:=E_{ij}+E_{ir}+E_{jr},
        \qquad
        S_F:=F_{ab}+F_{ac}+F_{bc}.
\]
A decomposition according to the overlap of the directed edge
$(p,q)$ with $\{i,j,r\}$ gives
\begin{align}
&\ee_\pi\left[
        Q(E,F;\pi)
        \,\middle|\,
        \pi_i=a,\pi_j=b,\pi_r=c
        \right]
=
2(E_{ij}F_{ab}+E_{ir}F_{ac}+E_{jr}F_{bc})
\nonumber\\
&\qquad+
\frac{2}{n-3}
\Bigl[
(E_{ij}+E_{ir})(F_{ab}+F_{ac})
+
(E_{ij}+E_{jr})(F_{ab}+F_{bc})
+
(E_{ir}+E_{jr})(F_{ac}+F_{bc})
\Bigr]
\nonumber\\
&\qquad+
\frac{4}{(n-3)(n-4)}S_E S_F .
\label{eq:Q-cond-three-values}
\end{align}
Indeed, the first line comes from the six directed edges internal to
\(\{i,j,r\}\). For the one-endpoint terms, if the fixed endpoint is
\(i\), then
\[
        \sum_{q\notin\{i,j,r\}}E_{iq}=-(E_{ij}+E_{ir}),
\]
and similarly
\[
        \sum_{d\notin\{a,b,c\}}F_{ad}=-(F_{ab}+F_{ac}).
\]
This yields the first product in the bracket in
\eqref{eq:Q-cond-three-values}; the fixed endpoints \(j\) and \(r\)
give the other two products. The factor \(2\) accounts for the two
orientations. Finally, the no-endpoint terms give
\[
        \frac{
        \left(
        \sum_{\substack{p\ne q\\p,q\notin\{i,j,r\}}}E_{pq}
        \right)
        \left(
        \sum_{\substack{d\ne e\\d,e\notin\{a,b,c\}}}F_{de}
        \right)
        }{(n-3)(n-4)}.
\]
Since the total directed off-diagonal sums of \(E\) and \(F\) are zero,
\[
        \sum_{\substack{p\ne q\\p,q\notin\{i,j,r\}}}E_{pq}
        =
        2S_E,
        \qquad
        \sum_{\substack{d\ne e\\d,e\notin\{a,b,c\}}}F_{de}
        =
        2S_F,
\]
which gives the final term in \eqref{eq:Q-cond-three-values}.

We now sum \eqref{eq:Q-cond-three-values} against
\[
        A_{ij}u_r H_{ab}v_c
\]
over distinct \(i,j,r\) and distinct \(a,b,c\) using the identities:
\begin{align*}
\sum_{\substack{i\ne j\\ r\notin\{i,j\}}}        A_{ij}u_rE_{ij}
& = -\Lambda_A, \\ 
\sum_{\substack{i\ne j\\ r\notin\{i,j\}}}
        A_{ij}u_rE_{ir}
& =\sum_{\substack{i\ne j\\ r\notin\{i,j\}}}
        A_{ij}u_rE_{jr}
= -\frac{1}{2}\Lambda_A.
\end{align*}
The first identity follows from $\sum_{r\notin\{i,j\}}u_r=-(u_i+u_j)$,
for the second; use $\sum_{j\notin\{i,r\}}A_{ij}=-A_{ir}$, which gives
\begin{align*}
\sum_{\substack{i\ne j\\ r\notin\{i,j\}}}
        A_{ij}u_rE_{ir}
=
\sum_{i\ne r}u_rE_{ir}
\sum_{j\notin\{i,r\}}A_{ij} 
=
-\sum_{i\ne r}u_rA_{ir}E_{ir}
=
-\frac{1}{2}\Lambda_A.
\end{align*}
The identity with \(E_{jr}\) is the same by symmetry. The corresponding
image-side identities are identical with \((A,E,u)\) replaced by \((H,F,v)\).

Hence the first line of \eqref{eq:Q-cond-three-values} contributes
\[
        2\left[
        \Lambda_A\Lambda_H
        +
        \frac{1}{4}\Lambda_A\Lambda_H
        +
        \frac{1}{4}\Lambda_A\Lambda_H
        \right]
        =
        3\Lambda_A\Lambda_H.
\]
For the bracketed one-endpoint term, the domain sums are
\[
\sum_{\substack{i\ne j\\ r\notin\{i,j\}}}
        A_{ij}u_r(E_{ij}+E_{ir})
=
-\frac{3}{2}\Lambda_A,
\]
\[
\sum_{\substack{i\ne j\\ r\notin\{i,j\}}}
        A_{ij}u_r(E_{ij}+E_{jr})
=
-\frac{3}{2}\Lambda_A,
\]
and
\[
\sum_{\substack{i\ne j\\ r\notin\{i,j\}}}
        A_{ij}u_r(E_{ir}+E_{jr})
=
-\Lambda_A.
\]
The image-side analogues give the same factors with \(\Lambda_H\).
Therefore the bracketed term contributes
\[
        \frac{2}{n-3}
        \left[
        \frac{9}{4}+\frac{9}{4}+1
        \right]
        \Lambda_A\Lambda_H
        =
        \frac{11}{n-3}\Lambda_A\Lambda_H.
\]
Finally,
\[
\sum_{\substack{i\ne j\\ r\notin\{i,j\}}}
        A_{ij}u_rS_E
=
\sum_{\substack{i\ne j\\ r\notin\{i,j\}}}
        A_{ij}u_r(E_{ij}+E_{ir}+E_{jr})
=
-2\Lambda_A,
\]
and similarly the image-side sum equals \(-2\Lambda_H\). Thus the
last term in \eqref{eq:Q-cond-three-values} contributes
\[
        \frac{4}{(n-3)(n-4)}
        (-2\Lambda_A)(-2\Lambda_H)
        =
        \frac{16}{(n-3)(n-4)}
        \Lambda_A\Lambda_H.
\]
Since
\[
        \pp(\pi_i=a,\pi_j=b,\pi_r=c)
        =
        \frac{1}{n(n-1)(n-2)}
\]
for distinct \(i,j,r\) and distinct \(a,b,c\), the total contribution
from \(r\notin\{i,j\}\) is
\begin{equation}\label{eq:QLQ-distinct-r-contribution}
        \frac{
        3+\frac{11}{n-3}
        +\frac{16}{(n-3)(n-4)}
        }{
        n(n-1)(n-2)
        }
        \Lambda_A\Lambda_H.
\end{equation}
Combining \eqref{eq:QLQ-same-r-contribution} and
\eqref{eq:QLQ-distinct-r-contribution}, we obtain
\begin{align*} 
\ee_\pi\left[
        Q(A, H;\pi)
        L(u, v;\pi)
        Q(E, F;\pi)
\right] &=
\left[
        \frac{1}{n(n-3)}
        +
        \frac{
        3+\frac{11}{n-3}
        +\frac{16}{(n-3)(n-4)}
        }{
        n(n-1)(n-2)
        }
\right]
\Lambda_A\Lambda_H \\ 
& =
\frac{n-2}{(n-1)(n-3)(n-4)}
\Lambda(A,E;u)\Lambda(H,F;v),
\end{align*}
as claimed.
\end{proof}

\subsubsection{Low-order score contractions and variance derivatives}

\begin{lmm}
\label{lem:score-contractions-centered}
Suppose that Assumptions~\ref{asmp:kernel} and~\ref{asmp:band} hold. Then, as $n\to\infty$,
\begin{equation}\label{eq:first-order-score-bound-centered}
        \abs{\ee_0\left[
        U_n^2\ee_0[S_{n,1}\mid\cR]\right]}
        =
        O(1/n),
\end{equation}
\begin{equation}\label{eq:H2-score-bound-centered}
        \abs{\ee_0\left[
        U_n^2G_{n,2}^{H_2}(\pi)\right]}
        =
        O(1/n),
\end{equation}
and
\begin{equation}\label{eq:pair-score-bound-centered}
        \abs{\ee_0\left[
        U_n^2G_{n,2}^{\rm pair}(\pi)\right]}
        =
        O(1/n).
\end{equation}
Moreover, the canonical quadratic component satisfies
\begin{equation}
\label{eq:D2-H2-score-bound-centered}
\left|
\ee_0\left[
        \widetilde D_n^2G_{n,2}^{H_2}(\pi)
\right]
\right|
=
O\left(\frac{1}{n^{3/2}h_1}\right)
=
o(1/n),
\end{equation}
and
\begin{equation}
\label{eq:D2-pair-score-bound-centered}
\left|
\ee_0\left[
        \widetilde D_n^2G_{n,2}^{\rm pair}(\pi)
\right]
\right|
=
O(1/n).
\end{equation}
Finally,
\begin{equation}\label{eq:first-mean-derivative-bound-centered}
        \Big  |
        \frac{d}{d\rho}\ee_\rho \hat\tau_n^2
        \Big |_{\rho=0} \Big |
        =
        O(n^{-1/2}).
\end{equation}
\end{lmm}

\begin{proof}
Recall the definition of $D_n$, $T_n$, $C_n$ and $b_{n0}$ in \eqref{lochd4} and note that
\[
        U_n
        = \hat\tau_n^2 - \ee_0[\hat\tau_n^2] = 
        D_n+2(T_n-\ee_0T_n)+ (b_{n0}-\ee_0b_{n0}).
\]
 
Let \(x=(x_i)_{i=1}^n\) and \(y=(y_k)_{k=1}^n\) be the centred
row-sum vectors defined in
Lemma~\ref{lem:null-canonical-centering-bounds}. By
\eqref{eq:canonical-linear-representations},
\[
        T_n-\ee_0T_n
        =
        \frac{1}{n^2h_1}
        \sum_{i=1}^nx_i y_{\pi_i}.
\]
We use throughout the bounds in \eqref{eq:xy-norms-higher}. We first prove \eqref{eq:first-order-score-bound-centered}. Since
\[
        \ee_0[S_{n,1}\mid\cR]=\sum_{r=1}^n\mu_r\mu_{\pi_r},
\]
the leading term is
\[
        \ee_0[(T_n-\ee_0T_n)^2\ee_0[S_{n,1}\mid\cR]].
\]
Expanding \(T_n-\ee_0T_n\), we get
\begin{equation}\label{eq:T2G1-expanded}
\ee_0[(T_n-\ee_0T_n)^2\ee_0[S_{n,1}\mid\cR]]
=
\frac{1}{n^4h_1^2}
\sum_{i,j,r=1}^n
x_ix_j\mu_r
\ee_\pi\left[
        y_{\pi_i}y_{\pi_j}\mu_{\pi_r}
\right].
\end{equation}
Let
\[
        M_{y^2}:=\sum_{k=1}^ny_k^2\mu_k.
\]
Since $\sum_ky_k=0$, and $\sum_k\mu_k=0$,
direct calculation gives
\begin{align}\label{eq:My2identity}
    \ee_\pi[y_{\pi_i}y_{\pi_j}\mu_{\pi_r}]
=
\begin{cases}
        M_{y^2}/n,
        & i=j=r,\\[0.6em]
        -M_{y^2}/\{n(n-1)\},
        & |\{i,j,r\}|=2,\\[0.6em]
        2M_{y^2}/\{n(n-1)(n-2)\},
        & |\{i,j,r\}|=3.
\end{cases}
\end{align}
Indeed, the first case is immediate. If \(i=j\ne r\), then
\[
\ee_\pi[y_{\pi_i}^2\mu_{\pi_r}]
=
\frac{1}{n(n-1)}
\sum_{k\ne \ell}y_k^2\mu_\ell
=
-\frac{1}{n(n-1)}\sum_{k=1}^ny_k^2\mu_k.
\]
The other two-index coincidences are identical. If \(i,j,r\) are all
distinct, then
\[
\sum_{k,\ell,m}^{*}y_ky_\ell\mu_m
=
\sum_m\mu_m
\left\{
        \left(\sum_{k\ne m}y_k\right)^2
        -
        \sum_{k\ne m}y_k^2
\right\}
=
2\sum_my_m^2\mu_m,
\]
because $\sum_ky_k=\sum_k\mu_k=0$. This proves \eqref{eq:My2identity}. Substituting \eqref{eq:My2identity} into \eqref{eq:T2G1-expanded}, and using
\(\sum_ix_i=\sum_i\mu_i=0\), yields
\[
\sum_{i,j,r}
x_ix_j\mu_r
\ee_\pi[y_{\pi_i}y_{\pi_j}\mu_{\pi_r}]
=
M_{y^2}
\left\{
        \frac{1}{n}+\frac{3}{n(n-1)}
        +\frac{4}{n(n-1)(n-2)}
\right\}
\sum_{i=1}^nx_i^2\mu_i.
\]
Consequently,
\begin{equation}\label{eq:T2G1-bound}
\abs{\ee_0[(T_n-\ee_0T_n)^2\ee_0[S_{n,1}\mid\cR]]}
\le
\frac{C}{n^5h_1^2}
\abs{\sum_{k=1}^ny_k^2\mu_k}
\abs{\sum_{i=1}^nx_i^2\mu_i}.
\end{equation}
By Cauchy–Schwarz inequality,
\[
        \abs{\sum_{k=1}^ny_k^2\mu_k}
        \le
        \left(\sum_ky_k^4\right)^{1/2}
        \left(\sum_k\mu_k^2\right)^{1/2}
        =
        O(n),
\]
by Lemma~\ref{lem:Gaussian-rank-score-identities} and since $\max_k|y_k|\le C$. Also,
\[
        \sum_i x_i^4
        \le
        \max_i x_i^2\sum_i x_i^2
        =
        O(n^2h_1^2)O(n^3h_1^3)
        =
        O(n^5h_1^5),
\]
so
\[
        \abs{\sum_{i=1}^nx_i^2\mu_i}
        \le
        \left(\sum_ix_i^4\right)^{1/2}
        \left(\sum_i\mu_i^2\right)^{1/2}
        =
        O(n^3h_1^{5/2}).
\]
Thus \eqref{eq:T2G1-bound} gives
\[
        \abs{\ee_0[(T_n-\ee_0T_n)^2\ee_0[S_{n,1}\mid\cR]]}
        =
        O\left(\frac{h_1^{1/2}}{n}\right)
        =
        o(1/n).
\]
For the remaining terms, conditional Jensen's inequality gives
\[
        \left\|
        \ee_0[S_{n,1}\mid\cR]
        \right\|_2
        \le
        \|S_{n,1}\|_2
        =
        n^{1/2},
\]
where the last equality holds because under $\pp_0$, the variables $X_iY_i$ are independent, centred, and have unit variance.

By Hölder's inequality and
Lemma~\ref{lem:null-component-fourth-moments},
\begin{align*}
\left|
\ee_0\left[
        D_n^2\ee_0[S_{n,1}\mid\cR]
\right]
\right|
&=
O\left(\frac{1}{n^{3/2}h_1}\right)
=
o(1/n),\\
\left|
\ee_0\left[
        D_n(T_n-\ee_0T_n)
        \ee_0[S_{n,1}\mid\cR]
\right]
\right|
&=
O(1/n),\\
\left|
\ee_0\left[
        D_n(b_{n0}-\ee_0b_{n0})
        \ee_0[S_{n,1}\mid\cR]
\right]
\right|
&=
O\left(\frac{1}{n^2h_1}\right)
=
o(1/n),\\
\left|
\ee_0\left[
        (T_n-\ee_0T_n)(b_{n0}-\ee_0b_{n0})
        \ee_0[S_{n,1}\mid\cR]
\right]
\right|
&=
O(n^{-3/2})
=
o(1/n),\\
\left|
\ee_0\left[
        (b_{n0}-\ee_0b_{n0})^2
        \ee_0[S_{n,1}\mid\cR]
\right]
\right|
&=
O\left(\frac{1}{n^{5/2}h_1}\right)
=
o(1/n).
\end{align*}
Here \(nh_1^2\to\infty\) implies \(nh_1\to\infty\).
Together with \eqref{eq:T2G1-bound}, these estimates prove
\eqref{eq:first-order-score-bound-centered}.

For the remaining second-order score bounds, we use the canonical quadratic component $\widetilde{D}_n$. By \eqref{eq:D-canonical-decomp},
\[
D_n=\widetilde{D}_n + \frac{2}{n-2}(T_n-\ee_0[T_n]),
\]
and hence
\[
U_n=\widetilde{D}_n+\frac{2(n-1)}{n-2}(T_n-\ee_0[T_n])
+(b_{n0}-\ee_0[b_{n0}]).
\]
This exact decomposition absorbs the linear component of $D_n$ into the existing linear term. The matrices defining $\widetilde{D}_n$ have zero diagonal and zero row sums, allowing us to apply the canonical permutation contraction identities below. Since $2(n-1)/(n-2)=O(1)$, this adjustment does not change the orders required for the linear and mixed terms.

Next, we prove \eqref{eq:H2-score-bound-centered}. Since
\(\sum_{i=1}^n\beta_i=0\), the usual second-moment formula for a
linear permutation statistic gives
\[
\left\|G_{n,2}^{H_2}\right\|_2^2
=
\frac{1}{4(n-1)}
\left(\sum_{i=1}^n\beta_i^2\right)^2
=
O(n).
\]
Repeating the three-index calculation leading to
\eqref{eq:T2G1-expanded}, with \(\mu\) replaced by \(\beta\), gives
\begin{align*}
\ee_0\left[
        (T_n-\ee_0T_n)^2G_{n,2}^{H_2}
\right] =
\frac{1}{2n^3h_1^2(n-1)(n-2)}
\left(\sum_{i=1}^nx_i^2\beta_i\right)
\left(\sum_{k=1}^ny_k^2\beta_k\right).
\end{align*}
By \eqref{eq:score-identities-centered} and
\eqref{eq:xy-norms-higher},
\[
\left|\sum_{k=1}^ny_k^2\beta_k\right|
\le
\left(\sum_{k=1}^ny_k^4\right)^{1/2}
\left(\sum_{k=1}^n\beta_k^2\right)^{1/2}
=
O(n),
\]
whereas
\[
\sum_{i=1}^nx_i^4
\le
\max_i x_i^2\sum_{i=1}^nx_i^2
=
O(n^5h_1^5),
\]
and hence
\[
\left|\sum_{i=1}^nx_i^2\beta_i\right|
\le
\left(\sum_{i=1}^nx_i^4\right)^{1/2}
\left(\sum_{i=1}^n\beta_i^2\right)^{1/2}
=
O(n^3h_1^{5/2}).
\]
Consequently,
\begin{equation}\label{eq:T2-H2-bound-centered}
\left|
\ee_0\left[
        (T_n-\ee_0T_n)^2G_{n,2}^{H_2}
\right]
\right|
=
O\left(\frac{h_1^{1/2}}{n}\right)
=
o(1/n).
\end{equation}

By H\"older's inequality, Lemma~\ref{lem:null-component-fourth-moments},
and the bound \(\|G_{n,2}^{H_2}\|_2=O(n^{1/2})\),
\begin{equation*}
\left|
\ee_0\left[
        \widetilde D_n^2G_{n,2}^{H_2}
\right]
\right|
=
O\left(\frac{1}{n^{3/2}h_1}\right)
=
o(1/n),
\end{equation*}
which gives \eqref{eq:D2-H2-score-bound-centered}.

Similarly,
\begin{align*}
\left|
\ee_0\left[
        \widetilde D_n(T_n-\ee_0T_n)G_{n,2}^{H_2}
\right]
\right|
&=
O(1/n),\\
\left|
\ee_0\left[
        \widetilde D_n(b_{n0}-\ee_0b_{n0})G_{n,2}^{H_2}
\right]
\right|
&=
O\left(\frac{1}{n^2h_1}\right)
=
o(1/n),\\
\left|
\ee_0\left[
        (T_n-\ee_0T_n)(b_{n0}-\ee_0b_{n0})
        G_{n,2}^{H_2}
\right]
\right|
&=
O(n^{-3/2})
=
o(1/n),\\
\left|
\ee_0\left[
        (b_{n0}-\ee_0b_{n0})^2G_{n,2}^{H_2}
\right]
\right|
&=
O\left(\frac{1}{n^{5/2}h_1}\right)
=
o(1/n).
\end{align*}

Since \(2(n-1)/(n-2)=O(1)\), expanding \(U_n^2\) by
\eqref{eq:U-canonical-decomp} and using
\eqref{eq:T2-H2-bound-centered} proves
\eqref{eq:H2-score-bound-centered}.

We now prove \eqref{eq:pair-score-bound-centered}. First note that by \eqref{eq:Gpair-canonical-decomp}, we have
\begin{align*}
\left|
\ee_0\left[
        \widetilde D_n^2G_{n,2}^{\rm pair}
\right]
\right|
\leq
\left|
\ee_0\left[
        \widetilde D_n^2\widetilde G_{n,2}
\right]
\right|
+
\frac{2}{n-2}
\left|
\ee_0\left[
        \widetilde D_n^2G_{n,2}^{H_2}
\right]
\right|.
\end{align*}
To bound this, we use the null canonical-centering identities from
Lemma~\ref{lem:null-canonical-centering-bounds} and the Gaussian
rank-score canonical-centering identities from
Lemma~\ref{lem:Gaussian-score-canonical-centering}.
In particular,
\(A^{\mathrm c}\), \(B^{\mathrm c}\), and \(M^{\mathrm c}\) have zero
row sums, and the bounds
\eqref{eq:a-can-pointwise}--\eqref{eq:m-can-local} hold.

With $Q(A, B; \pi)$ defined in~\eqref{eq:Q}, we have
\[
        \widetilde{D}_n
        =
        \frac{1}{n^2h_1}Q(A^{\mathrm c}, B^{\mathrm c};\pi),
        \qquad
        \widetilde{G}_{n,2}
        =
        \frac{1}{2} Q(M^{\mathrm c}, M^{\mathrm c};\pi).
\]
Hence
\[
\ee_0[
        \widetilde{D}_n^2\widetilde{G}_{n,2}]
=
\frac{1}{2n^4h_1^2}
\ee_\pi[
        Q(A^{\mathrm c}, B^{\mathrm c};\pi)^2Q(M^{\mathrm c}, M^{\mathrm c};\pi)].
\]
We therefore apply
Lemma~\ref{lem:cubic-Q-contraction} with
\[
        A=A^{\mathrm c},
        \qquad
        H=B^{\mathrm c},
        \qquad
        E=F=M^{\mathrm c}.
\]
By \eqref{eq:a-can-pointwise}--\eqref{eq:m-can-local},
\begin{align*}
        \mathcal E(A^{\mathrm c},M^{\mathrm c})
        \leq
        C\sum_{\substack{i\ne j\\
        |i-j|\le 2nh_1+2}}
        |m_{ij}^{\mathrm{c}}|
        +
        Ch_1^2\sum_{i\ne j}|m_{ij}^{\mathrm{c}}| =O(n^2h_1).
\end{align*}
Also,
\begin{align*}
        \mathcal T(A^{\mathrm c},M^{\mathrm c})
&\le
C\sum_{i,j,k}^{*}
\left\{
        \bone(|i-j|\le 2nh_1+2)+h_1
\right\} \\
&\hspace{2.5cm}\times
\left\{
        \bone(|i-k|\le 2nh_1+2)+h_1
\right\}
|m_{jk}^{\mathrm{c}}| \\
&=
O(n^3h_1^2).
\end{align*}
In the above, the last equality holds since, when both indicators occur,
$|j-k|\leq 4nh_1 + 4$, and there are $O(nh_1)$ possible values of $i$ for each such pair $(j,k)$. Hence this contribution is
\[
        O(nh_1)
        \sum_{\substack{j\ne k\\
        |j-k|\le4nh_1+4}}
        |m_{jk}^{\mathrm{c}}|
        =
        O(n^3h_1^2).
\]
The two terms containing one factor \(h_1\) satisfy, for example,
\begin{align*}
h_1
\sum_{i,j,k}^{*}
\bone(|i-j|\le 2nh_1+2)
|m_{jk}^{\mathrm c}|
\leq
Ch_1(nh_1+1)
\sum_{j\ne k}|m_{jk}^{\mathrm c}| 
=
O(n^3h_1^2),
\end{align*}
where we used \eqref{eq:m-can-global-local} and \(nh_1\to\infty\). Finally,
\[
h_1^2
\sum_{i,j,k}^{*}|m_{jk}^{\mathrm c}|
\le
Cnh_1^2
\sum_{j\ne k}|m_{jk}^{\mathrm c}|
=
O(n^3h_1^2).
\]
Thus
\[
\mathcal T(A^{\mathrm c}, M^{\mathrm c})=O(n^3h_1^2).
\]
Since \(b_{ij}^{\mathrm{c}}\) is uniformly bounded,
\[
        \mathcal E(B^{\mathrm{ c}},M^{\mathrm c})=O(n^2),
        \qquad
        \mathcal T(B^{\mathrm c},M^{\mathrm c})=O(n^3).
\]
Therefore \eqref{eq:cubic-Q-contraction} gives
\[
\left|
\ee_\pi\left[
Q(A^{\mathrm c},B^{\mathrm c};\pi)^2Q(M^{\mathrm c},M^{\mathrm c};\pi)
\right]
\right|
=
O(n^3h_1^2).
\]
We therefore conclude that
\begin{equation}\label{eq:D2-Gcan-bound}
        \left|
        \ee_0\left[
        \widetilde D_n^2\widetilde G_{n,2}
        \right]
        \right|
        =
        O(1/n).
\end{equation}
Combining \eqref{eq:D2-Gcan-bound} with
\eqref{eq:Gpair-canonical-decomp} and
\eqref{eq:D2-H2-score-bound-centered}, we obtain
\begin{align*}
\left|
\ee_0\left[
        \widetilde D_n^2G_{n,2}^{\rm pair}
\right]
\right|
\leq
\left|
\ee_0\left[
        \widetilde D_n^2\widetilde G_{n,2}
\right]
\right|
+
\frac{2}{n-2}
\left|
\ee_0\left[
        \widetilde D_n^2G_{n,2}^{H_2}
\right]
\right| 
=
O(1/n),
\end{align*}
which proves \eqref{eq:D2-pair-score-bound-centered}.

We next control the mixed term $\widetilde D_n(T_n-\ee_0T_n)\widetilde G_{n,2}$. 

Using \eqref{eq:a-can-pointwise}--\eqref{eq:m-can-local} and
$\max_i|x_i|\le Cnh_1$, we have
\begin{align*}
|\Lambda(A^{\mathrm c}, M^{\mathrm c};x)|
\le
2\max_i|x_i|
\sum_{i\ne j}
|a_{ij}^{\mathrm{c}}m_{ij}^{\mathrm{c}}| 
\le
Cnh_1
\left\{
\sum_{\substack{i\ne j\\
|i-j|\le2nh_1+2}}
|m_{ij}^{\mathrm{c}}|
+
h_1\sum_{i\ne j}|m_{ij}^{\mathrm{c}}|
\right\} 
=
O(n^3h_1^2).
\end{align*}
Similarly, since $b_{ij}^{\mathrm{c}}$ and $y_i$ are uniformly bounded,
\[
        |\Lambda(B^{\mathrm c}, M^{\mathrm c};y)|=O(n^2).
\]
It follows from Lemma~\ref{lem:QLQ-identity}
\[
\left|
\ee_\pi\left[
        Q(A^{\mathrm c},B^{\mathrm c};\pi)
        L(x,y;\pi)
        Q(M^{\mathrm c}, M^{\mathrm c};\pi)
\right]
\right|
=
O(n^3h_1^2).
\]
Consequently,
\begin{align}\label{eq:DT-Gcan-bound}
        \left|
        \ee_0\left[
        \widetilde D_n(T_n-\ee_0T_n)
        \widetilde G_{n,2}
        \right]
        \right|
        =
        O(1/n).
\end{align}
By Lemmas~\ref{lem:null-component-fourth-moments}
and~\ref{lem:Gtilde-L2},
\[
        \|\widetilde D_n\|_4
        =
        O\left(\frac{1}{nh_1^{1/2}}\right),
        \qquad
        \|T_n-\ee_0T_n\|_4
        =
        O\left(\frac{h_1^{1/2}}{n^{1/2}}\right),
\]
\[
        \|b_{n0}-\ee_0b_{n0}\|_4
        =
        O\left(\frac{1}{n^{3/2}h_1^{1/2}}\right),
        \qquad
        \|\widetilde G_{n,2}\|_2
        =
        O(n).
\]
Hölder's inequality and \eqref{eq:b-sharp-fourth} now give
\begin{align*}
\left|
\ee_0[
        \widetilde D_n(b_{n0}-\ee_0b_{n0})
        \widetilde G_{n,2}]
\right|
\leq
\|\widetilde D_n\|_4
\|b_{n0}-\ee_0b_{n0}\|_4
\|\widetilde G_{n,2}\|_2 
=
O\left(\frac{1}{n^{3/2}h_1}\right)
=
o(1/n),
\end{align*}
Similarly,
\begin{align*}
\left|
\ee_0[
        (T_n-\ee_0T_n) (b_{n0}-\ee_0b_{n0})
        \widetilde G_{n,2}]
\right|
\leq
\|T_n-\ee_0T_n\|_4
\|b_{n0}-\ee_0b_{n0}\|_4
\|\widetilde G_{n,2}\|_2 
=
O(1/n),
\end{align*}
and
\begin{align*}
\left|
\ee_0[
        (b_{n0}-\ee_0b_{n0})^2
        \widetilde G_{n,2}]
\right|
\leq
\|b_{n0}-\ee_0b_{n0}\|_4^2
\|\widetilde G_{n,2}\|_2 
=
O\left(\frac{1}{n^2h_1}\right)
=
o(1/n).
\end{align*}
It remains to control
\(
\ee_0[(T_n-\ee_0T_n)^2\widetilde G_{n,2}]
\).
Because \(M^{\mathrm c}\) has zero row sums, conditioning on one image
value gives
\[
\ee_\pi\left[
        Q(M^{\mathrm c},M^{\mathrm c};\pi)
        \,\middle|\,
        \pi_i=k
\right]
=0.
\]
Moreover, by \eqref{eq:Q-cond-two-values}, for \(i\ne j\) and
\(k\ne\ell\),
\[
\ee_\pi\left[
        Q(M^{\mathrm c},M^{\mathrm c};\pi)
        \,\middle|\,
        \pi_i=k,\pi_j=\ell
\right]
=
\frac{2(n-1)}{n-3}
        m_{ij}^{\mathrm c}m_{k\ell}^{\mathrm c}.
\]
Therefore,
\begin{align*}
\ee_\pi\left[
        L(x,y;\pi)^2Q(M^{\mathrm c},M^{\mathrm c};\pi)
\right]
&=
\sum_{i\ne j}x_ix_j
\frac{1}{n(n-1)}
\sum_{k\ne\ell}y_ky_\ell
\frac{2(n-1)}{n-3}
        m_{ij}^{\mathrm c}m_{k\ell}^{\mathrm c}\\
&=
\frac{2}{n(n-3)}
        (x^\top M^{\mathrm c}x)
        (y^\top M^{\mathrm c}y).
\end{align*}
Since
\[
        T_n-\ee_0T_n
        =
        \frac{1}{n^2h_1}L(x,y;\pi),
        \qquad
        \widetilde G_{n,2}
        =
        \frac{1}{2}Q(M^{\mathrm c},M^{\mathrm c};\pi),
\]
it follows that
\begin{equation}\label{eq:T2-Gcan-exact}
\ee_0\left[
        (T_n-\ee_0T_n)^2\widetilde G_{n,2}
\right]
=
\frac{
        (x^\top M^{\mathrm c}x)
        (y^\top M^{\mathrm c}y)
     }{
        n^5(n-3)h_1^2
     }.
\end{equation}
We next bound the two quadratic forms. From the definition of
\(M^{\mathrm c}\), the identities \(\sum_i x_i=0\) and
\(m_{ii}=1+\beta_i\) give
\begin{align*}
x^\top M^{\mathrm c}x
&=
\ee\left(\sum_{i=1}^nx_iZ_{i:n}\right)^2
-
\sum_{i=1}^nx_i^2(1+\beta_i)
-
\frac{2}{n-2}\sum_{i=1}^nx_i^2\beta_i.
\end{align*}
The function
\(
(z_1,\ldots,z_n)\mapsto\sum_{i=1}^nx_i z_{i:n}
\)
is almost everywhere differentiable, and its gradient is a permutation
of \(x\). Hence the Gaussian Poincar\'e inequality, together with
Lemma~\ref{lem:Lmu-bound}, yields
\begin{align*}
\ee\left(\sum_{i=1}^nx_iZ_{i:n}\right)^2
=
\var\left(\sum_{i=1}^nx_iZ_{i:n}\right)
+
\left(\sum_{i=1}^nx_i\mu_i\right)^2
\le
\sum_{i=1}^nx_i^2+O(n^2)
=
O(n^3h_1^3+n^2).
\end{align*}
Also,
\begin{align*}
\left|\sum_{i=1}^nx_i^2\beta_i\right|
\le
\left(\sum_{i=1}^nx_i^4\right)^{1/2}
\left(\sum_{i=1}^n\beta_i^2\right)^{1/2}
\le
\left(
        \max_i x_i^2\sum_{i=1}^nx_i^2
\right)^{1/2}
\left(\sum_{i=1}^n\beta_i^2\right)^{1/2}
=
O(n^3h_1^{5/2}).
\end{align*}
Since \(h_1\to0\),
\[
        n^3h_1^3
        =
        O(n^3h_1^{5/2}),
\]
and therefore
\begin{equation}\label{eq:x-Mc-x-bound}
        |x^\top M^{\mathrm c}x|
        =
        O\left(n^3h_1^{5/2}+n^2\right).
\end{equation}

On the other hand, by \eqref{eq:m-can-global-local} and
\eqref{eq:xy-norms-higher},
\begin{align}\label{eq:y-Mc-y-bound}
        |y^\top M^{\mathrm c}y|
        \le
        \|M^{\mathrm c}\|_{\mathrm F}
        \|yy^\top\|_{\mathrm F} 
        =
        \|M^{\mathrm c}\|_{\mathrm F}\|y\|_2^2
        =
        O(n^2).
\end{align}
Substituting \eqref{eq:x-Mc-x-bound} and
\eqref{eq:y-Mc-y-bound} into \eqref{eq:T2-Gcan-exact} gives
\begin{align*}
\left|
\ee_0\left[
        (T_n-\ee_0T_n)^2\widetilde G_{n,2}
\right]
\right|
=
O\left(
        \frac{h_1^{1/2}}{n}
        +
        \frac{1}{n^2h_1^2}
\right)
=
o(1/n),
\end{align*}
where the last equality follows from $h_1\to0$, and $nh_1^2\to\infty$.

Expanding $U_n$ using \eqref{eq:U-canonical-decomp} and applying
\eqref{eq:D2-Gcan-bound}, \eqref{eq:DT-Gcan-bound}, and the preceding
bounds therefore yields
\[
        \left|
        \ee_0[
        U_n^2\widetilde G_{n,2}]
        \right|
        =
        O(1/n).
\]
Using \eqref{eq:Gpair-canonical-decomp} once more,
\begin{align*}
\left|
\ee_0[
        U_n^2G_{n,2}^{\rm pair}]
\right|\le
\left|
\ee_0[
        U_n^2\widetilde G_{n,2}]
\right| +
\frac{2}{n-2}
\left|
\ee_0[
        U_n^2G_{n,2}^{H_2}]
\right|=
O(1/n).
\end{align*}
This proves \eqref{eq:pair-score-bound-centered}. 

Therefore, it remains to prove the mean-derivative bound
\eqref{eq:first-mean-derivative-bound-centered}. Let
\begin{align*}
    M_n(\rho):=\ee_\rho[\hat\tau_n^2].
\end{align*}
The Gaussian likelihood expansion gives
\[
        M_n'(0)
        =
        \ee_0[\hat\tau_n^2S_{n,1}]
        =
        \ee_0[\hat\tau_n^2\ee_0[S_{n,1}\mid\cR]].
\]
Since $\ee_0[S_{n,1}]=0$,
\[
    M_n'(0)=\ee_0[U_n\ee_0[S_{n,1}\mid\cR]].
\]
Using
\[
        U_n=D_n+2(T_n-\ee_0T_n)+ (b_{n0}-\ee_0b_{n0}),
\]
we consider the three terms separately. For the linear term,
\begin{align*}
\ee_0[(T_n-\ee_0T_n) \ee_0[S_{n,1}\mid\cR]]
&=
\frac{1}{n^2h_1}
\sum_{i=1}^n\sum_{r=1}^n
x_i\mu_r
\ee_\pi\left[y_{\pi_i}\mu_{\pi_r}\right].
\end{align*}
Let
\[
        M_{y}:=\sum_{k=1}^ny_k\mu_k.
\]
If \(r=i\), then
\[
        \ee_\pi[y_{\pi_i}\mu_{\pi_i}]
        =
        \frac{M_{y}}{n}.
\]
If \(r\ne i\), then, using
\(\sum_k y_k=\sum_k\mu_k=0\),
\begin{align*}
\ee_\pi[y_{\pi_i}\mu_{\pi_r}]
=
\frac{1}{n(n-1)}
\sum_{k\ne\ell}y_k\mu_\ell 
=
-\frac{M_{y}}{n(n-1)}.
\end{align*}
Consequently,
\begin{align*}
\ee_0[(T_n-\ee_0T_n) \ee_0[S_{n,1}\mid\cR]]
&=
\frac{M_{y}}{n^2h_1}
\left\{
\frac{1}{n}\sum_i x_i\mu_i
-\frac{1}{n(n-1)}
\sum_{i\ne r}x_i\mu_r
\right\}.
\end{align*}
Since
\[
        \sum_{i\ne r}x_i\mu_r
        =
        -\sum_i x_i\mu_i,
\]
we obtain
\begin{align*}
\big | \ee_0[(T_n-\ee_0T_n)\ee_0[S_{n,1}\mid\mathcal{R}]] \big | & =
\frac{1}{n^2h_1(n-1)}
\Big | \sum_i x_i\mu_i 
 \sum_k y_k\mu_k \Big | = O\Bigl(\frac{1}{nh_1}\Bigr) = 
        O(n^{-1/2}), 
\end{align*}
where we have used  Lemma~\ref{lem:Lmu-bound} and 
$  \left|\sum_k y_k\mu_k\right|=O(n) $ (note that $nh_1^2\to \infty$).

For fixed $i\ne j$, uniformity of the permutation and symmetry of $b_{k\ell}$ give
\begin{align*}
\ee_\pi\!\left[b_{\pi_i\pi_j}\mu_{\pi_j}\right] 
= 
\ee_\pi\!\left[b_{\pi_i\pi_j}\mu_{\pi_i}\right]
=
\frac{1}{n(n-1)}
\sum_{k=1}^n\mu_k
\sum_{\ell\ne k}b_{k\ell}
=
\frac{M_y}{n(n-1)}.
\end{align*}
If $r\notin\{i,j\}$, then
\begin{align*}
\ee_\pi\!\left[b_{\pi_i\pi_j}\mu_{\pi_r}\right]
=
\frac{1}{n(n-1)(n-2)}
\sum_{k\ne\ell}b_{k\ell}
\sum_{m\notin\{k,\ell\}}\mu_m = \frac{-1}{n(n-1)(n-2)}
\sum_{k\ne\ell}b_{k\ell}(\mu_k + \mu_\ell),
\end{align*}
since $\sum_m\mu_m=0$. Moreover,
\begin{align*}
\sum_{k\ne\ell}b_{k\ell}(\mu_k+\mu_\ell)
=
2\sum_{k=1}^n\mu_k
\sum_{\ell\ne k}b_{k\ell} 
=
2\sum_{k=1}^n\mu_ky_k
=
2M_y.
\end{align*}
Consequently for $r\notin\{i,j\}$,
\begin{align*}
    \ee_\pi\!\left[b_{\pi_i\pi_j}\mu_{\pi_r}\right]
    =
    -\frac{2M_y}{n(n-1)(n-2)}.
\end{align*}
Therefore
\begin{align*}
\ee_\pi\!\left[
        b_{\pi_i\pi_j}\ee_0[S_{n,1}\mid\cR]
\right]
&=
\sum_{r=1}^n
\mu_r\,
\ee_\pi\!\left[
        b_{\pi_i\pi_j}\mu_{\pi_r}
\right] \\
&=
\frac{M_y}{n(n-1)}(\mu_i+\mu_j)
-\frac{2M_y}{n(n-1)(n-2)}
\sum_{r\notin\{i,j\}}\mu_r \\
&= \frac{M_y}{(n-1)(n-2)}(\mu_i+\mu_j).
\end{align*}
Therefore,
\begin{align*}
\ee_0[D_n\ee_0[S_{n,1}\mid\cR]]
&=
\frac{M_y}{n^2h_1(n-1)(n-2)}
\sum_{i\ne j}a_{ij}(\mu_i+\mu_j) \\
&=
\frac{2}{n^2h_1(n-1)(n-2)}
\left(\sum_{i=1}^nL_i\mu_i\right)
\left(\sum_{k=1}^ny_k\mu_k\right) \\
&= \frac{2}{n^2h_1(n-1)(n-2)}
\left\{
        \sum_{i=1}^n(L_i-\bar L)\mu_i
\right\}
\left\{
        \sum_{k=1}^ny_k\mu_k
\right\}.
\end{align*}
By Lemma~\ref{lem:Lmu-bound},
\begin{align*}
        \left|
        \sum_{i=1}^n(L_i-\bar L)\mu_i
        \right|
        =
        O(n).
\end{align*}
By Cauchy--Schwarz and since $\sum_k y_k^2=O(n)$ and $\sum_k\mu_k^2=O(n)$,
\begin{align*}
        \left|\sum_{k=1}^ny_k\mu_k\right|
        \le
        \left(\sum_{k=1}^ny_k^2\right)^{1/2}
        \left(\sum_{k=1}^n\mu_k^2\right)^{1/2}
        =
        O(n).
\end{align*}
Hence
\begin{align*}
        \left|\ee_0[D_n\ee_0[S_{n,1}\mid\cR]]\right|
        =
        O\!\left(\frac{1}{n^2h_1}\right) = o(1/n).
\end{align*}
Finally, by \eqref{eq:canonical-linear-representations},
\begin{align*}
\ee_0\left[
        (b_{n0}-\ee_0b_{n0})
        \ee_0[S_{n,1}\mid\cR]
\right]
&=
\frac{1}{n^2h_1}
\ee_\pi\left[
        L(d,z;\pi)L(\mu,\mu;\pi)
\right]\\
&=
\frac{1}{n^2h_1(n-1)}
\left(\sum_{i=1}^nd_i\mu_i\right)
\left(\sum_{k=1}^nz_k\mu_k\right),
\end{align*}
where the last equality is the standard covariance identity for two
centred linear permutation statistics. By
\eqref{lochd16a}, \eqref{lochd16b}, and
\eqref{eq:score-identities-centered},
\[
\left|\sum_{i=1}^nd_i\mu_i\right|
\le
\|d\|_2\|\mu\|_2
=
O(nh_1^{1/2}),
\qquad
\left|\sum_{k=1}^nz_k\mu_k\right|
\le
\|z\|_2\|\mu\|_2
=
O(n).
\]
Thus
\[
\left|
\ee_0\left[
        (b_{n0}-\ee_0b_{n0})
        \ee_0[S_{n,1}\mid\cR]
\right]
\right|
=
O\left(\frac{1}{nh_1^{1/2}}\right)
=
o(n^{-1/2}),
\]
because \(nh_1\to\infty\).
Combining the three estimates yields
\[
        |M_n'(0)|=O(n^{-1/2}).
\]
This proves \eqref{eq:first-mean-derivative-bound-centered} and completes the
proof of the lemma.
\end{proof}
\begin{lmm}\label{lem:Vprime-bound-centered}
Suppose that Assumptions~\ref{asmp:kernel} and~\ref{asmp:band} hold.
Then
\[
        V_n'(0)
        =
        O\left(\frac{h_1^{1/2}}{n}\right).
\]
\end{lmm}

\begin{proof}
Recall that
\[
        U_n
        =
        \hat\tau_n^2-\ee_0[\hat\tau_n^2].
\]
By \eqref{eq:first-derivative-identity-centered},
\[
        V_n'(0)
        =
        \ee_0[U_n^2S_{n,1}].
\]
Since \(U_n\) is \(\cR\)-measurable,
\eqref{eq:first-score-conditional-centered} gives
\begin{equation}
\label{eq:Vprime-rank-score}
        V_n'(0)
        =
        \ee_\pi
        \left[
        U_n^2
        \sum_{i=1}^n\mu_i\mu_{\pi_i}
        \right],
\end{equation}
where, under the null, \(\pi\) is a uniform random permutation.

Let
\[
        J(i):=n+1-i,
        \qquad
        1\le i\le n,
\]
and define
\[
        a_{ij}^{J}
        :=
        a_{J(i),J(j)}.
\]
Correspondingly, set
\[
        \hat\tau_{n,J}^2
        :=
        \frac{1}{n^2h_1}
        \sum_{i,j=1}^n
        a_{ij}^{J}B_{\pi_i\pi_j},
        \qquad
        U_{n,J}
        :=
        \hat\tau_{n,J}^2
        -
        \ee_\pi[\hat\tau_{n,J}^2].
\]

We first compare the original and reflected kernel weights.
By symmetry of \(K\), a change of variables gives
\[
        a_{J(i),J(j)}
        =
        \frac1{h_1}
        \int_{1/n}^{1+1/n}
        K\left(\frac{v-i/n}{h_1}\right)
        K\left(\frac{v-j/n}{h_1}\right)\,dv.
\]
Consequently,
\begin{align}
a_{ij}-a_{J(i),J(j)}
&=
\frac1{h_1}
\int_0^{1/n}
K\left(\frac{v-i/n}{h_1}\right)
K\left(\frac{v-j/n}{h_1}\right)\,dv
\nonumber\\
&\quad-
\frac1{h_1}
\int_1^{1+1/n}
K\left(\frac{v-i/n}{h_1}\right)
K\left(\frac{v-j/n}{h_1}\right)\,dv.
\label{eq:a-reflection-difference}
\end{align}

For each fixed \(v\), the number of indices \(i\) for which
\[
        K\left(\frac{v-i/n}{h_1}\right)\ne0
\]
is \(O(nh_1)\). Hence, using boundedness of \(K\),
\[
\sum_{i=1}^n
\left|
K\left(\frac{v-i/n}{h_1}\right)
\right|
=
O(nh_1)
\]
uniformly for
\[
v\in[0,1/n]\cup[1,1+1/n].
\]
It follows from \eqref{eq:a-reflection-difference} that
\begin{align*}
\sum_{i,j=1}^n
\left|
a_{ij}-a_{J(i),J(j)}
\right|
&\le
\frac{C}{h_1}
\left(
\int_0^{1/n}
+
\int_1^{1+1/n}
\right)
(nh_1)^2\,dv
\nonumber\\
&=
O(nh_1).
\end{align*}
Since \(0\le B_{k\ell}\le1\),
\begin{equation}
\label{eq:tau-reflection-sup}
\left|
\hat\tau_n^2-\hat\tau_{n,J}^2
\right|
\le
\frac{1}{n^2h_1}
\sum_{i,j=1}^n
\left|
a_{ij}-a_{J(i),J(j)}
\right|
=
O(n^{-1})
\end{equation}
uniformly in \(\pi\).

We next exploit the reflection symmetry of the Gaussian rank scores.
Define the reflected permutation
\[
        \pi_i^J:=\pi_{J(i)}.
\]
Since \(\pi\) is uniform, so is \(\pi^J\). Moreover, by symmetry of
the standard normal distribution,
\[
        \mu_{J(i)}=-\mu_i.
\]
Therefore
\begin{align}
\sum_{i=1}^n
\mu_i\mu_{\pi_i^J}
&=
\sum_{i=1}^n
\mu_i\mu_{\pi_{J(i)}}
=
\sum_{r=1}^n
\mu_{J(r)}\mu_{\pi_r}
\nonumber\\
&=
-
\sum_{r=1}^n
\mu_r\mu_{\pi_r}.
\label{eq:first-score-reflection}
\end{align}
Also,
\[
        \hat\tau_n^2(\pi^J)
        =
        \hat\tau_{n,J}^2(\pi).
\]
Thus \(U_n(\pi^J)=U_{n,J}(\pi)\), and
\eqref{eq:Vprime-rank-score} together with
\eqref{eq:first-score-reflection} yields
\[
        V_n'(0)
        =
        -
        \ee_\pi
        \left[
        U_{n,J}^2
        \sum_{i=1}^n\mu_i\mu_{\pi_i}
        \right].
\]
Consequently,
\begin{align}
2V_n'(0)
&=
\ee_\pi
\left[
\left(
U_n^2-U_{n,J}^2
\right)
\sum_{i=1}^n\mu_i\mu_{\pi_i}
\right]
\nonumber\\
&=
\ee_\pi
\left[
(U_n-U_{n,J})(U_n+U_{n,J})
\sum_{i=1}^n\mu_i\mu_{\pi_i}
\right].
\label{eq:Vprime-reflection}
\end{align}

Since \(\hat\tau_n^2\) and \(\hat\tau_{n,J}^2\) have the same
distribution under the uniform permutation law, their expectations
are equal. Hence \eqref{eq:tau-reflection-sup} gives
\[
        \|U_n-U_{n,J}\|_\infty
        =
        O(n^{-1}).
\]
Moreover, by \eqref{eq:null-variance-order-centered},
\[
        \|U_n\|_2
        =
        O\left(\frac{h_1^{1/2}}{n^{1/2}}\right),
\]
and, by reflection invariance,
\[
        \|U_{n,J}\|_2
        =
        \|U_n\|_2
        =
        O\left(\frac{h_1^{1/2}}{n^{1/2}}\right).
\]
Thus
\[
        \|U_n+U_{n,J}\|_2
        =
        O\left(\frac{h_1^{1/2}}{n^{1/2}}\right).
\]

Finally, since \(\sum_i\mu_i=0\), the usual variance formula for a
linear permutation statistic and
\(\sum_i\mu_i^2=O(n)\) give
\begin{align*}
\left\|
        \sum_{i=1}^n\mu_i\mu_{\pi_i}
\right\|_2^2
&=
\frac{1}{n-1}
\left(\sum_{i=1}^n\mu_i^2\right)^2
=
O(n).
\end{align*}
Applying Cauchy–Schwarz inequality to
\eqref{eq:Vprime-reflection}, we obtain
\begin{align*}
|V_n'(0)|
&\le
\frac12
\|U_n-U_{n,J}\|_\infty
\|U_n+U_{n,J}\|_2
\left\|
        \sum_{i=1}^n\mu_i\mu_{\pi_i}
\right\|_2
\\
&=
O\left(\frac1n\right)
O\left(\frac{h_1^{1/2}}{n^{1/2}}\right)
O(n^{1/2})
\\
&=
O\left(\frac{h_1^{1/2}}{n}\right),
\end{align*}
which proves the assertion.
\end{proof}
\begin{lmm}\label{lem:Vsecond-bound-centered}
Suppose that Assumptions~\ref{asmp:kernel} and~\ref{asmp:band} hold.
Then we have
\[
        V_n''(0)=O(1/n).
\]
\end{lmm}

\begin{proof}
Recall the notation $M_n(\rho) = \ee_\rho[\hat\tau_n^2]$. Since $V_n(\rho)=\ee_\rho[(\hat\tau_n^2)^2]-M_n(\rho)^2$,
we have
\begin{align*}
        V_n''(0)
        & =\left.\frac{d^2}{d\rho^2}\ee_\rho (\hat\tau_n^2)^2\right|_{\rho=0}
        -2\{M_n'(0)\}^2
        -2M_n(0)M_n''(0) \\
        &= 2\ee_0((\hat\tau_n^2)^2S_{n,2})
        -2\{M_n'(0)\}^2
        -4M_n(0)\ee_0(\hat\tau_n^2S_{n,2}), 
\end{align*}
where we have used \eqref{eq:second-derivative-identity-centered} twice and \(M_n(0)=\ee_0[\hat\tau_n^2]\). 
As  \(\ee_0S_{n,2}=0\), it follows that 
\begin{align*}
        2\ee_0(U_n^2S_{n,2})
        &=2\ee_0\{(\hat\tau_n^2- \ee_0[\hat\tau_n^2])^2S_{n,2}\} =2\ee_0[(\hat\tau_n^2)^2S_{n,2}]-4M_n(0)\ee_0[\hat\tau_n^2S_{n,2}],
\end{align*}
which implies 
\begin{align}\label{eq:Vsecond-identity-centered}
        V_n''(0) = 2\ee_0(U_n^2S_{n,2})-2\{M_n'(0)\}^2
 \end{align}
Because \(U_n^2\) is \(\cR\)-measurable, using the decomposition \eqref{eq:second-score-conditional-centered}, we have
\[
        \ee_0(U_n^2S_{n,2})=\ee_0\left[U_n^2\ee_0(S_{n,2}\mid\cR)\right] =\ee_0\left[U_n^2G_{n,2}^{H_2}(\pi)\right]
        +\ee_0\left[U_n^2G_{n,2}^{\rm pair}(\pi)\right] = O(1/n), 
\]
where the last estimate is a consequence of 
\eqref{eq:H2-score-bound-centered}, \eqref{eq:pair-score-bound-centered} and \eqref{eq:first-mean-derivative-bound-centered}.
Substituting these two estimates into \eqref{eq:Vsecond-identity-centered} proves the assertion.
\end{proof}

\subsubsection{Higher-order score contractions}
Recall the definition of $L_{n,0}^{(r)}$
in \eqref{eq:likelihood-derivative-def} and define 
for $r\ge1$,
\begin{align} 
\label{lochd17}
        \Gamma_{n,r}
        :=
        \ee_0[L_{n,0}^{(r)}
        \mid \mathcal R].
\end{align}
Since \(\hat\tau_n^2\) is \(\mathcal R\)-measurable,
\begin{align}\label{eq:mean-rank-score-r}
        M_n^{(r)}(0)
        =
        \ee_0\left[\hat\tau_n^2\Gamma_{n,r}\right] = \ee_0\left[U_n\Gamma_{n,r}\right],
\end{align}
where the last equality holds since
\[
        \ee_0\Gamma_{n,r}
        =
        \left.
        \frac{d^r}{d\rho^r}\ee_0 L_{\rho}
        \right|_{\rho=0}
        =
        0.
\]
Define $R_n(\rho) := \ee_\rho[(\hat\tau_n^2)^2]$ and note that $V_n(\rho) = R_n(\rho) - M_n(\rho)^2$. 
Taking derivatives, we have
\begin{align}
    \begin{split}
V_n^{(4)}(0)
&=
R_n^{(4)}(0)
-2M_n(0)M_n^{(4)}(0) 
-8M_n'(0)M_n'''(0)
-6\{M_n''(0)\}^2  \\
&=  
        \ee_0[U_n^2\Gamma_{n,4}]
        -8M_n'(0)M_n'''(0)
        -6\{M_n''(0)\}^2, 
            \end{split}
            \label{eq:Vfourth-exact}
\end{align}
where the last equality follows from (note that $\ee_0[L_{n,0}^{(4)}] = 0$)
\begin{align*}
R_n^{(4)}(0)-2M_n(0)M_n^{(4)}(0)
&= \ee_0[\{\hat\tau_n^2-M_n(0)\}^2L_{n,0}^{(4)}] 
=
\ee_0[U_n^2\Gamma_{n,4}]
\end{align*}
If $\pi $ denotes the permutation corresponding to the ranks ${\mathcal R}$, we obtain from \eqref{eq:general-likelihood-derivative}
\begin{align*}
\nonumber
\Gamma_{n,4}(\pi)
&= \ee\biggl[\frac{\partial^4}{\partial \rho^4}L_{\rho} \Big |_{\rho=0} \Big | {\mathcal{R}} \biggr] \\
& = 
\sum_{i=1}^n
\gamma_i\gamma_{\pi_i}
+
4\sum_{i\ne j}
p_{ij}p_{\pi_i\pi_j} +
6\sum_{1\le i<j\le n}
q_{ij}q_{\pi_i\pi_j}
\nonumber\\
&\quad+
12\sum_{i=1}^n
\sum_{\substack{1\le j<k\le n\\
j,k\ne i}}
r_{ijk}r_{\pi_i\pi_j\pi_k}
+
24\sum_{1\le i<j<k<\ell\le n}
s_{ijk\ell}
s_{\pi_i\pi_j\pi_k\pi_\ell},
\end{align*}
where
\begin{align*}
    \gamma_i
    &:= \ee\!\left[H_4(Z_{i:n})\right],
    & & 1 \le i \le n, \\[2pt]
    p_{ij}
    &:= \ee\!\left[ H_3(Z_{i:n}) H_1(Z_{j:n}) \right]
      = \ee\!\left[ (Z_{i:n}^3 - 3 Z_{i:n}) Z_{j:n} \right],
    & & i \ne j, \\[2pt]
    q_{ij}
    &:= \ee\!\left[ H_2(Z_{i:n}) H_2(Z_{j:n}) \right]
      = \ee\!\left[ (Z_{i:n}^2 - 1)(Z_{j:n}^2 - 1) \right],
    & & i \ne j, \\[2pt]
    r_{ijk}
    &:= \ee\!\left[ H_2(Z_{i:n}) H_1(Z_{j:n}) H_1(Z_{k:n}) \right]
      = \ee\!\left[ (Z_{i:n}^2 - 1) Z_{j:n} Z_{k:n} \right],
    & & |\{i, j, k\}| = 3, \\[2pt]
    s_{ijk\ell}
    &:=
    \ee\!\left[
    H_1(Z_{i:n})H_1(Z_{j:n})H_1(Z_{k:n})H_1(Z_{\ell:n})
    \right] = \ee\!\left[Z_{i:n}Z_{j:n}Z_{k:n}Z_{\ell:n}\right], & & |\{i, j, k, \ell\}| = 4.
\end{align*}

\begin{lmm}
\label{lem:fixed-order-rank-score}
Let $A$ and $B$ be deterministic symmetric $n\times n$ matrices with zero diagonal, and let $\pi$ be a uniform random permutation on $\{1, \ldots, n\}$. For every fixed \(r\ge1\) and \(p\in\{1,2\}\),
\begin{equation}\label{eq:fixed-order-rank-score-bound}
        \left|
        \ee_\pi\left[
        Q(A,B;\pi)^p\Gamma_{n,r}(\pi)
        \right]
        \right|
        \le
        C_{r,p}
        \|A\|_{\mathrm{op}}^p
        \|B\|_{\mathrm{op}}^p,
\end{equation}
where $C_{r,p}$ depends only on $r$ and $p$, and not on $n$, $A$, or $B$.
\end{lmm}

\begin{proof}[Proof of Lemma~\ref{lem:fixed-order-rank-score}]

By Lemma~\ref{lem:likelihood-derivatives},
$L_{n,0}^{(r)}$ is given by \eqref{eq:general-likelihood-derivative}. Recall from \eqref{lochd17} that
\[
    \Gamma_{n,r}
    =
    \ee_0\!\left[
        L_{n,0}^{(r)}
        \mid\mathcal R
    \right].
\]
For a multi-index
\[
    \boldsymbol\alpha=(\alpha_1,\ldots,\alpha_s),
\]
define
\[
    t_{\boldsymbol\alpha}(k_1,\ldots,k_s)
    :=
    \ee\left[
        \prod_{u=1}^s
        H_{\alpha_u}(Z_{k_u:n})
    \right],
\]
where \(k_1,\ldots,k_s\) are distinct.

Under \(\rho=0\), after ordering
the observations by their \(X\)-ranks, the \(X\)- and \(Y\)-order
statistic vectors are independent, while the relative \(Y\)-rank
vector \(\pi\) is uniform, which gives 
\begin{align}
\Gamma_{n,r}(\pi)
&=
r!
\sum_{s=1}^r\frac{1}{s!}
\sum_{\substack{\alpha_1,\ldots,\alpha_s\ge1\\
\alpha_1+\cdots+\alpha_s=r}}
\frac{1}{\alpha_1!\cdots\alpha_s!}
\sum_{i_1,\ldots,i_s}^{*}
t_{\boldsymbol\alpha}(i_1,\ldots,i_s)
t_{\boldsymbol\alpha}
        (\pi_{i_1},\ldots,\pi_{i_s}).
\label{eq:conditional-general-score}
\end{align}
We first record two elementary moment estimates for Gaussian order
statistics. Put
\[
        W=(Z_{1:n},\ldots,Z_{n:n})^\top,
        \qquad
        h_q(W)
        :=
        \bigl(H_q(Z_{1:n}),\ldots,H_q(Z_{n:n})\bigr)^\top,
\]
where
\[
h_0(W) = (1, \ldots, 1).
\]
For every fixed \(q\ge1\) and \(m\ge1\),
\begin{align}\label{eq:Hermite-vector-moment}
        \left\|
        \|h_q(W)\|_2
        \right\|_{L^m}
        & =
        O(n^{1/2}) , \\
\label{eq:Hermite-sum-moment}
        \Big \|
        \sum_{i=1}^nH_q(Z_{i:n})
        \Big \|_{L^m}
        & =
        \Big \|
        \sum_{i=1}^nH_q(Z_i)
         \Big \|_{L^m}
        =
        O(n^{1/2}).
\end{align}
Indeed, for \(m\ge2\), the first estimate follows from 
\begin{align*}
\left\|
        \|h_q(W)\|_2
\right\|_{L^m}^2
& =
\Big \|
        \sum_{i=1}^nH_q(Z_{i:n})^2
\Big \|_{L^{m/2}} =
\Big \|
        \sum_{i=1}^nH_q(Z_i)^2
\Big \|_{L^{m/2}} \\
& \le
n 
\|H_q(Z_1)^2\|_{L^{m/2}}
=
n\|H_q(Z_1)\|_{L^m}^2,
\end{align*}
and the case \(1\le m<2\) follows by monotonicity of \(L^m\)-norms. To prove \eqref{eq:Hermite-sum-moment}, note that for $m\ge2$, Rosenthal's inequality gives
\begin{align*}
\ee\left|
        \sum_{i=1}^nH_q(Z_i)
\right|^m
\le
C_m\left\{
        n\ee|H_q(Z_1)|^m
        +
        \left(n\ee H_q(Z_1)^2\right)^{m/2}
        \right\}
=
O(n^{m/2}),
\end{align*}
because \(H_q(Z_i)\) are independent and centred,
\(\ee H_q(Z_1)^2=q!\), and all their fixed moments are finite.
Taking \(m\)th roots proves the assertion for \(m\ge2\), and  the case 
\(1\le m<2\) follows again by monotonicity of \(L^m\)-norms.

We next establish  a
bound for the sums in \eqref{eq:fixed-order-rank-score-bound} obtained after the matrix indices and the indices in the sum \eqref{eq:conditional-general-score}
are identified according to a fixed coincidence pattern and then
summed out. Fix \(v\le 2p+r\), a collection of \(p\) ordered
pairs
\[
        \mathcal I
        =
        \bigl((a_1,b_1),\ldots,(a_p,b_p)\bigr),
        \qquad
        a_\ell,b_\ell\in\{1,\ldots,v\},
        \qquad
        a_\ell\ne b_\ell,
\]
and labels
\[
        \boldsymbol d=(d_1,\ldots,d_v),
        \qquad
        d_j\in\{0,1,\ldots,r\},
        \qquad
        \sum_{j=1}^vd_j\le r.
\]
Assume that every \(j\) which does not appear in any of the pairs
\((a_\ell,b_\ell)\) satisfies \(d_j\ge1\). For the quantity 
\begin{align*}
\mathcal K_{\mathcal I,\boldsymbol d}(A)
&:=
\sum_{k_1,\ldots,k_v}^{*}
\left\{
\prod_{\ell=1}^p
A_{k_{a_\ell}k_{b_\ell}}
\right\}
\ee\left[
        \prod_{j=1}^vH_{d_j}(Z_{k_j:n})
\right],
\end{align*}
we show at the end of the proof that
\begin{equation}\label{eq:index-pattern-contraction-bound}
        \left|
        \mathcal K_{\mathcal I,\boldsymbol d}(A)
        \right|
        \le
        C_{r,p}\,
        n^{v/2}\|A\|_{\mathrm{op}}^p.
\end{equation}
Taking this estimate for granted, we now derive a bound for 
\[
        \ee_\pi\left[
        Q(A,B;\pi)^p\Gamma_{n,r}(\pi)
        \right].
\]
Expanding the \(p\) factors of \(Q(A,B;\pi)\), write their indices as
\[
        j_1,k_1,\ldots,j_p,k_p,
        \qquad
        j_\ell\ne k_\ell.
\]
Fix  an index \(s\), a multi-index 
\(\boldsymbol\alpha=(\alpha_1,\ldots,\alpha_s)\) in the sum \eqref{eq:conditional-general-score} and an admissible
equality pattern among the \(2p+s\) index positions
\[
        j_1,k_1,\ldots,j_p,k_p,i_1,\ldots,i_s.
\]
Here the indices \(i_1,\ldots,i_s\) come from $\Gamma_{n,r} (\pi) $ with positive degrees $\alpha_1, \ldots, \alpha_s\geq 1$  of the Hermite  polynomials $H_{\alpha_u}$ and remain pairwise distinct, as required by \eqref{eq:conditional-general-score}. The indices $j_\ell, k_\ell$ correspond to the entries $A_{j_\ell k_\ell}$. Suppose the pattern
contains \(v\) distinct indices, which we relabel as
\(1,\ldots,v\).

The pattern determines \(p\) ordered pairs
\[
        \mathcal I
        =
        \bigl((a_1,b_1),\ldots,(a_p,b_p)\bigr)
\]
and labels \(\boldsymbol d=(d_1,\ldots,d_v)\): if a distinct index coincides with an index \(i_u\), it receives label
\(\alpha_u\), and otherwise it receives label zero. Since the
 indices  \(i_1, \ldots , i_s\) are pairwise distinct, each distinct index receives at
most one positive label.

Now consider a distinct index $j$ that does not occur in any matrix factor. This means that its equality class contains none of the endpoint positions $j_1, k_1, \ldots, j_p, k_p$. Therefore, it must contain exactly  one of the index positions $i_u$ (note that the  indices $i_1, \ldots, i_s$ are pairwise distinct). Its label is consequently $d_j=\alpha_u \geq 1$. Thus
\eqref{eq:index-pattern-contraction-bound} applies.

Let
\[
        \kappa_{\boldsymbol\alpha}
        :=
        \frac{r!}
        {s!\,\alpha_1!\cdots\alpha_s!}.
\]
For fixed distinct domain indices
\(u_1,\ldots,u_v\), the vector
\[
        (\pi_{u_1},\ldots,\pi_{u_v})
\]
is a uniform ordered sample without replacement from
\(\{1,\ldots,n\}\). Consequently,
\begin{align*}
&\ee_\pi\left[
        \prod_{\ell=1}^p
        B_{\pi_{u_{a_\ell}}\pi_{u_{b_\ell}}}
        \,
        \ee\left\{
        \prod_{j=1}^v
        H_{d_j}(Z_{\pi_{u_j}:n})
        \right\}
        \right]\\
&\qquad=
        \frac{1}{(n)_v}
        \sum_{m_1,\ldots,m_v}^{*}
        \left\{
        \prod_{\ell=1}^p
        B_{m_{a_\ell}m_{b_\ell}}
        \right\}
        \ee\left[
        \prod_{j=1}^v
        H_{d_j}(Z_{m_j:n})
        \right] =
        \frac{
        \mathcal K_{\mathcal I,\boldsymbol d}(B)
        }{(n)_v},
\end{align*}
where
\[
        (n)_v=n(n-1)\cdots(n-v+1).
\]
Summing over the distinct domain indices shows that the contribution
of the fixed composition and equality pattern is
\[
        \kappa_{\boldsymbol\alpha}
        \frac{
        \mathcal K_{\mathcal I,\boldsymbol d}(A)
        \mathcal K_{\mathcal I,\boldsymbol d}(B)
        }{(n)_v}.
\]

If \(n<v\), the equality pattern has no admissible assignment and its
contribution is zero. If \(n\ge v\), then
\eqref{eq:index-pattern-contraction-bound} gives
\begin{align*}
\left|
        \kappa_{\boldsymbol\alpha}
        \frac{
        \mathcal K_{\mathcal I,\boldsymbol d}(A)
        \mathcal K_{\mathcal I,\boldsymbol d}(B)
        }{(n)_v}
\right|
\le
        C_{r,p}
        \frac{
        n^{v/2}\|A\|_{\mathrm{op}}^p
        n^{v/2}\|B\|_{\mathrm{op}}^p
        }{(n)_v}
=  C_{r,p} \frac{n^v}{(n)_v}
        \|A\|_{\mathrm{op}}^p
        \|B\|_{\mathrm{op}}^p   \le  C_{r,p}
        \|A\|_{\mathrm{op}}^p
        \|B\|_{\mathrm{op}}^p, 
\end{align*}
where the inequality follows from $
        n^v / (n)_v  =
        \prod_{j=0}^{v-1}n/(n-j)
        \le
        v^v / v! $ ($n\ge v$).

For fixed \(r\) and \(p\), the number of compositions and admissible
equality patterns is finite and depends only on \(r\) and \(p\), which yields
\[
        \left|
        \ee_\pi\left[
        Q(A,B;\pi)^p\Gamma_{n,r}(\pi)
        \right]
        \right|
        \le
        C_{r,p}
        \|A\|_{\mathrm{op}}^p
        \|B\|_{\mathrm{op}}^p 
\]
and proves \eqref{eq:fixed-order-rank-score-bound}. Therefore, it remains to prove \eqref{eq:index-pattern-contraction-bound}. We first write
\begin{align*}
\mathcal K_{\mathcal I,\boldsymbol d}(A)
&=
\ee\left[
\sum_{k_1,\ldots,k_v}^{*}
\left\{
\prod_{\ell=1}^p
A_{k_{a_\ell}k_{b_\ell}}
\right\}
\prod_{j=1}^vH_{d_j}(Z_{k_j:n})
\right].
\end{align*}
The distinctness restriction can be handled by finite
inclusion--exclusion. More precisely, the restricted sum over pairwise
distinct $k_1,\ldots,k_v$ is an exact finite linear combination of
unrestricted sums obtained by imposing selected equalities among these
indices. Since $v\le 2p+r$ and $r,p$ are fixed, the number of such coincidence patterns and
their inclusion--exclusion coefficients are bounded by constants
depending only on $r$ and $p$.

For a given coincidence pattern, all Hermite factors whose indices
have been identified are evaluated at the same order statistic. Their
product can be expanded, by the Hermite product formula, as a finite
linear combination of single Hermite polynomials of degree at most
$r$. If the two endpoints of a matrix factor are identified, the
term vanishes because $A_{ii}=0$.

Equivalently, let \(\mathcal P\) be a partition of
\(\{1,\ldots,v\}\). Each block \(S\in\mathcal P\) corresponds to one
index after the identifications. The Hermite factors associated with
that block are replaced by
\[
        P_S(z)
        :=
        \prod_{j\in S}H_{d_j}(z).
\]
Since the total Hermite degree is at most \(r\), the Hermite product
formula gives
\[
        P_S(z)
        =
        \sum_{q=0}^r c_{S,q}H_q(z),
\]
where the coefficients are bounded by constants depending only on
\(r\). Expanding these products reduces each inclusion--exclusion term
to a finite sum of unrestricted expressions containing at most two
factors from \(A\).

If an identification makes the two indices in one matrix factor
equal, the corresponding expression vanishes because \(A_{ii}=0\).
We therefore consider only the non-zero terms. Since
\(p\in\{1,2\}\), the part of any such term involving \(A\) has one of
the following forms.

If \(p=1\), it is a bilinear form
\[
        h_a(W)^\top A h_b(W),
        \qquad
        0\le a,b\le r.
\]
By the operator-norm inequality, Hölder's inequality, and
\eqref{eq:Hermite-vector-moment},
\begin{align*}
\left\|
        h_a(W)^\top A h_b(W)
\right\|_{L^m}
&\le
\|A\|_{\mathrm{op}}
\left\|
        \|h_a(W)\|_2\|h_b(W)\|_2
\right\|_{L^m}\\
&\le
\|A\|_{\mathrm{op}}
\left\|\|h_a(W)\|_2\right\|_{L^{2m}}
\left\|\|h_b(W)\|_2\right\|_{L^{2m}}\\
&=
O\left(n\|A\|_{\mathrm{op}}\right).
\end{align*}
Suppose now that $p=2$. If, after the imposed index
identifications, the two matrix entries depend on four distinct
quotient indices, the corresponding expression is a product of two
bilinear forms,
\[
        \bigl\{h_a(W)^\top A h_b(W)\bigr\}
        \bigl\{h_c(W)^\top A h_d(W)\bigr\},\qquad 0\le a,b,c,d\le r.
\]

By Hölder's inequality and the preceding bilinear-form estimate, its
fixed moments are
\[
        O\left(n^2\|A\|_{\mathrm{op}}^2\right).
\]

If, after the imposed identifications, the two matrix entries share
exactly one summation index, the
corresponding expression has the form
\[
        \sum_{i=1}^n
        H_a(W_i)
        \{Ah_b(W)\}_i
        \{Ah_c(W)\}_i .
\]
Using Hölder's inequality for the finite-dimensional norms,
\begin{align*}
\left|
        \sum_{i=1}^n
        H_a(W_i)
        \{Ah_b(W)\}_i
        \{Ah_c(W)\}_i
\right|
&\le
\|h_a(W)\|_2
\|Ah_b(W)\|_{\ell^4}
\|Ah_c(W)\|_{\ell^4}\\
&\le
\|A\|_{\mathrm{op}}^2
\|h_a(W)\|_2
\|h_b(W)\|_2
\|h_c(W)\|_2.
\end{align*}
Consequently, by Hölder's inequality in probability and
\eqref{eq:Hermite-vector-moment}, its fixed moments are
of order \[
        O\left(
        n^{3/2}\|A\|_{\mathrm{op}}^2
        \right).
\]

Finally, if, after the imposed identifications, the two matrix entries depend
on the same pair of quotient indices,
their product is \(A_{ij}^2\), since \(A\) is symmetric. The
corresponding expression is
\[
        h_a(W)^\top(A\circ A)h_b(W),
\]
where \(A\circ A=(A_{ij}^2)\). Since \(A\circ A\) is symmetric and
entrywise non-negative,
\begin{align*}
        \|A\circ A\|_{\mathrm{op}}
        \le
        \max_i\sum_{j=1}^nA_{ij}^2
        =
        \max_i\|Ae_i\|_2^2
        \le
        \|A\|_{\mathrm{op}}^2.
\end{align*}
The bilinear-form estimate therefore gives fixed moments of order
\[
        O\left(n\|A\|_{\mathrm{op}}^2\right).
\]

Thus, if \(w\) denotes the number of quotient indices which occur in
the matrix factors, then in all cases the part involving \(A\) has
fixed moments bounded by
\begin{equation}\label{eq:A-dependent-pattern-bound}
        C_{r,p}\,
        n^{w/2}\|A\|_{\mathrm{op}}^p.
\end{equation}

It remains to account for the quotient indices which do not occur in
a matrix factor. For such a partition block \(S\), the associated
factor is either
\[
        \sum_{i=1}^nH_q(Z_{i:n}),
        \qquad q\ge1,
\]
or the constant \(n\), corresponding to \(q=0\).

If \(q\ge1\), its fixed moments are of order \(O(n^{1/2})\) by
\eqref{eq:Hermite-sum-moment}, and therefore
\[
        O(n^{1/2})
        =
        O(n^{|S|/2}).
\]
If \(q=0\), the factor equals \(n\). In a non-zero term, an isolated
partition block cannot contain an original index appearing in a matrix
factor; otherwise that matrix factor would have been collapsed to a
diagonal entry. Hence every original index in \(S\) has strictly
positive label.  If \(|S|=1\), then
\(P_S=H_{d_j}\) with \(d_j\geq1\), whose Hermite expansion has zero
\(H_0\)-coefficient. Consequently, \(q=0\) can occur only if
\(|S|\geq2\), and therefore
\[    n
 \le  n^{|S|/2}.
\]
The number of factors involved is bounded solely in terms of \(r\)
and \(p\), so Hölder's inequality may be applied using fixed moments.
If \(\mathcal P_A\) denotes the collection of partition blocks whose
quotient indices occur in the matrix factors, then
\[
        w
        \le
        \sum_{S\in\mathcal P_A}|S|.
\]
Combining \eqref{eq:A-dependent-pattern-bound} with the bounds for the
remaining blocks gives
\begin{align*}
        C_{r,p}
        n^{w/2}\|A\|_{\mathrm{op}}^p
        \prod_{S\notin\mathcal P_A}n^{|S|/2}
\le
        C_{r,p}
        n^{\frac{1}{2}\sum_{S\in\mathcal P}|S|}
        \|A\|_{\mathrm{op}}^p
=
        C_{r,p}
        n^{v/2}\|A\|_{\mathrm{op}}^p.
\end{align*}
Since the number of partitions and Hermite-product terms depends only
on \(r\) and \(p\), this proves
\eqref{eq:index-pattern-contraction-bound}.
\end{proof}

\begin{lmm}
\label{lem:mixed-feature-bound}
Let \(\mathcal H_1\) and \(\mathcal H_2\) be finite-dimensional real
Hilbert spaces. For \(j=1,2\), let
\[
        F_j,G_j:\mathbb R^n\to\mathcal H_j
\]
be measurable functions such that \(F_j(X)\) depends on \(X\) only
through its rank vector \(R^x\), whereas \(G_j(Y)\) depends on \(Y\)
only through its rank vector \(R^y\). Suppose that, for \(j=1,2\),
\[
        \|F_j(X)\|_{\mathcal H_j}
        \quad\text{and}\quad
        \|G_j(Y)\|_{\mathcal H_j}
\]
are deterministic. Then, for every fixed integer \(r\geq1\),
\begin{align}
\left|
\ee_0\left[
        \langle F_1(X),G_1(Y)\rangle_{\mathcal H_1}
        \langle F_2(X),G_2(Y)\rangle_{\mathcal H_2}
        \Gamma_{n,r}
\right]
\right|
\leq
r!\,
\|F_1\|_{L^2}\|F_2\|_{L^2}
\|G_1\|_{L^2}\|G_2\|_{L^2}.
\label{eq:mixed-feature-bound}
\end{align}
Similarly,
\begin{equation}\label{eq:single-feature-bound}
\left|
\ee_0\left[
        \langle F_1(X),G_1(Y)\rangle_{\mathcal H_1}
        \Gamma_{n,r}
\right]
\right|
\leq
r!\,\|F_1\|_{L^2}\|G_1\|_{L^2}.
\end{equation}
Here each \(L^2\)-norm is taken with respect to \(\gamma_n\), with
values in the corresponding Hilbert space.
\end{lmm}

\begin{proof}
Set
\[
        F(X):=F_1(X)\otimes F_2(X),
        \qquad
        G(Y):=G_1(Y)\otimes G_2(Y),
\]
viewed as functions with values in
\(\mathcal H_1\otimes\mathcal H_2\). By the defining property of the
Hilbert tensor product,
\begin{align*}
\langle F(X),G(Y)\rangle_{\mathcal H_1\otimes\mathcal H_2}
&=
\langle F_1(X),G_1(Y)\rangle_{\mathcal H_1}
\langle F_2(X),G_2(Y)\rangle_{\mathcal H_2}.
\end{align*}
Since \(F_1(X)\) and \(F_2(X)\) are functions of the \(X\)-rank
vector, and \(G_1(Y)\) and \(G_2(Y)\) are functions of the \(Y\)-rank
vector, the preceding random variable is \(\mathcal R\)-measurable.
Hence, by the tower property and the definition of
\(\Gamma_{n,r}\),
\begin{align*}
&\ee_0\left[
        \langle F_1(X),G_1(Y)\rangle_{\mathcal H_1}
        \langle F_2(X),G_2(Y)\rangle_{\mathcal H_2}
        \Gamma_{n,r}
\right]
\\
&\qquad=
\ee_0\left[
        \langle F(X),G(Y)\rangle_{\mathcal H_1\otimes\mathcal H_2}
        L_{n,0}^{(r)}
\right]
\\
&\qquad=
\left.
\frac{d^r}{d\rho^r}
\ee_\rho\left[
        \langle F(X),G(Y)\rangle_{\mathcal H_1\otimes\mathcal H_2}
\right]
\right|_{\rho=0}.
\end{align*}
Proposition~\ref{prop:hilbert_mehler}, applied in the Hilbert space
\(\mathcal H_1\otimes\mathcal H_2\), gives
\begin{align*}
\ee_0\left[
        \langle F_1(X),G_1(Y)\rangle_{\mathcal H_1}
        \langle F_2(X),G_2(Y)\rangle_{\mathcal H_2}
        \Gamma_{n,r}
\right]
=
r!\,
\left\langle
        P_rF,P_rG
\right\rangle_
{L^2(\gamma_n;\mathcal H_1\otimes\mathcal H_2)}.
\end{align*}
Since \(P_r\) is an orthogonal projection, the Cauchy--Schwarz
inequality yields
\begin{align*}
\left|
\ee_0\left[
        \langle F_1(X),G_1(Y)\rangle_{\mathcal H_1}
        \langle F_2(X),G_2(Y)\rangle_{\mathcal H_2}
        \Gamma_{n,r}
\right]
\right|
\leq
r!\,
\|F\|_{L^2(\gamma_n;\mathcal H_1\otimes\mathcal H_2)}
\|G\|_{L^2(\gamma_n;\mathcal H_1\otimes\mathcal H_2)}.
\end{align*}
It remains to evaluate these two norms. By the defining property of
the Hilbert tensor-product norm,
\begin{align*}
\|F\|_{L^2(\gamma_n;\mathcal H_1\otimes\mathcal H_2)}^2
=
\ee_0\left[
        \|F_1(X)\otimes F_2(X)\|_
        {\mathcal H_1\otimes\mathcal H_2}^2
\right]
=
\ee_0\left[
        \|F_1(X)\|_{\mathcal H_1}^2
        \|F_2(X)\|_{\mathcal H_2}^2
\right].
\end{align*}
Since both pointwise Hilbert norms are deterministic,
\[
\|F\|_{L^2(\gamma_n;\mathcal H_1\otimes\mathcal H_2)}^2
=
\|F_1\|_{L^2}^2\|F_2\|_{L^2}^2.
\]
Similarly,
\[
\|G\|_{L^2(\gamma_n;\mathcal H_1\otimes\mathcal H_2)}^2
=
\|G_1\|_{L^2}^2\|G_2\|_{L^2}^2.
\]
Substituting these identities into the preceding bound proves
\eqref{eq:mixed-feature-bound}.

For \eqref{eq:single-feature-bound}, apply the same argument directly
in \(\mathcal H_1\), with \(F=F_1\) and \(G=G_1\).
\end{proof}

\begin{lmm}
\label{lem:Q-L-fourth-score-small}
Suppose that Assumptions~\ref{asmp:kernel} and~\ref{asmp:band} hold.
Let $A^{\mathrm c}$, $B^{\mathrm c}$, $x=(x_i)_{i=1}^n$, and $y=(y_i)_{i=1}^n$ be as in
Lemma~\ref{lem:null-canonical-centering-bounds}. Then
\begin{equation}
\label{eq:Q-L-fourth-score-small}
        \left|
        \ee_\pi\left[
        Q(A^{\mathrm c},B^{\mathrm c};\pi)L(x,y;\pi)\Gamma_{n,4}(\pi)
        \right]
        \right|
        =
        O(n^4h_1^{5/2})
        =
        o(n^4h_1^2).
\end{equation}
\end{lmm}

\begin{proof}[Proof of Lemma~\ref{lem:Q-L-fourth-score-small}]
Define two deterministic symmetric zero-diagonal matrices
\(H^X\) and
\(H^Y\) by
\[
        H^X_{ij}
        :=
        \begin{cases}
        x_i+x_j, & i\ne j,\\
        0, & i=j,
        \end{cases}
        \qquad
        H^Y_{ij}
        :=
        \begin{cases}
        y_i+y_j, & i\ne j,\\
        0, & i=j.
        \end{cases}
\]
Since
\[
        \sum_{i=1}^n x_i
        =
        \sum_{i=1}^n y_i
        =
        0,
\]
we have
\begin{align*}
Q(H^X,H^Y;\pi)
&=
\sum_{i\ne j}
(x_i+x_j)(y_{\pi_i}+y_{\pi_j})\\
&=
\sum_{i\ne j}x_i y_{\pi_i}
+
\sum_{i\ne j}x_i y_{\pi_j}
+
\sum_{i\ne j}x_j y_{\pi_i}
+
\sum_{i\ne j}x_j y_{\pi_j}.
\end{align*}
The first and fourth sums are equal to
\[
        (n-1)L(x,y;\pi),
\]
whereas
\[
\sum_{i\ne j}x_i y_{\pi_j}
=
\sum_{i=1}^n
x_i
\left(
\sum_{j=1}^n y_{\pi_j}-y_{\pi_i}
\right)
=
-L(x,y;\pi),
\]
and the same identity holds for the third sum. Consequently,
\begin{equation}
\label{eq:L-as-Q-sharp}
        Q(H^X,H^Y;\pi)
        =
        2(n-2)L(x,y;\pi).
\end{equation}

We next exploit the canonical zero-row-sum property of \(A^{\mathrm c}\) and
\(B^{\mathrm c}\). For every permutation \(\pi\),
\begin{align*}
Q(A^{\mathrm c},H^Y;\pi)
&=
\sum_{i\ne j}a^{\mathrm c}_{ij}(y_{\pi_i}+y_{\pi_j})\\
&=
\sum_{i=1}^n
y_{\pi_i}\sum_{j\ne i}a^{\mathrm c}_{ij}
+
\sum_{j=1}^n
y_{\pi_j}\sum_{i\ne j}a^{\mathrm c}_{ij}
=0.
\end{align*}
Similarly,
\begin{align*}
Q(H^X,B^{\mathrm c};\pi)
&=
\sum_{i\ne j}(x_i+x_j)b^{\mathrm c}_{\pi_i\pi_j}\\
&=
\sum_{i=1}^n
x_i\sum_{j\ne i}b^{\mathrm c}_{\pi_i\pi_j}
+
\sum_{j=1}^n
x_j\sum_{i\ne j}b^{\mathrm c}_{\pi_i\pi_j}
=0,
\end{align*}
because, as \(j\) ranges over
\(\{1,\ldots,n\}\setminus\{i\}\), the index \(\pi_j\) ranges over
\(\{1,\ldots,n\}\setminus\{\pi_i\}\), and \(B^{\mathrm c}\) has zero row sums.
Thus
\begin{equation}
\label{eq:cross-Q-zero}
        Q(A^{\mathrm c},H^Y;\pi)
        =
        Q(H^X,B^{\mathrm c};\pi)
        =
        0.
\end{equation}

We now record the operator-norm bounds for the two auxiliary matrices.
Since
\[
        H^X
        =
        x\mathbf 1^\top
        +
        \mathbf 1x^\top
        -
        2\operatorname{diag}(x),
\]
where $\mathbf 1 = (1, \ldots, 1)^\top\in\rr^n$, we obtain
\begin{align}
\|H^X\|_{\mathrm{op}}
\le
2\sqrt n\,\|x\|_2
+
2\|x\|_\infty
=
O(n^2h_1^{3/2})
+
O(nh_1)
=
O(n^2h_1^{3/2}),
\label{eq:Xsharp-op}
\end{align}
where we have used
\[
        \|x\|_2
        =
        O(n^{3/2}h_1^{3/2}),
        \qquad
        \|x\|_\infty
        =
        O(nh_1),
\]
from Lemma~\ref{lem:null-canonical-centering-bounds}. 

Similarly,
\[
        H^Y
        =
        y\mathbf 1^\top
        +
        \mathbf 1y^\top
        -
        2\operatorname{diag}(y),
\]
and therefore
\begin{equation}
\label{eq:Ysharp-op}
        \|H^Y\|_{\mathrm{op}}
        \le
        2\sqrt n\,\|y\|_2
        +
        2\|y\|_\infty
        =
        O(n),
\end{equation}
since
\[
        \|y\|_2=O(n^{1/2}),
        \qquad
        \|y\|_\infty=O(1).
\]

Set
\[
        \alpha_n:=n^{1/2}h_1^{1/4},
\]
and define
\[
        A_\pm
        :=
        \alpha_n A^{\mathrm c}
        \pm
        \alpha_n^{-1}H^X,
        \qquad
        B_+
        :=
        B^{\mathrm c} + H^Y.
\]
These matrices are symmetric and have zero diagonal. By
\eqref{eq:cross-Q-zero},
\begin{align*}
Q(A_\pm, B_+;\pi)
&=
\alpha_n Q(A^{\mathrm c},B^{\mathrm c};\pi)
\pm
\alpha_n^{-1}Q(H^X,H^Y;\pi).
\end{align*}
It follows that
\begin{align}
Q(A_+, B_+;\pi)^2-Q(A_-, B_+;\pi)^2
=
4Q(A^{\mathrm c},B^{\mathrm c};\pi)Q(H^X,H^Y;\pi).
\label{eq:polarization-Q}
\end{align}

Multiplying \eqref{eq:polarization-Q} by
\(\Gamma_{n,4}(\pi)\), taking expectations, and applying
Lemma~\ref{lem:fixed-order-rank-score} with \(r=4\) and \(p=2\), we
obtain
\begin{align}
\left|
\ee_\pi\left[
Q(A^{\mathrm c},B^{\mathrm c};\pi)
Q(H^X,H^Y;\pi)
\Gamma_{n,4}(\pi)
\right]
\right|
&\le
\frac{1}{4}
\sum_{\varepsilon\in\{+,-\}}
\left|
\ee_\pi\left[
Q(A_\varepsilon, B_+;\pi)^2
\Gamma_{n,4}(\pi)
\right]
\right|
\nonumber\\
&\le
C
\max_{\varepsilon\in\{+,-\}}
\|A_\varepsilon\|_{\mathrm{op}}^2
\|B_+\|_{\mathrm{op}}^2.
\label{eq:QQL-score-bound}
\end{align}

By Lemma~\ref{lem:null-canonical-centering-bounds},
\[
        \|A^{\mathrm c}\|_{\mathrm{op}}=O(nh_1),
        \qquad
        \|B^{\mathrm c}\|_{\mathrm{op}}=O(n).
\]
Together with \eqref{eq:Xsharp-op},
\begin{align*}
\|A_\pm\|_{\mathrm{op}}
&\le
\alpha_n\|A^{\mathrm c}\|_{\mathrm{op}}
+
\alpha_n^{-1}\|H^X\|_{\mathrm{op}}\\
&=
O\left(
n^{1/2}h_1^{1/4}\,nh_1
\right)
+
O\left(
n^{-1/2}h_1^{-1/4}\,
n^2h_1^{3/2}
\right)\\
&=
O(n^{3/2}h_1^{5/4}).
\end{align*}
Similarly, by \eqref{eq:Ysharp-op},
\[
        \|B_+\|_{\mathrm{op}}
        \le
        \|B^{\mathrm c}\|_{\mathrm{op}}
        +
        \|H^Y\|_{\mathrm{op}}
        =
        O(n).
\]
Substituting these bounds into \eqref{eq:QQL-score-bound} yields
\begin{equation*}
\left|
\ee_\pi\left[
Q(A^{\mathrm c},B^{\mathrm c};\pi)
Q(H^X,H^Y;\pi)
\Gamma_{n,4}(\pi)
\right]
\right|
=
O(n^5h_1^{5/2}).
\end{equation*}

Finally, by \eqref{eq:L-as-Q-sharp},
\begin{align*}
&\left|
\ee_\pi\left[
Q(A^{\mathrm c},B^{\mathrm c};\pi)
L(x,y;\pi)
\Gamma_{n,4}(\pi)
\right]
\right|\\
&\qquad=
\frac{1}{2(n-2)}
\left|
\ee_\pi\left[
Q(A^{\mathrm c},B^{\mathrm c};\pi)
Q(H^X,H^Y;\pi)
\Gamma_{n,4}(\pi)
\right]
\right|\\
&\qquad=
O(n^4h_1^{5/2}).
\end{align*}
Since \(h_1\to0\),
\[
        n^4h_1^{5/2}
        =
        o(n^4h_1^2),
\]
which proves \eqref{eq:Q-L-fourth-score-small}.
\end{proof}

\begin{lmm}
\label{lem:higher-score-contractions}
Suppose that Assumptions~\ref{asmp:kernel} and~\ref{asmp:band} hold. 
Then,
for \(r=2,3\),
\begin{align}
\left|
\ee_0[\widetilde D_n\Gamma_{n,r}]
\right|
&=O(1),
\nonumber\\
\left|
\ee_0[(T_n-\ee_0T_n)\Gamma_{n,r}]
\right|
&=O(h_1^{1/2}),
\nonumber\\
\left|
\ee_0[(b_{n0}-\ee_0b_{n0})\Gamma_{n,r}]
\right|
&=
O\left(\frac{1}{nh_1^{1/2}}\right).
\label{eq:higher-score-one-factor}
\end{align}
Moreover,
\begin{align}
\left|
\ee_0[\widetilde D_n^2\Gamma_{n,4}]
\right|
&=O(1),
\nonumber\\
\left|
\ee_0[
\widetilde D_n(T_n-\ee_0T_n)\Gamma_{n,4}]
\right|
&=O(h_1^{1/2}),
\nonumber\\
\left|
\ee_0[
\widetilde D_n(b_{n0}-\ee_0b_{n0})\Gamma_{n,4}]
\right|
&=
O\left(\frac{1}{nh_1}\right),
\nonumber\\
\left|
\ee_0[
(T_n-\ee_0T_n)^2\Gamma_{n,4}]
\right|
&=O(h_1),
\nonumber\\
\left|
\ee_0[
(T_n-\ee_0T_n)(b_{n0}-\ee_0b_{n0})\Gamma_{n,4}]
\right|
&=O(n^{-1}),
\nonumber\\
\left|
\ee_0[
(b_{n0}-\ee_0b_{n0})^2\Gamma_{n,4}]
\right|
&=
O\left(\frac{1}{n^2h_1}\right).
\label{eq:higher-score-fourth-components}
\end{align}
Consequently, for \(r=2,3\),
\begin{equation}
\ee_0[U_n\Gamma_{n,r}]
=
\ee_0[\widetilde D_n\Gamma_{n,r}]
+o(1)
=
O(1),
\label{eq:higher-score-reduction-linear}
\end{equation}
and
\begin{equation}
\ee_0[U_n^2\Gamma_{n,4}]
=
\ee_0[\widetilde D_n^2\Gamma_{n,4}]
+o(1)
=
O(1).
\label{eq:higher-score-reduction-fourth}
\end{equation}
\end{lmm}

\begin{proof}
Recall the decomposition
\eqref{eq:U-canonical-decomp},
\[
U_n
=
\widetilde D_n
+
\frac{2(n-1)}{n-2}(T_n-\ee_0T_n)
+
(b_{n0}-\ee_0b_{n0}).
\]
Let
\[
A^{\mathrm c}=(a_{ij}^{\mathrm c}),
\qquad
B^{\mathrm c}=(b_{ij}^{\mathrm c}),
\]
where the canonically centred matrices are defined in
Lemma~\ref{lem:null-canonical-centering-bounds}. Under \(\rho=0\),
after ordering the observations according to their \(X\)-ranks, the
relative \(Y\)-rank vector \(\pi\) is uniform and
\[
\widetilde D_n
=
\frac{1}{n^2h_1}
\sum_{i\ne j}A^{\mathrm c}_{ij}B^{\mathrm c}_{\pi_i\pi_j}.
\]
The same lemma gives
\[
T_n-\ee_0T_n
=
\frac{1}{n^2h_1}
\sum_{i=1}^n x_i y_{\pi_i},
\qquad
b_{n0}-\ee_0b_{n0}
=
\frac{1}{n^2h_1}
\sum_{i=1}^n d_i z_{\pi_i},
\]
together with the norm estimates
\eqref{lochd12}--\eqref{lochd16b}.

Equivalently, in the original observation coordinates,
\begin{align*}
\widetilde D_n
&=
\frac{1}{n^2h_1}
\left\langle A^{\mathrm c}[R^x],B^{\mathrm c}[R^y]\right\rangle_{\mathrm F},\\
T_n-\ee_0T_n
&=
\frac{1}{n^2h_1}
\left\langle x[R^x],y[R^y]\right\rangle,\\
b_{n0}-\ee_0b_{n0}
&=
\frac{1}{n^2h_1}
\left\langle d[R^x],z[R^y]\right\rangle .
\end{align*}

For \(r=2,3\), Lemma~\ref{lem:fixed-order-rank-score},
\eqref{lochd12}, and \eqref{lochd14} give
\begin{align*}
\left|
\ee_0[\widetilde D_n\Gamma_{n,r}]
\right|
\le
\frac{c_r}{n^2h_1}
\|A^{\mathrm c}\|_{\mathrm{op}}
\|B^{\mathrm c}\|_{\mathrm{op}}
=
O\left(
\frac{1}{n^2h_1}(nh_1)n
\right)
=
O(1).
\end{align*}
By the single-inner-product bound in
Lemma~\ref{lem:mixed-feature-bound} and
\eqref{eq:xy-norms-higher},
\begin{align*}
\left|
\ee_0[
(T_n-\ee_0T_n)\Gamma_{n,r}]
\right|
\le
\frac{c_r}{n^2h_1}
\|x\|_2\|y\|_2
=
O\left(
\frac{1}{n^2h_1}
n^{3/2}h_1^{3/2}n^{1/2}
\right)
=
O(h_1^{1/2}).
\end{align*}
Similarly, by \eqref{lochd16a} and \eqref{lochd16b},
\begin{align*}
\left|
\ee_0[
(b_{n0}-\ee_0b_{n0})\Gamma_{n,r}]
\right|
&\le
\frac{c_r}{n^2h_1}
\|d\|_2\|z\|_2\\
&=
O\left(
\frac{1}{n^2h_1}
(nh_1)^{1/2}n^{1/2}
\right)
=
O\left(\frac{1}{nh_1^{1/2}}\right).
\end{align*}
This proves \eqref{eq:higher-score-one-factor}.

We next consider the fourth-score terms. Lemma~\ref{lem:fixed-order-rank-score}, with \(p=2\) and \(r=4\), yields
\begin{align*}
\left|
\ee_0[
\widetilde D_n^2\Gamma_{n,4}]
\right|
\le
\left(\frac{1}{n^2h_1}\right)^2
c_4
\|A^{\mathrm c}\|_{\mathrm{op}}^2
\|B^{\mathrm c}\|_{\mathrm{op}}^2
=
O\left[
\frac{1}{n^4h_1^2}
(nh_1)^2n^2
\right]
=
O(1).
\end{align*}

The generic two-inner-product bound in
Lemma~\ref{lem:mixed-feature-bound} gives
\[
\left|
\ee_0[
\widetilde D_n(T_n-\ee_0T_n)\Gamma_{n,4}]
\right|
=
O(1).
\]
For the present fourth-score calculation, however, this bound can be
sharpened. Indeed,
\[
\widetilde D_n
=
\frac{1}{n^2h_1}
Q(A^{\mathrm c},B^{\mathrm c};\pi),
\qquad
T_n-\ee_0T_n
=
\frac{1}{n^2h_1}L(x,y;\pi),
\]
and therefore Lemma~\ref{lem:Q-L-fourth-score-small} gives
\begin{align*}
\left|
\ee_0[
\widetilde D_n(T_n-\ee_0T_n)\Gamma_{n,4}]
\right|
=
\frac{1}{n^4h_1^2}
\left|
\ee_\pi[
Q(A^{\mathrm c},B^{\mathrm c};\pi)
L(x,y;\pi)
\Gamma_{n,4}(\pi)]
\right|
=
O(h_1^{1/2}).
\end{align*}

For the remaining fourth-score terms, the two-inner-product bound in
Lemma~\ref{lem:mixed-feature-bound} gives
\begin{align*}
\left|
\ee_0[
\widetilde D_n(b_{n0}-\ee_0b_{n0})\Gamma_{n,4}]
\right|
\le
\left(\frac{1}{n^2h_1}\right)^2
c_4
\|A^{\mathrm c}\|_{\mathrm F}\|d\|_2
\|B^{\mathrm c}\|_{\mathrm F}\|z\|_2
=
O\left(\frac{1}{nh_1}\right),
\end{align*}
\begin{align*}
\left|
\ee_0[
(T_n-\ee_0T_n)^2\Gamma_{n,4}]
\right|
\le
\left(\frac{1}{n^2h_1}\right)^2
c_4
\|x\|_2^2\|y\|_2^2
=
O(h_1),
\end{align*}
\begin{align*}
\left|
\ee_0[
(T_n-\ee_0T_n)
(b_{n0}-\ee_0b_{n0})
\Gamma_{n,4}]
\right|
\le
\left(\frac{1}{n^2h_1}\right)^2
c_4
\|x\|_2\|d\|_2
\|y\|_2\|z\|_2
=
O(n^{-1}),
\end{align*}
and
\begin{align*}
\left|
\ee_0[
(b_{n0}-\ee_0b_{n0})^2
\Gamma_{n,4}]
\right|
&\le
\left(\frac{1}{n^2h_1}\right)^2
c_4
\|d\|_2^2\|z\|_2^2
=
O\left(\frac{1}{n^2h_1}\right).
\end{align*}
This proves \eqref{eq:higher-score-fourth-components}.

Since \(h_1\to0\) and Assumption~\ref{asmp:band} imply
\(nh_1\to\infty\), all terms in
\eqref{eq:higher-score-one-factor} other than the canonical term are
\(o(1)\). Hence, for \(r=2,3\),
\[
\ee_0[U_n\Gamma_{n,r}]
=
\ee_0[\widetilde D_n\Gamma_{n,r}]
+o(1).
\]
Similarly, expanding \(U_n^2\), using
\(2(n-1)/(n-2)=O(1)\), and applying
\eqref{eq:higher-score-fourth-components} gives
\[
\ee_0[U_n^2\Gamma_{n,4}]
=
\ee_0[\widetilde D_n^2\Gamma_{n,4}]
+o(1).
\]
The bounds
\[
\ee_0[\widetilde D_n\Gamma_{n,r}]=O(1),
\qquad
\ee_0[\widetilde D_n^2\Gamma_{n,4}]=O(1)
\]
then prove
\eqref{eq:higher-score-reduction-linear} and
\eqref{eq:higher-score-reduction-fourth}.
\end{proof}
\subsubsection{Canonical fourth-order analysis}
\begin{lmm}
\label{lem:canonical-discrete-regularity}
Assume that \(K\) is bounded, non-negative, compactly supported,
normalized by \(\int K=1\), and of bounded variation. Let $A^{\mathrm c}$ and $B^{\mathrm c}$ be the canonically centred matrices from
Lemma~\ref{lem:null-canonical-centering-bounds}. For a symmetric zero-diagonal matrix \(H=(H_{ij})\), define
\[
        \mathcal V(H)
        :=
        \sum_{i=1}^{n-1}
        \sum_{\substack{1\le j\le n\\
                        j\notin\{i,i+1\}}}
        |H_{i+1,j}-H_{ij}|.
\]
Then
\begin{equation}
\label{eq:A-discrete-variation}
        \mathcal V(A^{\mathrm c})=O(n).
\end{equation}
Moreover,
\begin{equation}
\label{eq:A-basic-l1-bounds}
        \max_{i,j}|a^{\mathrm c}_{ij}|=O(1),
        \qquad
        \max_i\sum_{j=1}^n|a^{\mathrm c}_{ij}|=O(nh_1),
        \qquad
        \sum_{i,j=1}^n|a^{\mathrm c}_{ij}|=O(n^2h_1),
\end{equation}
and
\begin{equation}
\label{eq:B-adjacent-bound}
        \max_{\substack{1\le i\le n-1\\
                        j\notin\{i,i+1\}}}
        |b^{\mathrm c}_{i+1,j}-b^{\mathrm c}_{ij}|
        =
        O(n^{-1}),
        \qquad
        \max_{i,j}|b^{\mathrm c}_{ij}|=O(1).
\end{equation}
\end{lmm}

\begin{proof}[Proof of Lemma~\ref{lem:canonical-discrete-regularity}]
The bounds in \eqref{eq:A-basic-l1-bounds}, as well as
\(\max_{i,j}|b^{\mathrm{c}}_{ij}|=O(1)\), follow from
Lemma~\ref{lem:null-canonical-centering-bounds}. Thus it remains to
prove \eqref{eq:A-discrete-variation} and the first assertion in
\eqref{eq:B-adjacent-bound}.

Extend \(K\) by zero outside its support and let
\(\operatorname{TV}(K)<\infty\) denote its total variation. Write
\[
        K_i(u):=
        K\left(\frac{u-i/n}{h_1}\right).
\]
Then
\[
        a_{ij}
        =
        \frac1{h_1}\int_0^1 K_i(u)K_j(u)\,du.
\]
For fixed \(j\),
\begin{align*}
\sum_{i=1}^{n-1}|a_{i+1,j}-a_{ij}|
&\le
\frac1{h_1}
\int_0^1
|K_j(u)|
\sum_{i=1}^{n-1}
|K_{i+1}(u)-K_i(u)|\,du.
\end{align*}
For each fixed \(u\), the points
\[
        \frac{u-i/n}{h_1},
        \qquad i=1,\ldots,n,
\]
form a monotone sequence. Hence, by the definition of total
variation,
\[
        \sum_{i=1}^{n-1}
        |K_{i+1}(u)-K_i(u)|
        \le
        \operatorname{TV}(K).
\]
Since \(K\ge0\) and \(\int K=1\),
\[
        \int_0^1 |K_j(u)|\,du
        \le
        h_1\int_{\mathbb R}|K(t)|\,dt
        =
        h_1.
\]
Consequently,
\[
        \sum_{i=1}^{n-1}|a_{i+1,j}-a_{ij}|
        \le
        \operatorname{TV}(K),
\]
uniformly in \(j\). Summing over \(j\) gives
\begin{equation}
\label{eq:a-total-discrete-variation}
        \sum_{i=1}^{n-1}\sum_{j=1}^n
        |a_{i+1,j}-a_{ij}|
        =
        O(n).
\end{equation}

Recall
\[
        L_i=\sum_{j\ne i}a_{ij}.
\]
By symmetry of \(a\),
\[
        a_{i+1,i}=a_{i,i+1},
\]
and therefore
\[
        L_{i+1}-L_i
        =
        \sum_{j\notin\{i,i+1\}}
        (a_{i+1,j}-a_{ij}).
\]
Hence \eqref{eq:a-total-discrete-variation} implies
\begin{equation}
\label{eq:L-total-variation}
        \sum_{i=1}^{n-1}|L_{i+1}-L_i|
        =
        O(n).
\end{equation}

For \(j\notin\{i,i+1\}\), the definition
\[
        a_{ij}^{\mathrm c}
        =
        a_{ij}
        -\frac{L_i+L_j}{n-2}
        +\frac{n\bar L}{(n-1)(n-2)}
\]
gives
\[
        a_{i+1,j}^{\mathrm c}-a_{ij}^{\mathrm c}
        =
        a_{i+1,j}-a_{ij}
        -
        \frac{L_{i+1}-L_i}{n-2}.
\]
Thus
\begin{align*}
\mathcal V(A^{\mathrm{c}})
\le
\sum_{i=1}^{n-1}
\sum_{j\notin\{i,i+1\}}
|a_{i+1,j}-a_{ij}|
+
\sum_{i=1}^{n-1}
\frac{n-2}{n-2}
|L_{i+1}-L_i|
=
O(n)
\end{align*}
by \eqref{eq:a-total-discrete-variation} and
\eqref{eq:L-total-variation}. This proves
\eqref{eq:A-discrete-variation}.

We next treat \(B^{\mathrm c}\). Since \(K\ge0\) and \(\int K=1\), \(\bar{K}\) is nondecreasing from \(0\)
to \(1\). For every \(\eta>0\),
\begin{align}
\int_{\rr}|\bar{K}(t)-\bar{K}(t-\eta)|\,dt
&=
\int_{\rr}\{\bar{K}(t)-\bar{K}(t-\eta)\}\,dt
\nonumber\\
&=
\int_{\rr}
\int_{t-\eta}^{t}K(s)\,ds\,dt
=
\eta\int_{\rr}K(s)\,ds
=
\eta.
\label{eq:F-translation-L1}
\end{align}
Consequently, since \(0\le \bar{K}\le1\),
\begin{align*}
|B_{i+1,j}-B_{ij}|
&\le
\int_0^1
\left|
\bar{K}\left(\frac{v-(i+1)/n}{h_2}\right)
-
\bar{K}\left(\frac{v-i/n}{h_2}\right)
\right|\,dv\\
&\le
h_2
\int_{\rr}
\left|
\bar{K}\left(t-\frac1{nh_2}\right)-\bar{K}(t)
\right|\,dt\\
&=
\frac{1}{n},
\end{align*}
where the last equality follows from
\eqref{eq:F-translation-L1}.

Recall
\[
        \bar b_i
        =
        \frac1{n-1}\sum_{j\ne i}B_{ij}.
\]
Again using symmetry,
\[
        B_{i+1,i}=B_{i,i+1},
\]
and hence
\begin{align*}
|\bar b_{i+1}-\bar b_i|
&\le
\frac1{n-1}
\sum_{j\notin\{i,i+1\}}
|B_{i+1,j}-B_{ij}|
\le
\frac1n.
\end{align*}
Since
\[
        b_{ij}
        =
        B_{ij}-\bar b_i-\bar b_j+\bar b,
        \qquad
        y_i=\bar b_i-\bar b,
\]
we have, whenever \(j\notin\{i,i+1\}\),
\[
        |b_{i+1,j}-b_{ij}|
        \le
        \frac2n,
        \qquad
        |y_{i+1}-y_i|
        \le
        \frac1n.
\]
Finally,
\[
        b^{\mathrm c}_{ij}
        =
        b_{ij}-\frac{y_i+y_j}{n-2},
        \qquad i\ne j,
\]
so
\[
        |b^{\mathrm c}_{i+1,j}-b^{\mathrm c}_{ij}|
        \le
        \frac{2}{n}+\frac1{n(n-2)}
        =
        O(n^{-1})
\]
uniformly over \(j\notin\{i,i+1\}\). This proves
\eqref{eq:B-adjacent-bound}.
\end{proof}

\begin{lmm}
\label{lem:canonical-Q-uniform-variance}
Assume the hypotheses of
Lemma~\ref{lem:canonical-discrete-regularity}. For
\[
        Q_n
        :=
        \sum_{i\ne j}
        a^{\mathrm c}_{R_i^xR_j^x}
        b^{\mathrm c}_{R_i^yR_j^y},
\]
we have
\begin{equation}
\label{eq:Q-uniform-variance}
        \sup_{|\rho|<1}
        \var_\rho(Q_n)
        =
        O(n^3).
\end{equation}
Equivalently, after ordering by the \(X\)-ranks,
\[
        \sup_{|\rho|<1}
        \var_\rho\{Q(A^{\mathrm c},B^{\mathrm c};\pi)\}
        =
        O(n^3).
\]
\end{lmm}

\begin{proof}[Proof of Lemma~\ref{lem:canonical-Q-uniform-variance}]
We first state a basic fact about ranks. Given a sample $X_1, \ldots, X_n$, take $X_i$ with rank $r$ and replace it with an i.i.d. copy $X_i^\prime$ with rank $s$. The effect of this replacement on the ranks of the unchanged observations can be captured by a bijection
\[
        \phi_{r,s}:D_r\longrightarrow D_s, \qquad D_r:=\{1,\ldots,n\}\setminus\{r\},
\]
where 
\[
\phi_{r,s}(k)
=
\begin{cases}
k+1, & s\le k<r,\quad s<r,\\
k-1, & r<k\le s,\quad r<s,\\
k,   & \text{otherwise}.
\end{cases}
\]
Then \(\phi_{r,s}\) is the identity except on one interval, on which
all ranks are shifted by exactly one.

If \(H\) is symmetric with zero diagonal, then
\begin{equation}
\label{eq:rank-shift-H}
\sum_{\substack{i,j\in D_r\\ i\ne j}}
|H_{\phi_{r,s}(i)\phi_{r,s}(j)}-H_{ij}|
\le
2\mathcal V(H) + 8n\|H\|_\infty.
\end{equation}
Indeed, by the triangle inequality,
\begin{align*}
|H_{\phi(i)\phi(j)} - H_{ij}|
\leq
|H_{\phi(i)\phi(j)} - H_{i,\phi(j)}|
+
|H_{i\phi(j)} - H_{ij}|.
\end{align*}
Whenever \(\phi(i)\ne i\), the first term is an adjacent difference
in the first coordinate. Except when the intermediate entry is
diagonal, these terms form a subcollection of the summands defining
\(\mathcal V(H)\), each used at most once. The diagonal exception can
occur only for an adjacent pair of ranks and hence for at most
\(O(n)\) ordered pairs; its contribution is bounded by
\(2\|H\|_\infty\) per pair. By symmetry the same argument applies to
the second term. This proves \eqref{eq:rank-shift-H}.

We now replace the first observation $Z_1:=(X_1,Y_1)$ by an independent copy \(Z_1'\), leaving the remaining observations
unchanged, and let \(Q_n'\) denote the corresponding value of \(Q_n\).

Let \(r,r'\) denote the old and new \(X\)-ranks of the first
observation, and let \(s,s'\) denote its old and new \(Y\)-ranks.
For the unchanged observations, their \(X\)-ranks are transformed by $\phi:=\phi_{r,r'}$,
and their \(Y\)-ranks are transformed by $\psi:=\phi_{s,s'}$.

Before replacement, let $\pi$ be the permutation mapping the old $X$-ranks to old $Y$-ranks.

We first consider the contribution to $Q_n$ from pairs of unchanged
observations. Before and after replacement these contributions are,
respectively,
\[
        S
        =
        \sum_{\substack{i,j\in D_r\\i\ne j}}
        a^{\mathrm c}_{ij}b^{\mathrm c}_{\pi(i)\pi(j)}
\]
and
\[
        S'
        =
        \sum_{\substack{i,j\in D_r\\i\ne j}}
        a^{\mathrm c}_{\phi(i)\phi(j)}
        b^{\mathrm c}_{\psi(\pi(i))\psi(\pi(j))}.
\]
Therefore
\begin{align}
|S'-S|
\leq
\|B^{\mathrm c}\|_\infty
\sum_{\substack{i,j\in D_r\\i\ne j}}
|a^{\mathrm c}_{\phi(i)\phi(j)}-a^{\mathrm c}_{ij}|
+
\sum_{\substack{i,j\in D_r\\i\ne j}}
|a^{\mathrm c}_{\phi(i)\phi(j)}|
\left|
b^{\mathrm c}_{\psi(\pi(i))\psi(\pi(j))}
-
b^{\mathrm c}_{\pi(i)\pi(j)}
\right|.
\label{eq:S-prime-minus-S}
\end{align}
By \eqref{eq:rank-shift-H}, applied with \(H=A^{\mathrm c}\),
\[
\sum_{\substack{i,j\in D_r\\ i\ne j}}
\left|
a^{\mathrm c}_{\phi(i)\phi(j)}
-a^{\mathrm c}_{ij}
\right|
\le
2\mathcal V(A^{\mathrm c})
+
8n\|A^{\mathrm c}\|_\infty
=
O(n),
\]
where we used
Lemma~\ref{lem:canonical-discrete-regularity} and
\(\|A^{\mathrm c}\|_\infty=O(1)\).
Since \(\|B^{\mathrm c}\|_\infty=O(1)\), the first term on the
right-hand side of \eqref{eq:S-prime-minus-S} is therefore \(O(n)\).
For the second sum, apply the same two-coordinate telescoping argument
to \(B^{\mathrm c}\). By \eqref{eq:B-adjacent-bound},
\[
\left|
b^{\mathrm c}_{\psi(k)\psi(\ell)}-b^{\mathrm c}_{k\ell}
\right|
=
O(n^{-1})
\]
except possibly when one of the intermediate entries is diagonal.
Such an exception can occur only when \(k\) and \(\ell\) are adjacent
ranks. There are only \(O(n)\) ordered adjacent pairs. Since
\(\pi\) is a bijection, there are therefore only \(O(n)\)
exceptional ordered pairs \((i,j)\).

For the non-exceptional pairs,
\begin{align*}
\sum_{\mathrm{nonexceptional}}
|a^{\mathrm c}_{\phi(i)\phi(j)}|
\left|
b^{\mathrm c}_{\psi(\pi(i))\psi(\pi(j))}
-
b^{\mathrm c}_{\pi(i)\pi(j)}
\right|\leq
\frac{C}{n}
\sum_{i\ne j}|a^{\mathrm c}_{ij}|
=
O(nh_1).
\end{align*}
For the exceptional pairs, using
\(\|A^{\mathrm c}\|_\infty=O(1)\) and \(\|B^{\mathrm c}\|_\infty=O(1)\),
\[
\sum_{\mathrm{exceptional}}
|a^{\mathrm c}_{\phi(i)\phi(j)}|
\left|
b^{\mathrm c}_{\psi(\pi(i))\psi(\pi(j))}
-
b^{\mathrm c}_{\pi(i)\pi(j)}
\right|
=
O(n).
\]
Consequently, $|S'-S|=O(n)$.

It remains to consider pairs involving the replaced observation.
Their total contribution before replacement is bounded by
\begin{align*}
2\|B^{\mathrm c}\|_\infty
\sum_{j\ne r}|a^{\mathrm c}_{rj}|
=
O(nh_1),
\end{align*}
and the same bound holds after replacement. Since \(h_1\to0\), there exists a constant $C$ independent of $\rho$ such that
\begin{equation}
\label{eq:Q-bounded-difference}
        |Q_n-Q_n'|
        \le
        Cn.
\end{equation}
Since $\{(X_i, Y_i)\}_{i=1}^n$ is an i.i.d. sample, the
Efron--Stein inequality and \eqref{eq:Q-bounded-difference} give
\[
\var_\rho(Q_n)
\le
\frac12
\sum_{i=1}^n
\ee_\rho
\left[
(Q_n-Q_n^{(i)})^2
\right]
\le
Cn^3,
\]
where \(Q_n^{(i)}\) denotes the statistic after replacing the
\(i\)-th observation by an independent copy. This finishes the proof of \eqref{eq:Q-uniform-variance}.
\end{proof}

\begin{lmm}
\label{lem:canonical-fourth-score-cancellation}
Suppose that Assumptions~\ref{asmp:kernel} and~\ref{asmp:band} hold. Assume in
addition that \(K\) is of bounded variation, and suppose that, for
some fixed \(\eta>0\),
\[
        nh_1^{2+\eta}\longrightarrow\infty.
\]
Then
\begin{equation}
\label{eq:canonical-fourth-score-cancellation}
\ee_\pi\left[
        Q(A^{\mathrm c},B^{\mathrm c};\pi)^2\Gamma_{n,4}(\pi)
\right]
-
6
\left\{
\ee_\pi\left[
        Q(A^{\mathrm c},B^{\mathrm c};\pi)\Gamma_{n,2}(\pi)
\right]
\right\}^2
=
o(n^4h_1^2).
\end{equation}
\end{lmm}

\begin{proof}[Proof of Lemma~\ref{lem:canonical-fourth-score-cancellation}]
Let
\[
        Q_n =
        \sum_{i\ne j}
        a^{\mathrm c}_{R_i^xR_j^x}
        b^{\mathrm c}_{R_i^yR_j^y} = Q(A^{\mathrm c},B^{\mathrm c};\pi)
        ,
\]
where $\pi$ is the permutation induced by reordering observations by their \(X\)-ranks. For \(|\rho|<1\), define
\[
         W_n(\rho)
        :=
        \var_\rho(Q_n)
        =
        g_n(\rho)-q_n(\rho)^2,\qquad q_n(\rho)
        :=
        \ee_\rho[Q_n],
        \qquad
        g_n(\rho)
        :=
        \ee_\rho[Q_n^2].
\]
We first identify \(W_n^{(4)}(0)\). Since \(A^{\mathrm c}\) and \(B^{\mathrm c}\) have zero
row sums, they also have zero total sums. Under \(\rho=0\), the permutation \(\pi\) is uniform, and hence $ q_n(0)=0$. We also claim that
\begin{equation}
\label{eq:q-prime-zero}
        q_n'(0)=0.
\end{equation}
Indeed, the first conditional likelihood score is
\[
        \Gamma_{n,1}(\pi)
        =
        L(\mu,\mu;\pi)
        =
        \sum_{r=1}^n\mu_r\mu_{\pi_r},
\]
where \(\sum_r\mu_r=0\). Thus
\[
        q_n'(0)
        =
        \ee_\pi[
        Q(A^{\mathrm c},B^{\mathrm c};\pi)L(\mu,\mu;\pi)] = \sum_{r = 1}^n\sum_{i\neq j}a^{\mathrm{c}}_{ij}\mu_r\ee_\pi[b^{\mathrm{c}}_{\pi_i\pi_j}\mu_{\pi_r}].
\]
For fixed \(i\ne j\), consider
\[
        \ee_\pi[
        b^{\mathrm c}_{\pi_i\pi_j}\mu_{\pi_r}].
\]
If \(r=i\), then since \(B^{\mathrm c}\) has zero row sums, we have 
\[
\ee_\pi[
        b^{\mathrm c}_{\pi_i\pi_j}\mu_{\pi_i}]
=
\frac{1}{n(n-1)}
\sum_{i\ne j}b^{\mathrm c}_{ij}\mu_i = \frac{1}{n(n-1)}
\sum_{i = 1}^n \mu_i \sum_{j\neq i} b^{\mathrm c}_{ij}
=
0.
\]
The case \(r=j\) is identical. If
\(r\notin\{i,j\}\), then
\begin{align*}
\ee_\pi[
        b^{\mathrm c}_{\pi_i\pi_j}\mu_{\pi_r}]
=
\frac{1}{(n)_3}
\sum_{i,j,k}^{*}b^{\mathrm c}_{ij}\mu_k
=
-\frac{1}{(n)_3}
\sum_{i\ne j}b^{\mathrm c}_{ij}(\mu_i+\mu_j)
=
0,
\end{align*}
again by the zero-row-sum property of \(B^{\mathrm c}\). Therefore
\eqref{eq:q-prime-zero} follows.

Because \(Q_n\) is rank measurable, differentiation of the likelihood
at zero gives, for every fixed \(r\),
\begin{equation}
\label{eq:q-g-score-derivatives}
        q_n^{(r)}(0)
        =
        \ee_\pi[
        Q(A^{\mathrm c},B^{\mathrm c};\pi)\Gamma_{n,r}(\pi)],
        \qquad
        g_n^{(r)}(0)
        =
        \ee_\pi[
        Q(A^{\mathrm c},B^{\mathrm c};\pi)^2\Gamma_{n,r}(\pi)].
\end{equation}
Since
\[
        (q_n^2)^{(4)}
        =
        2q_nq_n^{(4)}
        +
        8q_n'q_n^{(3)}
        +
        6(q_n'')^2,
\]
the identities \(q_n(0)=q_n'(0)=0\) yield
\begin{align*}
W_n^{(4)}(0)
&=
\ee_\pi[
        Q(A^{\mathrm c},B^{\mathrm c};\pi)^2\Gamma_{n,4}(\pi)]
-
6
\left\{
\ee_\pi[
        Q(A^{\mathrm c},B^{\mathrm c};\pi)\Gamma_{n,2}(\pi)]
\right\}^2.
\end{align*}
Thus to show \eqref{eq:canonical-fourth-score-cancellation}, it remains to show
\begin{equation*}
        W_n^{(4)}(0)
        =
        o(n^4h_1^2).
\end{equation*}
Let \(\mathcal H_n\) be the Hilbert space of real \(n\times n\)
matrices equipped with the Frobenius inner product. Recall that
\[
        Q_n
        =
        \left\langle
        A^{\mathrm c}[R^x],
        B^{\mathrm c}[R^y]
        \right\rangle_{\mathcal H_n}.
\]
Since \(R^x\) and \(R^y\) are permutations of
\(\{1,\ldots,n\}\), deterministically
\[
        \|A^{\mathrm c}[R^x]\|_{\mathcal H_n}
        =
        \|A^{\mathrm c}\|_{\mathrm F},
        \qquad
        \|B^{\mathrm c}[R^y]\|_{\mathcal H_n}
        =
        \|B^{\mathrm c}\|_{\mathrm F}.
\]
By Proposition~\ref{prop:hilbert_mehler},
\[
        q_n(\rho)
        =
        \sum_{r=0}^{\infty}\alpha_{n,r}\rho^r
\]
with
\[
        \sum_{r=0}^{\infty}|\alpha_{n,r}|
        \le
        \|A^{\mathrm c}\|_{\mathrm F}
        \|B^{\mathrm c}\|_{\mathrm F}.
\]
The canonical norm bounds give
\[
        \|A^{\mathrm c}\|_{\mathrm F}
        =
        O(nh_1^{1/2}),
        \qquad
        \|B^{\mathrm c}\|_{\mathrm F}
        =
        O(n),
\]
and therefore
\begin{equation}
\label{eq:alpha-l1}
        \sum_{r=0}^{\infty}|\alpha_{n,r}|
        =
        O(n^2h_1^{1/2}).
\end{equation}

Similarly,
\[
        Q_n^2
        =
        \left\langle
        A^{\mathrm c}[R^x]\otimes A^{\mathrm c}[R^x],
        B^{\mathrm c}[R^y]\otimes B^{\mathrm c}[R^y]
        \right\rangle_
        {\mathcal H_n\otimes\mathcal H_n}.
\]
Applying Proposition~\ref{prop:hilbert_mehler} in the tensor-product
Hilbert space gives
\[
        g_n(\rho)
        =
        \sum_{r=0}^{\infty}\beta_{n,r}\rho^r
\]
and
\begin{align}
\sum_{r=0}^{\infty}|\beta_{n,r}|
&\le
\|A^{\mathrm c}\|_{\mathrm F}^2
\|B^{\mathrm c}\|_{\mathrm F}^2
=
O(n^4h_1).
\label{eq:beta-l1}
\end{align}
Write
\[
        W_n(\rho)
        =
        \sum_{r=0}^{\infty}\omega_{n,r}\rho^r.
\]
Since
\[
        W_n=g_n-q_n^2,
\]
\eqref{eq:alpha-l1} and \eqref{eq:beta-l1} imply
\begin{align}
\sum_{r=0}^{\infty}|\omega_{n,r}|
&\le
\sum_{r=0}^{\infty}|\beta_{n,r}|
+
\left(
\sum_{r=0}^{\infty}|\alpha_{n,r}|
\right)^2
=
O(n^4h_1).
\label{eq:omega-l1}
\end{align}

We now convert the uniform bound on $W_n(\rho)$ into a bound on
its fourth derivative at the origin.

Fix
\[
        \kappa:=\frac{\eta}{8},
\]
and choose a fixed even integer $R\ge6$ sufficiently large that
\begin{equation}
\label{eq:R-choice-fourth}
        \kappa(R-2)>1.
\end{equation}
The integer \(R\) depends only on \(\eta\), and not on \(n\).

We first sharpen the coefficient bound for all the finitely many
orders \(r\le R\). From \eqref{eq:q-g-score-derivatives},
\[
        r!\alpha_{n,r}
        =
        \ee_\pi[
        Q(A^{\mathrm c},B^{\mathrm c};\pi)
        \Gamma_{n,r}(\pi)],
\]
and
\[
        r!\beta_{n,r}
        =
        \ee_\pi[
        Q(A^{\mathrm c},B^{\mathrm c};\pi)^2
        \Gamma_{n,r}(\pi)].
\]
For every fixed \(r\), Lemma~\ref{lem:fixed-order-rank-score},
together with
\[
        \|A^{\mathrm c}\|_{\mathrm{op}}
        =
        O(nh_1),
        \qquad
        \|B^{\mathrm c}\|_{\mathrm{op}}
        =
        O(n),
\]
gives
\[
        |\alpha_{n,r}|
        =
        O(n^2h_1),
        \qquad
        |\beta_{n,r}|
        =
        O(n^4h_1^2).
\]
Since $q_n(0)=0$, the coefficient of $\rho^r$ in $q_n(\rho)^2$ is
\[
        \sum_{k=1}^{r-1}
        \alpha_{n,k}\alpha_{n,r-k},
\]
and therefore, for every fixed \(r\),
\begin{equation}
\label{eq:omega-fixed-order}
        |\omega_{n,r}|
        =
        O(n^4h_1^2).
\end{equation}
In particular, \eqref{eq:omega-fixed-order} holds uniformly over
the finite set $r=6,8,\ldots,R$. 

Now set $t_n:=h_1^\kappa$ and note that $t_n\to0$. Consider the symmetric fourth finite difference
\[
\Delta_{t_n}^4W_n(0)
:=
W_n(2t_n)-4W_n(t_n)+6W_n(0)
-4W_n(-t_n)+W_n(-2t_n).
\]
Since the odd powers cancel,
\begin{align*}
\Delta_{t_n}^4W_n(0)
&=
24\omega_{n,4}t_n^4
+
\sum_{\substack{6\le r\le R\\ r\ {\rm even}}}
(2^{r+1}-8)\omega_{n,r}t_n^r
+
\mathcal R_{n,R},
\end{align*}
where
\[
        \mathcal R_{n,R}
        :=
        \sum_{\substack{r\ge R+2\\ r\ {\rm even}}}
        (2^{r+1}-8)\omega_{n,r}t_n^r.
\]
Consequently,
\begin{align}
W_n^{(4)}(0)
&=
\frac{\Delta_{t_n}^4W_n(0)}{t_n^4}
-
\sum_{\substack{6\le r\le R\\ r\ {\rm even}}}
(2^{r+1}-8)\omega_{n,r}t_n^{r-4}
-
\frac{\mathcal R_{n,R}}{t_n^4}.
\label{eq:Wfourth-finite-difference-new}
\end{align}

We bound the three terms on the right-hand side separately.

First, Lemma~\ref{lem:canonical-Q-uniform-variance} gives
\[
        \sup_{|\rho|<1}W_n(\rho)
        =
        O(n^3).
\]
Since \(t_n\to0\), \(2t_n<1\) for all sufficiently large \(n\), and
therefore
\[
        |\Delta_{t_n}^4W_n(0)|
        =
        O(n^3).
\]
Hence
\begin{align}
\frac{1}{n^4h_1^2}
\frac{|\Delta_{t_n}^4W_n(0)|}{t_n^4}
=
O\left(
        \frac{1}{nh_1^2t_n^4}
\right)
=
O\left(
        \frac{1}{nh_1^{2+4\kappa}}
\right).
\label{eq:finite-difference-main-small}
\end{align}
Because \(4\kappa=\eta/2<\eta\),
\[
        nh_1^{2+4\kappa}
        =
        \{nh_1^{2+\eta}\}
        h_1^{4\kappa-\eta}
        \longrightarrow\infty,
\]
and thus the quantity in
\eqref{eq:finite-difference-main-small} tends to zero.

Second, by \eqref{eq:omega-fixed-order},
\begin{align*}
\frac{1}{n^4h_1^2}
\left|
\sum_{\substack{6\le r\le R\\ r\ {\rm even}}}
(2^{r+1}-8)\omega_{n,r}t_n^{r-4}
\right|
\le
C
\sum_{\substack{6\le r\le R\\ r\ {\rm even}}}
t_n^{r-4}
=
o(1),
\end{align*}
because \(R\) is fixed and \(t_n\to0\).

Finally, for all sufficiently large \(n\), \(2t_n\le1/2\).
Using \eqref{eq:omega-l1},
\begin{align*}
|\mathcal R_{n,R}|
\le
C
\sum_{r\ge R+2}
2^r|\omega_{n,r}|t_n^r
\le
C_R t_n^{R+2}
\sum_{r\ge R+2}|\omega_{n,r}|
=
O(n^4h_1t_n^{R+2}).
\end{align*}
Therefore
\begin{align*}
\frac{1}{n^4h_1^2}
\frac{|\mathcal R_{n,R}|}{t_n^4}
=
O\left(
        \frac{t_n^{R-2}}{h_1}
\right)
=
O\left(
        h_1^{\kappa(R-2)-1}
\right)
=
o(1),
\end{align*}
where the last step follows from
\eqref{eq:R-choice-fourth}.

Substituting the preceding three estimates into
\eqref{eq:Wfourth-finite-difference-new} gives
\[
        W_n^{(4)}(0)
        =
        o(n^4h_1^2),
\]
which proves
\eqref{eq:canonical-fourth-score-cancellation}.
\end{proof}

\begin{cor}
\label{cor:canonical-fourth-derivative-local}
Assume the hypotheses of
Lemma~\ref{lem:canonical-fourth-score-cancellation}, and define
\[
        W_n^D(\rho)
        :=
        \var_\rho(\widetilde D_n).
\]
Then
\begin{equation}
\label{eq:canonical-D-fourth-zero}
        (W_n^D)^{(4)}(0)=o(1).
\end{equation}
Moreover, there exists a constant $C>0$, independent of $n$, such
that, for all sufficiently large $n$,
\begin{equation}
\label{eq:canonical-D-fourth-local-bound}
        |(W_n^D)^{(4)}(\rho)|
        \le
        |(W_n^D)^{(4)}(0)|
        +
        C|\rho|
        +
        C\frac{\rho^2}{h_1},
        \qquad
        |\rho|\le\frac14.
\end{equation}
Consequently, if
\[
        \delta_n
        :=
        \Bigl(\frac{h_1}{n}\Bigr)^{1/4},
\]
then, for every fixed $C_0>0$,
\begin{equation}
\label{eq:canonical-D-fourth-uniform}
        \sup_{|\rho|\le C_0\delta_n}
        |(W_n^D)^{(4)}(\rho)|
        =
        o(1).
\end{equation}
\end{cor}

\begin{proof}[Proof of Corollary~\ref{cor:canonical-fourth-derivative-local}]
For
\[
        Q_n
        :=
        Q(A^{\mathrm c},B^{\mathrm c};\pi),
\]
write
\[
        \mathcal W_n(\rho)
        :=
        \var_\rho(Q_n)
        =
        \sum_{r=0}^{\infty}\omega_{n,r}\rho^r.
\]
The Hermite-coefficient estimates established in the proof of
Lemma~\ref{lem:canonical-fourth-score-cancellation} give
\begin{equation*}
        \sum_{r=0}^{\infty}|\omega_{n,r}|
        =
        O(n^4h_1),
        \qquad
        |\omega_{n,r}|
        =
        O(n^4h_1^2)
\end{equation*}
for every fixed $r$; see
\eqref{eq:omega-l1} and \eqref{eq:omega-fixed-order}.

Since, for every $\rho$,
\[
        \widetilde D_n
        =
        \frac{1}{n^2h_1}Q_n,
\]
we have
\[
        W_n^D(\rho)
        =
        \frac{\mathcal W_n(\rho)}{n^4h_1^2}
        =
        \sum_{r=0}^{\infty}w_{n,r}\rho^r,
        \qquad
        w_{n,r}
        :=
        \frac{\omega_{n,r}}{n^4h_1^2}.
\]
Hence
\begin{equation}
\label{eq:canonical-D-coefficients-cor}
        \sum_{r=0}^{\infty}|w_{n,r}|
        =
        O(h_1^{-1}),
        \qquad
        |w_{n,r}|=O(1)
\end{equation}
for every fixed \(r\).

By Lemma~\ref{lem:canonical-fourth-score-cancellation},
\[
\begin{aligned}
(W_n^D)^{(4)}(0)
&=
\frac{1}{n^4h_1^2}
\biggl[
\ee_\pi\left[
Q_n^2\Gamma_{n,4}(\pi)
\right]
-
6
\left\{
\ee_\pi\left[
Q_n\Gamma_{n,2}(\pi)
\right]
\right\}^2
\biggr]
\\
&=
o(1),
\end{aligned}
\]
which proves \eqref{eq:canonical-D-fourth-zero}.

Differentiating the power series four times gives
\[
(W_n^D)^{(4)}(\rho)
=
(W_n^D)^{(4)}(0)
+
5!w_{n,5}\rho
+
\sum_{r=6}^{\infty}
(r)_4w_{n,r}\rho^{r-4},
\]
where
\[
        (r)_4:=r(r-1)(r-2)(r-3).
\]
By \eqref{eq:canonical-D-coefficients-cor},
\[
        |w_{n,5}|=O(1).
\]
Furthermore, for \(|\rho|\le1/4\),
\[
        \sup_{r\ge6}(r)_4|\rho|^{r-6}
        \le C,
\]
and therefore
\begin{align*}
\left|
\sum_{r=6}^{\infty}
(r)_4w_{n,r}\rho^{r-4}
\right|
&\le
C\rho^2
\sum_{r=6}^{\infty}|w_{n,r}|
\\
&=
O\left(\frac{\rho^2}{h_1}\right).
\end{align*}
This proves \eqref{eq:canonical-D-fourth-local-bound}.

Finally, for every fixed \(C_0>0\), \(C_0\delta_n\le1/4\) for all
sufficiently large \(n\). Thus
\begin{align*}
\sup_{|\rho|\le C_0\delta_n}
|(W_n^D)^{(4)}(\rho)|
&\le
|(W_n^D)^{(4)}(0)|
+
CC_0\delta_n
+
CC_0^2\frac{\delta_n^2}{h_1}.
\end{align*}
Now
\[
        (W_n^D)^{(4)}(0)=o(1),
        \qquad
        \delta_n\to0,
\]
and
\[
        \frac{\delta_n^2}{h_1}
        =
        \frac{1}{\sqrt{nh_1}}
        \longrightarrow0.
\]
Hence \eqref{eq:canonical-D-fourth-uniform} follows.
\end{proof}

\subsubsection{Higher variance derivatives}
\begin{lmm}
\label{lem:Vfourth-zero}
Suppose that Assumptions~\ref{asmp:kernel} and~\ref{asmp:band} hold and $K$ is of bounded variation, and that, for some fixed $\eta>0$,
\[
        nh_1^{2+\eta}\longrightarrow\infty.
\]
Then
\begin{equation*}
        V_n^{(4)}(0)=o(1).
\end{equation*}
\end{lmm}

\begin{proof}[Proof of Lemma~\ref{lem:Vfourth-zero}]

Taking $r=2$ in \eqref{eq:mean-rank-score-r} gives $M_n''(0)=\ee_0[U_n\Gamma_{n,2}]$. Applying \eqref{eq:higher-score-reduction-linear} and the first bound in \eqref{eq:higher-score-one-factor} therefore yields
\begin{equation*}
M_n''(0)
=
\ee_0[\widetilde D_n\Gamma_{n,2}]
+o(1),
\qquad
\ee_0[\widetilde D_n\Gamma_{n,2}]
=
O(1).
\end{equation*}
Consequently,
\begin{equation}
\label{eq:Msecond-square-canonical-reduction}
\{M_n''(0)\}^2
=
\left\{
\ee_0[\widetilde D_n\Gamma_{n,2}]
\right\}^2
+o(1).
\end{equation}
Lemma~\ref{lem:higher-score-contractions} also gives
\begin{equation}
\label{eq:U2-fourth-canonical-reduction}
\ee_0[U_n^2\Gamma_{n,4}]
=
\ee_0[\widetilde D_n^2\Gamma_{n,4}]
+o(1).
\end{equation}

Under \(\rho=0\), after ordering the observations according to their
\(X\)-ranks,
\[
\widetilde D_n
=
\frac{1}{n^2h_1}
Q(A^{\mathrm c},B^{\mathrm c};\pi).
\]
Hence Lemma~\ref{lem:canonical-fourth-score-cancellation} yields
\begin{align*}
&
\ee_0[\widetilde D_n^2\Gamma_{n,4}]
-
6
\left\{
\ee_0[\widetilde D_n\Gamma_{n,2}]
\right\}^2\\
&\qquad=
\frac{1}{n^4h_1^2}
\biggl[
\ee_\pi[
Q(A^{\mathrm c},B^{\mathrm c};\pi)^2
\Gamma_{n,4}(\pi)]
-
6
\left\{
\ee_\pi[
Q(A^{\mathrm c},B^{\mathrm c};\pi)
\Gamma_{n,2}(\pi)]
\right\}^2
\biggr]\\
&\qquad=o(1).
\end{align*}
Combining this with
\eqref{eq:Msecond-square-canonical-reduction} and
\eqref{eq:U2-fourth-canonical-reduction} gives
\begin{equation*}
\ee_0[U_n^2\Gamma_{n,4}]
-
6\{M_n''(0)\}^2
=
o(1).
\end{equation*}

It remains to control the term involving the first and third mean
derivatives. By
\eqref{eq:first-mean-derivative-bound-centered},
\[
M_n'(0)=O(n^{-1/2}),
\]
while \eqref{eq:mean-rank-score-r} and
Lemma~\ref{lem:higher-score-contractions} give
\[
M_n'''(0)
=
\ee_0[U_n\Gamma_{n,3}]
=
O(1).
\]
Therefore
\[
M_n'(0)M_n'''(0)
=
O(n^{-1/2})
=
o(1).
\]

Finally, the exact identity \eqref{eq:Vfourth-exact} gives
\begin{align*}
V_n^{(4)}(0)
&=
\ee_0[U_n^2\Gamma_{n,4}]
-
8M_n'(0)M_n'''(0)
-
6\{M_n''(0)\}^2 =
o(1),
\end{align*}
which proves the assertion.
\end{proof}

\begin{lmm}
\label{lem:Vfourth-uniform}
Suppose that Assumptions~\ref{asmp:kernel} and~\ref{asmp:band} hold and
\[
        V_n^{(4)}(0)=o(1).
\]
Then there exists a constant $C>0$, independent of $n$, such that,
for all sufficiently large $n$,
\begin{equation}
\label{eq:Vfourth-local-bound}
        |V_n^{(4)}(\rho)|
        \le
        |V_n^{(4)}(0)|
        +
        C|\rho|
        +
        C\frac{\rho^2}{h_1},
        \qquad
        |\rho|\le\frac{1}{4}.
\end{equation}
Consequently, for every fixed $C_0>0$,
\begin{equation}
\label{eq:Vfourth-uniform-local}
        \sup_{|\rho|\le C_0(h_1/n)^{1/4}}
        |V_n^{(4)}(\rho)|
        =
        o(1).
\end{equation}
\end{lmm}

\begin{proof}
We first show that
\begin{equation}
\label{eq:Vfifth-zero-bound}
        V_n^{(5)}(0)=O(1).
\end{equation}

Recall that
\[
        M_n(\rho)=\ee_\rho[\hat\tau_n^2],
        \qquad
        R_n(\rho)=\ee_\rho[(\hat\tau_n^2)^2],
        \qquad
        V_n(\rho)=R_n(\rho)-M_n(\rho)^2.
\]
Since
\[
        M_n^{(r)}(0)
        =
        \ee_0[U_n\Gamma_{n,r}],
\]
the same argument as in the proof of
Lemma~\ref{lem:higher-score-contractions}, using
Lemma~\ref{lem:fixed-order-rank-score} for the fixed orders $r=2,3,4$, gives
\begin{equation}
\label{eq:mean-derivatives-up-to-four}
        M_n''(0)=O(1),
        \qquad
        M_n'''(0)=O(1),
        \qquad
        M_n^{(4)}(0)=O(1).
\end{equation}
Moreover, applying the same quadratic-feature argument with $p=2$ and $r=5$ yields
\begin{equation}
\label{eq:U2-Gamma5-bound}
        \left|
        \ee_0[U_n^2\Gamma_{n,5}]
        \right|
        =
        O(1).
\end{equation}
Differentiating
\[
        V_n(\rho)=R_n(\rho)-M_n(\rho)^2
\]
five times at zero, and using $\ee_0\Gamma_{n,5}=0$, gives
\begin{align*}
V_n^{(5)}(0)
&=
\ee_0[U_n^2\Gamma_{n,5}]
-
10M_n'(0)M_n^{(4)}(0)
-
20M_n''(0)M_n'''(0).
\end{align*}
Since \(M_n'(0)=O(n^{-1/2})\), equations
\eqref{eq:mean-derivatives-up-to-four} and
\eqref{eq:U2-Gamma5-bound} prove
\eqref{eq:Vfifth-zero-bound}.

We next derive an absolutely summable Hermite expansion of the
variance. Let \(\mathcal H_n\) denote the Hilbert space of real
\(n\times n\) matrices equipped with the Frobenius inner product, and
define the rank-matrix-valued functions
\[
        \mathcal A_n(X)
        :=
        \bigl(a_{R_i^xR_j^x}\bigr)_{1\le i,j\le n},
        \qquad
        \mathcal B_n(Y)
        :=
        \bigl(B_{R_i^yR_j^y}\bigr)_{1\le i,j\le n}.
\]
By \eqref{eq:tau-square-rank-representation},
\begin{equation}
\label{eq:tau-Hilbert-representation}
        \hat\tau_n^2
        =
        \frac{1}{n^2h_1}
        \left\langle
        \mathcal A_n(X),
        \mathcal B_n(Y)
        \right\rangle_{\mathcal H_n}.
\end{equation}
Since \(R^x\) and \(R^y\) are permutations of
\(\{1,\ldots,n\}\), their Frobenius norms are deterministic:
\[
        \|\mathcal A_n(X)\|_{\mathcal H_n}
        =
        \|a\|_{\mathrm F},
        \qquad
        \|\mathcal B_n(Y)\|_{\mathcal H_n}
        =
        \|B\|_{\mathrm F}.
\]
Hence Proposition~\ref{prop:hilbert_mehler} applied to
\eqref{eq:tau-Hilbert-representation} gives
\begin{equation*}
        M_n(\rho)
        =
        \sum_{q=0}^{\infty}m_{n,q}\rho^q,
\end{equation*}
where, by Parseval's identity,
\begin{align}
\sum_{q=0}^{\infty}|m_{n,q}|
&\le
\frac{1}{n^2h_1}
\|\mathcal A_n\|_{L^2(\gamma_n;\mathcal H_n)}
\|\mathcal B_n\|_{L^2(\gamma_n;\mathcal H_n)}
\nonumber\\
&=
\frac{1}{n^2h_1}
\|a\|_{\mathrm F}\|B\|_{\mathrm F}
=
O(h_1^{-1/2}),
\label{eq:M-coefficient-l1}
\end{align}
where the last equality follows from
\eqref{eq:kernel-weight-a-norms} and
\eqref{eq:kernel-weight-B-norms}.

For the second moment, note that
\[
(\hat\tau_n^2)^2
=
\frac{1}{n^4h_1^2}
\left\langle
\mathcal A_n(X)\otimes\mathcal A_n(X),
\mathcal B_n(Y)\otimes\mathcal B_n(Y)
\right\rangle_{\mathcal H_n\otimes\mathcal H_n}.
\]
Applying Proposition~\ref{prop:hilbert_mehler} in the tensor-product
Hilbert space gives
\begin{equation*}
        R_n(\rho)
        =
        \sum_{q=0}^{\infty}r_{n,q}\rho^q,
\end{equation*}
with
\begin{align}
\sum_{q=0}^{\infty}|r_{n,q}|
&\le
\frac{1}{n^4h_1^2}
\|\mathcal A_n\otimes\mathcal A_n\|_
{L^2(\gamma_n;\mathcal H_n\otimes\mathcal H_n)}
\|\mathcal B_n\otimes\mathcal B_n\|_
{L^2(\gamma_n;\mathcal H_n\otimes\mathcal H_n)}
\nonumber\\
&=
\frac{1}{n^4h_1^2}
\|a\|_{\mathrm F}^2
\|B\|_{\mathrm F}^2
=
O(h_1^{-1}).
\label{eq:R-coefficient-l1}
\end{align}
Here we used
\[
        \|\mathcal A_n(X)\otimes\mathcal A_n(X)\|
        =
        \|\mathcal A_n(X)\|_{\mathcal H_n}^2
        =
        \|a\|_{\mathrm F}^2,
\]
and similarly for \(\mathcal B_n\).

Now write
\[
        M_n(\rho)^2
        =
        \sum_{q=0}^{\infty}s_{n,q}\rho^q,
        \qquad
        s_{n,q}
        :=
        \sum_{j=0}^q
        m_{n,j}m_{n,q-j}.
\]
The absolute summability in \eqref{eq:M-coefficient-l1} implies
\begin{align}
        \sum_{q=0}^{\infty}|s_{n,q}|
        &\le
        \left(
        \sum_{q=0}^{\infty}|m_{n,q}|
        \right)^2
        =
        O(h_1^{-1}).
\label{eq:M2-coefficient-l1}
\end{align}
Since
\[
        V_n(\rho)
        =
        R_n(\rho)-M_n(\rho)^2,
\]
we may therefore write
\begin{equation}
\label{eq:V-Hermite-series}
        V_n(\rho)
        =
        \sum_{q=0}^{\infty}v_{n,q}\rho^q,
        \qquad
        v_{n,q}:=r_{n,q}-s_{n,q},
\end{equation}
and \eqref{eq:R-coefficient-l1} and
\eqref{eq:M2-coefficient-l1} yield
\begin{equation}
\label{eq:V-coefficient-l1}
        \sum_{q=0}^{\infty}|v_{n,q}|
        =
        O(h_1^{-1}).
\end{equation}
Furthermore, $4!v_{n,4} = V_n^{(4)}(0)$ and, by \eqref{eq:Vfifth-zero-bound}, $5!v_{n,5} = V_n^{(5)}(0) = O(1)$.

Differentiating \eqref{eq:V-Hermite-series} four times gives
\begin{equation*}
V_n^{(4)}(\rho)
=
V_n^{(4)}(0)
+
5!v_{n,5}\rho
+
\sum_{q=6}^{\infty}
(q)_4v_{n,q}\rho^{q-4},
\end{equation*}
where $(q)_4:=q(q-1)(q-2)(q-3)$.
For \(|\rho|\le1/4\),
\begin{align*}
\left|
\sum_{q=6}^{\infty}
(q)_4v_{n,q}\rho^{q-4}
\right|
\le
|\rho|^2
\sum_{q=6}^{\infty}
(q)_4|v_{n,q}||\rho|^{q-6}
\le
C|\rho|^2
\sum_{q=6}^{\infty}|v_{n,q}|
\le
C\frac{\rho^2}{h_1},
\end{align*}
because
\[
        \sup_{q\ge6}(q)_4(1/4)^{q-6}<\infty
\]
and \eqref{eq:V-coefficient-l1} holds. This proves
\eqref{eq:Vfourth-local-bound}.

Finally, set
\[
        \delta_n
        :=
        \Bigl(\frac{h_1}{n}\Bigr)^{1/4}.
\]
For every fixed \(C_0>0\), \(C_0\delta_n\le1/4\) for all sufficiently
large \(n\). Hence \eqref{eq:Vfourth-local-bound} gives
\begin{align*}
\sup_{|\rho|\le C_0\delta_n}
|V_n^{(4)}(\rho)|
&\le
|V_n^{(4)}(0)|
+
CC_0\delta_n
+
CC_0^2\frac{\delta_n^2}{h_1}.
\end{align*}
Since
\[
        |V_n^{(4)}(0)|=o(1),
        \qquad
        \delta_n\to0,
        \qquad
        \frac{\delta_n^2}{h_1}
        =
        \frac{1}{\sqrt{nh_1}}
        \longrightarrow0,
\]
we obtain \eqref{eq:Vfourth-uniform-local}.
\end{proof}

\begin{lmm}
\label{lem:Vthird-zero}
Suppose that Assumptions~\ref{asmp:kernel} and~\ref{asmp:band} hold, that
$K$ is of bounded variation, and that, for some fixed $\eta\in(0,1)$,
\[
    nh_1^{2+\eta}\longrightarrow\infty.
\]
Set
\[
    \delta_n
    :=
    \Bigl(\frac{h_1}{n}\Bigr)^{1/4}.
\]
Then
\begin{equation}
\label{eq:Vthird-zero-small}
        V_n^{(3)}(0)
        =
        o(\delta_n).
\end{equation}
\end{lmm}

\begin{proof}[Proof of Lemma~\ref{lem:Vthird-zero}]
Recall that
\[
        V_n(\rho)
        =
        \var_\rho(\hat\tau_n^2).
\]
By \eqref{eq:null-variance-order-centered},
Lemma~\ref{lem:Vprime-bound-centered}, and
Lemma~\ref{lem:Vsecond-bound-centered},
\begin{equation*}
        V_n(0)=O(\delta_n^4),
        \qquad
        V_n'(0)
        =
        O\left(\frac{h_1^{1/2}}{n}\right),
        \qquad
        V_n''(0)=O(n^{-1}).
\end{equation*}

Moreover, by Lemma~\ref{lem:Vfourth-zero},
\begin{equation*}
        \varepsilon_n
        :=
        |V_n^{(4)}(0)|
        \rightarrow 0.
\end{equation*}
We now choose a slowly diverging sequence
\begin{equation*}
s_n
:=
\min\left\{
        \delta_n^{-1/4},
        (\varepsilon_n+\delta_n)^{-1/2},
        (nh_1)^{1/12}
\right\}.
\end{equation*}
It is guaranteed that $s_n\rightarrow\infty$, since $\delta_n\to0$, $\varepsilon_n\to0$, and $nh_1\to\infty$.

Moreover,
\begin{equation}
\label{eq:sn-properties-third-new}
        s_n\delta_n\longrightarrow0,
        \qquad
        s_n\varepsilon_n\longrightarrow0,
        \qquad
        s_n^2\delta_n\longrightarrow0,
        \qquad
        \frac{s_n^3}{(nh_1)^{1/2}}
        \longrightarrow0.
\end{equation}
Since $s_n\delta_n\to0$, the interval $|\rho|\le s_n\delta_n$ is contained in $[-1/4,1/4]$ for all sufficiently large $n$.
Therefore, by \eqref{eq:Vfourth-local-bound},
\begin{align*}
s_n
\sup_{|\rho|\le s_n\delta_n}
|V_n^{(4)}(\rho)|
&\le
s_n\varepsilon_n
+
Cs_n^2\delta_n
+
C\frac{s_n^3\delta_n^2}{h_1}.
\end{align*}

But
\[
        \frac{\delta_n^2}{h_1}
        =
        \frac{1}{(nh_1)^{1/2}},
\]
and hence \eqref{eq:sn-properties-third-new} implies
\begin{equation}
\label{eq:fourth-small-enlarged-window}
        s_n
        \sup_{|\rho|\le s_n\delta_n}
        |V_n^{(4)}(\rho)|
        \longrightarrow 0.
\end{equation}
If $V_n^{(3)}(0)=0$, there is nothing to prove. Otherwise define
\begin{equation*}
        t_n
        :=
        -\operatorname{sgn}(V_n^{(3)}(0))s_n\delta_n.
\end{equation*}
Since $s_n\delta_n\to0$, we have $t_n\to0$, and hence
$t_n\in(-1,1)$ for all sufficiently large $n$.

Taylor's theorem gives, for some $\xi_n$ lying between $0$ and $t_n$,
\begin{align}
V_n(t_n)
&=
V_n(0)
+t_nV_n'(0)
+\frac{t_n^2}{2}V_n''(0)
+\frac{t_n^3}{6}V_n^{(3)}(0)
+\frac{t_n^4}{24}V_n^{(4)}(\xi_n).
\label{eq:Taylor-third-small}
\end{align}
By the definition of \(t_n\),
\[
        t_n^3V_n^{(3)}(0)
        =
        -|t_n|^3|V_n^{(3)}(0)|.
\]
Also, $V_n(t_n)\ge0$, since $V_n(t_n)$ is a variance. Consequently,
\eqref{eq:Taylor-third-small} implies
\begin{align*}
\frac{|t_n|^3}{6}|V_n^{(3)}(0)|
\le
V_n(0)
+
|t_n||V_n'(0)|
+
\frac{|t_n|^2}{2}|V_n''(0)|
+
\frac{|t_n|^4}{24}
\sup_{|\rho|\le|t_n|}
|V_n^{(4)}(\rho)|.
\end{align*}
Dividing by $|t_n|^3=(s_n\delta_n)^3$, we obtain
\begin{align}
|V_n^{(3)}(0)|
&\le
\frac{6V_n(0)}{s_n^3\delta_n^3}
+
\frac{6|V_n'(0)|}{s_n^2\delta_n^2}
+
\frac{3|V_n''(0)|}{s_n\delta_n}
+
\frac{s_n\delta_n}{4}
\sup_{|\rho|\le s_n\delta_n}
|V_n^{(4)}(\rho)|.
\label{eq:third-small-four-terms}
\end{align}

We show that every term on the right-hand side is $o(\delta_n)$. For the first term, using $V_n(0)=O(\delta_n^4)$,
\[
        \frac{V_n(0)}
        {s_n^3\delta_n^3}
        =
        O\left(\frac{\delta_n}{s_n^3}\right)
        =
        o(\delta_n),
\]
because $s_n\to\infty$.

For the second term, the first-derivative estimate gives
\begin{align*}
\frac{|V_n'(0)|}{s_n^2\delta_n^2}
&=
O\left(
\frac{h_1^{1/2}}
{n s_n^2\delta_n^2}
\right)
=
O\left(
\frac{\delta_n}
{s_n^2(nh_1)^{1/4}}
\right)
=
o(\delta_n).
\end{align*}
For the third term,
\begin{align*}
\frac{|V_n''(0)|}{s_n\delta_n}
&=
O\left(
\frac{1}{ns_n\delta_n}
\right)
=
O\left(
\frac{\delta_n}{s_n(nh_1)^{1/2}}
\right)
=
o(\delta_n).
\end{align*}
Finally, by
\eqref{eq:fourth-small-enlarged-window},
\begin{align*}
s_n\delta_n
\sup_{|\rho|\le s_n\delta_n}
|V_n^{(4)}(\rho)|
&= \delta_n
\left\{
s_n
\sup_{|\rho|\le s_n\delta_n}
|V_n^{(4)}(\rho)|
\right\}=
o(\delta_n).
\end{align*}
Thus every term on the right-hand side of
\eqref{eq:third-small-four-terms} is \(o(\delta_n)\). Therefore
\[
        |V_n^{(3)}(0)|
        =
        o(\delta_n),
\]
which proves \eqref{eq:Vthird-zero-small}.
\end{proof}

\subsection{Central limit theorem under local alternatives}\label{subsec:gaussian-local-clt}

\subsubsection{Proof of Theorem~\ref{thm:Gaussian-local-CLT}}
By Lemma~\ref{lem:local-linear-reduction}, we have 
\[
\sqrt{\frac{n}{h_1}}
\left[
\hat\tau_n^2-\ee_{\rho_n}\hat\tau_n^2
-\frac{2(n-1)}{n-2}
\{T_n-\ee_{\rho_n}T_n\}
\right]
\stackrel{\pp_{\rho_n}}{\rightarrow}0,
\]
where $T_n$ is defined in \eqref{lochd4}.
On the other hand, Lemma~\ref{lem:Tn-local-CLT}  and Slutsky's theorem give under $\pp_{\rho_n}$
\[
\frac{2(n-1)}{n-2}
\sqrt{\frac{n}{h_1}}
\{T_n-\ee_{\rho_n}T_n\}
\stackrel{\mathcal D}{\rightarrow} \mathcal{N}(0,4\sigma_0^2).
\]
Combining the preceding two displays proves
the assertion.
\subsubsection{Auxiliary results for the local central limit theorem}
\begin{lmm}
\label{lem:local-linear-reduction}
Suppose that Assumptions~\ref{asmp:kernel} and~\ref{asmp:band} hold,
that
$K$ is of bounded variation, and that, for some fixed $\eta\in(0,1)$,
\[
    nh_1^{2+\eta}\longrightarrow\infty.
\]
Let
\[
        \rho_n
        =
        O\left\{
        \Bigl(\frac{h_1}{n}\Bigr)^{1/4}
        \right\}.
\]
Then, under any bivariate distribution with continuous marginals and
Gaussian copula \(C_{\rho_n}\),
\begin{equation}
\label{eq:Dtilde-local-negligible}
        \var_{\rho_n}(\widetilde D_n)
        =
        o\left(\frac{h_1}{n}\right),
\end{equation}
and
\begin{equation}
\label{eq:b-local-negligible}
        \var_{\rho_n}(b_{n0})
        =
        o\left(\frac{h_1}{n}\right).
\end{equation}
Consequently,
\begin{equation}
\label{eq:tau-T-local-reduction}
\frac{n}{h_1}\ee_{\rho_n}
\left(\left[
\hat\tau_n^2-\ee_{\rho_n}\hat\tau_n^2
-
\frac{2(n-1)}{n-2}
\{T_n-\ee_{\rho_n}T_n\}
\right]^2\right)
\longrightarrow 0.
\end{equation}
Moreover,
\begin{equation}
\label{eq:Tn-local-variance}
        \var_{\rho_n}(T_n)
        =
        \sigma_0^2\frac{h_1}{n}
        +
        o\left(\frac{h_1}{n}\right).
\end{equation}
\end{lmm}

\begin{proof}
Since all statistics appearing in the lemma are rank measurable, their
joint distribution depends only on the copula. Hence it suffices to
work with standard bivariate normal observations having correlation $\rho$.

Set
\[
        \delta_n
        :=
        \Bigl(\frac{h_1}{n}\Bigr)^{1/4}.
\]
We first prove \eqref{eq:Dtilde-local-negligible}. Recall that
\[
        W_n^D(\rho)
        =
        \var_\rho(\widetilde D_n).
\]
Under \(\rho=0\),
\[
        \widetilde D_n
        =
        \frac{1}{n^2h_1}
        Q(A^{\mathrm c},B^{\mathrm c};\pi),
\]
where $\pi$ is a uniform permutation. Since
$A^{\mathrm c}$ and $B^{\mathrm c}$ have zero row sums, $\ee_0\widetilde D_n=0$. Recall from the proof of Lemma~\ref{lem:canonical-fourth-score-cancellation} that \eqref{eq:q-prime-zero} establishes $q_n'(0)=0$, where $q_n(\rho)=\ee_\rho[Q(A^{\mathrm c},B^{\mathrm c};\pi)]$. Since $\ee_\rho\widetilde D_n=q_n(\rho)/(n^2h_1)$, we obtain
\begin{equation}
\label{eq:Dtilde-mean-first-zero-local}
        \left.
        \frac{d}{d\rho}
        \ee_\rho\widetilde D_n
        \right|_{\rho=0}
        =0.
\end{equation}
Using Lemma~\ref{lem:null-component-fourth-moments}, we have
\begin{equation}
\label{eq:WD-zero-local}
        W_n^D(0)
        \le
        \|\widetilde D_n\|_4^2
        =
        O\left(\frac{1}{n^2h_1}\right)
        =
        o\left(\frac{h_1}{n}\right),
\end{equation}
where the last equality is because $nh_1^2\to\infty$, which follows from $nh_1^{2+\eta}\to\infty$.

We next consider the second derivative. Recall \eqref{eq:likelihood-derivative-def} and \eqref{lochd17}. By
\eqref{eq:Dtilde-mean-first-zero-local},
\[
        W_n^{D\,\prime\prime}(0)
        =
        \ee_0[
        \widetilde D_n^2\Gamma_{n,2}].
\]
Additionally recall that
\[
        \Gamma_{n,2}
        =
        2G_{n,2}^{H_2}
        +
        2G_{n,2}^{\rm pair}.
\]
By \eqref{eq:D2-H2-score-bound-centered} and
\eqref{eq:D2-pair-score-bound-centered},
\begin{align}
|W_n^{D\,\prime\prime}(0)|
&\le
2\left|
\ee_0[
        \widetilde D_n^2G_{n,2}^{H_2}]
\right|
+
2\left|
\ee_0[
        \widetilde D_n^2G_{n,2}^{\rm pair}]
\right|
= O(1/n).
\label{eq:WD-second-local}
\end{align}
We next control the fourth derivative on the relevant local
neighbourhood. By
Corollary~\ref{cor:canonical-fourth-derivative-local}, for every
fixed \(C_0>0\),
\begin{equation}
\label{eq:WD-fourth-local-uniform}
        \sup_{|\rho|\le C_0\delta_n}
        |W_n^{D\,{(4)}}(\rho)|
        =
        o(1).
\end{equation}
Choose $C_0>0$ such that $|\rho_n|\le C_0\delta_n$ for all sufficiently large $n$.

Note that we can now avoid any estimate of $W_n^{D\,\prime}(0)$ or
$W_n^{D\,{(3)}}(0)$. Since $W_n^D(\rho)$ is a variance, for any $\rho\in (-1, 1)$
\[
        W_n^D(\rho)\ge 0.
\]
Applying Taylor's formula through order three to $W_n^D(\rho_n)$
and $W_n^D(-\rho_n)$, and adding the two expansions, cancels all odd
terms and gives
\begin{align*}
W_n^D(\rho_n)
&\le
W_n^D(\rho_n)+W_n^D(-\rho_n)\\
&\le
2W_n^D(0)
+
\rho_n^2|W_n^{D\,\prime\prime}(0)|
+
\frac{\rho_n^4}{12}
\sup_{|\rho|\le|\rho_n|}
|W_n^{D\,{(4)}}(\rho)|.
\end{align*}
By \eqref{eq:WD-zero-local},
\eqref{eq:WD-second-local}, and
\eqref{eq:WD-fourth-local-uniform},
\begin{align*}
W_n^D(\rho_n)
&=
o\left(\frac{h_1}{n}\right)
+
O\left(\frac{\delta_n^2}{n}\right)
+
O(\delta_n^4)o(1).
\end{align*}
Now $\delta_n^4= h_1 / n$, and
\[
\frac{\delta_n^2/n}{h_1/n}
=
\frac{\delta_n^2}{h_1}
=
\frac{1}{\sqrt{nh_1}}
\longrightarrow0.
\]
Therefore
\[
        W_n^D(\rho_n)
        =
        o\left(\frac{h_1}{n}\right),
\]
which proves \eqref{eq:Dtilde-local-negligible}.

We next prove \eqref{eq:b-local-negligible}. Let
\[
        \bar a_{\rm d}
        :=
        \frac1n\sum_{i=1}^na_{ii},
        \qquad
        \bar B_{\rm d}
        :=
        \frac1n\sum_{k=1}^nB_{kk}.
\]
Recall from Lemma~\ref{lem:null-canonical-centering-bounds} that
\[
        d_i=a_{ii}-\bar a_{\rm d},
        \qquad
        z_k=B_{kk}-\bar B_{\rm d},
\]
with
\[
        \sum_i d_i=\sum_k z_k=0,
        \qquad
        \|d\|_2=O\{(nh_1)^{1/2}\},
        \qquad
        \|z\|_2=O(n^{1/2}).
\]
For every permutation $\pi$, not only under the null,
\begin{align*}
b_{n0}
=
\frac{1}{n^2h_1}
\sum_{i=1}^n
(\bar a_{\rm d}+d_i)
(\bar B_{\rm d}+z_{\pi_i})
=
\frac{\bar a_{\rm d}\bar B_{\rm d}}{nh_1}
+
\frac{1}{n^2h_1}
\sum_{i=1}^n d_i z_{\pi_i}.
\end{align*}
Consequently, uniformly in the distribution of $\pi$,
\begin{align*}
\var_\rho(b_{n0})
\le
\frac{1}{n^4h_1^2}
\sup_{\pi}
\left|
\sum_{i=1}^n d_i z_{\pi_i}
\right|^2
\le
\frac{1}{n^4h_1^2}
\|d\|_2^2\|z\|_2^2
=
O\left(\frac{1}{n^2h_1}\right)
=
o\left(\frac{h_1}{n}\right),
\end{align*}
where the last equality is because $nh_1^2\to\infty$. This proves
\eqref{eq:b-local-negligible}.

We next prove \eqref{eq:tau-T-local-reduction}. The canonical
decomposition is an algebraic identity and therefore holds for every
permutation
\begin{equation*}
        \hat\tau_n^2
        =
        \widetilde{D}_n
        +
        \frac{2(n-1)}{n-2}T_n
        -
        \frac{n}{n-2}C_n
        +
        b_{n0}.
\end{equation*}
Since $C_n$ is deterministic, subtracting expectations under $\pp_{\rho_n}$ gives
\begin{align*}
\hat\tau_n^2-\ee_{\rho_n}\hat\tau_n^2
-
\frac{2(n-1)}{n-2}
\{T_n-\ee_{\rho_n}T_n\}
=
\widetilde D_n-\ee_{\rho_n}\widetilde D_n
+
b_{n0}-\ee_{\rho_n}b_{n0}.
\end{align*}
Thus, by \eqref{eq:Dtilde-local-negligible} and
\eqref{eq:b-local-negligible},
\begin{align*}
\frac{n}{h_1}
\ee_{\rho_n}
\left(\left[
\hat\tau_n^2-\ee_{\rho_n}\hat\tau_n^2
-
\frac{2(n-1)}{n-2}
\{T_n-\ee_{\rho_n}T_n\}
\right]^2\right)
\le
\frac{2n}{h_1}
\left\{
\var_{\rho_n}(\widetilde D_n)
+
\var_{\rho_n}(b_{n0})
\right\}
\longrightarrow0,
\end{align*}
which proves \eqref{eq:tau-T-local-reduction}.

It remains to prove \eqref{eq:Tn-local-variance}. First consider the null model. By
\eqref{eq:Tn-variance-limit},
\[
        \var_0(T_n)
        =
        \sigma_0^2\frac{h_1}{n}
        +
        o\left(\frac{h_1}{n}\right).
\]
Applying the preceding reduction with \(\rho=0\) therefore yields
\[
        \var_0(\hat\tau_n^2)
        =
        4\sigma_0^2\frac{h_1}{n}
        +
        o\left(\frac{h_1}{n}\right).
\]
On the other hand, Theorem~\ref{thm:variance-stability2} gives
\[
        \var_{\rho_n}(\hat\tau_n^2)
        =
        \var_0(\hat\tau_n^2)
        +
        o\left(\frac{h_1}{n}\right)
        =
        4\sigma_0^2\frac{h_1}{n}
        +
        o\left(\frac{h_1}{n}\right).
\]
By \eqref{eq:tau-T-local-reduction},
\[
\left\|
\hat\tau_n^2-\ee_{\rho_n}\hat\tau_n^2
-
\frac{2(n-1)}{n-2}\{T_n-\ee_{\rho_n}T_n\}
\right\|_2
=
o\left(\sqrt{\frac{h_1}{n}}\right).
\]
The reverse triangle inequality in \(L^2\) therefore gives
\[
\left|
\sqrt{\var_{\rho_n}(\hat\tau_n^2)}
-
\frac{2(n-1)}{n-2}\sqrt{\var_{\rho_n}(T_n)}
\right|
=
o\left(\sqrt{\frac{h_1}{n}}\right).
\]
It follows that
\[
        \var_{\rho_n}(T_n)
        =
        \sigma_0^2\frac{h_1}{n}
        +
        o\left(\frac{h_1}{n}\right),
\]
which proves \eqref{eq:Tn-local-variance} and completes the proof.
\end{proof}

\begin{lmm}
\label{lem:Tn-local-CLT}
Suppose that Assumptions~\ref{asmp:kernel} and~\ref{asmp:band} hold,
that \(K\) is of bounded variation, and that $nh_1^{2 + \eta}\rightarrow\infty$ for some fixed $\eta\in(0, 1)$. Let
\[
        \rho_n
        =
        O\left\{
        \Bigl(\frac{h_1}{n}\Bigr)^{1/4}
        \right\}.
\]
Then, under any bivariate distribution with continuous marginals and
Gaussian copula $C_{\rho_n}$,
\begin{equation*}
        \sqrt{\frac{n}{h_1}}
        \{T_n-\ee_{\rho_n}T_n\}
        \stackrel{\mathcal D}{\rightarrow}
        \mathcal{N}(0,\sigma_0^2).
\end{equation*}
\end{lmm}

\begin{proof}[Proof of Lemma~\ref{lem:Tn-local-CLT}]
Since \(T_n\) is rank measurable, its distribution depends only on
the copula; therefore we assume throughout the proof that $(X_i,Y_i)$, $1\le i\le n$, are i.i.d.\ standard bivariate normal vectors with correlation $\rho_n$.

Recall that
\[
        x_i=L_i-\bar L,
        \qquad
        y_k=\bar b_k-\bar b.
\]
The algebraic identity used in
Lemma~\ref{lem:local-linear-reduction} gives
\begin{equation*}
        \sqrt{\frac{n}{h_1}}(T_n-C_n)
        =
        \frac{1}{n^{3/2}h_1^{3/2}}
        \sum_{i=1}^n x_i y_{\pi_i}.
\end{equation*}
We first record one consequence of the deterministic bounds already
established for $x_i$. Let $I_n$ and $\ell_n$ be as in Lemma~\ref{lem:deterministic-kernel-weight-bounds}. Since $L_i=\ell_n$ for $i\in I_n$, while
\[
        |I_n^c|=O(nh_1),
        \qquad
        \max_i|L_i-\ell_n|=O(nh_1),
\]
we have
\begin{align}
\sum_{i=1}^n|x_i|
&\le
\sum_{i=1}^n|L_i-\ell_n|
+
n|\bar L-\ell_n|
\le
2\sum_{i\notin I_n}|L_i-\ell_n|
=
O(n^2h_1^2).
\label{eq:x-l1-local-CLT}
\end{align}
We also need a Lipschitz interpolation of the $Y$-scores. Let $J_n:[0,1]\to\mathbb R$ be continuous and piecewise linear with
\[
        J_n(k/n)=y_k,
        \qquad 1\le k\le n,
\]
and let \(J_n(t)=y_1\) for \(0\le t\le1/n\).
By \eqref{eq:B-row-centered-bounds},
\[
        \sup_n\sup_{0\le t\le1}|J_n(t)|<\infty.
\]
Moreover, the bound established in the proof of
Lemma~\ref{lem:canonical-discrete-regularity},
\[
        \max_{1\le k<n}|y_{k+1}-y_k|=O(n^{-1}),
\]
shows that there exists \(C<\infty\), independent of \(n\), such that
\begin{equation}
\label{eq:Jn-Lipschitz-local-CLT}
        |J_n(s)-J_n(t)|
        \le
        C|s-t|,
        \qquad s,t\in[0,1].
\end{equation}
Since \(\sum_{k=1}^ny_k=0\), while
\eqref{eq:y-square-asymptotic-Tn} gives
\[
        \frac{1}{n}\sum_{k=1}^ny_k^2
        \rightarrow
        \frac{1}{45},
\]
the Riemann-sum approximation implied by
\eqref{eq:Jn-Lipschitz-local-CLT} yields, for
\(U\sim{\rm Unif}[0,1]\),
\begin{equation}
\label{eq:Jn-uniform-variance-local}
        \var(J_n(U))
        \rightarrow
        \frac{1}{45}.
\end{equation}
We now condition on the $X$-sample. Order the observations according
to their $X$-coordinates,
\[
    X_{1:n}<\cdots<X_{n:n},
\]
and denote by $Y_{[i]}$ the observation associated with $X_{i:n}$.
Let
\[
    \mathcal X_n:=\sigma(X_1,\ldots,X_n).
\]
Conditional on $\mathcal X_n$, the variables
$Y_{[1]},\ldots,Y_{[n]}$ are independent and
\[
        Y_{[i]}\mid\mathcal X_n
        \sim
        N\left(
        \rho_nX_{i:n},
        1-\rho_n^2
        \right).
\]
Let $\pi_i$ be the rank of $Y_{[i]}$,
\[
        \pi_i
        =
        1+\sum_{j\ne i}
        \bone\{Y_{[j]}\le Y_{[i]}\}.
\]
It follows that
\begin{align}
\ee_{\rho_n}
\left[
\frac{\pi_i}{n}
\,\middle|\,
\mathcal X_n,Y_{[i]}
\right]
=
\frac{1}{n}
\left[
1+
\sum_{j\ne i}
\Phi\left(
\frac{Y_{[i]}-\rho_nX_{j:n}}
{\sqrt{1-\rho_n^2}}
\right)
\right],
\label{eq:rank-cond-mean-local-CLT}
\end{align}
and
\begin{equation}
\label{eq:rank-cond-var-local-CLT}
\ee_{\rho_n}
\left[
\left(
\frac{\pi_i}{n}
-
\ee_{\rho_n}
\left[
\frac{\pi_i}{n}
\,\middle|\,
\mathcal X_n,Y_{[i]}
\right]
\right)^2
\,\middle|\,
\mathcal X_n,Y_{[i]}
\right]
\le
\frac{1}{4n}.
\end{equation}
With these preliminary observations in place, we are ready to
introduce the surrogate $\widetilde{T}_n$ for the normalized linear
component
\[
        \sqrt{\frac{n}{h_1}}(T_n-C_n) = \frac{1}{n^{3/2}h_1^{3/2}}
        \sum_{i=1}^n x_i y_{\pi_i} = \frac{1}{n^{3/2}h_1^{3/2}}
        \sum_{i=1}^n x_i J_n(\pi_i/n).
\]
For this, we approximate the rank score $J_n(\pi_i/n)=y_{\pi_i}$ by replacing
the normalized rank by its conditional mean. Define
\begin{equation*}
\widetilde T_n
:=
\frac{1}{n^{3/2}h_1^{3/2}}
\sum_{i=1}^n
x_i
J_n\left(
\ee_{\rho_n}
\left[
\frac{\pi_i}{n}
\,\middle|\,
\mathcal X_n,Y_{[i]}
\right]
\right).
\end{equation*}
By \eqref{eq:Jn-Lipschitz-local-CLT} and
\eqref{eq:rank-cond-var-local-CLT},
\begin{align*}
&\ee_{\rho_n}
\left[
\left|
y_{\pi_i}
-
J_n\left(
\ee_{\rho_n}
\left[
\frac{\pi_i}{n}
\,\middle|\,
\mathcal X_n,Y_{[i]}
\right]
\right)
\right|^2
\,\middle|\,
\mathcal X_n
\right]
\le
\frac{C}{n}.
\end{align*}
Therefore, by Minkowski's inequality and
\eqref{eq:x-l1-local-CLT},
\begin{align*}
\left\{
\ee_{\rho_n}
\left[
\left|
\sqrt{\frac{n}{h_1}}(T_n-C_n)
-
\widetilde T_n
\right|^2
\,\middle|\,
\mathcal X_n
\right]
\right\}^{1/2}
\le
\frac{C}{n^2h_1^{3/2}}
\sum_{i=1}^n|x_i|
=
O(h_1^{1/2}).
\end{align*}
Since the bound is deterministic and $h_1\to0$,
\begin{equation}
\label{eq:T-oracle-local-L2-unconditional}
\ee_{\rho_n}
\left[
\left|
\sqrt{\frac{n}{h_1}}(T_n-C_n)
-
\widetilde T_n
\right|^2
\right]
\rightarrow 0,
\end{equation}
therefore $\widetilde T_n$ is indeed a valid surrogate for $\sqrt{n/h_1}(T_n-C_n)$. 

We next identify the conditional variance of \(\widetilde T_n\). For \(|\rho|\le1/2\), we have
\begin{align}\label{eq:normal-cdf-perturbation-local}
&\sup_{y\in\rr}\left|
\Phi\left(
\frac{y-\rho x}{\sqrt{1-\rho^2}}
\right)
-\Phi(y)
\right|\nonumber\\
&\qquad\leq 
\sup_{y\in\rr}
\left|
\Phi\left(
\frac{y-\rho x}{\sqrt{1-\rho^2}}
\right)
-
\Phi\left(
\frac{y}{\sqrt{1-\rho^2}}
\right)
\right| 
+ 
\sup_{y\in\mathbb R}
\left|
\Phi\left(
\frac{y}{\sqrt{1-\rho^2}}
\right)
-
\Phi\left(y\right)
\right| \nonumber\\
&\qquad\leq 
C|\rho x| + C\left|
\frac{1}{\sqrt{1-\rho^2}}-1
\right| \nonumber\\
&\qquad\leq 
C(|\rho x| + \rho^2),
\end{align}
where we have used the fact that the standard normal density \(\varphi\) is uniformly bounded and
\(\sqrt{1-\rho^2}\) is bounded away from zero for
\(|\rho|\le1/2\) and $\sup_{z\in\mathbb R}|z|\varphi(z)<\infty$.

Since
\[
        |\rho_n|
        =
        O\left\{
        \Bigl(\frac{h_1}{n}\Bigr)^{1/4}
        \right\}
        =
        O(n^{-1/4})
\]
and
\[
        \max_{1\le i\le n}|X_i|
        =
        O_{\pp_{\rho_n}}(\sqrt{\log n}),
\]
\eqref{eq:normal-cdf-perturbation-local} gives
\begin{equation*}
\max_{1\le i\le n}
\sup_{y\in\mathbb R}
\left|
\Phi\left(
\frac{y-\rho_nX_{i:n}}
{\sqrt{1-\rho_n^2}}
\right)
-\Phi(y)
\right|
\stackrel{\pp_{\rho_n}}{\rightarrow}0.
\end{equation*}
By the triangle inequality, it follows from \eqref{eq:rank-cond-mean-local-CLT} that, uniformly in $i$,
\begin{align}
&\left|
\ee_{\rho_n}
\left[
\frac{\pi_i}{n}
\,\middle|\,
\mathcal X_n,Y_{[i]}
\right]
-
\Phi\left(
\frac{Y_{[i]}-\rho_nX_{i:n}}
{\sqrt{1-\rho_n^2}}
\right)
\right| \nonumber\\
&\qquad\le \frac{1}{n}\left|
1
-
\Phi\left(
\frac{Y_{[i]}-\rho_nX_{i:n}}
{\sqrt{1-\rho_n^2}}
\right)
\right| + \frac{1}{n}\sum_{j \neq i}\left|
\Phi\left(
\frac{Y_{[i]}-\rho_nX_{j:n}}
{\sqrt{1-\rho_n^2}}
\right)
-
\Phi\left(
\frac{Y_{[i]}-\rho_nX_{i:n}}
{\sqrt{1-\rho_n^2}}
\right)
\right| \nonumber\\
&\qquad\le \frac{1}{n} 
+
\left(\frac{n - 1}{n}\right)\left|
\Phi\left(
\frac{Y_{[i]}-\rho_nX_{i:n}}
{\sqrt{1-\rho_n^2}}
\right)
-
\Phi\left(Y_{[i]}\right)
\right|
+ 
\frac{1}{n}\sum_{j \neq i}\left|
\Phi\left(
\frac{Y_{[i]}-\rho_nX_{j:n}}
{\sqrt{1-\rho_n^2}}
\right)
-
\Phi\left(Y_{[i]}\right)
\right|  \nonumber\\
&\qquad\le
\frac{1}{n}
+
2
\max_{1\le j\le n}
\sup_{y\in\rr}
\left|
\Phi\left(
\frac{y-\rho_nX_{j:n}}
{\sqrt{1-\rho_n^2}}
\right)
-\Phi(y)
\right|
\stackrel{\pp_{\rho_n}}{\longrightarrow}0.
\label{eq:rank-mean-PIT-local}
\end{align}
Conditional on \(\mathcal X_n\),
\[
        \Phi\left(
        \frac{Y_{[i]}-\rho_nX_{i:n}}
        {\sqrt{1-\rho_n^2}}
        \right)
        \sim{\rm Unif}(0,1).
\]
Thus, by the uniform boundedness and Lipschitz property of \(J_n\),
\eqref{eq:rank-mean-PIT-local} and
\eqref{eq:Jn-uniform-variance-local} imply
\begin{align}
\max_{1\le i\le n}
\bigg|
&\var_{\rho_n}
\left[
J_n\left(
\ee_{\rho_n}
\left[
\frac{\pi_i}{n}
\,\middle|\,
\mathcal X_n,Y_{[i]}
\right]
\right)
\,\middle|\,
\mathcal X_n
\right]
-\frac{1}{45}
\bigg|
\stackrel{\pp_{\rho_n}}{\longrightarrow}0.
\label{eq:individual-cond-var-local}
\end{align}

Conditional on \(\mathcal X_n\), the summands defining
\(\widetilde T_n\) are independent, since the \(Y_{[i]}\)'s are
conditionally independent. Hence
\begin{align}
\var_{\rho_n}(
\widetilde T_n\mid\mathcal X_n)
&=
\frac{1}{n^3h_1^3}
\sum_{i=1}^n
x_i^2
\var_{\rho_n}
\left[
J_n\left(
\ee_{\rho_n}
\left[
\frac{\pi_i}{n}
\,\middle|\,
\mathcal X_n,Y_{[i]}
\right]
\right)
\,\middle|\,
\mathcal X_n
\right].
\label{eq:Ttilde-cond-var-local}
\end{align}
By \eqref{eq:x-square-asymptotic-Tn},
\[
        \frac{1}{n^3h_1^3}
        \sum_{i=1}^n x_i^2
        \longrightarrow
        45\sigma_0^2.
\]
Combining this with \eqref{eq:individual-cond-var-local} gives
\begin{equation}
\label{eq:Ttilde-cond-var-limit-local}
        \var_{\rho_n}(
        \widetilde T_n\mid\mathcal X_n)
        \stackrel{\pp_{\rho_n}}{\longrightarrow}
        \sigma_0^2.
\end{equation}
We now apply the Lindeberg--Feller theorem conditionally on
\(\mathcal X_n\). The centred summands of \(\widetilde T_n\) are
conditionally independent, and by \eqref{eq:L-centered-bounds} and
the uniform boundedness of \(J_n\), their maximum absolute value is
bounded by
\[
        C\frac{\max_i|x_i|}
        {n^{3/2}h_1^{3/2}}
        =
        O\left(\frac{1}{\sqrt{nh_1}}\right)
        =
        o(1).
\]
Thus the conditional Lindeberg condition holds. Together with
\eqref{eq:Ttilde-cond-var-limit-local}, this gives the conditional
central limit theorem.

For completeness, we make the passage from the conditional to the
unconditional limit explicit. For \(t\in\mathbb R\), let
\[
\varphi_n(t)
:=
\ee_{\rho_n}
\left[
\exp\left\{
it\left(
\widetilde T_n
-
\ee_{\rho_n}[
\widetilde T_n\mid\mathcal X_n]
\right)
\right\}
\,\middle|\,
\mathcal X_n
\right].
\]
We claim that
\begin{equation}
\label{eq:Ttilde-cond-cf-local}
\varphi_n(t)
\stackrel{\pp_{\rho_n}}{\longrightarrow}
\exp\left(-\frac{t^2\sigma_0^2}{2}\right).
\end{equation}
Indeed, consider an arbitrary subsequence. By
\eqref{eq:Ttilde-cond-var-limit-local}, it contains a further
subsequence along which
\[
        \var_{\rho_n}(
        \widetilde T_n\mid\mathcal X_n)
        \longrightarrow
        \sigma_0^2
\]
almost surely. Along this further subsequence, the conditional
Lindeberg--Feller theorem yields, for almost every realization of
$\mathcal X_n$,
\[
        \varphi_n(t)
        \longrightarrow
        \exp\left(-\frac{t^2\sigma_0^2}{2}\right).
\]
Since every subsequence has such a further subsequence,
\eqref{eq:Ttilde-cond-cf-local} follows. Moreover, $|\varphi_n(t)|\le1$. Hence
\eqref{eq:Ttilde-cond-cf-local} also implies convergence in $L^1$,
and therefore
\begin{align*}
\ee_{\rho_n}
\left[
\exp\left\{
it\left(
\widetilde T_n
-
\ee_{\rho_n}[
\widetilde T_n\mid\mathcal X_n]
\right)
\right\}
\right]
=
\ee_{\rho_n}[\varphi_n(t)]
\longrightarrow
\exp\left(-\frac{t^2\sigma_0^2}{2}\right).
\end{align*}
By L\'evy's continuity theorem,
\begin{equation}
\label{eq:Ttilde-cond-CLT-local}
\widetilde T_n - \ee_{\rho_n}[
\widetilde T_n\mid\mathcal X_n]
\stackrel{\mathcal D}{\longrightarrow}
\mathcal N(0,\sigma_0^2).
\end{equation}
It remains to replace the random conditional centring by the
unconditional mean. By Lemma~\ref{lem:local-linear-reduction},
\[
\var_{\rho_n}
\left\{
\sqrt{\frac{n}{h_1}}(T_n-C_n)
\right\}
\rightarrow
\sigma_0^2,
\]
and hence by
\eqref{eq:T-oracle-local-L2-unconditional}, 
\begin{equation}
\label{eq:Ttilde-unconditional-var-local}
        \var_{\rho_n}(\widetilde T_n)
        \longrightarrow
        \sigma_0^2.
\end{equation}
The conditional variances in
\eqref{eq:Ttilde-cond-var-local} are uniformly bounded, because
\(J_n\) is uniformly bounded and
\[
        \frac{1}{n^3h_1^3}
        \sum_{i=1}^n x_i^2
        =
        O(1).
\]
Hence \eqref{eq:Ttilde-cond-var-limit-local} also yields
\[
\ee_{\rho_n}
\left[
\var_{\rho_n}(
\widetilde T_n\mid\mathcal X_n)
\right]
\rightarrow
\sigma_0^2.
\]
By the law of total variance and
\eqref{eq:Ttilde-unconditional-var-local},
\[
\var_{\rho_n}
\left(
\ee_{\rho_n}[
\widetilde T_n\mid\mathcal X_n]
\right)
\rightarrow0.
\]
Therefore
\begin{equation}
\label{eq:Ttilde-cond-mean-local}
\ee_{\rho_n}[
\widetilde T_n\mid\mathcal X_n]
-
\ee_{\rho_n}\widetilde T_n
\stackrel{L^2(\pp_{\rho_n})}{\longrightarrow}0.
\end{equation}
Combining
\eqref{eq:Ttilde-cond-CLT-local} and
\eqref{eq:Ttilde-cond-mean-local} with Slutsky's theorem gives
\[
        \widetilde T_n-\ee_{\rho_n}\widetilde T_n
        \stackrel{\mathcal D}{\rightarrow} \mathcal{N}(0,\sigma_0^2).
\]

Finally, \eqref{eq:T-oracle-local-L2-unconditional} implies
\[
\begin{aligned}
&\ee_{\rho_n}
\biggl[
\bigg|
\sqrt{\frac{n}{h_1}}
\{T_n-\ee_{\rho_n}T_n\}
-
\{\widetilde T_n-\ee_{\rho_n}\widetilde T_n\}
\bigg|^2
\biggr]
\rightarrow0.
\end{aligned}
\]
Another application of Slutsky's theorem therefore yields
\[
        \sqrt{\frac{n}{h_1}}
        \{T_n-\ee_{\rho_n}T_n\}
        \stackrel{\mathcal D}{\rightarrow}
         \mathcal{N}(0,\sigma_0^2),
\]
as claimed.
\end{proof}

\subsection{Proof of Proposition~\ref{prop:Gaussian-superlocal-power}}\label{subsec:superlocal-power}
\begin{proof}
Since all statistics under consideration are rank measurable, their
distributions depend only on the copula. Thus it suffices to work with
standard bivariate normal observations having correlation $\rho_n$.

Set
\[
        \delta_n
        :=
        \Bigl(\frac{h_1}{n}\Bigr)^{1/4}.
\]
Then \eqref{eq:superlocal-rho-range} is equivalent to
\[
        \delta_n=o(|\rho_n|),
        \qquad
        |\rho_n|=o(h_1^{1/2}).
\]
Notice that $nh_1^{2+\eta}\to\infty$ and $h_1\to0$ imply $nh_1^2\to\infty$ and $nh_1\to\infty$.

We first prove the mean expansion. Recall the decomposition
\[
\ee_\rho[\hat\tau_n^2]
=
T_{n,\rho}^{\mathrm{off}}
+
T_{n,\rho}^{\mathrm{diag}},
\]
where
\[
T_{n,\rho}^{\mathrm{off}}
=
T_{n0}^{\mathrm{off}}
+
\rho T_{n1}^{\mathrm{off}}
+
\rho^2T_{n2}^{\mathrm{off}}
+
\rho^3\mathcal R_n^{\mathrm{off}}(\rho).
\]

For the diagonal term, Lemma~\ref{lmm:Tdiag} gives, uniformly for
\(\rho\) in a fixed neighbourhood of zero,
\[
T_{n,\rho}^{\mathrm{diag}}
-
T_{n,0}^{\mathrm{diag}}
=
O\left(
        \frac{|\rho|}{n^2h_1^{3/2}}
\right)
+
O\left(
        \frac{\rho^2}{nh_1}
\right).
\]
Consequently,
\begin{align*}
\frac{
|T_{n,\rho_n}^{\mathrm{diag}}
-
T_{n,0}^{\mathrm{diag}}|
}{
\rho_n^2
}
&=
O\left(
        \frac{1}{n^2h_1^{3/2}|\rho_n|}
\right)
+
O\left(\frac{1}{nh_1}\right)
\\
&=
O\left[
        \frac{\delta_n}{|\rho_n|}
        \frac{1}{(nh_1)^{7/4}}
\right]
+
O\left(\frac{1}{nh_1}\right)
\longrightarrow0.
\end{align*}
Hence
\begin{equation}
\label{eq:superlocal-diagonal}
T_{n,\rho_n}^{\mathrm{diag}}
-
T_{n,0}^{\mathrm{diag}}
=
o(\rho_n^2).
\end{equation}

By Lemma~\ref{lmm:Toff1},
\[
        T_{n1}^{\mathrm{off}}
        =
        O\left(\frac{1}{nh_1}\right).
\]
Therefore
\begin{align*}
\frac{
|\rho_nT_{n1}^{\mathrm{off}}|
}{
\rho_n^2
}
=
O\left(\frac{1}{nh_1|\rho_n|}\right)
=
O\left[
        \frac{\delta_n}{|\rho_n|}
        \frac{1}{nh_1\delta_n}
\right].
\end{align*}
Now
\[
        nh_1\delta_n
        =
        n^{3/4}h_1^{5/4}
        =
        (nh_1^2)^{5/8}n^{1/8}
        \longrightarrow\infty.
\]
Since \(\delta_n/|\rho_n|\to0\), it follows that
\begin{equation*}
        \rho_nT_{n1}^{\mathrm{off}}
        =
        o(\rho_n^2).
\end{equation*}

By Lemma~\ref{lmm:Toff2},
\begin{equation*}
        \rho_n^2T_{n2}^{\mathrm{off}}
        =
        \frac{\rho_n^2}{2\pi\sqrt3}
        +
        o(\rho_n^2).
\end{equation*}

Finally, Lemma~\ref{lmm:aggregate-third-order-hermite} gives,
uniformly for \(|\rho|<1\),
\[
        \left|
        \rho\mathcal R_n^{\mathrm{off}}(\rho)
        \right|
        \le
        C\frac{|\rho|}{h_1^{1/2}}.
\]
Therefore
\begin{align*}
\frac{
|\rho_n^3\mathcal R_n^{\mathrm{off}}(\rho_n)|
}{
\rho_n^2
}
=
\left|
\rho_n\mathcal R_n^{\mathrm{off}}(\rho_n)
\right|
\le
C\frac{|\rho_n|}{h_1^{1/2}}
\longrightarrow0
\end{align*}
by \eqref{eq:superlocal-rho-range}. Thus
\begin{equation}
\label{eq:superlocal-Hermite-remainder}
        \rho_n^3
        \mathcal R_n^{\mathrm{off}}(\rho_n)
        =
        o(\rho_n^2).
\end{equation}

Combining
\eqref{eq:superlocal-diagonal}--
\eqref{eq:superlocal-Hermite-remainder} gives
\[
\ee_{\rho_n}[\hat\tau_n^2]
-
\ee_0[\hat\tau_n^2]
=
\frac{\rho_n^2}{2\pi\sqrt3}
+
o(\rho_n^2).
\]
Since
\[
        \hat r_n=6\hat\tau_n^2-2,
        \qquad
        b_n=\ee_0[\hat r_n],
\]
we conclude that
\[
        \ee_{\rho_n}[\hat r_n]-b_n
        =
        \frac{\sqrt3}{\pi}\rho_n^2
        +
        o(\rho_n^2),
\]
which proves \eqref{eq:superlocal-mean}.

We next establish the variance estimate needed for consistency. Set
\[
        V_n(\rho)
        :=
        \var_\rho(\hat\tau_n^2).
\]
We shall prove
\begin{equation}
\label{eq:superlocal-variance-tau-new}
        V_n(\rho_n)=o(\rho_n^4).
\end{equation}

First, by the null variance bound,
\[
        V_n(0)
        =
        O\left(\frac{h_1}{n}\right)
        =
        O(\delta_n^4).
\]
Thus
\begin{equation}
\label{eq:superlocal-V0-new}
        \frac{V_n(0)}{\rho_n^4}
        =
        O\left\{
        \left(\frac{\delta_n}{|\rho_n|}\right)^4
        \right\}
        \longrightarrow0.
\end{equation}

Also, Lemma~\ref{lem:Vsecond-bound-centered} gives
\[
        V_n''(0)=O(1/n).
\]
Hence
\begin{align}
\frac{
\rho_n^2|V_n''(0)|
}{
\rho_n^4
}
=
O\left(\frac{1}{n\rho_n^2}\right)
=
O\left[
        \left(\frac{\delta_n}{|\rho_n|}\right)^2
        \frac{1}{n\delta_n^2}
\right]
\longrightarrow0,
\label{eq:superlocal-V2-new}
\end{align}
because
\[
        n\delta_n^2
        =
        (nh_1)^{1/2}
        \longrightarrow\infty.
\]
It remains to control the fourth derivative on the enlarged
super-local interval. By Lemmas~\ref{lem:Vfourth-zero} and
\ref{lem:Vfourth-uniform}, for all sufficiently large \(n\),
\begin{align}
\sup_{|\rho|\le|\rho_n|}
|V_n^{(4)}(\rho)|
\le
|V_n^{(4)}(0)|
+
C|\rho_n|
+
C\frac{\rho_n^2}{h_1}
=
o(1).
\label{eq:superlocal-V4-uniform-new}
\end{align}
Indeed,
\[
        V_n^{(4)}(0)=o(1),
        \qquad
        \rho_n\to0,
        \qquad
        \frac{\rho_n^2}{h_1}\to0,
\]
where the last relation follows from
\(|\rho_n|=o(h_1^{1/2})\).

We now use the nonnegativity of the variance to avoid the first and third variance derivatives altogether. Since $V_n(\rho_n)\ge0$ and $V_n(-\rho_n)\ge0$,
\[
        V_n(\rho_n)
        \le
        V_n(\rho_n)+V_n(-\rho_n).
\]
Taylor's formula through order three, applied at $\rho_n$ and $-\rho_n$, gives
\begin{align*}
V_n(\rho_n) + V_n(-\rho_n)
&=
2V_n(0)
+
\rho_n^2V_n''(0)
+
R_{n,+}+R_{n,-},
\end{align*}
where
\[
        |R_{n,+}|+|R_{n,-}|
        \le
        \frac{\rho_n^4}{12}
        \sup_{|\rho|\le |\rho_n|}
        |V_n^{(4)}(\rho)|.
\]
Consequently,
\begin{align*}
0\le V_n(\rho_n)
&\le
2V_n(0)
+
\rho_n^2|V_n''(0)|
+
\frac{\rho_n^4}{12}
\sup_{|\rho|\le|\rho_n|}
|V_n^{(4)}(\rho)|.
\end{align*}
Dividing by \(\rho_n^4\), and using
\eqref{eq:superlocal-V0-new},
\eqref{eq:superlocal-V2-new}, and
\eqref{eq:superlocal-V4-uniform-new}, yields $V_n(\rho_n)/\rho_n^4 \rightarrow0$.
This proves \eqref{eq:superlocal-variance-tau-new}. Since
\(\hat r_n=6\hat\tau_n^2-2\),
\begin{equation}
\label{eq:superlocal-variance}
        \var_{\rho_n}(\hat r_n)
        =
        36V_n(\rho_n)
        =
        o(\rho_n^4).
\end{equation}

It remains to prove consistency. The rejection rule \eqref{test} is
equivalent to
\[
        \hat r_n-b_n
        >
        12\sigma_0u_{1-\alpha}
        \sqrt{\frac{h_1}{n}}
        =
        12\sigma_0u_{1-\alpha}\delta_n^2.
\]
By \eqref{eq:superlocal-mean},
\[
        \ee_{\rho_n}[\hat r_n]-b_n
        =
        \frac{\sqrt3}{\pi}\rho_n^2\{1+o(1)\}.
\]
Since $\delta_n = o(|\rho_n|)$, there exists $c_0>0$ such that, for all sufficiently large $n$,
\[
\ee_{\rho_n}[\hat r_n]-b_n
-
12\sigma_0u_{1-\alpha}\delta_n^2
\ge
c_0\rho_n^2.
\]
By Chebyshev's inequality,
\begin{align*}
\pp_{\rho_n}
\left\{
\sqrt{\frac{n}{h_1}}
(\hat r_n-b_n)
\le
12\sigma_0u_{1-\alpha}
\right\}
\le
\pp_{\rho_n}
\left\{
\left|
\hat r_n-\ee_{\rho_n}\hat r_n
\right|
\ge
c_0\rho_n^2
\right\}
\le
\frac{
\var_{\rho_n}(\hat r_n)
}{
c_0^2\rho_n^4
}
\longrightarrow0,
\end{align*}
where the last convergence follows from
\eqref{eq:superlocal-variance}. Therefore
\[
\pp_{\rho_n}
\left\{
        \sqrt{\frac{n}{h_1}}
        (\hat r_n-b_n)
        >
        12\sigma_0u_{1-\alpha}
\right\}
\longrightarrow1,
\]
which proves \eqref{eq:superlocal-power-one}.
\end{proof}

\section{Corrected variance}\label{app:corrected-variance} 
In this section we provide the correct formula for the asymptotic variance in the central limit theorem in \citet{dette2013copula}. It was shown in \citet{shidrtonhan2021a} that a revised version of Equations (24)--(26) in \citet{dette2013copula} is
\begin{align*}
    n^{1 / 2}\left(\hat{r}_n - r\right) = \frac{12}{n^{1 / 2}} \sum_{i=1}^n\left(Z_i- \ee [Z_i]\right) + o_p(1), 
\end{align*}
where $Z_i = Z_{i1} - Z_{i2} - Z_{i3}$ with
\begin{align*}
& Z_{i1} = \int_0^1 \bone\{F_{2}(Y_i) \leq v\} \tau(F_{1}(X_i), v)\,dv, \\
& Z_{i2} = \int_{[0, 1]^2} \bone\{F_{1}(X_i) \leq u\} \tau(u, v) \partial_1\tau(u, v)\,du\,dv, \\
& Z_{i3} = \int_{[0, 1]^2} \bone\{F_{2}(Y_i) \leq v\} \tau(u, v) \partial_2 \tau(u, v)\,du\,dv.
\end{align*}
Note that 
\begin{align}
\nonumber
    \var(Z_{i1}) &= \int_{[0,1]^3} \tau(s, v \wedge w) \tau(s, v) \tau(s, w)\,ds\,dv\,dw-\Bigl(\int_{[0,1]^2} \tau^2(u, v)\,du\,dv\Bigr)^2, 
    \\
    \nonumber
    \var(Z_{i2}) &= \frac{1}{4}\Bigl[\int_{[0, 1]^3} \tau^2(u, v)\tau^2(u, w)\,dv\,dw\,du -\Bigl(\int_{[0, 1]^2} \tau(u, v)^2\,du\,dv\Bigr)^2\Bigr],
    \\
    \nonumber 
    \var(Z_{i3}) &= \frac{1}{4}\Bigl[\int_{[0, 1]^3} \tau^2(u, v)\tau^2(w, v)\,du\,dv\,dw -\Bigl(\int_{[0, 1]^2} \tau(u, v)^2\,du\,dv\Bigr)^2\Bigr],
    \\
    \nonumber
    2\cov(Z_{i1}, Z_{i2}) &= -\int_{[0, 1]^3}\tau^2(u, v)\tau^2(u, w)\,du\,dv\,dw + \Bigl(\int_{[0, 1]^2} \tau(u, v)^2\,du\,dv\Bigr)^2,
    \\
    2\cov(Z_{i1}, Z_{i3}) &= -\int_{[0,1]^4} \tau^2(u, y_1) \tau(x_1, w) \partial_2 \tau(x_1, y_1) \bone\{y_1 \leq w\}\,dx_1\,dy_1\,du\,dw +\nonumber\\
    &\qquad \Bigl(\int_{[0, 1]^2} \tau(u, v)^2\,du\,dv\Bigr)^2,\label{eq:covar13}\\
    2\cov(Z_{i2}, Z_{i3}) &= \frac{1}{2}\Bigl[\int_{[0, 1]^4}\tau^2(u, v^\prime)\tau^2(u^\prime, v)\partial_2\tau(u, v)\,du^\prime\,dv^\prime\,du\,dv - \Bigl(\int_{[0, 1]^2} \tau(u, v)^2\,du\,dv\Bigr)^2\Bigr]. \label{eq:covar23}
\end{align}
Putting these together we have 
\begin{align*}
    \var(Z_{i1} - Z_{i2} - Z_{i3}) &= \int_{[0,1]^3} \tau(s, v \wedge w) \tau(s, v) \tau(s, w)\,ds\,dv\,dw \\
    & +\frac{5}{4} \int_{[0,1]^3} \tau^2(u, v) \tau^2(u, w)\,du\,dv\,dw \\
    & +\frac{1}{4} \int_{[0,1]^3} \tau^2(u, v) \tau^2(w, v)\,du\,dv\,dw \\
    & +\int_{[0,1]^4} \tau^2(u, y_1) \tau(x_1, w) \partial_2 \tau(x_1, y_1) \bone\{y_1 \leq w\}\,dx_1\,dy_1\,du\,dw \\
    & +\frac{1}{2} \int_{[0,1]^4} \tau^2(u, v^{\prime}) \tau^2(u^{\prime}, v) \partial_2 \tau(u, v)\,du^{\prime}\,dv^{\prime}\,du\,dv \\
    & -4\Bigl(\int_{[0,1]^2} \tau^2(u, v)\,du\,dv\Bigr)^2 .
\end{align*}
Note that the variance $\sigma^2 = \var(Z_{i1} - Z_{i2} - Z_{i3})$ derived above differs from the corresponding expression in~\citet{dette2013copula}. This discrepancy arises because the quantities corresponding to \eqref{eq:covar13} and \eqref{eq:covar23} were not computed correctly in~\citet{dette2013copula}. In addition, a sign error in that paper, detected by~\citet{shidrtonhan2021a}, led to the use of $\var(Z_{i1} + Z_{i2} + Z_{i3})$ in place of the correct $\var(Z_{i1} - Z_{i2} - Z_{i3})$.
Moreover, a simple calculation shows that in the case of independence, that is $C(u,v)=uv$, we have $\tau(u,v)=v$ and hence $\sigma^2=0$.

\printbibliography[title={References for the Appendix}]
\end{document}